\documentclass[12pt]{amsart}

\usepackage[shortlabels]{enumitem}

\usepackage[colorlinks,linkcolor={blue},
citecolor={blue},urlcolor={red},]{hyperref}

\usepackage{}
\usepackage{amsfonts}
\usepackage{amsmath, amsrefs, amsthm, amssymb}
\usepackage[active]{srcltx}		
\usepackage{pdfsync}					
\usepackage{graphicx}
\usepackage{amsthm}
\usepackage{graphicx}
\usepackage{framed}
\usepackage{multicol,multienum}
\usepackage{graphicx} 
\usepackage{booktabs} 
\usepackage{mathrsfs}
\usepackage{tcolorbox}
\usepackage{bm}
\usepackage{subfigure}
\usepackage{epstopdf}

\renewcommand\div{{\rm div}}

\newcommand{\wuhao}{\fontsize{10.5pt}{\baselineskip}\selectfont}

\newtheorem{theorem}{Theorem}[section]

\newtheorem{lemma}[theorem]{Lemma}

\theoremstyle{definition}
\newtheorem{definition}[theorem]{Definition}

\newtheorem{remark}[theorem]{Remark}

\numberwithin{equation}{section}

\begin{document}

\title [Martingale solutions for 3D SCNS system]{Global martingale weak solutions for the three-dimensional stochastic chemotaxis-Navier-Stokes system with L\'{e}vy processes}

\author{Lei Zhang}
\address{School of Mathematics and Statistics, Hubei Key Laboratory of Engineering Modeling  and Scientific Computing, Huazhong University of Science and Technology,  Wuhan 430074, Hubei, P.R. China.}
\email{lei\_zhang@hust.edu.cn (L. Zhang)}

\author{Bin Liu}
\address{School of Mathematics and Statistics, Hubei Key Laboratory of Engineering Modeling  and Scientific Computing, Huazhong University of Science and Technology,  Wuhan 430074, Hubei, P.R. China.}
\email{binliu@mail.hust.edu.cn (B. Liu)}


\date{\today}

\begin{abstract}
This paper considers the three-dimensional stochastic chemotaxis-Navier-Stokes (SCNS) system subjected to a L\'{e}vy-type random external force in a bounded domain. Until recently, the existed results concerning global solvability of SCNS system mainly concentrated on the case of two spatial dimensions, little is known about the SCNS system in dimension three. We prove in the present work that the initial-boundary value problem for the three-dimensional SCNS system has at least one global martingale solution under proper assumptions, which is weak in both the analytical and stochastic sense. A new stochastic analogue of entropy-energy inequality and an uniform boundedness estimate are derived, which enable us to construct global-in-time approximate solutions from a suitably regularized SCNS system via the Contraction Mapping Principle. The proof of the existence of martingale solution is based on a stochastic compactness method and an elaborate identification of the limits procedure, where the Jakubowski-Skorokhod Theorem is applied to deal with the phase spaces equipped with weak topology.
\end{abstract}

\maketitle
\section{Introduction}
\subsection{Formulation of problem}
The interaction of bacterial populations with a surrounding fluid in which the chemical substances is consumed has already been recognized by several authors \cite{dombrowski2004self,tuval2005,fujikawa1989fractal}. This experiment intensively revealed complex facets of the spatio-temporal behavior in colonies of the aerobic species \emph{Bacillus subtilis} when suspended in sessile water drops. On one hand, both the density of bacteria and the evolution of chemical substrates   change over time in response to the flow of liquid environment; on the other hand, the motion of fluid is affected by certain external body force, which can be produced by various physical mechanism such as gravity, centrifugal, electric or magnetic forces as well as some uncertainties. The chemotaxis-Navier-Stokes (CNS) system rigorously derived by Tuval et al. \cite{tuval2005} has accurately characterized such a mutual chemotaxis \& fluid interaction. This paper is dedicated to the study of a three-dimensional stochastic chemotaxis-Navier-Stokes (SCNS) system affected by a L\'{e}vy-type random external force:
\begin{equation}\label{SCNS}
\left\{
\begin{aligned}
&\mathrm{d} n +u\cdot \nabla n \mathrm{d}t = D_n\Delta n \mathrm{d}t- \div\left(n\chi(c)\nabla c\right)\mathrm{d}t,   &  \textrm{in} ~\mathbb{R}^+\times \mathcal {O},\\
& \mathrm{d} c+ u\cdot \nabla c  \mathrm{d}t = D_c\Delta c\mathrm{d}t-nf(c)\mathrm{d}t,  & \textrm{in} ~~\mathbb{R}^+\times \mathcal {O},\\
& \mathrm{d} u+ \left((u\cdot \nabla) u+\nabla P\right)\mathrm{d}t \\
&\quad = \delta\Delta u\mathrm{d}t+ n\nabla \Phi\mathrm{d}t+ h(t,u) \mathrm{d} t+ g(t,u)  \mathrm{d}W_t\\
& \quad + \int_{Z_0} K(u(t-),z)\widetilde{\pi}(\mathrm{d}t,\mathrm{d}z) +  \int_{Z\backslash Z_0} G(u(t-),z) \pi (\mathrm{d}t,\mathrm{d}z), &\textrm{in} ~~\mathbb{R}^+\times \mathcal {O},\\
&  \div  u  =0, &\textrm{in} ~~\mathbb{R}^+\times \mathcal {O},
\end{aligned}
\right.
\end{equation}
with the boundary conditions
\begin{equation} \label{SCNS-c1}
   \frac{\partial n}{\partial \nu} =\frac{\partial c}{\partial \nu}=0 \quad \textrm{and} \quad u =0, \quad \textrm{in} ~  \mathbb{R}^+\times\partial \mathcal {O},
 \end{equation}
and the initial conditions
\begin{equation} \label{SCNS-c2}
n|_{t=0}=n_0,~c|_{t=0}=c_0,~ u|_{t=0}=u_0,\quad \textrm{in}~  \mathcal {O}.
 \end{equation}
Here,  $\mathcal {O}\subset \mathbb{R}^3$ is a bounded domain with smooth boundary $\partial \mathcal {O}$, and $\nu$ stands for the inward normal on the boundary. In \eqref{SCNS}, the unknowns are $n=n(t, x): \mathbb{R}^+ \times \mathcal {O} \rightarrow \mathbb{R}^+$, $c=c(t, x): \mathbb{R}^+ \times \mathcal {O} \rightarrow \mathbb{R}^+$, $u=u(t, x): \mathbb{R}^+ \times \mathcal {O} \rightarrow \mathbb{R}^3$ and $P=P(t, x): \mathbb{R}^+ \times \mathcal {O} \rightarrow \mathbb{R}$, denoting the density of bacteria, concentration of substrate, velocity field of fluid and associated pressure, respectively. The term  $n\nabla \Phi$ denotes the external force exerted by the bacteria on the fluid  through a given gravitational potential $\Phi=\Phi(x):\mathbb{R}^3\mapsto \mathbb{R}$. The positive constants $D_n,D_c$ and $\delta$ represent the diffusion coefficients for the bacteria, substrate and fluid, respectively. $\chi(c)$ is the chemotactic sensitivity and $f(c)$ is the consumption rate of the substrate by the bacteria.

Besides the gravity $n\nabla \Phi$ caused by the mass of bacteria, the fluid dynamics behavior
in \eqref{SCNS} is also influenced by a L\'{e}vy-type random external force:
\begin{equation} \label{noise}
h( t,u) + g(t,u)\frac{\mathrm{d} W(t)}{\mathrm{d} t}+ \int_{Z_0} K(u(t-),z)\widetilde{\pi}(\mathrm{d}t,\mathrm{d}z)+  \int_{Z\backslash Z_0} G(u(t-),z) \pi (\mathrm{d}t,\mathrm{d}z),
 \end{equation}
which is typically applied to describe the unpredictable noises in measurement or the errors of model in  approximating physical phenomena (cf. \cites{cyr2018euler,brzezniak20132d,chen2019martingale,nguyen2021nonlinear,zhai2015large}). Particularly, $h(t,u)$ represents the deterministic external force, $g(t,u)\frac{\mathrm{d} W(t)}{\mathrm{d} t}$ denotes the random force that influences the system continuously in time, and the terms $\int_{Z_0} K(u(t-),z)\widetilde{\pi}(\mathrm{d}t,\mathrm{d}z)$ and $  \int_{Z\backslash Z_0} G(u(t-),z) \pi (\mathrm{d}t,\mathrm{d}z)$ stand for the random forces that exhibit both small and large jumps, which influence the system discretely in time as impulses.

The random noises in \eqref{noise} are defined over a fixed probability space $(\Omega,\mathcal {F},\mathfrak{F},\mathbb{P})$ equipped with a filtration $\mathfrak{F}=(\mathcal {F} _t)_{t\geq 0}$ that satisfies the usual hypotheses. Specifically, $W$ is a $\mathfrak{F}$-adapted cylindrical Wiener process with values in a separable Hilbert space $U$. $\pi $, independent of $W$, is a time-homogeneous Poisson random measure on $[0, \infty) \times Z$ with intensity measure $\mathrm{d}t\otimes\mathrm{d}\mu $, where $Z$ is a complete and separable metric space with Borel $\sigma$-field $\mathscr{B}(Z)$ and $\mu$ is a $\sigma$-finite measure on $( Z,\mathscr{B}(Z))$. The compensated Poisson random measure is denoted by $\widetilde{\pi}(\mathrm{d}t,\mathrm{d}z) =\pi(\mathrm{d}t,\mathrm{d}z)-\mu(\textrm{d}z)\mathrm{d}t$. In \eqref{SCNS} or \eqref{noise}, we suppose that there is a subset $Z_0 \in \mathscr{B}(Z)$ such that $\mu (Z\backslash Z_0)<\infty$. For example, let $L$ be a L\'{e}vy process in a separable Hilbert space $U$ and $\pi $ be the Poisson random measure corresponding to the jump process $\Delta L$
defined by
$$
\Delta L(t)\overset{\textrm{def}}{=} L(t)-L\left(t-\right),\quad t\geq 0.
$$
The intensity measure of $\pi$ takes the form of $\textrm{d}t\otimes \textrm{d}\mu$, where $\mu$ is a L\'{e}vy measure of $L$ (cf. \cite[Theorem 4.23]{peszat2007stochastic}). Setting $Z= U\backslash \{0\}$, one can endow the space $Z$ with a complete separable metric (cf. \cite[Example A.8]{cyr2018euler}) such that $B\subseteq Z$ is bounded if and only if $B$ is separated from $0$, i.e., $0\notin \overline{B}$. If we choose
$$
Z_0=\left\{x \in U:  0<\|x\|_U<1\right\},
$$
then it is clear that $Z\backslash Z_0$ is separated from $0$, and hence $\mu (Z\backslash Z_0)<\infty$.

Regarding basic assumptions for the functions $\Phi ,f,\chi$ and the coefficients $h, g,K,G$, we refer to $(\textbf{A}_1)$-$(\textbf{A}_4)$ stated in Subsection \ref{sub1.3} for details. Concrete examples that satisfy these assumptions are provided in the Remark \ref{remmain}.

\subsection{History}

This subsection provides a summary of the studies relevant to the qualitative theory for the deterministic and stochastic CNS systems.

\textsc{CNS system in $\mathbb{R}^2$ or $\mathbb{R}^3$.} The coupled CNS system (where $h$, $g$, $K$ and $G$ vanish in \eqref{SCNS}) has gained a lot of interest in recent years. For $\mathbb{R}^3$, Duan, Lorz  and  Markowich \cite{duan2010global} obtained the global existence of classical solutions near constant states, while for $\mathbb{R}^2$, they established the existence of global weak solutions provided that the external force is weak or the substrate concentration is small. Later in \cite{tan2014decay}, Tan  and  Zhang obtained the optimal convergence rates of classical solutions for small initial perturbation around constant states, which partially improves the results in \cite{duan2010global}. By virtue of a refined  energy functional, Liu  and  Lorz \cite{liu2011coupled} proved the global existence of weak solutions in two spatial dimensions with some conditions on the sensitivity and consumption rate. In \cite{chae2014global}, Chae, Kang  and  Lee established the local existence of regular solutions, and presented some blow-up criteria when the concentration equation is parabolic-type and hyperbolic-type, respectively. In the case that the initial mass is bounded and integrable, Kang  and  Kim \cite{kang2017existence} construct a solution to the CNS system such that the density of biological organism belongs to the absolutely continuous curves in the Wasserstein space. More recently, Jeong and Kang \cite{jeong2022well} investigated the local well-posedness and blow-up criteria in Sobolev spaces for both partially inviscid and fully inviscid cases by performing a weighted Gagliardo-Nirenberg-Sobolev inequality.

\textsc{CNS system in bounded domain.} The first work concerning the CNS system in bounded domain is due to Lorz \cite{lorz2010coupled}, in which the author demonstrated the existence of a local weak solution using Schauder's fixed point theory. If the domain is further assumed to be convex, by using some delicate entropy-energy estimates, Winkler \cite{winkler2012global} established the existence of a unique global classical solution with general initial data in two spatial dimensions, and he also provided an existence result in three dimensions without the effect of convection term. Note that the convex assumption on the region was then removed by Jiang, Wu  and  Zheng in \cite{mizoguchi2014nondegeneracy}. Later in \cite{winkler2014stabilization}, Winkler studied the stability of solutions obtained in \cite{winkler2012global} in two-dimensional situation, and it was discovered that the solution converges to a constant state as time approaches infinity. When the bounded domain $\mathcal {O}\subseteq \mathbb{R}^3$ is convex and smooth, Winker \cite{winkler2016global} construct a global weak solution with suitable weak conditions on initial data by using an energy-type inequality. Furthermore, it was recently demonstrated in \cite{winkler2017far} that after some relaxation time, these weak solutions constructed in \cite{winkler2016global} enjoy  additional regularity qualities and thereby comply with the eventual energy solutions. The possibility for singularities of these weak energy solutions on small time-scales has also been shown to arise solely on sets of measure zero \cite{winkler2022does}. For the CNS system of production type, we refer to \cites{hillen2009user,arumugam2021keller} and references therein for details.

\textsc{About SCNS system.} In order to account for the unpredictable environment around the fluid at a macro- and microscopic level, Zhai and Zhang \cite{zhai20202d} first considered the two-dimensional CNS system with a fluid driven by cylindrical Wiener process in a convex and bounded domain. They affirmed the existence and uniqueness of both pathwise mild solutions and weak solutions with the help of an entropy-energy inequality. After submitting this paper, we have learned a new result by Hausenblas et al. \cite{hausenblas2023existence}, in which the authors studied the initial-boundary value problem for the two-dimensional SCNS system with an additional random noise imposed on the $n$-equation. In comparison with its deterministic counterpart, the mathematical literature on the three-dimensional SCNS system is less developed, and as far as we know, \cite{zhai20202d,hausenblas2023existence} are the only results relevant to the coupled SCNS system in two spatial dimensions. For completeness, we would like to mention the works \cites{mayorcas2021blow,hausenblas2022one,misiats2022global,huang2021microscopic,
shang2019asymptotic,martini2022additive} that discuss the existence, uniqueness, and blow-up criteria for the decoupled stochastic Keller-Segel type systems.  Currently, the  well-posed problem for the coupled SCNS system in dimension three is still open,  such as the global existence in generalized solution frameworks.

The \emph{main contribution} of this paper is to provide an affirmative answer to the question of global solvability for the three-dimensional SCNS system in a bounded domain. More precisely, under appropriate assumptions, we shall prove that the three-dimensional SCNS system \eqref{SCNS} perturbed by both continuous and discontinuous random external forces possesses at least
one global martingale solution. These solutions are weak in the analytical sense (derivatives
exists only in the sense of distributions) and weak in the stochastic sense (the underlying
probability space is not a priori given but part of the problem).

\subsection{Main result}\label{sub1.3}
For any $p \in[1,+\infty]$, $ k \in \mathbb{N}$, $W^{k, p}(\mathcal {O})$ stands for the standard Sobolev spaces equipped with the norm $\|\cdot\|_{W^{k, p} }$. As a subspace, we define
$$
W^{k, p}_{0,\sigma}(\mathcal {O})=\left\{u\in W^{k, p}(\mathcal {O});\ \div u=0 \ \textrm{and} \ u |_{\partial\mathcal {O}}=0\right\}.
$$
For a Banach space $X$, its dual space is indicated by $X^*$, and the duality between $X$ and $X^*$ is denoted by $\langle\cdot, \cdot\rangle_{X^*, X}$, or simply by $\langle\cdot, \cdot\rangle$ when  there is no confusion. We denote by $\mathcal {C}([0,T];X)$ the space of all $X$-valued continuous functions, and by $\mathcal {D}[0,T];X)$ the space of all $X$-valued \emph{c\`{a}dl\`{a}g} functions, i.e., the right continuous functions with left-hand limits from $[0,T]$ into $X$. Given two separable Hilbert spaces $U$ and $G$, we use the symbol $\mathcal {L}_2(U;G)$ to denote the space of all Hilbert-Schmidt operators from $U$ into $G$.

Now let us give the rigorous definition of martingale weak solutions to SCNS system \eqref{SCNS}. Without loss of generality, we will take $D_n=D_c=\delta=1$ throughout the paper.

\begin{definition}\label{def}
A quantity
$$
\left((\Omega,\mathcal {F},\mathfrak{F} ,\mathbb{P}),W,\pi,n,c,u\right)
$$
is called a \emph{global martingale weak solution} to the problem \eqref{SCNS}-\eqref{SCNS-c2},  provided:
\begin{itemize}[leftmargin=0.9cm]
\item [$\bullet$] $\left(\Omega,\mathcal {F},\mathfrak{F},\mathbb{P}\right)$ is a filtered probability space with filtration $\mathfrak{F}=(\mathcal {F}_t)_{t\geq 0}$; $W$ is an $\mathfrak{F}$-adapted $U$-valued cylindrical Wiener process; $\pi$, independent of $W$, is a time-homogeneous Poisson random measure on $[0,\infty)\times Z$ with intensity measure $\textrm{d}t\otimes \textrm{d}\mu$.

\item [$\bullet$] $ n\in L^2\left(\Omega; L^1_{loc}(0,\infty; W^{1,1}(\mathcal {O}))\right)$ is an $\mathfrak{F}$-progressively measurable process with $\mathbb{P}$-a.s. continuous sample paths in $(W^{\frac{13}{11},11}(\mathcal {O}))^*$; $c\in L^2\left(\Omega; L^1_{loc}(0,\infty;W^{1,1}(\mathcal {O}))\right)$ is an $\mathfrak{F}$-progressively measurable process with $\mathbb{P}$-a.s. continuous sample paths in $(W^{2,\frac{5}{2}}(\mathcal {O}))^*$; $u\in L^2 \left(\Omega;  L^1_{loc} (0,\infty; (W^{1,1}_{0,\sigma}(\mathcal {O}) )^3 ) \right)$ is an $\mathfrak{F}$-progressively measurable process with $\mathbb{P}$-a.s. c\`{a}dl\`{a}g sample paths in $(W^{1, 5}_{0,\sigma}(\mathcal {O}) )^*$.

\item [$\bullet$] for $\mathbb{P}$-a.a. $\omega\in \Omega$, there hold
\begin{align*}
&(u\otimes u)(\cdot,\omega) \in L^1_{loc}\left(\mathcal {O}\times [0,\infty);\mathbb{R}^{3\times 3}\right),~(nf(c))(\cdot,\omega) \in L^1_{loc}\left(0,\infty;L^1(\mathcal {O})\right),\\
&(n\chi(c)\nabla c)(\cdot,\omega),~(nu)(\cdot,\omega),~(cu)(\cdot,\omega) \in L^1_{loc}\left(0,\infty;(L^1(\mathcal {O}))^3\right).
\end{align*}

\item [$\bullet$] there hold $\mathbb{P}$-a.s.
\begin{align*}
-  \int_\mathcal {O} n_0 \varphi|_{t=0}\mathrm{d}x  &=  \int_ 0^\infty  \int_\mathcal {O} \left(n \varphi _t+nu \cdot \nabla \varphi-\nabla n \cdot\nabla \varphi+n\chi(c) \nabla c \cdot \nabla \varphi \right)\mathrm{d}x \mathrm{d} t,\\
- \int_\mathcal {O} c_0 \varphi|_{t=0}\mathrm{d}x&=  \int_ 0^\infty  \int_\mathcal {O} \left( c \varphi _t+cu \cdot \nabla \varphi  -\nabla c \cdot\nabla \varphi-nf(c) \varphi    \right)\mathrm{d}x \mathrm{d} t,
\end{align*}
for all $\varphi\in \mathcal {C}^\infty_{0}(\mathcal {O}\times [0,\infty);\mathbb{R})$.

\item [$\bullet$] there holds $\mathbb{P}$-a.s.
\begin{align*}
&- \int_\mathcal {O} u_0 \cdot\overline{\varphi}|_{t=0}\mathrm{d}x=  \int_ 0^\infty  \int_\mathcal {O}  \left(u\cdot \overline{\varphi} _t+u\otimes u : \nabla \overline{\varphi}-\nabla u \cdot\nabla \overline{\varphi} +n \nabla \Phi\cdot \overline{\varphi  }  \right)\mathrm{d}x \mathrm{d} t\\
& \quad + \int_ 0^\infty  \int_\mathcal {O}   h( t,u) \overline{\varphi}\mathrm{d}x \mathrm{d} t+  \int_ 0^\infty  \int_\mathcal {O}  g(t,u)\overline{\varphi} \mathrm{d}x\mathrm{d}W_t \\
&\quad + \int_ 0^\infty  \int_{Z_0} \int_\mathcal {O} K(u(x,t-),z)\overline{\varphi}\mathrm{d} x \widetilde{\pi}(\mathrm{d}t,\mathrm{d}z) + \int_ 0^\infty  \int_{Z\backslash Z_0} \int_\mathcal {O} G(u(x,t-),z)\overline{\varphi}\mathrm{d} x \pi(\mathrm{d}t,\mathrm{d}z),
\end{align*}
for all $\overline{\varphi}\in \mathcal {C}^\infty_0(\mathcal {O}\times [0,\infty);\mathbb{R}^3)$ with $\div \overline{\varphi}=0$.
\end{itemize}
\end{definition}

In this paper, we make the following assumptions.

\begin{itemize}[leftmargin=1cm]
\item [$(\textbf{A}_1)$] For the initial datum $n_0,c_0$ and $u_0$, we assume that
 \begin{equation*}
\left\{
\begin{aligned}
 &u_0\in L^2_\sigma(\mathcal {O}),~~n_0 \in L \log L(\mathcal {O}) ~~ \textrm{is non-negative and}  ~n_0\not\equiv 0,\\
& c_0 \in L^\infty(\mathcal {O})~~ \textrm{is non-negative}, ~~\sqrt{c_0}\in W^{1,2}(\mathcal {O}),
\end{aligned}
\right.
\end{equation*}
where $L\log L(\mathcal {O})$ denotes the standard Orlicz space associated with the Young function  $\theta(x)=x\ln(1+x)$ on $(0,\infty)$, and $L^2_\sigma(\mathcal {O}) \overset{\textrm{def}}{=} \{u\in L^2(\mathcal {O});\ \div u=0 \}
$ stands for the solenoidal subspace of $L^2(\mathcal {O})$.

\vskip1mm
\item [$(\textbf{A}_2)$] For the potential $\Phi$, the chemotactic sensitivity $\chi$ and the signal consumption rate $f$ in SCNS \eqref{SCNS}, we assume that
 \begin{equation*}
\left\{
\begin{aligned}
&  \Phi: \mathcal {O}\mapsto \mathbb{R}~ \textrm{is locally Lipschitz countinuous, i.e.,} ~\Phi\in W^{1,\infty}(\mathcal {O};\mathbb{R}),\\
& \chi :[0,\infty)\mapsto (0,\infty)~ \textrm{is a $C^2$ function},\\
&f:[0,\infty)\mapsto [0,\infty)~ \textrm{is a $C^2$ function with} ~f(0)=0, ~~f>0~\textrm{on} ~(0,\infty) ,  \\
&\left(\frac{f}{\chi}\right)'>0,~~\left(\frac{f}{\chi}\right)''\leq 0,  ~~\textrm{and} ~ \left(f \chi\right)'\geq0~ \textrm{on}~ [0,\infty).
\end{aligned}
\right.
\end{equation*}

\vskip1mm
\item [$(\textbf{A}_3)$] For any $\alpha \in \{0,\frac{1}{2}\}\cup (\frac{3}{4},1)$, the mappings $h(\cdot,\cdot):[0,\infty)\times \mathscr{D}(\textrm{A}^\alpha)\mapsto \mathscr{D}(\textrm{A}^\alpha)$ and $g(\cdot,\cdot):[0,\infty)\times \mathscr{D}(\textrm{A}^\alpha)\mapsto \mathcal {L}_2(U;\mathscr{D}(\textrm{A}^\alpha))$ are Borel measurable, and there exists a constant $C>0$ such that
\begin{equation*}
\begin{split}
 &\|h(t,u)\|_{\mathscr{D}(\textrm{A}^\alpha)}\leq  C\left(1+\|u\|_{\mathscr{D}(\textrm{A}^\alpha)} \right), \\
 &\|h(t,u_1)-h(t,u_2)\|_{\mathscr{D}(\textrm{A}^\alpha)}\leq  C\|u_1-u_2\|_{\mathscr{D}(\textrm{A}^\alpha)},\\
 &\|g(t,u)\|_{\mathcal {L}_2(U;\mathscr{D}(\textrm{A}^\alpha))}\leq  C\left(1+\|u\|_{\mathscr{D}(\textrm{A}^\alpha)} \right),\\
 &\|g(t,u_1)-g(t,u_2)\|_{\mathcal {L}_2(U;\mathscr{D}(\textrm{A}^\alpha))}\leq  C\|u_1-u_2\|_{\mathscr{D}(\textrm{A}^\alpha)}
 \end{split}
\end{equation*}
hold uniformly in time for all $u,u_1,u_2 \in \mathscr{D}(\textrm{A}^\alpha)$. Here, $\textrm{A}\overset{\textrm{def}}{=}- \mathcal {P}\Delta$ denotes the Stokes operator with the domain $ \mathscr{D}(\textrm{A}) = W^{2,2}(\mathcal {O})\cap W^{1,2}_{0,\sigma}(\mathcal {O})$.

 \vskip1mm
\item [$(\textbf{A}_4)$] For any $\alpha \in \{0,\frac{1}{2}\}\cup (\frac{3}{4},1)$, the mappings $K(\cdot ,\cdot): \mathscr{D}(\textrm{A}^\alpha)\times Z_0 \rightarrow\mathscr{D}(\textrm{A}^\alpha)$ and $G(\cdot ,\cdot): \mathscr{D}(\textrm{A}^\alpha)\times Z\backslash Z_0 \rightarrow\mathscr{D}(\textrm{A}^\alpha)$ are Borel measurable, and there exists a  $C>0$ such that
\begin{equation*}
\begin{split}
 & \int_{Z_0}\|K(u ,z)\|_{\mathscr{D}(\textrm{A}^\alpha)}^p\mu( \mathrm{d}z)+ \int_{Z\backslash Z_0}\|G(u ,z)\|_{\mathscr{D}(\textrm{A}^\alpha)}^p\mu( \mathrm{d}z)\\
&\quad\leq C \left(1+\|u\|^p_{\mathscr{D}(\textrm{A}^\alpha)} \right),\quad \textrm{for all} ~~ p\geq 2,\\
 & \int_{Z_0}\|K(u_1 ,z)-K(u_2 ,z)\|_{\mathscr{D}(\textrm{A}^\alpha)}^2\mu( \mathrm{d}z)+ \int_{Z\backslash Z_0}\|G(u_1 ,z)-G(u_2 ,z)\|_{\mathscr{D}(\textrm{A}^\alpha)}^2\mu( \mathrm{d}z) \\
 & \quad\leq C  \| u_1- u_2 \|_{\mathscr{D}(\textrm{A}^\alpha)}^2
 \end{split}
\end{equation*}
hold for all $u,u_1,u_2 \in \mathscr{D}(\textrm{A}^\alpha)$.
\end{itemize}

The main result in this work can now be stated by the following theorem.
\begin{theorem}\label{thm}
Under the assumptions $(\textbf{A}_1)-(\textbf{A}_4)$, the initial-boundary value problem $\eqref{SCNS}-\eqref{SCNS-c2}$ possesses at least one global martingale weak solution $
((\widehat{\Omega},\widehat{\mathcal {F}},\mathfrak{F},\widehat{\mathbb{P}}),\widehat{W},\widehat{\pi},\widehat{n},\widehat{c},\widehat{u}) $ in the sense of Definition \ref{def}. Moreover, for any  $T>0$, the weak entropy-energy inequality
\begin{equation}\label{1.5}
\begin{split}
&- \int_0^T\phi'(t) \int_\mathcal {O} \left(\widehat{n} \ln \widehat{n} + \frac{1}{2}|\nabla \Psi(\widehat{c} )|^2+c^\dag |\widehat{u} |^2\right)\mathrm{d} x\mathrm{d}t \\
&\quad +  \int_0^T\phi(t)  \int_\mathcal {O}\left( \frac{1}{2}\frac{ |\nabla \widehat{n} |^{2}}{\widehat{n}}  +  d_1\frac{|\nabla \widehat{c} |^4}{\widehat{c}^3} + d_2\frac{ |\Delta \widehat{c}|^2}{\widehat{c}}  +  \frac{c^\dag}{2}|\nabla \widehat{u}|^{2}\right)\mathrm{d}x\mathrm{d}t\\
 & \quad\leq \phi(0)\int_\mathcal {O}\left(n_0 \ln n_0 + \frac{1}{2}|\nabla \Psi(c_0)|^2+c^\dag |u_0|^2\right)\mathrm{d} x \\
 &\quad+ C \int_0^T\phi(t)\left( \int_\mathcal {O} |\widehat{u}(t) | ^2\mathrm{d} x+1\right)\mathrm{d}t +  \int_0^T\phi(t)\mathrm{d}\mathcal {M}_E
\end{split}
\end{equation}
holds true $\widehat{\mathbb{P}}$-a.s. for any deterministic test function $\phi (t)\geq 0$ with $\phi(T)=0$, where
$$
\Psi(\sigma)\stackrel{\mathrm{def}}{=} \int_1^\sigma\sqrt{\frac{\chi(r)}{f(r)}} \mathrm{d}r,\quad \textrm{for all}~ \sigma>0.
$$
Here $\mathcal {M}_E$ is a real-valued martingale with bounded $p$-th order momentum
\begin{equation}\label{1.8}
\begin{split}
 \widehat{\mathbb{E}}\sup_{t\in [0,T]}|\mathcal {M}_E|^p \leq C\left[\mathbb{E}\left( \int_\mathcal {O} \left(n_0\ln n_0+ |\nabla \Psi(c_0)|^2+ |u_0|^2\right)\mathrm{d} x\right)^p +1 \right],
\end{split}
\end{equation}
for $1\leq p <\infty$, where the positive constants $d_1,d_2 ,c^\dagger,C$ depend only on $p,c_0,f,\chi$ and $\Phi$.
\end{theorem}

\begin{remark} \label{remmain}
The assumptions $(\textbf{A}_1)$-$(\textbf{A}_4)$ are reasonable, for instance:
\begin{itemize}[leftmargin=0.9cm]
\item [$\bullet$] for $(\textbf{A}_1)$-$(\textbf{A}_2)$,  since $\mathcal {O}$ is a bounded domain in $\mathbb{R}^3$, the initial data may be selected as $
(n_0,c_0,u_0)=\left(|x|^2+1,|x|^4,\textrm{curl}(|x|^2)\right)$. Concerning the chemotactic sensitivity and the consumption rate of substrate, it is of interest to consider $f(s)=s$, for all $s\geq0$  and $\chi\equiv\textrm{const.}>0 $,
which corresponds to the prototypical deterministic CNS system (cf. \cite{hillen2009user,arumugam2021keller}).

\item [$\bullet$] for $(\textbf{A}_3)$, one can choose $h(t,u)=u$, for all $t\geq 0$. Concerning the Hilbert-Schmidt operator $g(t,u)$, one can take $U=L^2 (\mathcal {O})$ and fix an element  $\psi\in \mathscr{D}(\textrm{A}^\alpha)$, and define the linear mapping $g(t,\cdot):\mathscr{D}(\textrm{A}^\alpha)\rightarrow \mathcal {L}_2(U;\mathscr{D}(\textrm{A}^\alpha))$ by
\begin{equation*}
\begin{split}
g(t,u)v\overset{\textrm{def}}{=}(u, v)_{L^2}\cdot\psi,\quad \forall u\in \mathscr{D}(\textrm{A}^\alpha),  v\in U.
\end{split}
\end{equation*}
Apparently, there hold $g(t,0)=0$ and
$$
\|g(t,u_1)-g(t,u_2)\|_{\mathcal {L}_2(U;\mathscr{D}(\textrm{A}^\alpha))}^2\leq C\| u_1-u_2\|_{\mathscr{D}(\textrm{A}^\alpha)}^2\|\psi\|_{\mathscr{D}(\textrm{A}^\alpha)}^2,\quad \forall u_1,u_2 \in \mathscr{D}(\textrm{A}^\alpha),
$$
which imply that $g(\cdot,\cdot)$ satisfies the assumption $(\textbf{A}_3)$.

\item [$\bullet$] for $(\textbf{A}_4)$, let $\pi$ be the Poisson random measure with intensity measure $\mathrm{d}t\otimes\mathrm{d}\mu$ induced from a L\'{e}vy process $L$ on a separable Hilbert space $U$, where $\mu$ is a $\sigma$-finite L\'{e}vy measure such that $\int_{U\backslash\{0\}}(\|z\|_U^2 \wedge 1) \mu (\textrm{d}z)<\infty$ (cf. \cite[Theorem 4.23]{peszat2007stochastic}). By choosing $Z=U\backslash \{0\}$ and $Z_0=\{z\in U;~0<\|z\|_U<1\}$ (and hence $\mu(Z\backslash Z_0)<\infty$), one can define the measurable mapping $K(\cdot,\cdot): \mathscr{D}(\textrm{A}^\alpha) \times Z_0\mapsto \mathscr{D}(\textrm{A}^\alpha)$ by
\begin{equation*}
\begin{split}
K(u ,z)\overset{\textrm{def}}{=} \|z\|_U\cdot u, \ \ \textrm{for all}\ (u,z)\in \mathscr{D}(\textrm{A}^\alpha)\times Z_0.
\end{split}
\end{equation*}
Direct calculation shows that $K(0,z)=0$, and for all $p\geq 2$
\begin{equation*}
\begin{split}
\int_{Z_0}\|K(u_1,z)-K(u_2,z)\|^p_{\mathscr{D}(\textrm{A}^\alpha)} \mu (\textrm{d}z)&=\|u_1-u_2\|_{\mathscr{D}(\textrm{A}^\alpha)}^p\int_{Z_0}\|z\|_U^p \mu (\textrm{d}z)\\
&\leq C\|u_1-u_2\|_{\mathscr{D}(\textrm{A}^\alpha)}^p,
\end{split}
\end{equation*}
which implies that $K(\cdot,\cdot)$ satisfies the assumption $(\textbf{A}_4)$. In a similar manner, the measurable mapping $G(\cdot,\cdot): \mathscr{D}(\textrm{A}^\alpha) \times Z\backslash Z_0\mapsto \mathscr{D}(\textrm{A}^\alpha)$ defined by
\begin{equation*}
\begin{split}
G(u ,z)\overset{\textrm{def}}{=} \frac{1}{\|z\|_U}\cdot u, \ \ \textrm{for all}\ (u,z)\in \mathscr{D}(\textrm{A}^\alpha)\times Z\backslash Z_0
\end{split}
\end{equation*}
 satisfies the assumption $(\textbf{A}_4)$ due to the fact of $\mu (Z\backslash Z_0)<0$.
\end{itemize}
\end{remark}

\begin{remark}
As far as we aware, Theorem \ref{thm} appears to be the first result to consider the global solvability for SCNS system in three spatial dimensions, extending the earlier works in the deterministic case by Winkler \cite{winkler2012global,winkler2016global} and in the two-dimensional stochastic cases by Zhai and Zhang \cite{zhai20202d} and Hausenblas et al. \cite{hausenblas2023existence}. It should also be noted that Theorem \ref{thm} does not require the convexity of the bounded domain $\mathcal {O}$, which is technically assumed in \cite{winkler2012global,zhai20202d,winkler2016global}. The lack of convexity of the domain makes the derivation of several important a priori estimations more delicate, please see the associated statement in Subsection 1.4. Additionally, our result includes the effects of both continuous and discontinuous
random external forces surrounding the fluid, which is more natural from a physical point
of view, and we refer to \cite{chen2019martingale,sakthivel2012martingale,dong2011ergodicity,zhai2015large,manna2021well,
dong2011martingale,hausenblas2020existence,nguyen2021nonlinear} for a few efficient applications of L\'{e}vy processes in the stochastic hydrodynamics.
\end{remark}

\begin{remark}
In dimension two, one can demonstrate the pathwise uniqueness of solutions to the SCNS system. This fact together with  the existence of global martingale solutions as well as the Yamada-Watanabe Theorem for SPDEs (cf. \cite[Lemma 1.1]{gyongy1980stochastic}) yields that the two dimensional SCNS system admits a unique global pathwise solution, we refer the reader to \cite{zhai20202d} for the proof. In dimension three, however, due to the lack of spatial regularity of solutions, the uniqueness part can not be guaranteed. As a matter of fact, the uniqueness or smoothness of solutions to the three dimensional deterministic/stochastic Navier-Stokes equations is a long standing open problem. In this aspect, we would like to mention the recent breakthrough on the non-uniqueness of weak solutions to the deterministic Navier-Stokes equations \cite{de2013dissipative,de2014dissipative,albritton2022non,buckmaster2019nonuniqueness} and to the stochastic Navier-Stokes equations \cite{hofmanova2019non,hofmanova2021global,hofmanova2021ill,chen2022sharp} based on the convex integration method. This also inspires us to pursue instead the global martingale (weak in probability) weak solutions. In future, a fascinating but challenging problem is to investigate the regularity and uniqueness of solutions constructed in Theorem \ref{thm}.
\end{remark}

\subsection{Outline of proof and main  ideas}

Theorem \ref{thm} will be proved through the following steps. First,  we consider a regularized SCNS system \eqref{SCNS-1} by applying Leray's regularization and the weakly increasing approximation in \cite{winkler2022reaction,winkler2022does}. However, due to the irregular random
external forces, the standard methodology used in \cite{winkler2016global} is not enough  to get the existence and uniqueness of approximate solutions. To overcome this,  proper working spaces are selected such that the processes are bounded in space variable, and several cut-off operators, which depends on the size of $\|n_\epsilon\|_{L^\infty}$, $\|c_\epsilon\|_{W^{1,q}}$ and $\|\textrm{A}^\alpha u_{\epsilon}\|_{L^2}$, are used to deal with the nonlinear terms and stochastic integrals (cf. \eqref{SCNS-2}). Based on a contraction mapping argument, one can prove the existence of a unique mild solution on any interval. Then we construct a maximal local mild solution $\left(n_{\epsilon},c_{\epsilon},u_{\epsilon},\widetilde{\tau}_\epsilon\right)$ to the regularized system \eqref{SCNS-1} by removing the cut-off operators via a sequence of stopping times. Therein, the smoothing effect of heat semigroup and Stokes semigroup \cite{winkler2010aggregation,giga1986solutions} plays an important role in a series of estimations.

The second step is to show that the maximal local solution obtained in Lemma \ref{lem1} is global in time $\mathbb{P}$-almost surely, that is, $
\mathbb{P}[\widetilde{\tau}_\epsilon=\infty]=1, $
where the key component is a newly derived stochastic analogue of entropy-energy type inequality and a crucial uniform estimate for the approximate solutions $(n_{\epsilon},c_{\epsilon},u_{\epsilon})$ within an unnecessarily convex region (see Lemma \ref{lem4.5}-\ref{lem4.6}). Note that in \cite{winkler2012global,zhai20202d}, the authors obtained the entropy-energy estimates by dropping the negative term $ \int_0^T\int_{\partial\mathcal {O}} \frac{1}{\theta(c_{\epsilon})}  \frac{\partial|\nabla c_{\epsilon}|^{2}}{\partial \nu}\mathrm{d}\Sigma\mathrm{d}t\leq 0$ under the convex condition of domain. Without the convexity condition, this term can not be dropped directly which makes the estimation more subtle. Fortunately, by means of the Mizoguchi-Souplet Lemma stated in Lemma \ref{lem-ms}, the term $\frac{\partial|\nabla c_{\epsilon}|^{2}}{\partial \nu}$ can be controlled by the quantity $|\nabla c_{\epsilon}|^2$ on the boundary $\partial\mathcal {O}$, which together with a proper Sobolev Embedding Theorem and an interpolation argument leads to the key inequality \eqref{(4.13)}. Then we obtain the desired estimate by absorbing the second and third terms on the R.H.S. of \eqref{(4.13)} in terms of the fine structure of the system \eqref{SCNS-1}. To proceed, several stopping times should be introduced which assist us to establish some a priori estimates by a careful application of the BDG inequality. We remark that for the classical Keller-Segel system and the ones coupled to fluid, the entropy-energy dissipative structure constituted crucial ingredients in the development of global existence theories \cite{duan2010global,liu2011coupled,wang2021local}.

The third step is to take a limit as $\epsilon\rightarrow 0$ and prove the existence of global martingale solutions. At this stage, the main difficulty comes from the fact that the weak convergence inherited from the uniform bounds is too weak to proceeding a compactness argument as that in deterministic situation. We overcome this problem by adapting the original idea of Skorokhod \cite{skorokhod1961existence} to investigating the tightness of approximate solutions, which can be carried out by applying the uniform energy estimate together with a generalized  Aubin-Lions Lemma \cite{nguyen2021nonlinear},  Flandoli's Lemma \cite{flandoli1995martingale} and a tightness criteria in $\mathcal {D} ([0,T];(W^{1, 5}_{0,\sigma}(\mathcal {O}))^*)$ due to Aldous \cite{aldous1978stopping,aldous1989stopping}. Then with the Jakubowski-Skorokhod Theorem (which is valid for quasi-polish spaces such as the separable Banach spaces with weak topology) one can  construct a new probability space $(\widehat{\Omega}, \widehat{\mathcal {F}}, \widehat{\mathbb{P}})$, on which a sequence of random variables $(\widehat{W}_{\epsilon_j},\widehat{\pi}_{\epsilon_j},\widehat{n}_{\epsilon_j},
\widehat{c}_{\epsilon_j},\widehat{u}_{\epsilon_j})_{j\geq 1}$ can be defined. These random elements have the same distributions as the original ones and in addition converge almost surely to an element $(\widehat{W} ,\widehat{\pi} ,\widehat{n} ,\widehat{c },\widehat{u})$.  Making use of Bensoussan's argument \cite{bensoussan1995stochastic} but with a crucial generalization on dealing with the stochastic integrals, we show that $(\widehat{n}_{\epsilon_j},\widehat{c}_{\epsilon_j},\widehat{u}_{\epsilon_j})_{j\geq 1}$ satisfies the approximate system \eqref{SCNS-1} under the new stochastic basis $((\widehat{\Omega}, \widehat{\mathcal {F}}, (\widehat{\mathcal {F}} _t), \widehat{\mathbb{P}}),\widehat{W}_{\epsilon_j},\widehat{\pi}_{\epsilon_j})$. With above results at hand, by carefully using the Uniform Integrability Theorem, the Vitali Convergence Theorem, the Dominated Convergence Theorem together with the almost surely convergence in Lemma \ref{lem5.5}, one can take the limit $j\rightarrow\infty$ in \eqref{SCNS-1} to verify  that the quantity $(\widehat{W} ,\widehat{\pi} ,\widehat{n} ,\widehat{c },\widehat{u})$ is indeed a global martingale weak solution to \eqref{SCNS} in the sense of Definition \ref{def}.

\subsection{Organization of the paper}

Section \ref{sec3} is devoted to the existence of maximal local mild solutions to the regularized SCNS system.
In Section \ref{sec4}, we first verify that the solution constructed in Section \ref{sec3} is indeed a variational solution associated to a Gelfand triple, then we derive a stochastic analogue of   entropy-energy type inequality and also an uniform bounds for approximate solutions, based on which we show that the approximate solutions are global ones. Section \ref{sec5} focuses on the identification of limits by using the stochastic compactness method and the Jakubowski-Skorokhod Theorem.

\section{Local existence of approximate solutions}\label{sec3}

In this section, we first introduce a regularized SCNS system, and then prove the existence and uniqueness of maximal local mild solution in appropriate function spaces.

\subsection{Regularized SCNS system}
Suppose that $\varrho\in C_0^\infty(\mathbb{R}^3)$ is a standard mollifier, for any $\epsilon> 0$, consider the Leray regularization defined by
$$
\textbf{L}_\epsilon u\overset {\textrm{def}} {=}\varrho _\epsilon \star u= \frac{1}{\epsilon^3} \int_{\mathbb{R}^3}\varrho\left(\frac{x-y}{\epsilon}\right)u(y)\mathrm{d} y,
$$
and utilize the weakly increasing approximation  $\textbf{h}_\epsilon$ given by $
\textbf{h}_\epsilon(s)\overset {\textrm{def}} {=}\frac{1}{\epsilon}\ln (1+\epsilon s)$, $s\geq 0$. Then the regularized SCNS system takes the form of
\begin{equation}\label{SCNS-1}
\left\{
\begin{aligned}
&\mathrm{d} n_\epsilon +u_\epsilon\cdot \nabla n_\epsilon \mathrm{d}t =  \Delta n_\epsilon\mathrm{d}t - \div\left(n_\epsilon \textbf{h}_\epsilon' (n_\epsilon)\chi(c_\epsilon)\nabla c_\epsilon\right)\mathrm{d}t,\\
&\mathrm{d} c_\epsilon+ u_\epsilon\cdot \nabla c _\epsilon \mathrm{d}t = \Delta c_\epsilon\mathrm{d}t-\textbf{h}_\epsilon(n_\epsilon) f(c_\epsilon)\mathrm{d}t,\\
&\mathrm{d} u_\epsilon+ \mathcal {P} (\textbf{L}_\epsilon u_\epsilon\cdot \nabla) u_\epsilon  \mathrm{d}t \\
&\quad \ =-\textrm{A} u_\epsilon\mathrm{d}t+\mathcal {P}(n_\epsilon\nabla \Phi )\mathrm{d}t +\mathcal {P}h(t,u_\epsilon) \mathrm{d}t+ \mathcal {P}g(t,u_\epsilon) \mathrm{d}W_t\\
 &\quad \ + \int_{Z_0} \mathcal {P}K(u_\epsilon(t-),z)\widetilde{\pi}(\mathrm{d}t,\mathrm{d}z)+  \int_{Z\backslash Z_0} \mathcal {P}G(u_\epsilon(t-),z) \pi (\mathrm{d}t,\mathrm{d}z),\\
& \frac{\partial n_\epsilon}{\partial \nu} \Big|_{\partial\mathcal {O}}=\frac{\partial c_\epsilon}{\partial \nu} \Big|_{\partial\mathcal {O}}=0,~~ u_\epsilon|_{\partial\mathcal {O}}=0,\\
 & n_\epsilon|_{t=0}=n_{ \epsilon0},~c_\epsilon|_{t=0}=c_{ \epsilon0},~ u_\epsilon|_{t=0}=u_{ \epsilon0},
 \end{aligned}
\right.
\end{equation}
where $\mathcal {P}:L^2(\mathcal {O})\mapsto L^2_\sigma(\mathcal {O})$ is the Leray-Holmholtz projection, and the family of approximate initial datum $(n_{ \epsilon0},c_{ \epsilon0},u_{ \epsilon0})$ have the following properties:
\begin{equation*}
\left\{
\begin{aligned}
&0\leq n_{\epsilon0}\in \mathcal {C}^\infty_{0}(\mathcal {O}),~\|n_{ \epsilon0}\|_{L^1}=\|n_{ 0}\|_{L^1},~\|n_{ \epsilon0}-n_{ 0}\|_{L \textrm{log} L(D)}\stackrel{\epsilon\rightarrow0}{\longrightarrow} 0,\\
& 0\leq \sqrt{c_{ \epsilon0}}\in \mathcal {C}^\infty_{0}(\mathcal {O}),~\|c_{ \epsilon0}\|_{L^\infty}\leq \|c_{0}\|_{L^\infty},~\|\sqrt{c_{ \epsilon0}}-\sqrt{c_{ 0}}\|_{1,2}\stackrel{\epsilon\rightarrow0}{\longrightarrow} 0,\\
&c_{ \epsilon0}\stackrel{\epsilon\rightarrow0}{\longrightarrow} c_{ 0}~\textrm{a.e. in} ~\mathcal {O},\\
& u_{ \epsilon0}\in \mathcal {C}^\infty_{0,\sigma}(\mathcal {O}),~\|u_{ \epsilon0}\|_{L^2}\leq \|u_{0}\|_{L^2},~\|u_{ \epsilon0}-u_{ 0}\|_{1,2}\stackrel{\epsilon\rightarrow0}{\longrightarrow} 0,~u_{ \epsilon0}\stackrel{\epsilon\rightarrow0}{\longrightarrow} u_{ 0}~\textrm{a.e. in} ~\mathcal {O}.
 \end{aligned}
\right.
\end{equation*}
Note that the above smooth functions can be found by standard mollification of $n_0$ and $c_0$ for the first two component (cf. \cite{winkler2016global}); the construction for the fluid component is based on the same idea (cf. \cite[Theorem III.4.1]{galdi2011introduction}).

Let us clarify the definition of a local mild solution  to \eqref{SCNS-1}, which allows us to exploit the regularizing properties of the heat and Stokes semigroups.

\begin{definition}\label{def-1}
Let $(\Omega,\mathcal {F},\mathfrak{F},\mathbb{P})$ be the stochastic basis given as before, then a quadruplet $( n_\epsilon ,c_\epsilon,u_\epsilon,\tau_\epsilon)$ is called a \emph{local mild solution} to the system \eqref{SCNS-1} provided:
\begin{itemize}[leftmargin=0.9cm]
\item [$\bullet$] $\tau_\epsilon$ is a $\mathbb{P}$-a.s. strictly positive $\mathfrak{F}$-stopping time.

\item [$\bullet$] for any $q>3$, $(n_\epsilon(\cdot\wedge \tau_\epsilon),c_\epsilon(\cdot\wedge \tau_\epsilon)) \in L^2\left(\Omega;L^\infty(0,\infty; \mathcal{C}^0(\overline{\mathcal {O}})\times W^{1,q}(\mathcal {O}))\right)
$ is an $\mathfrak{F}$-progressively measurable process with $\mathbb{P}$-a.s. continuous sample paths in  $\mathcal{C}^0(\overline{\mathcal {O}})\times W^{1,q}(\mathcal {O})$, such that $n_\epsilon (\cdot\wedge \tau_\epsilon)\geq0$,  $c_\epsilon (\cdot\wedge \tau_\epsilon) \geq0$, $ \mathbb{P}$-a.s.; for any $\alpha \in (\frac{3}{4},1)$, $u_\epsilon(\cdot\wedge \tau_\epsilon) \in L^2(\Omega; L^\infty(0,\infty;\mathscr{D}(\textrm{A}^\alpha)))$ is an $\mathfrak{F}$-progressively measurable process with $\mathbb{P}$-a.s. c\`{a}dl\`{a}g sample paths in $\mathscr{D}(\textrm{A}^\alpha)$.

\item [$\bullet$] for all $t\geq0$, there holds $\mathbb{P}$-a.s.
\begin{equation*}
\begin{split}
n_\epsilon  (t\wedge \tau_\epsilon)  = & e^{t\Delta} n_{ \epsilon0}
- \int_0^{t\wedge \tau_\epsilon}  e^{(t-s)\Delta} \big( u_\epsilon\cdot \nabla n_\epsilon+ \div\left(n_\epsilon \textbf{h}_\epsilon' (n_\epsilon)\chi(c_\epsilon)\nabla c_\epsilon\right)\big)\mathrm{d}s,\\
c_\epsilon  (t\wedge \tau_\epsilon)=& e^{t\Delta} c_{ \epsilon0}
- \int_0^{t\wedge \tau_\epsilon}  e^{(t-s)\Delta}  \big(u_\epsilon\cdot \nabla c _\epsilon  +\textbf{h}_\epsilon(n_\epsilon) f(c_\epsilon)\big)\mathrm{d}s,\\
u_\epsilon  (t\wedge \tau_\epsilon) =& e^{-t\textrm{A}} u_{ \epsilon0}
- \int_0^{t\wedge \tau_\epsilon}  e^{-(t-s)\textrm{A}}  \mathcal {P}\big((\textbf{L}_\epsilon u_\epsilon\cdot \nabla) u_\epsilon -   n_\epsilon\nabla \Phi-h(s,u_\epsilon) \big)\mathrm{d}s\\
+&  \int_0^{t\wedge \tau_\epsilon}  e^{-(t-s)\textrm{A}}   \mathcal {P}g(s,u_\epsilon) \mathrm{d}W_s+ \int_0^{t\wedge \tau_\epsilon} \int_{Z_0} e^{-(t-s)\textrm{A}} \mathcal {P}K(u_\epsilon(s-),z)\widetilde{\pi}(\mathrm{d}s,\mathrm{d}z)\\
 +&  \int_0^{t\wedge \tau_\epsilon} \int_{Z\backslash Z_0} e^{-(t-s)\textrm{A}} \mathcal {P}G(u_\epsilon(s-),z) \pi (\mathrm{d}s,\mathrm{d}z).
\end{split}
\end{equation*}
\end{itemize}

\end{definition}
Here and in what follows, we denote by $(e^{t\Delta})_{t\geq0}$ the Neumann heat semigroup  and by $(e^{-t\textrm{A}} )_{t\geq0}$ the Stokes semigroup with Dirichlet boundary data (cf. \cite{winkler2010aggregation,giga1986solutions}).

\begin{definition}\label{def-2}
A quadruplet $( n_\epsilon ,c_\epsilon,u_\epsilon,\widetilde{\tau}_\epsilon)$ is called a \emph{maximal local mild solution} to the system \eqref{SCNS-1} provided:
\begin{itemize}[leftmargin=0.9cm]
\item [$\bullet$] $\widetilde{\tau}_\epsilon$ is an $\mathbb{P}$-a.s. strictly positive $\mathfrak{F}$-stopping time.

\item [$\bullet$]  $(\tau_R)_{R\geq1}$ is an increasing sequence of $\mathfrak{F}$-stopping times such that $\tau_R<\widetilde{\tau}_\epsilon$ on $[\widetilde{\tau}_\epsilon<\infty]$, $\lim\limits_{R\rightarrow\infty}\tau_R=\widetilde{\tau}_\epsilon$ $\mathbb{P}$-a.s., and
$$
\sup_{t\in [0,\tau_R]} \|(n_\epsilon,c_\epsilon,u_\epsilon)(t)\|_{L^\infty\times W^{1,q}  \times \mathscr{D}(\textrm{A}^\alpha)}\geq R \quad \textrm{on} \quad [\widetilde{\tau}_\epsilon<\infty].
$$

\item [$\bullet$] $( n_\epsilon ,c_\epsilon,u_\epsilon,\tau_R)$ is a local mild solution in the sense of Definition \ref{def-1}.

\end{itemize}

\end{definition}

The main goal of this section is to prove the following result.
\begin{lemma}\label{lem1}
Suppose that the assumptions $(\textbf{A}_1)-(\textbf{A}_4)$ hold. Then the initial-boundary value problem \eqref{SCNS-1} has a unique maximal local mild solution $( n_\epsilon ,c_\epsilon,u_\epsilon,\widetilde{\tau}_\epsilon)$ in the sense of Definition \ref{def-2}.
\end{lemma}

\subsection{Regularized system with truncation}

Let $R>0$,  and $\varphi_R : [0,\infty)\rightarrow [0,1]$ be a smooth cut-off function such that $\varphi_R (s)\equiv 1$ for $0\leq s \leq R$ and $ \varphi_R (s)\equiv0$ for $s > 2R$. Then the regularized SCNS system with cut-off operators takes the following form:
\begin{equation}\label{SCNS-2}
\left\{
\begin{aligned}
&\mathrm{d} n_\epsilon +\varphi_R\left(\|u_\epsilon\|_{\mathscr{D}(\textrm{A}^\alpha)}\right)\varphi_R\left(\|n_\epsilon\|_{L^\infty}\right)u_\epsilon\cdot \nabla n_\epsilon \mathrm{d}t \\
&\quad \  = \Delta n_\epsilon\mathrm{d}t- \varphi_R(\|n_\epsilon\|_{L^\infty})\varphi_R(\|c_\epsilon\|_{W^{1,q}})\div\left(n_\epsilon \textbf{h}_\epsilon' (n_\epsilon)\chi(c_\epsilon)\nabla c_\epsilon\right)\mathrm{d}t,\\
&\mathrm{d} c_\epsilon+ \varphi_R\left(\|u_\epsilon\|_{\mathscr{D}(\textrm{A}^\alpha)}\right)\varphi_R(\|c_\epsilon\|_{W^{1,q}})u_\epsilon\cdot \nabla c _\epsilon \mathrm{d}t \\
&\quad\  =\Delta c_\epsilon\mathrm{d}t -\varphi_R(\|n_\epsilon\|_{L^\infty})\varphi_R(\|c_\epsilon\|_{W^{1,q}})\textbf{h}_\epsilon(n_\epsilon) f(c_\epsilon)\mathrm{d}t,\\
&\mathrm{d} u_\epsilon+ \varphi_R\left(\|u_\epsilon\|_{\mathscr{D}(\textrm{A}^\alpha)}\right) \mathcal {P}(\textbf{L}_\epsilon u_\epsilon\cdot \nabla) u_\epsilon  \mathrm{d}t =-\textrm{A} u_\epsilon\mathrm{d}t\\
&\quad \ + \varphi_R(\|n_\epsilon\|_{L^\infty})\mathcal {P} (n_\epsilon\nabla \Phi) \mathrm{d}t+  \varphi_R\left(\|u_\epsilon\|_{\mathscr{D}(\textrm{A}^\alpha)}\right)\mathcal {P}h(t,u_\epsilon) \mathrm{d}t\\
&\quad\ +  \varphi_R\left(\|u_\epsilon\|_{\mathscr{D}(\textrm{A}^\alpha)}\right)\mathcal {P}g(t,u_\epsilon) \mathrm{d}W_t + \int_{Z_0}   \varphi_R\left(\|u_\epsilon\|_{\mathscr{D}(\textrm{A}^\alpha)}\right)\mathcal {P}K(u_\epsilon(t-),z)\widetilde{\pi}(\mathrm{d}t,\mathrm{d}z)\\
&\quad\  + \int_{Z\backslash Z_0} \varphi_R\left(\|u_\epsilon\|_{\mathscr{D}(\textrm{A}^\alpha)}\right)\mathcal {P}G(u_\epsilon(t-),z) \pi (\mathrm{d}t,\mathrm{d}z),
\end{aligned}
\right.
\end{equation}
together with the initial-boundary conditions:
\begin{align}
&\frac{\partial n_\epsilon}{\partial \nu}  \bigg|_{\partial\mathcal {O}}=\frac{\partial c_\epsilon}{\partial \nu} \bigg|_{\partial\mathcal {O}}=  u_\epsilon  \bigg|_{\partial\mathcal {O}}=0,\label{cc1}\\
&n_\epsilon|_{t=0}=n_{ \epsilon0},~c_\epsilon|_{t=0}=c_{ \epsilon0},~ u_\epsilon|_{t=0}=u_{ \epsilon0}.\label{cc2}
\end{align}

Here and in the proof of Lemma \ref{lem2} below, we shall drop the subscript $R$ in $(n_{R,\epsilon},c_{R,\epsilon},u_{R,\epsilon})$ to avoid abundant notations. We remark that the concept of maximal local mild solutions to \eqref{SCNS-2} can be defined as same as Definitions \ref{def-1}-\ref{def-2}, and we shall omit the details.

\begin{lemma}\label{lem2}
Assume that the conditions $(\textbf{A}_1)-(\textbf{A}_4)$ hold. Then for any $T>0$, the initial-boundary value problem $\eqref{SCNS-2}-\eqref{cc2}$ has a unique mild solution $(n_\epsilon,c_\epsilon,u_\epsilon)$ on $[0,T]$.
\end{lemma}

\begin{proof} [\textbf{\emph{Proof.}}]

\textsc{Step 0:}  The proof involves a reasoning based on the Banach Fixed Point Theorem. To this end, we introduce a Banach space $\mathscr{X}_T$ which consists of all $\mathfrak{F}$-progressively measurable processes with values in $\mathcal{C}^0(\overline{\mathcal {O}}) \times W^{1,q}(\mathcal {O}) \times \mathscr{D}(\textrm{A}^\alpha)$, denoted by
\begin{equation*}
\begin{split}
\mathscr{X}_T &\overset{\textrm{def}}{=}  L^2 \left(\Omega; L^\infty(0,T;\mathcal{C}^0(\overline{\mathcal {O}}))\right) \times L^2\left(\Omega; L^\infty(0,T;W^{1,q}(\mathcal {O}))\right)\\
&\times  L^2\left(\Omega; L^\infty(0,T;\mathscr{D}(\textrm{A}^\alpha))\right), \quad \textrm{for any}~~q \in (3,\infty) ~~\textrm{and} ~\alpha \in (\frac{3}{4},1),
\end{split}
\end{equation*}
which is equipped with the norm
$
\| (n_\epsilon,c_\epsilon,u_\epsilon)\|_{\mathscr{X}_T}^2 \overset{\textrm{def}}{=} \mathbb{E}(\|n_\epsilon\|_{L^\infty(0,T;L^\infty)}^2) +\mathbb{E}(\|c_\epsilon\|_{L^\infty(0,T;W^{1,q})}^2)
+\mathbb{E}(\|u_\epsilon\|_{L^\infty(0,T;\mathscr{D}(\textrm{A}^\alpha))}^2
).
$
With $r>0$ to be specified below, we consider a closed ball  in $\mathscr{X}_T$:
$
B_T(r)= \{(n_\epsilon,c_\epsilon,u_\epsilon)\in \mathscr{X}_T;~ \| (n_\epsilon,c_\epsilon,u_\epsilon)\|_{\mathscr{X}_T} ^2 \leq r\}$. For any $(n_\epsilon,c_\epsilon,u_\epsilon)\in \mathscr{X}_T$, we explore a mapping $\mathscr{F}=(\mathscr{F}_1,\mathscr{F}_2,\mathscr{F}_3)$ given by
\begin{equation*}
\begin{split}
\mathscr{F}_1(n_\epsilon,c_\epsilon,u_\epsilon)(t)&\overset {\textrm{def}} {=}   e^{t\Delta}  n_{\epsilon 0} - \int_0^t  e^{(t-s)\Delta} \big(\varphi_R\left(\|u_\epsilon\|_{\mathscr{D}(\textrm{A}^\alpha)}\right)\varphi_R(\|n_\epsilon\|_{L^\infty})u_\epsilon\cdot \nabla n_\epsilon \\
& + \varphi_R(\|n_\epsilon\|_{L^\infty})\varphi_R(\|c_\epsilon\|_{W^{1,q}})\div\left(n_\epsilon \textbf{h}_\epsilon' (n_\epsilon)\chi(c_\epsilon)\nabla c_\epsilon\right) \big)\mathrm{d}s,\\
\end{split}
\end{equation*}
\begin{equation*}
\begin{split}\mathscr{F}_2(n_\epsilon,c_\epsilon,u_\epsilon)(t)&\overset {\textrm{def}} {=} e^{t\Delta}  c_{\epsilon 0} - \int_0^t  e^{(t-s)\Delta} \big(\varphi_R\left(\|u_\epsilon\|_{\mathscr{D}(\textrm{A}^\alpha)}\right)\varphi_R(\|c_\epsilon\|_{W^{1,q}})u_\epsilon\cdot \nabla c _\epsilon\\
&  +\varphi_R(\|n_\epsilon\|_{L^\infty})\varphi_R(\|c_\epsilon\|_{W^{1,q}})\textbf{h}_\epsilon(n_\epsilon) f(c_\epsilon)  \big)\mathrm{d}s,\\
\end{split}
\end{equation*}
and
\begin{equation*}
\begin{split}
&\mathscr{F}_3(n_\epsilon,c_\epsilon,u_\epsilon)(t)\overset {\textrm{def}} {=}  e^{-t\textrm{A}}  u_{\epsilon 0} \\
&\quad- \int_0^t  e^{-(t-s)\textrm{A}} \mathcal {P}\Big(\varphi_R\left(\|u_\epsilon\|_{\mathscr{D}(\textrm{A}^\alpha)}\right)(\textbf{L}_\epsilon u_\epsilon\cdot \nabla) u_\epsilon-\varphi_R(\|n_\epsilon\|_{L^\infty})n_\epsilon\nabla \Phi -h(s,u_\epsilon)\Big)\mathrm{d}s \\
&\quad+  \int_0^t  e^{-(t-s)\textrm{A}} \varphi_R\left(\|u_\epsilon\|_{\mathscr{D}(\textrm{A}^\alpha)}\right)\mathcal {P}g(s,u_\epsilon) \mathrm{d}W_s  \\
&\quad + \int_0^t\int_{Z_0} e^{-(t-s)\textrm{A}}  \varphi_R\left(\|u_\epsilon\|_{\mathscr{D}(\textrm{A}^\alpha)}\right)\mathcal {P}K(u_\epsilon(s-),z)\widetilde{\pi}(\mathrm{d}s,\mathrm{d}z)\\
&\quad + \int_0^t\int_{Z\backslash Z_0} e^{-(t-s)\textrm{A}}\varphi_R\left(\|u_\epsilon\|_{\mathscr{D}(\textrm{A}^\alpha)}\right)\mathcal {P}G(u_\epsilon(s-),z) \pi (\mathrm{d}s,\mathrm{d}z)  \\
 &\quad\overset{\textrm{def}}{=} e^{-t\textrm{A}}  u_{\epsilon 0}+\mathscr{F}_{31} (t)+\mathscr{F}_{32} (t)+\mathscr{F}_{33} (t)+\mathscr{F}_{34} (t).
\end{split}
\end{equation*}
In the next two steps, we need to show that the mapping $\mathscr{F}(\cdot)$ satisfies the conditions for  applying the Contraction Mapping Principle.

\textsc{Step 1:}  $\mathscr{F}(\cdot): \mathscr{X}_T \mapsto \mathscr{X}_T$ is well-defined. For $\mathscr{F}_1(n_\epsilon,c_\epsilon,u_\epsilon)$, since $q>3$, one can  pick a $\gamma\in (0,\frac{1}{2})$ such that $ 2\gamma-\frac{3}{q}>0$ and  $\mathscr{D}((1-\Delta)^\gamma)\subset \mathcal{C}^0(\overline{\mathcal {O}})$ by the Sobolev  embedding,  where $(1-\Delta)^\gamma$ denotes the Bessel potential.  Applying the smooth effect of Neumann heat semigroup (cf. \cite{winkler2015boundedness}) and the identity $\div (u_\epsilon n_\epsilon)=u_\epsilon \cdot \nabla n_\epsilon$, we have for $t\in [0,T]$
\begin{equation}\label{(3.5)}
\begin{split}
 &\|\mathscr{F}_1(n_\epsilon,c_\epsilon,u_\epsilon)(t)\|_{L^\infty}\\
&\quad\leq \| n_{\epsilon 0}\|_{L^\infty} +C \int_0^t (t-s)^{-\frac{2\gamma+1}{2}}\varphi_R\left(\|u_\epsilon\|_{\mathscr{D}(\textrm{A}^\alpha)}\right)\varphi_R(\|n_\epsilon\|_{L^\infty})\|  u_\epsilon   n_\epsilon  \|_{L^q} \mathrm{d}s  \\
 &\quad +C \int_0^t (t-s)^{-\frac{2\gamma+1}{2}}\varphi_R(\|n_\epsilon\|_{L^\infty})\varphi_R(\|c_\epsilon\|_{W^{1,q}}) \|  n_\epsilon \textbf{h}_\epsilon' (n_\epsilon)\chi(c_\epsilon)\nabla c_\epsilon \|_{L^q} \mathrm{d}s\\
&\quad \leq \| n_{ 0}\|_{L^\infty} +C_RT^{\frac{1}{2}-\gamma} +C_R T <\infty,
 \end{split}
\end{equation}
where the second inequality used the condition $(\textbf{A}_2)$ and the fact of $0\leq\textbf{h}_\epsilon' (n_\epsilon)\leq1$.

For $\mathscr{F}_2(n_\epsilon,c_\epsilon,u_\epsilon)$, by choosing a $\delta \in (\frac{1}{2},1)$ such that $2\delta-\frac{3}{q}>1-\frac{3}{q}$, which makes sense of the estimate $\|\cdot \|_{W^{1,q}}\leq C \|(1-\Delta)^\delta \cdot \|_{L^{q}}$, we gain
\begin{equation}\label{(3.6)}
\begin{split}
 & \|\mathscr{F}_2(n_\epsilon,c_\epsilon,u_\epsilon)(t)\|_{W^{1,q}}\\
 &\quad\leq \| c_{\epsilon 0} \|_{W^{1,q}}+ \int_0^t \Big\|(1-\Delta)^\delta  e^{(t-s)\Delta}\Big[\varphi_R\left\|u_\epsilon\|_{\mathscr{D}(\textrm{A}^\alpha)}\right) \varphi_R(\|c_\epsilon\|_{W^{1,q}})u_\epsilon\cdot \nabla c _\epsilon\\
& \quad+\varphi_R(\|n_\epsilon\|_{L^\infty})\varphi_R(\|c_\epsilon\|_{W^{1,q}})\textbf{h}_\epsilon(n_\epsilon) f(c_\epsilon)  \Big]\Big\|_{L^q}\mathrm{d}s\\
&\quad \leq  \| c_{ 0} \|_{W^{1,q}}+C T^{1-\delta}<\infty,\quad \textrm{for all}~  t\in [0,T].
 \end{split}
\end{equation}

Now let us estimate $\mathscr{F}_3(n_\epsilon,c_\epsilon,u_\epsilon)$. For $\mathscr{F}_{31}(t)$, it follows from the fact of $\alpha\in (\frac{3}{4},1)$ that
$\|\cdot\|_{W^{1,2}}\leq C \|\cdot\|_{\mathscr{D}(\textrm{A}^\alpha)}$ and $\|\cdot\|_{L^\infty}\leq C \|\cdot\|_{\mathscr{D}(\textrm{A}^\alpha)}.$
Since the operator $\textbf{L}_\epsilon $ is bounded from $ L^2(\mathcal {O})$ into $ L^\infty(\mathcal {O})$, we deduce from the smooth effect of Stokes semigroup (cf. \cite{giga1986solutions}) that
\begin{equation}\label{(3.7)}
\begin{split}
&\|\textrm{A}^\alpha(e^{-t\textrm{A}}  u_{\epsilon 0}+\mathscr{F}_{31} (t))\|_{L^2} \\
&\quad \leq \|\textrm{A}^\alpha u_{\epsilon 0}\|_{L^2}+ \int_0^t (t-s)^{-\alpha}\Big[ \varphi_R\left(\|u_\epsilon\|_{\mathscr{D}(\textrm{A}^\alpha)}\right)\|(\textbf{L}_\epsilon u_\epsilon\cdot \nabla) u_\epsilon \|_{L^2}\\
&\quad +\varphi_R(\|n_\epsilon\|_{L^\infty})\|n_\epsilon\nabla \Phi \|_{L^2}+ \varphi_R\left(\|u_\epsilon\|_{\mathscr{D}(\textrm{A}^\alpha)}\right)\left(1+\|u_\epsilon\|_{ L^2} \right) \Big] \mathrm{d}s\\
&\quad\leq\| u_{ 0}\|_{\mathscr{D}(\textrm{A}^\alpha)}+C T ^{1-\alpha} <\infty,\quad \textrm{for all}~  t\in [0,T].
  \end{split}
  \end{equation}
For $\mathscr{F}_{32}(t)$, it will be convenient to rewrite the integral into the evolution equation
$
\textrm{A}^\alpha\mathscr{F}_{32}( t)= \int_0^t\textrm{A}^{\alpha+1}\mathscr{F}_{32}(s)\mathrm{d}s+  \int_0^t \varphi_R\left(\|u_\epsilon\|_{\mathscr{D}(\textrm{A}^\alpha)}\right)\textrm{A}^\alpha\mathcal {P}g(s,u_\epsilon) \mathrm{d}W.
$
By applying the Burkholder-Davis-Gundy (BDG) inequality (cf. \cite[Theorem 3.50]{peszat2007stochastic}), the boundedness of the projection $\mathcal {P}$ as well as the assumption $(\textbf{A}_3)$, we obtain
\begin{equation*}
\begin{split}
\mathbb{E}\sup_{t\in [0,T]}\|\textrm{A}^\alpha\mathscr{F}_{32} ( t)\|_{L^2}^2&\leq \mathbb{E}  \int_0^T\varphi_R^2\left(\|u_\epsilon\|_{\mathscr{D}(\textrm{A}^\alpha)}\right)\|\textrm{A}^\alpha\mathcal {P}g(t,u_\epsilon)\|_{\mathcal {L}_2(U;L^2)}^2 \mathrm{d}t
 \\
  &   +2\mathbb{E}\sup_{t\in [0,T]}\left| \int_0^t\varphi_R\left(\|u_\epsilon\|_{\mathscr{D}(\textrm{A}^\alpha)}\right)\langle \textrm{A}^\alpha\mathscr{F}_{32}, \textrm{A}^\alpha\mathcal {P}g(s,u_\epsilon) \mathrm{d}W_s\rangle_{L^2}\right|\\
   & \leq C  T+ C  \mathbb{E} \int_0^T\sup_{s\in [0,t]} \left\| \textrm{A}^\alpha\mathscr{F}_{32}(s)\right\|_{L^2}^2 \mathrm{d}t,
 \end{split}
\end{equation*}
which combined with the Gronwall inequality leads to
\begin{equation}\label{(3.8)}
\begin{split}
 \mathbb{E}\sup_{t\in [0,T]}\|\textrm{A}^\alpha\mathscr{F}_{32} ( t)\|_{L^2}^2 \leq  C  T.
 \end{split}
\end{equation}
For $\mathscr{F}_{33}+\mathscr{F}_{34} \overset {\textrm{def}} {=}\mathscr{F}_{35}$, note that $\textrm{A}^\alpha\mathscr{F}_{35}$ satisfies the identity
\begin{equation*}
\begin{split}
\textrm{A}^\alpha\mathscr{F}_{35}(t)&=  \int_0^t \textrm{A}^{\alpha+1}\mathscr{F}_{35} (s)\mathrm{d}s+ \int_0^t \int_{Z_0}  \varphi_R\left(\|u_\epsilon\|_{\mathscr{D}(\textrm{A}^\alpha)}\right)\textrm{A}^\alpha\mathcal {P}K(u_\epsilon(s-),z)\widetilde{\pi}(\mathrm{d}s,\mathrm{d}z)\\
&+ \int_0^t \int_{Z\backslash Z_0} \varphi_R\left(\|u_\epsilon\|_{\mathscr{D}(\textrm{A}^\alpha)}\right) \textrm{A}^\alpha \mathcal {P}G(u_\epsilon(s-),z) \pi (\mathrm{d}s,\mathrm{d}z).
 \end{split}
\end{equation*}
By applying the It\^{o} formula in Hilbert spaces (cf. \cite[Theorem 3.50]{peszat2007stochastic} or \cite[Theorem 1]{gyongy1980stochastic}) to the process $ \|\textrm{A}^\alpha\mathscr{F}_{35}  (t)\|_{L^2}^2$, after integrating by parts and using the identity $\widetilde{\pi}(\textrm{d}t,\textrm{d}z)=\pi(\textrm{d}t,\textrm{d}z)-\mu(\textrm{d}z)\textrm{d}t$ to rearrange the integrals, we obtain
{\wuhao
\begin{equation}\label{(3.9)}
\begin{split}
 &\|\textrm{A}^\alpha\mathscr{F}_{35}  (t)\|_{L^2}^2+ 2 \int_0^t\| \textrm{A}^{\alpha+\frac{1}{2}}\mathscr{F}_{35} (s)\|_{L^2} \mathrm{d}s \\
 &\quad\leq\bigg| \int_0^t \int_{Z\backslash Z_0}2\varphi_R\left(\|u_\epsilon\|_{\mathscr{D}(\textrm{A}^\alpha)}\right)\langle \textrm{A}^\alpha\mathscr{F}_{35}  (s-), \textrm{A}^\alpha \mathcal {P}G(u_\epsilon(s-),z)\rangle_{L^2} \mu(\mathrm{d} z)\mathrm{d}s\bigg|\\
 & \quad+ \int_0^t \int_{Z_0} \varphi_R^2\left(\|u_\epsilon\|_{\mathscr{D}(\textrm{A}^\alpha)}\right)\|\textrm{A}^\alpha \mathcal {P}K(u_\epsilon(s-),z)\|_{L^2}^2\mu(\mathrm{d} z)\mathrm{d}s\\
 & \quad+ \int_0^t \int_{Z\backslash Z_0}  \varphi_R^2\left(\|u_\epsilon\|_{\mathscr{D}(\textrm{A}^\alpha)}\right)\|\textrm{A}^\alpha \mathcal {P}G (u_\epsilon(x,s-),z) \|_{L^2}^2 \mu(\mathrm{d} z)\mathrm{d}s\\
 & \quad +\bigg| \int_0^t \int_{Z_0} \Big(2\varphi_R\left(\|u_\epsilon\|_{\mathscr{D}(\textrm{A}^\alpha)}\right)\langle \textrm{A}^\alpha\mathscr{F}_{35}  (s-),\textrm{A}^\alpha \mathcal {P}K(u_\epsilon(s-),z)\rangle_{L^2}\\
 & \quad  + \varphi_R^2\left(\|u_\epsilon\|_{\mathscr{D}(\textrm{A}^\alpha)}\right)\|\textrm{A}^\alpha \mathcal {P}K(u_\epsilon(s-),z) \|_{L^2}^2 \Big)\widetilde{\pi}(\mathrm{d}s,\mathrm{d}z)\bigg|\\
 & \quad+\bigg| \int_0^t \int_{Z\backslash Z_0} \Big( 2\varphi_R\left(\|u_\epsilon\|_{\mathscr{D}(\textrm{A}^\alpha)}\right)\langle \textrm{A}^\alpha\mathscr{F}_{35}  (s-), \textrm{A}^\alpha \mathcal {P}G(u_\epsilon(s-),z)\rangle_{L^2}\\
 & \quad + \varphi_R^2\left(\|u_\epsilon\|_{\mathscr{D}(\textrm{A}^\alpha)}\right)\|\textrm{A}^\alpha \mathcal {P}G (u_\epsilon(x,s-),z)\|_{L^2}^2 \Big)\widetilde{\pi}(\mathrm{d}s,\mathrm{d}z)\bigg| \\
 &\quad\overset{\textrm{def}}{=}  I_1(t)+\cdot\cdot\cdot+I_5 (t).
 \end{split}
\end{equation}}\noindent
Taking the supremum for $I_1(t)$ over interval $[0,T]$, and then taking the mathematical expectation, we get by assumption $(\textbf{A}_4)$ and the BDG inequality that
\begin{equation}\label{(3.10)}
\begin{split}
\mathbb{E}\sup_{t\in [0,T]}I_1(t)
 &\leq  \mathbb{E}  \int_0^T \|\textrm{A}^\alpha\mathscr{F}_{35}  \|_{L^2}^2\mathrm{d}s+C_\mu\mathbb{E} \int_0^T\varphi_R^2\left(\|u_\epsilon\|_{\mathscr{D}(\textrm{A}^\alpha)}\right)
 \left(1+\|u_\epsilon\|^2_{\mathscr{D}(\textrm{A}^\alpha)} \right)\mathrm{d}s\\
 &\leq  \mathbb{E}  \int_0^T \|\textrm{A}^\alpha\mathscr{F}_{35}  (s)\|_{L^2}^2\mathrm{d}s+C T.
 \end{split}
\end{equation}
Here $C_\mu\overset{\textrm{def}}{=} \int_{Z\backslash Z_0} \mu(\mathrm{d} z) = \mu (Z\backslash Z_0)<\infty$. For $I_2(t)$ and $I_3(t)$, we conclude that
\begin{equation}\label{(3.11)}
\begin{split}
\mathbb{E}\sup_{t\in [0,T]}(I_2+I_3)(t) \leq C  \int_0^T\varphi_R^2\left(\|u_\epsilon\|_{\mathscr{D}(\textrm{A}^\alpha)}\right)\left(1+\|u_\epsilon\|^2_{\mathscr{D}(\textrm{A}^\alpha)} \right)\mathrm{d}s  \leq C T.
 \end{split}
\end{equation}
For $I_4(t)$, on  the one hand, we get by applying the BDG inequality that
\begin{equation*}
\begin{split}
\mathbb{E}\sup_{t\in [0,T]} I_{41}(t)
 & \leq C\mathbb{E}\left[\sup_{t\in [0,T]}\|\textrm{A}^\alpha\mathscr{F}_{35}  (t)\|_{L^2} \left( \int_0^T\varphi_R^2\left(\|u_\epsilon\|_{\mathscr{D}(\textrm{A}^\alpha)}\right) \left(1+\|u_\epsilon\|^2_{\mathscr{D}(\textrm{A}^\alpha)} \right) \mathrm{d} s\right)^{\frac{1}{2}}\right]\\
 &\leq \epsilon_1\mathbb{E} \sup_{t\in [0,T]}\|\textrm{A} ^\alpha \mathscr{F}_{35}  (t)\|_{L^2}^2+C  \int_0^T\varphi_R\left(\|u_\epsilon\|_{\mathscr{D}(\textrm{A}^\alpha)}\right)\left(1+\|u_\epsilon\|^2_{\mathscr{D}(\textrm{A}^\alpha)} \right)\mathrm{d}s\\
 &\leq \epsilon_1\mathbb{E} \sup_{t\in [0,T]}\|\textrm{A}^\alpha\mathscr{F}_{35}  (t)\|_{L^2}^2+C  T,\quad \textrm{for any}~ \epsilon_1 >0.
 \end{split}
\end{equation*}
On the other hand,  it follows from the assumption on $K$ and the BDG inequality  that
\begin{equation*}
\begin{split}
\mathbb{E}\sup_{t\in [0,T]} I_{42}(t) &\leq 2\mathbb{E}\left( \int_0^T \int_{Z_0}  \varphi_R^4(\|u_\epsilon\|_{\mathscr{D}(\textrm{A}^\alpha)})\|\textrm{A}^\alpha  K \|_{L^2}^4 \mu(\mathrm{d}z) \mathrm{d} s\right)^{\frac{1}{2}}\\
&\leq C\mathbb{E}\left( \int_0^T\varphi_R^4(\|u_\epsilon\|_{\mathscr{D}(\textrm{A}^\alpha)})\left(1+\|u_\epsilon\|^4_{\mathscr{D}(\textrm{A}^\alpha)} \right) \mathrm{d} s\right)^{\frac{1}{2}} \leq C  T^{\frac{1}{2}}.
 \end{split}
\end{equation*}
The last two inequalities imply that
\begin{equation}\label{(3.12)}
\begin{split}
\mathbb{E}\sup_{t\in [0,T]} I_{4}(t) \leq \epsilon_1\mathbb{E} \sup_{t\in [0,T]}\|\textrm{A}^\alpha\mathscr{F}_{35}  (t)\|_{L^2}^2+C  T+C  T^{\frac{1}{2}}.
 \end{split}
\end{equation}
Proceeding similarly, by using the BDG inequality and the assumption on $G$, one has
\begin{equation}\label{(3.13)}
\begin{split}
&\mathbb{E}\sup_{t\in [0,T]} I_{5}(t) \leq \epsilon_2\mathbb{E} \sup_{t\in [0,T]}\|\textrm{A}^\alpha\mathscr{F}_{35}  (t)\|_{L^2}^2+C  T+C  T^{\frac{1}{2}},\quad \textrm{for any}~ \epsilon_2>0.
 \end{split}
\end{equation}
Substituting the estimate \eqref{(3.10)}-\eqref{(3.13)} into \eqref{(3.9)}, after choosing $\epsilon_1$, $\epsilon_2>0$ small enough such that $\epsilon_1+\epsilon_2\leq \frac{1}{2}$, we obtain
\begin{equation*}
\begin{split}
 \mathbb{E}\sup_{t\in [0,T]} \|\textrm{A}^\alpha\mathscr{F}_{35}  (t)\|_{L^2}^2\leq  \mathbb{E}  \int_0^T \|\textrm{A}^\alpha\mathscr{F}_{35}  (t)\|_{L^2}^2\mathrm{d}t+C \left(T+T^{\frac{1}{2}}\right),
 \end{split}
\end{equation*}
which implies
\begin{equation}\label{(3.14)}
\begin{split}
 \mathbb{E}\sup_{t\in [0,T]} \|\textrm{A}^\alpha(\mathscr{F}_{33} +\mathscr{F}_{34}) (t)\|_{L^2}^2\leq  C \left(T+T^{\frac{1}{2}}\right).
 \end{split}
\end{equation}
As a result, it follows from \eqref{(3.7)}, \eqref{(3.8)} and \eqref{(3.14)} that
\begin{equation*}
\begin{split}
 &\mathbb{E}\sup_{t\in [0,T]} \|\textrm{A}^\alpha\mathscr{F}_{3} (n_\epsilon,c_\epsilon,u_\epsilon)  (t)\|_{L^2}^2\leq  4\| u_{ 0}\|_{\mathscr{D}(\textrm{A}^\alpha)}^2 +C \left(T^{\frac{1}{2}}+T+T^2+T ^{2-2\alpha}\right) ,
 \end{split}
\end{equation*}
which together with \eqref{(3.5)} and \eqref{(3.6)} leads to
\begin{equation}\label{vv}
\begin{split}
 \|\mathscr{F}(n_\epsilon,c_\epsilon,u_\epsilon)\|_{\mathscr{X}_T}^2 &\leq 3\| n_{ 0}\|_{L^\infty} ^2 + 2\| c_{ 0} \|_{W^{1,q}}^2 +4\| u_{ 0}\|_{\mathscr{D}(\textrm{A}^\alpha)}^2\\
 &+C \left(T+T^2 +T^{\frac{1}{2}}+T^{1-2\gamma}+T^{2-2\delta}+T ^{2-2\alpha}\right).
 \end{split}
\end{equation}
Choosing $r>0$ such that $r^2 = 6\| n_{ 0}\|_{L^\infty} ^2 + 4\| c_{ 0} \|_{W^{1,q}}^2 +8\| u_{ 0}\|_{\mathscr{D}(\textrm{A}^\alpha)}^2$
and then $T>0$ small enough such that
$C  (T+T^2 +T^{\frac{1}{2}}+T^{1-2\gamma}+T^{2-2\delta}+T ^{2-2\alpha} )\leq\frac{r^2}{2}$, we get from \eqref{vv} that
\begin{equation*}
\begin{split}
 \|\mathscr{F}(n_\epsilon,c_\epsilon,u_\epsilon)\|_{\mathscr{X}_T} \leq r,\quad \forall (n_\epsilon,c_\epsilon,u_\epsilon)\in B_T(r) .
 \end{split}
\end{equation*}

On the other hand, by using the strong continuity of the semigroups, the assumptions ($\textbf{A}_3$)-($\textbf{A}_4$) and performing the similar estimates as before, one can prove that
\begin{equation}\label{mn}
\begin{split}
\lim_{t\rightarrow s}\mathbb{E}\left(\|\mathscr{F}(n_\epsilon,c_\epsilon,u_\epsilon)(t) -\mathscr{F}(n_\epsilon,c_\epsilon,u_\epsilon)(s) \|_{L^\infty\times W^{1,q}\times \mathscr{D}(\textrm{A}^\alpha)}^2\right)=0,
 \end{split}
\end{equation}
for all $0\leq s \leq t \leq T$, which implies the stochastic continuity of  $\mathscr{F}(n_\epsilon,c_\epsilon,u_\epsilon)$. Hence, as the stochastic continuity of an adapted process implies the existence of a predictable version (cf. \cite[Proposition 3.21]{peszat2007stochastic}), it follows from \eqref{mn} that the mapping $\mathscr{F}: B_T(r)\subseteq \mathscr{X}_T \mapsto B_T(r)$ is well-defined.

\textsc{Step 2.}  $\mathscr{F}(\cdot):B_T(r)\mapsto B_T(r)$ is a contraction for $T>0$ small enough. For any $(n_\epsilon,c_\epsilon,u_\epsilon)$, $(\bar{n}_\epsilon,\bar{c}_\epsilon,\bar{u}_\epsilon)\in B_T(r)$, and any $\gamma\in (\frac{3}{2q},\frac{1}{2})$, we get for the $n_\epsilon$-equation that
\begin{equation}\label{(3.15)}
\begin{split}
&\|\mathscr{F}_1(n_\epsilon,c_\epsilon,u_\epsilon)(t)-\mathscr{F}_1(\bar{n}_\epsilon,\bar{c}_\epsilon,
\bar{u}_\epsilon)(t)\|_{L^\infty}\\
&\quad  \leq  \int_0^t(t-s)^{-\gamma-\frac{1}{2}} \big\| \varphi_R(\|u_\epsilon\|_{\mathscr{D}(\textrm{A}^\alpha)})\varphi_R(\|n_\epsilon\|_{L^\infty})u_\epsilon  n_\epsilon\\
& \quad \quad \quad\quad -\varphi_R(\|\bar{u}_\epsilon\|_{\mathscr{D}(\textrm{A}^\alpha)})\varphi_R(\|\bar{n}_\epsilon\|_{L^\infty})\bar{u}_\epsilon \bar{n}_\epsilon \big\|_{L^q}\mathrm{d}s\\
 & \quad + \int_0^t (t-s)^{-\gamma-\frac{1}{2}}\big\| \varphi_R(\|n_\epsilon\|_{L^\infty})\varphi_R(\|c_\epsilon\|_{W^{1,q}}) n_\epsilon \textbf{h}_\epsilon' (n_\epsilon)\chi(c_\epsilon)\nabla c_\epsilon \\
&\quad \quad \quad\quad -\varphi_R(\|\bar{n}_\epsilon\|_{L^\infty})\varphi_R(\|\bar{c}_\epsilon\|_{W^{1,q}}) \bar{n}_\epsilon \textbf{h}_\epsilon' (\bar{n}_\epsilon)\chi(\bar{c}_\epsilon)\nabla \bar{c}_\epsilon \big\|_{L^q}\mathrm{d}s \\
& \quad\overset{\textrm{def}}{=}  J_1(t) +J_2(t).
 \end{split}
\end{equation}
The term $J_1(t)$ can be handled in a similar manner as that in \cite{zhai20202d} and we have
\begin{equation}\label{(3.16)}
\begin{split}
J_1(t) \leq C_R T^{\frac{1}{2}-\gamma} \left(\|n_\epsilon-\bar{n}_\epsilon\|_{L^\infty(0,T;L^\infty)}
+\|u_\epsilon-\bar{u}_\epsilon\|_{L^\infty(0,T;\mathscr{D}(\textrm{A}^\alpha))}\right).
 \end{split}
\end{equation}
For $J_2(t)$, define
$$
\mathcal {G}_\epsilon \overset{\textrm{def}}{=} \varphi_R(\|n_\epsilon\|_{L^\infty})\varphi_R(\|c_\epsilon\|_{W^{1,q}})  n_\epsilon \textbf{h}_\epsilon' (n_\epsilon)\chi(c_\epsilon)\nabla c_\epsilon -\varphi_R(\|\bar{n}_\epsilon\|_{L^\infty})\varphi_R(\|\bar{c}_\epsilon\|_{W^{1,q}}) \bar{n}_\epsilon \textbf{h}_\epsilon' (\bar{n}_\epsilon)\chi(\bar{c}_\epsilon)\nabla \bar{c}_\epsilon.
$$
We will divide the estimate for $J_2$ into several cases.
\begin{itemize}[leftmargin=0.9cm]
\item [$\bullet$] If $\max\{\|n_\epsilon\|_{L^\infty(0,T;L^\infty)}, \|c_\epsilon\|_{L^\infty(0,T;W^{1,q})}\}> 2R$ and  $\max\{\|\bar{n}_\epsilon\|_{L^\infty(0,T;L^\infty)}, \|\bar{c}_\epsilon\|_{L^\infty(0,T;W^{1,q})}\}> 2R$, then $
\|\mathcal {G}_\epsilon\|_{L^q}=0 $.
\item [$\bullet$] If $\max\{\|n_\epsilon\|_{L^\infty(0,T;L^\infty)}, \|c_\epsilon\|_{L^\infty(0,T;W^{1,q})}\}> 2R$ and $\|\bar{n}_\epsilon\|_{L^\infty(0,T;L^\infty)}  \leq2R$, then
\begin{equation*}
\begin{split}
\|\mathcal {G}_\epsilon\|_{L^q}
&\leq \max_{s\in [0,Cr]}\chi(s)\varphi_R(\|\bar{n}_\epsilon\|_{L^\infty})\varphi_R(\|\bar{c}_\epsilon\|_{W^{1,q}})\| \bar{n}_\epsilon \|_{ L^\infty}^2\|\nabla \bar{c}_\epsilon\|_{L^q}\\
&\leq 8R^3\max_{s\in [0,Cr]}\chi(s)(\varphi_R(\|\bar{n}_\epsilon\|_{L^\infty})-\varphi_R(\|n_\epsilon\|_{L^\infty})) \leq C \|n_\epsilon-\bar{n}_\epsilon\|_{L^\infty(0,T;L^\infty)},
 \end{split}
\end{equation*}
where we used $\|\bar{c}_\epsilon\|_{L^\infty}\leq C\|\bar{c}_\epsilon\|_{W^{1,q}}\leq Cr$.

\item [$\bullet$] If $\max\{\|n_\epsilon\|_{L^\infty(0,T;L^\infty)}, \|c_\epsilon\|_{L^\infty(0,T;W^{1,q})}\}> 2R$ and $\|\bar{c}_\epsilon\|_{L^\infty(0,T;L^\infty)}  \leq2R$,  then
\begin{equation*}
\begin{split}
\|\mathcal {G}_\epsilon\|_{L^q}
&\leq \max_{s\in [0,Cr]}\chi(s)\varphi_R(\|\bar{n}_\epsilon\|_{L^\infty})\varphi_R(\|\bar{c}_\epsilon\|_{W^{1,q}})\| \bar{n}_\epsilon \|_{ L^\infty}^2\|\nabla \bar{c}_\epsilon\|_{L^q}\leq C\|c_\epsilon-\bar{c}_\epsilon\|_{L^\infty(0,T;L^\infty)}.
 \end{split}
\end{equation*}

\item [$\bullet$] By symmetry, if $\max\{\|\bar{n}_\epsilon\|_{L^\infty(0,T;L^\infty)}, \|\bar{c}_\epsilon\|_{L^\infty_TW^{1,q}}\}> 2R$, $\|n_\epsilon\|_{L^\infty(0,T;L^\infty)}  \leq2R$, then
\begin{equation*}
\begin{split}
\|\mathcal {G}_\epsilon\|_{L^q} \leq C \|n_\epsilon-\bar{n}_\epsilon\|_{L^\infty(0,T;L^\infty)}.
 \end{split}
\end{equation*}

\item [$\bullet$] By symmetry, if $\max\{\|\bar{n}_\epsilon\|_{L^\infty(0,T;L^\infty)}, \|\bar{c}_\epsilon\|_{L^\infty(0,T;W^{1,q})}\}> 2R$, $\|c_\epsilon\|_{L^\infty(0,T;L^\infty)}  \leq2R$, then
\begin{equation*}
\begin{split}
\|\mathcal {G}_\epsilon\|_{L^q} \leq C \|c_\epsilon-\bar{c}_\epsilon\|_{L^\infty(0,T;L^\infty)}.
 \end{split}
\end{equation*}

\item [$\bullet$] If $\max\{\|n_\epsilon\|_{L^\infty(0,T;L^\infty)}, \|c_\epsilon\|_{L^\infty(0,T;W^{1,q})} , \|\bar{n}_\epsilon\|_{L^\infty(0,T;L^\infty)}, \|\bar{c}_\epsilon\|_{L^\infty(0,T;W^{1,q})}\} \leq 2R$, then
\begin{equation*}
\begin{split}
\|\mathcal {G}_\epsilon\|_{L^q} &\leq
\varphi_R(\|n_\epsilon\|_{L^\infty})\varphi_R(\|c_\epsilon\|_{W^{1,q}}) \| (n_\epsilon-\bar{n}_\epsilon) \textbf{h}_\epsilon' (n_\epsilon)\chi(c_\epsilon)\nabla c_\epsilon \|_{L^q} \\
&+\varphi_R(\|n_\epsilon\|_{L^\infty})\varphi_R(\|c_\epsilon\|_{W^{1,q}}) \|\bar{n}_\epsilon  ( \textbf{h}_\epsilon' (n_\epsilon)  -   \textbf{h}_\epsilon' (\bar{n}_\epsilon))\chi(c_\epsilon)\nabla c_\epsilon )\|_{L^q}\\
&+\varphi_R(\|n_\epsilon\|_{L^\infty})\varphi_R(\|c_\epsilon\|_{W^{1,q}}) \|\bar{n}_\epsilon \textbf{h}_\epsilon' (\bar{n}_\epsilon) ( \chi(c_\epsilon)  -    \chi(\bar{c}_\epsilon))\nabla c_\epsilon\|_{L^q}  \\
&+\varphi_R(\|n_\epsilon\|_{L^\infty})\varphi_R(\|c_\epsilon\|_{W^{1,q}}) \|\bar{n}_\epsilon \textbf{h}_\epsilon' (\bar{n}_\epsilon) \chi(\bar{c}_\epsilon)( \nabla c_\epsilon -\nabla \bar{c}_\epsilon )\|_{L^q}\\
&+\big|\varphi_R(\|n_\epsilon\|_{L^\infty}) -\varphi_R(\|\bar{n}_\epsilon\|_{L^\infty})\big|\varphi_R(\|c_\epsilon\|_{W^{1,q}}) \|\bar{n}_\epsilon \textbf{h}_\epsilon' (\bar{n}_\epsilon)\chi(\bar{c}_\epsilon)\nabla \bar{c}_\epsilon\|_{L^q}\\
&+\varphi_R(\|\bar{n}_\epsilon\|_{L^\infty})\big|\varphi_R(\|c_\epsilon\|_{W^{1,q}}) - \varphi_R(\|\bar{c}_\epsilon\|_{W^{1,q}})\big| \|\bar{n}_\epsilon \textbf{h}_\epsilon' (\bar{n}_\epsilon)\chi(\bar{c}_\epsilon)\nabla \bar{c}_\epsilon\|_{L^q}.
   \end{split}
\end{equation*}
According to the definition of $\varphi_R$, $\textbf{h}_\epsilon $ and $\chi$, the Mean Value Theorem as well as the embedding $W^{1,q}(\mathcal {O})\subset L^\infty (\mathcal {O})$ for $q>3$, one can deduce that
\begin{equation*}
\begin{split}
\|\mathcal {G}_\epsilon\|_{L^q} &\leq
C_{R,\chi} \Big(\| n_\epsilon-\bar{n}_\epsilon \|_{L^\infty}  + \left\| \textbf{h}_\epsilon' (n_\epsilon)  -   \textbf{h}_\epsilon' (\bar{n}_\epsilon) \right\|_{L^\infty} + \| \chi(c_\epsilon)  -    \chi(\bar{c}_\epsilon) \|_{L^\infty} \\
&+ \| \nabla c_\epsilon -\nabla \bar{c}_\epsilon \|_{L^q} + \big|\varphi_R(\|n_\epsilon\|_{L^\infty}) -\varphi_R(\|\bar{n}_\epsilon\|_{L^\infty})\big| \\
&+  \big|\varphi_R(\|c_\epsilon\|_{W^{1,q}}) - \varphi_R(\|\bar{c}_\epsilon\|_{W^{1,q}})\big|  \Big)\\
&\leq
C (\| n_\epsilon-\bar{n}_\epsilon \|_{L^\infty}+ \| \nabla c_\epsilon -\nabla \bar{c}_\epsilon \|_{L^q}).
   \end{split}
\end{equation*}
\end{itemize}
We conclude from the above discussion that
\begin{equation}\label{(3.17)}
\begin{split}
J_2(t)&\leq C  \int_0^T (T-s)^{-\gamma-\frac{1}{2}} (\| n_\epsilon-\bar{n}_\epsilon \|_{L^\infty}+ \| \nabla c_\epsilon -\nabla \bar{c}_\epsilon \|_{L^q})\mathrm{d}s\\
& \leq CT^{\frac{1}{2}-\gamma} \left(\|n_\epsilon-\bar{n}_\epsilon\|_{L^\infty(0,T;L^\infty)}
+\|c_\epsilon-\bar{c}_\epsilon\|_{L^\infty(0,T;W^{1,q})}\right).
 \end{split}
\end{equation}
From estimates \eqref{(3.15)}-\eqref{(3.17)}, we obtain for any $\gamma\in (\frac{3}{2q},\frac{1}{2})$
\begin{equation}\label{(3.18)}
\begin{split}
 &\mathbb{E}\sup_{t\in [0,T]}\|\mathscr{F}_1(n_\epsilon,c_\epsilon,u_\epsilon)-\mathscr{F}_1(\bar{n}_\epsilon,\bar{c}_\epsilon,
\bar{u}_\epsilon)\|_{L^\infty}^2\\
 &\quad\leq  C T^{1-2\gamma} \Big(\mathbb{E}\|n_\epsilon-\bar{n}_\epsilon\|_{L^\infty(0,T;L^\infty)}^2
 +\mathbb{E}\|c_\epsilon-\bar{c}_\epsilon\|_{L^\infty(0,T;W^{1,q})}^2+\mathbb{E}\|u_\epsilon-\bar{u}_\epsilon\|_{L^\infty(0,T;\mathscr{D}(\textrm{A}^\alpha))}^2\Big).
 \end{split}
\end{equation}
Similarly, for the $c_\epsilon$-equation, one can prove that for any $\delta \in (\frac{1}{2},1)$
\begin{equation}\label{(3.19)}
\begin{split}
&\mathbb{E}\sup_{t\in [0,T]} \|\mathscr{F}_2(n_\epsilon,c_\epsilon,u_\epsilon)-\mathscr{F}_2(\bar{n}_\epsilon,\bar{c}_\epsilon,
\bar{u}_\epsilon)\|_{W^{1,q}}^2\\
 &\quad\leq  C T^{2-2\delta} \Big(\mathbb{E}\|n_\epsilon-\bar{n}_\epsilon\|_{L^\infty(0,T;L^\infty)}^2
 +\mathbb{E}\|c_\epsilon-\bar{c}_\epsilon\|_{L^\infty(0,T;W^{1,q})}^2+\mathbb{E}\|u_\epsilon-\bar{u}_\epsilon\|_{L^\infty(0,T;\mathscr{D}(\textrm{A}^\alpha))}^2\Big).
 \end{split}
\end{equation}
It remains to estimate $\|\mathscr{F}_3(n_\epsilon,c_\epsilon,u_\epsilon)(t)-\mathscr{F}_3(\bar{n}_\epsilon,\bar{c}_\epsilon
,\bar{u}_\epsilon)(t)\|_{\mathscr{D}(\textrm{A}^\alpha)}$ with respect to the $u_\epsilon$-equation. Thanks to  the smooth effect of  Stokes semigroup, we have
\begin{equation}\label{(3.20)}
\begin{split}
&\|\textrm{A}^\alpha(\mathscr{F}_3(n_\epsilon,c_\epsilon,u_\epsilon) -\mathscr{F}_3(\bar{n}_\epsilon,\bar{c}_\epsilon
,\bar{u}_\epsilon))\|_{L^2}\\
&\quad\leq  \int_0^t  (t-s)^{-\alpha}\left\|\varphi_R\left(\|u_\epsilon\|_{\mathscr{D}(\textrm{A}^\alpha)}\right)(\textbf{L}_\epsilon u_\epsilon\cdot \nabla) u_\epsilon-\varphi_R(\|\bar{u}_\epsilon\|_{\mathscr{D}(\textrm{A}^\alpha)})(\textbf{L}_\epsilon \bar{u}_\epsilon\cdot \nabla) \bar{u}_\epsilon\right\|_{L^2}\mathrm{d}s\\
& \quad+  \int_0^t (t-s)^{-\alpha}\left\| \varphi_R(\|n_\epsilon\|_{L^\infty})n_\epsilon\nabla \Phi-\varphi_R(\|\bar{n}_\epsilon\|_{L^\infty})\bar{n}_\epsilon\nabla \Phi \right\|_{L^2}\mathrm{d}s \\
&\quad +  \int_0^t (t-s)^{-\alpha}\left\| \varphi_R(\|u_\epsilon\|_{L^\infty})h(s,u_\epsilon)-\varphi_R(\|\bar{u}_\epsilon\|_{L^\infty})h(s,\bar{u}_\epsilon) \right\|_{L^2}\mathrm{d}s\\
&  \quad+  \Bigg\| \int_0^t\textrm{A}^\alpha  e^{-(t-s)\textrm{A}} \mathcal {P}\Big(\varphi_R\left(\|u_\epsilon\|_{\mathscr{D}(\textrm{A}^\alpha)}\right) g(s,u_\epsilon)-\varphi_R(\|\bar{u}_\epsilon\|_{\mathscr{D}(\textrm{A}^\alpha)})g(s,\bar{u}_\epsilon)\Big) \mathrm{d}W_s\Bigg\|_{L^2} \\
 &\quad +\Bigg\| \int_0^t \int_{Z_0}\textrm{A}^\alpha  e^{-(t-s)\textrm{A}}  \Big(\varphi_R\left(\|u_\epsilon\|_{\mathscr{D}(\textrm{A}^\alpha)}\right)\mathcal {P}K(u_\epsilon(s-),z) \\
 &\quad\quad\quad\quad\quad\quad-\varphi_R(\|\bar{u}_\epsilon\|_{\mathscr{D}(\textrm{A}^\alpha)}) \mathcal {P}K(\bar{u}_\epsilon(s-),z)\Big) \widetilde{\pi}(\mathrm{d}s,\mathrm{d}z) \\
 &\quad\quad\quad\quad\quad\quad\quad \quad +  \int_0^t \int_{Z\backslash Z_0} \textrm{A}^\alpha  e^{-(t-s)\textrm{A}}  \Big(\varphi_R\left(\|u_\epsilon\|_{\mathscr{D}(\textrm{A}^\alpha)}\right)\mathcal {P}G(u_\epsilon(s-),z) \\
 &\quad\quad\quad\quad\quad\quad\quad\quad\quad\quad \quad\quad -\varphi_R(\|\bar{u}_\epsilon\|_{\mathscr{D}(\textrm{A}^\alpha)}) \mathcal {P}G(\bar{u}_\epsilon(s-),z)\Big) \pi(\mathrm{d}s,\mathrm{d}z)\Bigg\|_{L^2}\\
 &\quad \overset{\textrm{def}}{=} K_1(t)+\cdot\cdot\cdot+\|K_5(t)\|_{L^2}.
 \end{split}
\end{equation}
The estimate for $K_1(t)$ will be divided into four cases.

\begin{itemize}[leftmargin=0.9cm]
\item [$\bullet$] If $\max\{\|u_\epsilon\|_{L^\infty(0,T;\mathscr{D}(\textrm{A}^\alpha))}, \|\bar{u}_\epsilon\|_{L^\infty(0,T;\mathscr{D}(\textrm{A}^\alpha))}\}> 2R$, then $K_1(t)\equiv0.$

\item [$\bullet$] If $\|u_\epsilon\|_{L^\infty(0,T;\mathscr{D}(\textrm{A}^\alpha))} > 2R$ and $\|\bar{u}_\epsilon\|_{L^\infty(0,T;\mathscr{D}(\textrm{A}^\alpha))}\leq 2R$, then it follows from the embedding $W^{2\alpha,2}(\mathcal {O})\subset W^{1,2}(\mathcal {O})$ with $ \alpha > \frac{3}{4}$ that
\begin{equation*}
\begin{split}
K_1(t)
&\leq C \int_0^t  (t-s)^{-\alpha}\left(\varphi_R(\|\bar{u}_\epsilon\|_{\mathscr{D}(\textrm{A}^\alpha)})-\varphi_R\left(\|u_\epsilon\|_{\mathscr{D}(\textrm{A}^\alpha)}\right)\right)\| \bar{u}_\epsilon\|_{L^2}\| \textrm{A}^\alpha\bar{u}_\epsilon\|_{L^2}\mathrm{d}s\\
&\leq C  T^{1-\alpha}\|\bar{u}_\epsilon-u_\epsilon\|_{\mathscr{D}(\textrm{A}^\alpha)}.
 \end{split}
\end{equation*}

\item [$\bullet$] Similar to ($K_{12}$), if $\|u_\epsilon\|_{L^\infty(0,T;\mathscr{D}(\textrm{A}^\alpha))} \leq 2R$ and $\|\bar{u}_\epsilon\|_{L^\infty(0,T;\mathscr{D}(\textrm{A}^\alpha))}> 2R$, then
\begin{equation*}
\begin{split}
J_1(t) \leq C  T^{1-\alpha}\|\bar{u}_\epsilon-u_\epsilon\|_{\mathscr{D}(\textrm{A}^\alpha)}.
 \end{split}
\end{equation*}

\item [$\bullet$] If $\max\{\|u_\epsilon\|_{L^\infty(0,T;\mathscr{D}(\textrm{A}^\alpha))}, \|\bar{u}_\epsilon\|_{L^\infty(0,T;\mathscr{D}(\textrm{A}^\alpha))}\}\leq 2R$, then
\begin{equation*}
\begin{split}
&\left\|\varphi_R\left(\|u_\epsilon\|_{\mathscr{D}(\textrm{A}^\alpha)}\right)(\textbf{L}_\epsilon u_\epsilon\cdot \nabla) u_\epsilon-\varphi_R(\|\bar{u}_\epsilon\|_{\mathscr{D}(\textrm{A}^\alpha)})(\textbf{L}_\epsilon \bar{u}_\epsilon\cdot \nabla) \bar{u}_\epsilon\right\|_{L^2}\\
&\quad\leq C  \|u_\epsilon-\bar{u}_\epsilon\|_{\mathscr{D}(\textrm{A}^\alpha)} \|(\textbf{L}_\epsilon u_\epsilon\cdot \nabla) u_\epsilon\|_{L^2} + \varphi_R\left(\|\bar{u}_\epsilon\|_{\mathscr{D}(\textrm{A}^\alpha)}\right) \|(\textbf{L}_\epsilon u_\epsilon-\textbf{L}_\epsilon \bar{u}_\epsilon)\cdot \nabla  u_\epsilon \|_{L^2}\\
&\quad\quad + \varphi_R\left(\|\bar{u}_\epsilon\|_{\mathscr{D}(\textrm{A}^\alpha)}\right) \|\textbf{L}_\epsilon \bar{u}_\epsilon\cdot \nabla ( u_\epsilon- \bar{u}_\epsilon)\|_{L^2}\\
&\quad\leq C   \|\bar{u}_\epsilon-u_\epsilon\|_{\mathscr{D}(\textrm{A}^\alpha)},
 \end{split}
\end{equation*}
which implies that $
K_1(t) \leq C  T^{1-\alpha}\|\bar{u}_\epsilon-u_\epsilon\|_{\mathscr{D}(\textrm{A}^\alpha)}$.
\end{itemize}
In conclusion, we get
\begin{equation}\label{(3.21)}
\begin{split}
\mathbb{E}\sup_{t\in[0,T]} |K_1(t)|^2 \leq C T^{2-2\alpha}\mathbb{E}\left(\|\bar{u}_\epsilon-u_\epsilon\|_{L^\infty(0,T;\mathscr{D}(\textrm{A}^\alpha))}^2\right) .
 \end{split}
\end{equation}
For $K_2(t)$, we prove by a similar method that
\begin{equation}\label{(3.22)}
\begin{split}
\mathbb{E}\sup_{t\in[0,T]} |K_2(t)|^2 \leq C T^{2-2\alpha}\mathbb{E}\left(\|\bar{n}_\epsilon-n_\epsilon\|_{L^\infty(0,T;L^\infty)}^2\right).
 \end{split}
\end{equation}
For $K_3(t)$, in virtue of the assumption on $h$, we have
\begin{equation*}
\begin{split}
|K_3(t)| \leq&  \int_0^t (t-s)^{-\alpha} |\varphi_R(\|u_\epsilon\|_{L^\infty})-\varphi_R(\|\bar{u}_\epsilon\|_{L^\infty})| \left(1+\| u_\epsilon \|_{L^2}^2\right)\mathrm{d}s\\
+& \int_0^t (t-s)^{-\alpha} \varphi_R(\|\bar{u}_\epsilon\|_{L^\infty})  \| u_\epsilon - \bar{u}_\epsilon \|_{L^2}\mathrm{d}s,
 \end{split}
\end{equation*}
which implies that
\begin{equation}\label{(3.23)}
\begin{split}
\mathbb{E}\sup_{t\in[0,T]} |K_3(t)|^2\leq C T^{2-2\alpha}\mathbb{E}\left(\|\bar{u}_\epsilon-u_\epsilon\|_{L^\infty(0,T;\mathscr{D}(\textrm{A}^\alpha))}^2\right) .
 \end{split}
\end{equation}
For $K_4(t)$, we rewrite it into the form of It\^{o} type as $\mathrm{d} K_4 (t)= \textrm{A} K_4 (t)\mathrm{d} t+ \textrm{A}^\alpha \hbar(t)\mathrm{d}W_t$ with $K_4 (0)=0$, where $\hbar(t) \overset{\textrm{def}}{=} \mathcal {P}\left(\varphi_R(\|u_\epsilon\|_{\mathscr{D}(\textrm{A}^\alpha)}) g(t,u_\epsilon)-\varphi_R(\|\bar{u}_\epsilon\|_{\mathscr{D}(\textrm{A}^\alpha)})g(t,\bar{u}_\epsilon) \right)$. By applying the It\^{o} formula and the BDG inequality, we get
\begin{equation*}
\begin{split}
 \mathbb{E}\sup_{t\in[0,T]} \|K_4 (t)\|_{L^2}^2
\leq&\mathbb{E}  \int_0^T \|\textrm{A}^\alpha \hbar(t)\|_{\mathcal {L}_2(U;L^2)}^2 \mathrm{d} t +2\mathbb{E}\left[\sup_{t\in[0,T]}\|K_4 (t)\|_{L^2} \left( \int_0^T\|\textrm{A}^\alpha \hbar (t)\|_{\mathcal {L}_2(U;L^2)} ^2 \mathrm{d} t \right)^{\frac{1}{2}}\right]\\
\leq& \frac{1}{2}\mathbb{E}\sup_{t\in[0,T]} \|K_4 (t)\|_{L^2}^2+C\mathbb{E}  \int_0^T \|\textrm{A}^\alpha \hbar(t)\|_{\mathcal {L}_2(U;L^2)}^2 \mathrm{d} t .
 \end{split}
\end{equation*}
For the last term on the R.H.S., we have
\begin{equation*}
\begin{split}
\|\textrm{A}^\alpha \hbar(t)\|_{\mathcal {L}_2(U;L^2)}^2 \leq& 2 \left|\varphi_R\left(\|u_\epsilon\|_{\mathscr{D}(\textrm{A}^\alpha)}\right)-\varphi_R\left(\|\bar{u}_\epsilon\|_{\mathscr{D}(\textrm{A}^\alpha)}\right)\right|^2\|\textrm{A}^\alpha g(t,u_\epsilon)\|_{\mathcal {L}_2(U;L^2)}^2 \\
+&2 \varphi_R^2\left(\|\bar{u}_\epsilon\|_{\mathscr{D}(\textrm{A}^\alpha)}\right)\left\|\textrm{A}^\alpha [g(t,u_\epsilon)-g(t,\bar{u}_\epsilon) ]\right\|_{\mathcal {L}_2(U;L^2)}^2\\
\leq& \frac{C}{R}\left(1+R^2\right)\|u_\epsilon-\bar{u}_\epsilon\|_{\mathscr{D}(\textrm{A}^\alpha)}^2.
 \end{split}
\end{equation*}
It follows from the last two estimates that
\begin{equation}\label{(3.24)}
\begin{split}
 \mathbb{E}\sup_{t\in[0,T]} \|K_4 (t)\|_{L^2}^2\leq C T\mathbb{E}\left(\|u_\epsilon-\bar{u}_\epsilon\|_{L^\infty(0,T; \mathscr{D}(\textrm{A}^\alpha))}^2\right).
 \end{split}
\end{equation}
To estimate  $K_5(t)$, we first note that
\begin{equation*}
\begin{split}
 K_5(t)=&  \int_0^t \textrm{A} K_5(s)\mathrm{d}s+  \int_0^t \int_{Z_0}\textrm{A}^\alpha \mathcal {P}V_1(s,z)\widetilde{\pi}(\mathrm{d}s,\mathrm{d}z) +  \int_0^t \int_{Z\backslash Z_0}\textrm{A}^\alpha\mathcal {P}  V_2(s,z)\pi(\mathrm{d}s,\mathrm{d}z),
 \end{split}
\end{equation*}
where
\begin{equation*}
\begin{split}
V_1(t,z)&= \varphi_R\left(\|u_\epsilon\|_{\mathscr{D}(\textrm{A}^\alpha)}\right) K(u_\epsilon(t-),z)-\varphi_R\left(\|\bar{u}_\epsilon\|_{\mathscr{D}(\textrm{A}^\alpha)}\right)K(\bar{u}_\epsilon(t-),z),\\
V_2(t,z)&= \varphi_R\left(\|u_\epsilon\|_{\mathscr{D}(\textrm{A}^\alpha)}\right) G(u_\epsilon(t-),z)-\varphi_R\left(\|\bar{u}_\epsilon\|_{\mathscr{D}(\textrm{A}^\alpha)}\right)G(\bar{u}_\epsilon(t-),z).
 \end{split}
\end{equation*}
By using the assumptions on $K$ and $G$, one can verify that for all $t \in [0,T]$
\begin{equation}\label{(3.25)}
\begin{split}
& \int_{Z_0}\|V_1(t,z)\|_{\mathscr{D}(\textrm{A}^\alpha)}^2 \mu(\mathrm{d}z)+ \int_{Z\backslash Z_0}\|V_2(t,z)\|_{\mathscr{D}(\textrm{A}^\alpha)}^2 \mu(\mathrm{d}z) \leq C  \|u_\epsilon-\bar{u}_\epsilon\|_{\mathscr{D}(\textrm{A}^\alpha)}^2.
 \end{split}
\end{equation}
By applying the It\^{o} formula to $\|K_5(t)\|_{L^2}^2$ and using the fact of $\widetilde{\pi}=\pi-\textrm{d} \mu\otimes\textrm{d}t $, we get
\begin{equation}\label{(3.26)}
\begin{split}
\|K_5(t)\|_{L^2}^2
 &\leq 2\left| \int_0^t  \int_{Z\backslash Z_0}\langle K_5(s-),\textrm{A}^\alpha \mathcal{P}V_2(s,z)\rangle \mu(\mathrm{d}z) \mathrm{d} s\right|\\
& + \int_0^t  \int_{Z_0}\| \textrm{A}^\alpha \mathcal {P}V_1(s-,z)\|_{L^2}^2 \mu(\mathrm{d}z) \mathrm{d} s + \int_0^t  \int_{Z\backslash Z_0}\|\textrm{A}^\alpha \mathcal{P}V_2(s-,z)\|_{L^2}^2  \mu(\mathrm{d}z) \mathrm{d} s\\
& +\left| \int_0^t  \int_{Z_0}\Big(2\langle K_5(s-),\textrm{A}^\alpha \mathcal{P}V_1(s,z)\rangle + \|\textrm{A}^\alpha \mathcal{P}V_1(s-,z)\|_{L^2} ^2 \Big) \widetilde{\pi}(\mathrm{d}s,\mathrm{d}z)\right|\\
&  +\left| \int_0^t  \int_{Z\backslash Z_0}\Big(2\langle K_5(s-),\textrm{A}^\alpha \mathcal{P}V_2(s-,z)\rangle + \|\textrm{A}^\alpha \mathcal{P}V_2(s-,z)\|_{L^2}^2 \Big)\widetilde{\pi} (\mathrm{d}s,\mathrm{d}z)\right|\\
&  \overset{\textrm{def}}{=} L_1(t)+\cdot\cdot\cdot+L_5(t).
 \end{split}
\end{equation}
For $L_1(t)$, it follows from \eqref{(3.25)} and the H\"{o}lder inequality  that
\begin{equation*}
\begin{split}
 L_1(t) &\leq  \int_0^t \|K_5(s)\|_{L^2}^2\mathrm{d} s+ \int_0^t\left| \int_{Z\backslash Z_0}\| \textrm{A}^\alpha \mathcal{P}V_2(s-,z) \|_{L^2} \mu(\mathrm{d}z)\right|^2 \mathrm{d} s\\
  &\leq  \int_0^t \|K_5(s)\|_{L^2}^2\mathrm{d} s+ C T\|u_\epsilon-\bar{u}_\epsilon\|_{L^\infty(0,T;\mathscr{D}(\textrm{A}^\alpha))}^2.
 \end{split}
\end{equation*}
By \eqref{(3.25)} again, we have
\begin{equation*}
\begin{split}
 L_2(t)+L_3(t)  \leq C  \int_0^T \|u_\epsilon-\bar{u}_\epsilon\|_{\mathscr{D}(\textrm{A}^\alpha)}^2\mathrm{d} s\leq C T\|u_\epsilon-\bar{u}_\epsilon\|_{L^\infty(0,T;\mathscr{D}(\textrm{A}^\alpha))}^2.
 \end{split}
\end{equation*}
For $L_4(t)$, by applying the BDG inequality, we get for any $\eta>0$
\begin{equation*}
\begin{split}
 &\mathbb{E}\sup_{t\in [0,T]} \left| \int_0^t  \int_{Z_0} \langle K_5(s-),\textrm{A}^\alpha \mathcal{P}V_1(t,z)\rangle_{L^2}\widetilde{\pi}(\mathrm{d}t,\mathrm{d}z)\right|\\
 &\quad\leq \eta \mathbb{E}  \sup_{t\in [0,T]}\| K_5(t)\|_{L^2}  + C T\mathbb{E}\left(\|u_\epsilon-\bar{u}_\epsilon\|_{L^\infty(0,T;\mathscr{D}(\textrm{A}^\alpha))}^2\right),
 \end{split}
\end{equation*}
and
\begin{equation*}
\begin{split}
 &\mathbb{E}\sup_{t\in [0,T]}  \left| \int_0^T  \int_{Z_0} \|\textrm{A}^\alpha \mathcal{P}V_1(s-,z)\|_{L^2}^2 \widetilde{\pi}(\mathrm{d}t,\mathrm{d}z)\right|\\
  &\quad\leq \eta \mathbb{E} \sup_{t\in [0,T]}\|u_\epsilon-\bar{u}_\epsilon\|_{\mathscr{D}(\textrm{A}^\alpha)}^2+ C T\mathbb{E}\left(\|u_\epsilon-\bar{u}_\epsilon\|_{L^\infty(0,T;\mathscr{D}(\textrm{A}^\alpha))}^2
  \right),
 \end{split}
\end{equation*}
which implies that
\begin{equation*}
\begin{split}
 \mathbb{E}\sup_{t\in [0,T]} L_4(t)\leq 2\eta \mathbb{E}  \sup_{t\in [0,T]}\| K_5(t)\|_{L^2}  + C T\mathbb{E}\left(\|u_\epsilon-\bar{u}_\epsilon\|_{L^\infty(0,T;\mathscr{D}(\textrm{A}^\alpha))}^2\right),\quad \forall\eta>0.
 \end{split}
\end{equation*}
In a similar manner, by assumptions on $G$ and \eqref{(3.25)}, we also have
\begin{equation*}
\begin{split}
 \mathbb{E}\sup_{t\in [0,T]} L_5(t)\leq 2 \eta \mathbb{E}  \sup_{t\in [0,T]}\| K_5(t)\|_{L^2}  + C T\mathbb{E}\left(\|u_\epsilon-\bar{u}_\epsilon\|_{L^\infty(0,T;\mathscr{D}(\textrm{A}^\alpha))}^2\right),\quad \forall\eta>0.
 \end{split}
\end{equation*}
Substituting the above estimates for $L_1(t)\sim L_5(t)$ into \eqref{(3.26)} and then choosing $\eta>0$ small enough, we obtain
\begin{equation*}
\begin{split}
 \mathbb{E}\sup_{t\in [0,T]} \|K_5(t)\|_{L^2}^2\leq  \int_0^T \|K_5(t)\|_{L^2}^2\mathrm{d}t+ C T\mathbb{E}\left(\|u_\epsilon-\bar{u}_\epsilon\|_{L^\infty(0,T;\mathscr{D}(\textrm{A}^\alpha))}^2\right).
 \end{split}
\end{equation*}
Applying the Gronwall inequality to the above inequality  leads to
\begin{equation}\label{(3.27)}
\begin{split}
 \mathbb{E}\sup_{t\in [0,T]} \|K_5(t)\|_{L^2}^2\leq C T\mathbb{E}\left(\|u_\epsilon-\bar{u}_\epsilon\|_{L^\infty(0,T;\mathscr{D}(\textrm{A}^\alpha))}^2\right).
 \end{split}
\end{equation}
It follows from the estimates \eqref{(3.21)}-\eqref{(3.24)} and \eqref{(3.27)} that
\begin{equation*}
\begin{split}
&\mathbb{E}\sup_{t\in [0,T]}\left\|\textrm{A}^\alpha(\mathscr{F}_3(n_\epsilon,c_\epsilon,u_\epsilon) -\mathscr{F}_3(\bar{n}_\epsilon,\bar{c}_\epsilon
,\bar{u}_\epsilon))\right\|_{L^2}^2\leq C  \left(T+T^{2-2\alpha}\right)\\
&\quad\times\left(\mathbb{E}\|\bar{c}_\epsilon-c_\epsilon\|_{L^\infty(0,T;W^{1,q})}^2+\mathbb{E}
 \|\bar{n}_\epsilon-n_\epsilon\|_{L^\infty(0,T;L^\infty)}^2+ \mathbb{E}\|u_\epsilon-\bar{u}_\epsilon\|_{L^\infty(0,T;\mathscr{D}(\textrm{A}^\alpha))}^2\right),
 \end{split}
\end{equation*}
which together with \eqref{(3.18)} and \eqref{(3.19)} yields that
\begin{equation}\label{mm}
\begin{split}
& \left\|\mathscr{F}(n_\epsilon,c_\epsilon,u_\epsilon) -\mathscr{F}(\bar{n}_\epsilon,\bar{c}_\epsilon
,\bar{u}_\epsilon) \right\|_{\mathscr{X}_T }^2\\
& \quad\leq  C  \left(T+T^{2-2\alpha}+T^{1-2\gamma}+T^{2-2\delta}\right)\|(n_\epsilon,c_\epsilon,u_\epsilon) -(\bar{n}_\epsilon,\bar{c}_\epsilon
,\bar{u}_\epsilon)\|_{\mathscr{X}_T}^2.
 \end{split}
\end{equation}
If we take $T_R>0$ small enough such that $C  (T_R+T_R^{2-2\alpha}+T_R^{1-2\gamma}+T_R^{2-2\delta})\leq \frac{1}{2}$, then \eqref{mm} implies that $\mathscr{F}:B_{T_R}(r)\mapsto B_{T_R}(r)$ is a contraction mapping.  By applying the Banach Fixed Point Theorem, there is a unique triple $(n_\epsilon,c_\epsilon,u_\epsilon)\in B_{T_R}(r) $ such that $ \mathscr{F}(n_\epsilon,c_\epsilon,u_\epsilon)=(n_\epsilon,c_\epsilon,u_\epsilon)$, i.e., the process $(n_\epsilon,c_\epsilon,u_\epsilon)$ satisfies the mild form of \eqref{SCNS-2} on the interval $[0,T_R]$. Observing that $T_R>0$ is independent of the initial data $(n_0,c_0,u_0)$, one can repeat above arguments to \eqref{SCNS-2} with the new initial data $(n_{\epsilon}(T_R), c_{\epsilon}(T_R), u_{\epsilon}(T_R))$. In this way, one can construct a unique solution on intervals $[T_R,2T_R]$, $[2T_R,3T_R]$, $\cdots$, which implies that the triple $(n_\epsilon,c_\epsilon,u_\epsilon)$ satisfies the mild form of \eqref{SCNS-2} in $\mathscr{X}_T$ for any $T>0$.

\textsc{Step 3.} We show that $(n_\epsilon,c_\epsilon)$ is a $\mathcal {C}^0(\mathcal {O}) \times W^{1,q}(\mathcal {O})$-valued continuous process and $u_\epsilon$ is a $\mathscr{D}(\textrm{A})$-valued c\`{a}dl\`{a}g process, $\mathbb{P}$-a.s. Note that it is sufficient to consider the stochastic convolution terms, since the deterministic integrals have continuous trajectories. Indeed, by using $\widetilde{\pi}(\mathrm{d}t, \mathrm{d}z)=\pi(\mathrm{d}t, \mathrm{d}z)-\mu(\mathrm{d}z)\mathrm{d}t$, the stochastic convolution terms can be written as
\begin{equation*}
\begin{split}
&\int_0^{t}  e^{-(t-s)\textrm{A}}  \varphi_R\left(\|u_\epsilon\|_{\mathscr{D}(\textrm{A}^\alpha)}\right) \mathcal {P}g(s,u_\epsilon) \mathrm{d}W_s\\
&\quad+ \int_0^{t} \int_{Z} e^{-(t-s)\textrm{A}}\varphi_R\left(\|u_\epsilon\|_{\mathscr{D}(\textrm{A}^\alpha)}\right)\mathcal {P}\mathcal {R}(u_\epsilon(s-),z)\widetilde{\pi}(\mathrm{d}s,\mathrm{d}z)\\
 &\quad-  \int_0^{t}\int_{Z\backslash Z_0} e^{-(t-s)\textrm{A}}\varphi_R\left(\|u_\epsilon\|_{\mathscr{D}(\textrm{A}^\alpha)}\right)\mathcal {P}G(u_\epsilon(s-),z) \mu (\mathrm{d}z)\mathrm{d}s  \overset{\textrm{def}}{=}\mathscr{L}_t^1+\mathscr{L}_t^2+\mathscr{L}_t ^3,
\end{split}
\end{equation*}
where $\mathcal {R}(u_\epsilon(s-),z)=K(u_\epsilon(s-),z)$ for $z \in Z_0$ and $\mathcal {R}(u_\epsilon(s-),z)= G(u_\epsilon(s-),z)$ for $z \in Z\backslash Z_0$. Clearly, by virtue of the condition ($\textbf{A}_4$) and the fact of $\mu(Z\backslash Z_0)<\infty$, it is straightforward to verify that the deterministic integral $\mathscr{L}_t ^3$  is continuous in $\mathscr{D}(\textrm{A})$, $\mathbb{P}$-a.s. To deal with  $\mathscr{L}_t^1$ and $\mathscr{L}_t^2$, notice that $\mathscr{D}(\textrm{A}^\alpha)$ equipped with the norm $\|\cdot\|_{\mathscr{D}(\textrm{A})}=\|\textrm{A}^\alpha\cdot\|_{L^2}$ is a separable Hilbert space, which infers that $\mathscr{D}(\textrm{A}^\alpha)$ is a  martingale type 2 or 2-smooth Banach space (cf. \cite[Definition A.1]{Brzezniak2019Maximal}). Meanwhile, since $U$ is a separable Hilbert space, the space $\gamma(U;\mathscr{D}(\textrm{A}^\alpha))$ consisting of all $\gamma$-radonifying operators is isometrically isomorphic to the space of  Hilbert-Schmidt operators from $U$ into $\mathscr{D}(\textrm{A}^\alpha)$, i.e., $\gamma(U;\mathscr{D}(\textrm{A}^\alpha))\simeq\mathcal {L}_2(U;\mathscr{D}(\textrm{A}^\alpha))$ (cf. \cite[Example 2.8]{2010Stochastic}).  Furthermore, by using the fact of $u_\epsilon \in L^2(\Omega; L^\infty(0,T;\mathscr{D}(\textrm{A})))$ and the assumptions $(\textbf{A}_3)$-$(\textbf{A}_4)$, one can deduce that
\begin{equation*}
\begin{split}
&\mathbb{E}\int_0^T\left\|\varphi_R\left(\|u_\epsilon\|_{\mathscr{D}(\textrm{A}^\alpha)}\right)\mathcal {P}g(s,u_\epsilon)\right\|_{\mathcal {L}_2(U;\mathscr{D}(\textrm{A}))} ^2 \textrm{d}t\\
&\quad +
\mathbb{E}\int_0^T\int_Z\left\|\varphi_R\left(\|u_\epsilon\|_{\mathscr{D}(\textrm{A}^\alpha)}\right)\mathcal {P}\mathcal {R}(u_\epsilon(t-),z)\right\|_{\mathscr{D}(\textrm{A})}^2 \mu (\textrm{d}z)ds \leq  C.
\end{split}
\end{equation*}
Therefore,  by applying the results in \cite[Theorem 3.6]{Brzezniak2019Maximal} and \cite[Corollary 5.1]{Brzezniak2017Maximal}, we obtain that the stochastic convolution  $\mathscr{L}_t^1+\mathscr{L}_t^2$  has a c\`{a}dl\`{a}g modification, which implies that $u_\epsilon$ has $\mathbb{P}$-a.s. c\`{a}dl\`{a}g sample paths in $\mathscr{D}(\textrm{A})$. The  proof of Lemma \ref{lem2} is completed.
\end{proof}

\subsection{Removing the cut-off operator}

 \begin{proof}[\emph{\textbf{Proof of Lemma \ref{lem1}.}}]

Assume that $(n_\epsilon,c_\epsilon,u_\epsilon,\tau ^\epsilon)$ and $(\overline{n}_\epsilon,\overline{c}_\epsilon,\overline{u}_\epsilon,\overline{\tau}^\epsilon )$ are two local mild solutions to the system \eqref{SCNS-1} with the same initial data. For $l>0$, we set
$
\tau^*_\epsilon=\tau^\epsilon \wedge \overline{\tau}^\epsilon$, $\tau^{ *}_{\epsilon,l} =\tau_l^\epsilon\wedge \overline{\tau}_l^\epsilon,
$
where
\begin{equation*}
\begin{split}
\tau_l^\epsilon&=\inf\left\{0\leq t <\infty ;~\| (n_\epsilon,c_\epsilon,u_\epsilon)\|_{\mathscr{X}_T} \geq l\right\}\wedge \tau ^\epsilon,\\
\overline{\tau}_l^\epsilon&=\inf\left\{0\leq t <\infty;~\| (\overline{n}_\epsilon,\overline{c}_\epsilon,\overline{u}_\epsilon)\|_{\mathscr{X}_T} \geq l\right\}\wedge \overline{\tau}^\epsilon .
 \end{split}
\end{equation*}
Since $(n_\epsilon,c_\epsilon,u_\epsilon)$ and $(\overline{n}_\epsilon,\overline{c}_\epsilon,\overline{u}_\epsilon)$ are $\mathfrak{F}$-adapted processes with $\mathbb{P}$-a.s. right-continuous sample paths by Lemma \ref{lem2}, it follows from a well-known result \cite[Proposition 4.6 in Chapter I]{Revuz1999}  that $\tau_l^\epsilon$ and $\overline{\tau}_l^\epsilon$ are stopping times with respect to the filtration $\mathfrak{F}$.

Denote $n_\epsilon^*=n_\epsilon-\overline{n}_\epsilon$, $c_\epsilon^*=c_\epsilon-\overline{c}_\epsilon$, and $u_\epsilon^*=u_\epsilon-\overline{u}_\epsilon$. The proof of uniqueness result depends upon some energy estimates for $(n_\epsilon^*,c_\epsilon^*,u_\epsilon^*)$. For $n_\epsilon^*$, we get by integrating by parts and the divergence-free condition that, for any $t\in [0,\tau^{ *}_{\epsilon,l}  \wedge T)$,
\begin{equation}\label{111}
\begin{split}
& \|n_\epsilon^* (t)\|_{L^2}^2 + \int_0^t  \|\nabla n_\epsilon^*\|_{L^2}^2\mathrm{d}s\\
&\quad\leq \int_0^t  \int_\mathcal {O}\big|n_\epsilon^* u_\epsilon^*\cdot \nabla n_\epsilon\big|\mathrm{d}x  \mathrm{d}s+ \int_0^t \int_\mathcal {O}  \big|n_\epsilon^* \textbf{h}_\epsilon' (n_\epsilon)\chi(c_\epsilon)\nabla n_\epsilon^*\cdot\nabla c_\epsilon  \big| \mathrm{d}x  \mathrm{d}s\\
&\quad+ \int_0^t \int_\mathcal {O} \big|\overline{n}_\epsilon  (\textbf{h}_\epsilon' (n_\epsilon)-\textbf{h}_\epsilon' (\overline{n}_\epsilon))\chi(c_\epsilon)\nabla n_\epsilon^*\cdot\nabla c_\epsilon  \big|\mathrm{d}x  \mathrm{d}s\\
&\quad+ \int_0^t \int_\mathcal {O}  \big| \overline{n}_\epsilon \textbf{h}_\epsilon' (\overline{n}_\epsilon)(\chi(c_\epsilon)-\chi(\overline{c}_\epsilon))\nabla n_\epsilon^*\cdot\nabla c_\epsilon  \big| \mathrm{d}x  \mathrm{d}s\\
&\quad+   \int_0^t \int_\mathcal {O} \big| \overline{n}_\epsilon \textbf{h}_\epsilon' (\overline{n}_\epsilon)\chi(\overline{c}_\epsilon)\nabla n_\epsilon^*\cdot\nabla c_\epsilon^* \big| \mathrm{d}x  \mathrm{d}s  \overset{\textrm{def}}{=}  N_1 +\cdot\cdot\cdot+N_5.
 \end{split}
\end{equation}
For $N_1$, by using the embedding $W^{1,q}(\mathcal {O})\subset L^\infty(\mathcal {O})$ for $q>3$ and integrating by parts $ \int_\mathcal {O}n_\epsilon^* u_\epsilon^*\cdot \nabla n_\epsilon\mathrm{d}x=- \int_\mathcal {O}n_\epsilon u_\epsilon^*\cdot\nabla n_\epsilon^* \mathrm{d}x$ because $\div u_\epsilon =0$, we get for any $\eta> 0$
\begin{equation*}
\begin{split}
N_1+N_2  \leq \eta \int_0^t \|\nabla n_\epsilon^*\|_{L^2}^2 \mathrm{d}s+ C   \int_0^t \left(\| u_\epsilon^*\|_{L^2}^2+\|n_\epsilon^* \|_{L^2}^2\right) \mathrm{d}s , \quad t\in [0,\tau^{ *}_{\epsilon,l}  \wedge T).
 \end{split}
\end{equation*}
For $N_3$ and $N_5$, we get from the boundedness of $\textbf{h}_\epsilon'' (\cdot)$ that, for any $\eta> 0$,
\begin{equation*}
\begin{split}
N_3+N_5\leq \eta \int_0^t \|\nabla n_\epsilon^*\|_{L^2}^2 \mathrm{d}s+ C  \int_0^t \left(\| n_\epsilon^*\|_{L^2}^2+\|c_\epsilon^* \|_{L^2}^2\right)\mathrm{d}s , \quad t\in [0,\tau^{ *}_{\epsilon,l}  \wedge T).
 \end{split}
\end{equation*}
For $N_4$, recalling the following interpolation result: If $B_0 \subset B \subset B_1$ are three Banach spaces, the embedding $B_0 \subset B$ being compact, $B \subset B_1$ being continuous, then for every $\eta>0$ there is a constant $C_{\eta}>0$ such that
$
\|x\|_B \leq \eta\|x\|_{B_0}+C_{\eta}\|x\|_{B_1}
$
for every $x \in B_0$. We apply this fact to $ W^{1,2}(\mathcal {O})\subset L^{\frac{2q}{q-2}}(\mathcal {O})\subset W^{-1,2}(\mathcal {O})$ with $q>3$: For any $\eta>0$, there exists a $C_\eta>0$ such that
$$
\|f\|_{L^{\frac{2q}{q-2}}}^2\leq \eta \|\nabla f\|_{L^{2}}^2+C _\eta\|f\|_{L^2}^2.
$$
It then follows from the locally Lipschitz continuity of $\chi$ that
\begin{equation*}
\begin{split}
N_4
&\leq  \eta  \int_0^t \|\nabla n_\epsilon^*\|_{L^2}^2 \mathrm{d}s+ C  \int_0^t\|\overline{n}_\epsilon\|_{L^\infty}^2\| c_\epsilon^*\|_{L^{\frac{2q}{q-2}}}^2 \|\nabla c_\epsilon \|_{L^q}^2 \mathrm{d}s\\
&\leq \eta  \int_0^t \|\nabla n_\epsilon^*\|_{L^2}^2 \mathrm{d}s+ C  \int_0^t\| c_\epsilon^*\|_{L^{\frac{2q}{q-2}}}^2 \mathrm{d}s\\
&\leq \eta  \int_0^t \|\nabla n_\epsilon^*\|_{L^2}^2 \mathrm{d}s+ \eta  \int_0^t\|  \nabla c_\epsilon^* \|_{L^2}^2 \mathrm{d}s+ C   \int_0^t\| c_\epsilon^*  \|_{L^2}^2 \mathrm{d}s.
 \end{split}
\end{equation*}
Therefore plugging the estimates for terms $N_1-N_6$ into \eqref{111} leads to
\begin{equation}\label{(3.28)}
\begin{split}
& \mathbb{E}\sup_{s\in [0,t\wedge \tau^{ *}_{\epsilon,l} ]}\|n_\epsilon^* (s)\|_{L^2}^2 + (1-3\eta) \int_0^{t\wedge \tau^{ *}_{\epsilon,l} }  \|\nabla n_\epsilon^*\|_{L^2}^2\mathrm{d}s\\
&\quad\leq \eta  \int_0^{t\wedge \tau^{ *}_{\epsilon,l} }\|  \nabla c_\epsilon^* \|_{L^2}^2 \mathrm{d}s + C   \int_0^{t\wedge \tau^{ *}_{\epsilon,l} } \big(\| u_\epsilon^*\|_{L^2}^2 + \|n_\epsilon^* \|_{L^2}^2 + \| c_\epsilon^*  \|_{L^2}^2\big) \mathrm{d}s.
 \end{split}
\end{equation}
Proceeding similarly to the equations satisfied by $c_\epsilon$ and $\overline{c}_\epsilon $, we obtain
\begin{equation}\label{(3.29)}
\begin{split}
 \mathbb{E}\sup_{s\in [0,t\wedge \tau^{ *}_{\epsilon,l} ]}\|c_\epsilon^* (s)\|_{L^2}^2 +   \int_0^{t\wedge \tau^{ *}_{\epsilon,l} }  \|\nabla c_\epsilon^*\|_{L^2}^2\mathrm{d}s \leq C  \int_0^{t\wedge \tau^{ *}_{\epsilon,l} } \big(\| u_\epsilon^*\|_{L^2}^2 + \|n_\epsilon^* \|_{L^2}^2 + \| c_\epsilon^*  \|_{L^2}^2\big) \mathrm{d}s.
 \end{split}
\end{equation}
To estimate $\| u_\epsilon^*\|_{L^2}$, we apply the It\^{o} formula to infer that for any $t\in [0,\tau^{ *}_{\epsilon,l}  \wedge T)$
{\wuhao\begin{equation}\label{(3.30)}
\begin{split}
& \|u_\epsilon^*(t)\|_{L^2}^2+2 \int_0^t\|\textrm{A}^{\frac{1}{2}}u_\epsilon^*(s)\|_{L^2}^2\mathrm{d}s \leq  \int_0^t\|\mathcal {P}[g(s,u_\epsilon)-g(s,\overline{u}_\epsilon)]\|_{\mathcal {L}_2(U;L^2)}^2 \mathrm{d}s\\
&\quad+2\left| \int_0^t\big\langle u_\epsilon^*(t),\mathcal {P} [(\textbf{L}_\epsilon u_\epsilon\cdot \nabla) u_\epsilon-\textbf{L}_\epsilon \overline{u}_\epsilon\cdot \nabla) \overline{u}_\epsilon]  - \mathcal {P}(n_\epsilon^*\nabla \Phi ) \big\rangle_{L^2}\mathrm{d}s\right|\\
 &\quad+ \int_0^t \int_{Z_0} \|\mathcal {P}[ K(u_\epsilon(x,s-),z)- K(\overline{u}_\epsilon(x,s-),z)]\|_{L^2}^2 \pi(\mathrm{d}s,\mathrm{d}z)\\
 &\quad+ \int_0^t \int_{Z\backslash Z_0} \|\mathcal {P}[ G(u_\epsilon(x,s-),z)- G(\overline{u}_\epsilon(x,s-),z)]\|_{L^2}^2 \pi(\mathrm{d}s,\mathrm{d}z) \\
&\quad+2\left| \int_0^t \int_{Z\backslash Z_0}
 \big\langle u_\epsilon^*(s),\mathcal {P}[ G(u_\epsilon(x,s-),z)- G(\overline{u}_\epsilon(x,s-),z)]\big\rangle_{L^2}\mu(\mathrm{d}z)\mathrm{d}s\right|\\
 &\quad+2\left| \int_0^t\big\langle u_\epsilon^*(s), \mathcal {P}[g(s,u_\epsilon)-g(s,\overline{u}_\epsilon)] \mathrm{d}W_s\big\rangle_{L^2}\right|\\
&\quad+2\left| \int_0^t \int_{Z_0}
 \big\langle u_\epsilon^*(s),\mathcal {P}[ K(u_\epsilon(x,s-),z)- K(\overline{u}_\epsilon(x,s-),z)]\big\rangle_{L^2}\widetilde{\pi}(\mathrm{d}s,\mathrm{d}z)\right|\\
&\quad+2\left| \int_0^t \int_{Z\backslash Z_0}
 \big\langle u_\epsilon^*(s),\mathcal {P}[ G(u_\epsilon(x,s-),z)- G(\overline{u}_\epsilon(x,s-),z)]\big\rangle_{L^2}\widetilde{\pi}(\mathrm{d}s,\mathrm{d}z)\right|\\
 &\quad \overset{\textrm{def}}{=}  K_1(t)+\cdot\cdot\cdot+K_8(t),
\end{split}
\end{equation}}\noindent
where we used the fact of  $\widetilde{\pi}(\textrm{d}t,\textrm{d}z)=\pi(\textrm{d}t,\textrm{d}z)-\mu(\textrm{d}z)\textrm{d}t$. Now we shall estimate each integral terms on the R.H.S. of \eqref{(3.30)}. For $K_1$, by the hypothesis on $g$, we have
\begin{equation*}
\begin{split}
\mathbb{E}\sup_{s\in [0,t\wedge \tau^{ *}_{\epsilon,l} ]}K_1 (s) \leq C\mathbb{E} \int_0^{t\wedge \tau^{ *}_{\epsilon,l} }\| g(s,u_\epsilon)-g(s,\overline{u}_\epsilon) \|_{\mathcal {L}_2(U;L^2)}^2 \mathrm{d}s \leq C\mathbb{E} \int_0^{t\wedge \tau^{ *}_{\epsilon,l} }\|u_\epsilon^*\|_{L^2}^2\mathrm{d}s.
 \end{split}
\end{equation*}
For $K_2$, by virtue of $\left\langle u_\epsilon^*,(\textbf{L}_\epsilon \overline{u}_\epsilon\cdot \nabla) u_\epsilon^*\right\rangle=0$ and $\left\langle u_\epsilon^*,(\textbf{L}_\epsilon \overline{u}_\epsilon^*\cdot \nabla) u_\epsilon \right\rangle =-\left\langle u_\epsilon,(\textbf{L}_\epsilon \overline{u}_\epsilon^*\cdot \nabla) u_\epsilon^* \right\rangle $ because of $\div (\textbf{L}_\epsilon \overline{u}_\epsilon)=0$, we get for any $\eta>0$
\begin{equation*}
\begin{split}
\mathbb{E}\sup_{s\in [0,t\wedge \tau^{ *}_{\epsilon,l} ]}K_2(s)  &\leq 2\mathbb{E}\sup_{t\in [0,t\wedge \tau^{ *}_{\epsilon,l} ]}\left| \int_0^{t\wedge \tau^{ *}_{\epsilon,l} }\langle u_\epsilon^*,(\textbf{L}_\epsilon u_\epsilon^*\cdot \nabla) u_\epsilon+(\textbf{L}_\epsilon \overline{u}_\epsilon\cdot \nabla) u_\epsilon^*\rangle_{L^2}\mathrm{d}s\right| \\
&+2\mathbb{E}  \int_0^{t\wedge \tau^{ *}_{\epsilon,l} }\| u_\epsilon^*\|_{L^2}\|n_\epsilon^*\|_{L^2}\|\nabla \Phi  \|_{L^\infty}\mathrm{d}s\\
&\leq \eta\mathbb{E}  \int_0^{t\wedge \tau^{ *}_{\epsilon,l} } \|\nabla  u_\epsilon^*\|_{L^2}^2\mathrm{d}s +C \mathbb{E}  \int_0^{t\wedge \tau^{ *}_{\epsilon,l} }\left(\| u_\epsilon^*\|_{L^2}^2+\|n_\epsilon^*\|_{L^2}^2\right)\mathrm{d}s .
 \end{split}
\end{equation*}
For $K_3$, we have
\begin{equation*}
\begin{split}
\mathbb{E}\sup_{s\in [0,t\wedge \tau^{ *}_{\epsilon,l} ]}K_3(s)&\leq
C\mathbb{E} \int_0^{t\wedge \tau^{ *}_{\epsilon,l} } \int_{Z\backslash Z_0} \|G(u_\epsilon(x,s-),z)- G(\overline{u}_\epsilon(s-),z)\|_{L^2}^2 \pi(\mathrm{d}s,\mathrm{d}z)\\
&=C
\mathbb{E} \int_0^{t\wedge \tau^{ *}_{\epsilon,l} } \int_{Z\backslash Z_0} \|G(u_\epsilon(x,s-),z)- G(\overline{u}_\epsilon(s-),z)\|_{L^2}^2 \mu(\mathrm{d}z)\mathrm{d}s\\
&\leq C
\mathbb{E} \int_0^{t\wedge \tau^{ *}_{\epsilon,l} }  \| u_\epsilon^*(s)\|_{L^2}^2 \mathrm{d}s.
 \end{split}
\end{equation*}
Similarly, we have
\begin{equation*}
\begin{split}
\mathbb{E}\sup_{s\in [0,t\wedge \tau^{ *}_{\epsilon,l} ]}(K_4(s)+K_5(s)) \leq  C \mathbb{E}  \int_0^{t\wedge \tau^{ *}_{\epsilon,l} } \| u_\epsilon^*(s)\|_{L^2}^2 \mathrm{d}s.
 \end{split}
\end{equation*}
For $K_6$, the assumption on $g$ and BDG inequality imply that, for any $\eta>0$,
\begin{equation*}
\begin{split}
\mathbb{E}\sup_{s\in [0,t\wedge \tau^{ *}_{\epsilon,l} ]}K_6(s)&\leq  C \mathbb{E}\left(\sup_{s\in [0,t\wedge \tau^{ *}_{\epsilon,l} ]}\| u_\epsilon^*(s)\|^2  \int_0^{t\wedge \tau^{ *}_{\epsilon,l} } \|\mathcal {P}[g(s,u_\epsilon)-g(s,\overline{u}_\epsilon)\|_{\mathcal {L}_2(U;L^2)}^2\mathrm{d}s \right)^{\frac{1}{2}}\\
&\leq \eta\mathbb{E} \sup_{s\in [0,t\wedge \tau^{ *}_{\epsilon,l} ]}\| u_\epsilon^*(s)\|^2 + C \mathbb{E}  \int_0^{t\wedge \tau^{ *}_{\epsilon,l} } \| u_\epsilon^*(s)\|_{L^2}^2 \mathrm{d}s.
 \end{split}
\end{equation*}
For $K_{7}$, by using the assumption on $K$ and the BDG inequality, we obtain for any $\eta>0$
\begin{equation*}
\begin{split}
&\mathbb{E}\sup_{s\in [0,t\wedge \tau^{ *}_{\epsilon,l} ]}K_{7 }(s)\\
&\quad\leq  C \mathbb{E}\left[\sup_{s\in [0,t\wedge \tau^{ *}_{\epsilon,l} ]}\|u_\epsilon^*(s)\|_{L^2} \left( \int_0^{t\wedge \tau^{ *}_{\epsilon,l} } \int_{Z_0}
\|K(u_\epsilon(s-),z)- K(\overline{u}_\epsilon(s-),z)\big\|_{L^2} ^2\mu(\mathrm{d}z)\mathrm{d}s\right)^{\frac{1}{2}}\right]\\
&\quad\leq  \eta \mathbb{E} \sup_{s\in [0,t\wedge \tau^{ *}_{\epsilon,l} ]}\|u_\epsilon^*(s)\|_{L^2}^2+ C \mathbb{E}  \int_0^{t\wedge \tau^{ *}_{\epsilon,l} } \| u_\epsilon^*(s)\|_{L^2}^2 \mathrm{d}s.
 \end{split}
\end{equation*}
In a similar manner, the term $K_{8}$ can be estimated by
\begin{equation*}
\begin{split}
&\mathbb{E}\sup_{s\in [0,t\wedge \tau^{ *}_{\epsilon,l} ]}K_{8}(s) \leq  \eta \mathbb{E} \sup_{s\in [0,t\wedge \tau^{ *}_{\epsilon,l} ]}\|u_\epsilon^*(s)\|_{L^2}^2+ C \mathbb{E}  \int_0^{t\wedge \tau^{ *}_{\epsilon,l} } \| u_\epsilon^*(s)\|_{L^2}^2 \mathrm{d}s,\quad \forall\eta>0.
 \end{split}
\end{equation*}
Taking the supremum to \eqref{(3.30)} on time over $[0, t\wedge \tau^{ *}_{\epsilon,l} ]$ and then taking expectation on both sides, after  substituting the above estimates for $K_{1}\sim K_{8}$ and choosing $\eta>0$ small enough, we arrive at
\begin{equation}\label{(3.31)}
\begin{split}
 \mathbb{E}\sup_{s\in [0,t\wedge \tau^{ *}_{\epsilon,l} ]}\|u_\epsilon^*(s)\|_{L^2}^2\leq  C \mathbb{E}  \int_0^{t\wedge \tau^{ *}_{\epsilon,l} }\left(\| u_\epsilon^*(s)\|_{L^2}^2+\|n_\epsilon^*(s)\|_{L^2}^2\right)\mathrm{d}s.
\end{split}
\end{equation}
By virtue of the estimates  \eqref{(3.28)}, \eqref{(3.29)} and \eqref{(3.31)} and taking small $\eta>0$, we obtain
\begin{equation}
\begin{split}
 &\mathbb{E}\sup_{s\in [0,t\wedge \tau^{ *}_{\epsilon,l} ]}\left(\|u_\epsilon^*(s)\|_{L^2}^2+ \|c_\epsilon^* (s)\|_{L^2}^2 + \|n_\epsilon^* (s)\|_{L^2}^2\right) \\
 & \quad\leq  C \mathbb{E} \int_0^{t} \sup_{\varsigma\in [0,s\wedge \tau^{ *}_{\epsilon,l} )}\left(\| u_\epsilon^*(\varsigma)\|_{L^2}^2 + \|n_\epsilon^* (\varsigma)\|_{L^2}^2 + \| c_\epsilon^* (\varsigma) \|_{L^2}^2\right) \mathrm{d}s,
\end{split}
\end{equation}
which combined with  the Gronwall inequality implies that $\mathbb{E}\sup_{s\in [0,t\wedge \tau^{ *}_{\epsilon,l} ]}(\|u_\epsilon^*(s)\|_{L^2}^2+ \|c_\epsilon^* (s)\|_{L^2}^2 + \|n_\epsilon^* (s)\|_{L^2}^2)=0$. By taking $t=\tau^*$ and letting $l\rightarrow\infty$, this proves the \emph{uniqueness} result.

Now let us extend the solution $(n_{\epsilon},c_{\epsilon},u_{\epsilon})$ to a maximal time of existence $\widetilde{\tau}_{\epsilon}$. For any $R>0$, we define
\begin{equation*}
\tau_R\overset {\textrm{def}} {=}
\left\{
 \begin{array}{ll}
   \inf\left\{0\leq t<\infty ;~ \| c_\epsilon\|_{L^\infty(0,t;W^{1,q})}^2+
 \| n_\epsilon\|_{L^\infty(0,t;L^\infty)}^2+  \|u_\epsilon \|_{L^\infty(0,t;\mathscr{D}(\textrm{A}^\alpha))}^2\geq R\right\},  \\
    +\infty,  \quad \textrm{if the above set $\{\cdot\cdot\cdot\}$ is empty.}
  \end{array}
\right.
\end{equation*}
Due to the right-continuity of $(n_\epsilon,c_\epsilon,u_\epsilon)$, $(\tau_R)_{R>0}$ is a  sequence of stopping times \cite{Revuz1999}, and since $\| c_{\epsilon 0}\|_{W^{1,q}}^2+
 \| n_{\epsilon 0}\|_{L^\infty}^2+ \|u_{\epsilon 0}\|_{\mathscr{D}(\textrm{A}^\alpha)}^2$ is bounded from above by some constant independent of $R$,  $\tau_R>0$ $\mathbb{P}$-a.s. for all $R$ large enough. According to the definition of cut-off operator, we see that the solution $(n_\epsilon,c_\epsilon,u_\epsilon)$ constructed in Lemma \ref{lem2} restricted on $[0,\tau_R]$ is indeed a local mild solution to the regularized system \eqref{SCNS-1}. Moreover,  it follows from the uniqueness of solutions that $\tau_R\leq  \tau_{R+1}$ $\mathbb{P}$-a.s., and for $t\in [0,\tau_R]$, there holds $
(n_{R+1,\epsilon},c_{R+1,\epsilon},u_{R+1,\epsilon})= (n_{R,\epsilon},c_{R,\epsilon},u_{R,\epsilon}).$

Define a stopping time
$
\widetilde{\tau}_{\epsilon}(\omega)=\lim_{R\rightarrow\infty} \tau_R(\omega),
$
and a triplet
$
(n_{\epsilon},c_{\epsilon},u_{\epsilon})=(n_{R,\epsilon},c_{R,\epsilon},u_{R,\epsilon})$ on $[0,\tau_R]$.
One can conclude from the definition of $(\tau_R)_{R>0}$ that $(n_{\epsilon},c_{\epsilon},u_{\epsilon},\widetilde{\tau}_{\epsilon})$ is actually a unique maximal local mild solution to \eqref{SCNS-1}.  The proof of Lemma \ref{lem1} is completed.
\end{proof}

\section{Global-in-time approximate solutions} \label{sec4}

In this section, we prove that the local solution constructed in Section \ref{sec3} is actually a global one. The proof is based on a new entropy-energy estimate, which can be regarded as an extension of the deterministic counterpart but without convex condition on the domain. Let us begin with an observation on spatio-temporal regularity of solutions.

\subsection{Variational formulation}
For $q>3$ and $\alpha\in(\frac{3}{4},1)$, we infer from Definition \ref{def-2} that
\begin{equation} \label{(4.1)}
\begin{split}
(c_\epsilon,u_\epsilon) \in L^2\left([0,\widetilde{\tau}_\epsilon); W^{1,2}(\mathcal {O})\times \left(W^{1,2}_{0,\sigma}(\mathcal {O})\right)^3\right),\quad \mathbb{P}\textrm{-a.s.}
 \end{split}
\end{equation}
Since  $n_\epsilon \in L^\infty \left([0,\widetilde{\tau}_\epsilon);  \mathcal {C}^0(\mathcal {O}) \right)$, \eqref{(4.1)} implies that $u_\epsilon\cdot \nabla n_\epsilon=\div(u_\epsilon n_\epsilon)\in L^2 \left([0,\widetilde{\tau}_\epsilon); L^{2} (\mathcal {O}) \right)$ and $\div\left(n_\epsilon \textbf{h}_\epsilon' (n_\epsilon)\chi(c_\epsilon)\nabla c_\epsilon\right)\in L^2 \left([0,\widetilde{\tau}_\epsilon);  W^{-1,2} (\mathcal {O}) \right)$. By making use of the $n_\epsilon$-equation in \eqref{SCNS-1} and the maximal $L^p-L^q$ regularity for parabolic equations (cf. \cite{giga1991abstract}), we obtain
$n_\epsilon \in L^2  \left([0,\widetilde{\tau}_\epsilon); W^{1,2}(\mathcal {O})\right)$, $\mathbb{P}$-a.s.

Setting
$$
\textbf{V}=W^{1,2}(\mathcal {O})\times W^{1,2}(\mathcal {O})\times \left(W^{1,2}_{0,\sigma}(\mathcal {O})\right)^3, \quad \textbf{H}=L^{2}(\mathcal {O})\times L^{2}(\mathcal {O})\times \left(L^{2}_\sigma(\mathcal {O})\right)^3.
$$
Then, $\textbf{V}\subset\textbf{H}\equiv \textbf{H}^*\subset\textbf{V}^*$ formulates a Gelfand inclusion, where $\textbf{H}$ is identified with its dual $\textbf{H}^*$ and $\textbf{V}^*$ is the dual space  of $\textbf{V}$. As  a consequence, the local mild solution $(n_\epsilon,c_\epsilon,u_\epsilon)$ constructed in Lemma \ref{lem1} is actually a \emph{variational solution}, and the variational formulation of \eqref{SCNS-1} in $\textbf{V}^*$ is given by
\begin{subequations}
\begin{align}
& \langle n_\epsilon (t), \phi\rangle + \int_0^t \langle u_\epsilon\cdot \nabla n_\epsilon , \phi\rangle \mathrm{d}s =  \int_0^t\langle  \Delta n_\epsilon , \phi\rangle  \mathrm{d}s -  \int_0^t \langle\div\left(n_\epsilon \textbf{h}_\epsilon' (n_\epsilon)\chi(c_\epsilon)\nabla c_\epsilon\right), \phi\rangle \mathrm{d}s,\label{Va}\\
& \langle c_\epsilon(t), \varphi\rangle +  \int_0^t\langle u\cdot \nabla c _\epsilon, \varphi\rangle  \mathrm{d}s = \int_0^t\langle  \Delta c_\epsilon, \varphi\rangle \mathrm{d}s- \int_0^t \langle\textbf{h}_\epsilon(n_\epsilon) f(c_\epsilon), \varphi\rangle \mathrm{d}s,\label{Vb}\\
& \langle u_\epsilon(t), \psi\rangle +  \int_0^t \langle\mathcal {P} (\textbf{L}_\epsilon u_\epsilon\cdot \nabla) u_\epsilon , \psi\rangle   \mathrm{d}s \label{Vc}=- \int_0^t\langle \textrm{A} u_\epsilon , \psi\rangle \mathrm{d}s \\
&\quad+ \int_0^t\langle\mathcal {P}(n_\epsilon\nabla \Phi)+\mathcal {P} h(s,u_\epsilon), \psi\rangle  \mathrm{d}s\notag+  \int_0^t \langle\mathcal {P}g(s,u_\epsilon) \mathrm{d}W_s, \psi\rangle  \\
&\quad + \int_0^t\int_{Z_0} \langle\mathcal {P}K(u_\epsilon(s-),z)\widetilde{\pi}(\mathrm{d}s,\mathrm{d}z),\psi\rangle \notag +  \int_0^t\int_{Z\backslash Z_0} \langle\mathcal {P}G(u_\epsilon(s-),z) \pi (\mathrm{d}s,\mathrm{d}z),\psi\rangle,\notag
\end{align}
\end{subequations}
for any $(\phi,\varphi,\psi) \in \textbf{V}$, and any $t\in [0,T]\subset [0,\widetilde{\tau}_\epsilon)$,  $\mathbb{P}$-a.s.

\begin{remark}
In view of the assumptions on $g$, $K$ and $G$ (see $(\textbf{A}_3)$-$(\textbf{A}_4)$),  the solution $(n_\epsilon,c_\epsilon,u_\epsilon)$ can be regarded as a $\textbf{V}^*$-valued c\`{a}dl\`{a}g semimartingale with the stochastic integrals being local square integrable martingale with values in $\textbf{H}$. As a result, a generalized It\^{o}  formula in Banach spaces provided by Gy\"{o}ngy-Krylov (cf.  \cite[Theorem 1]{gyongy1980stochastic}) can be applied to the system \eqref{SCNS-1}, which is the basic analytical tool for deriving some entropy-energy estimates and uniform bounds in the following argument. Meanwhile, the It\^{o} formula for infinite-dimensional system driven by Wiener process could be found in \cite[Theorem I.3.1]{KR79} and \cite[Theorem 4.2.5]{prevot2007concise}. Concerning the It\^{o} formula for the system driven by L\'{e}vy noises, we refer to the results in \cite[Theorem D.2]{peszat2007stochastic}, \cite[Theorem A.1]{brzezniak20132d} and \cite[Appendix C]{brzezniak2019weak}.
\end{remark}

\subsection{A new entropy-energy inequality}
To begin with, let us recall an identity which was initially derived by Winkler \cite{winkler2012global} for the deterministic counterpart in the case of $\textbf{L}_\epsilon \equiv \textrm{Id}$. Due to the divergence-free condition $\div u_\epsilon =0$ and the fact that the derivation of this identity mainly depends on the $n_\epsilon$-equation and $c_\epsilon$-equation, similar result also holds for the first two random PDEs in \eqref{SCNS-1}.

\begin{lemma} \label{lem4.1}
For any given $\epsilon \in (0,1)$ and any $T>0$, the solution
\begin{equation}\label{(4.3)}
\begin{split}
 \begin{aligned}
&\mathrm{d}\left( \int_\mathcal {O} n_{\epsilon} \ln n_{\epsilon}\mathrm{d}x+\frac{1}{2}  \int_\mathcal {O} |\nabla \Psi (c_{\epsilon} ) |^{2}\mathrm{d} x\right)+ \int_\mathcal {O} \frac{ |\nabla n_{\epsilon} |^{2}}{n_{\epsilon}}\mathrm{d}x\mathrm{d}t+ \int_\mathcal {O} \theta (c_{\epsilon} ) |D^{2} \rho (c_{\epsilon} ) |^{2}\mathrm{d}x\mathrm{d}t \\
&\quad=  \int_\mathcal {O} \emph{\textbf{h}}_{\epsilon}(n_{\epsilon})\left(\frac{f(c_{\epsilon}) \theta^{\prime}(c_{\epsilon})}{2 \theta^{2}(c_{\epsilon})}
-\frac{f^{\prime}(c_{\epsilon})}{\theta(c_{\epsilon})}\right)  |\nabla c_{\epsilon}|^{2} \mathrm{d}x\mathrm{d}t\\
&\quad+ \int_\mathcal {O} \bigg(\frac{1}{\theta c_{\epsilon})}\Delta c_{\epsilon} -\frac{1}{2} \frac{\theta^{\prime} (c_{\epsilon})}{\theta^{2}(c_{\epsilon})}|\nabla c_{\epsilon}|^{2}\bigg)(u_{\epsilon} \cdot \nabla c_{\epsilon})\mathrm{d}x\mathrm{d}t\\
&\quad+\frac{1}{2}  \int_\mathcal {O} \frac{\theta^{\prime \prime}(c_{\epsilon})}{\theta^{2}(c_{\epsilon})}|\nabla c_{\epsilon}|^{4}\mathrm{d}x\mathrm{d}t
+\frac{1}{2}  \int_{\partial\mathcal {O}} \frac{1}{\theta(c_{\epsilon})}  \frac{\partial|\nabla c_{\epsilon}|^{2}}{\partial \nu}\mathrm{d}\Sigma\mathrm{d}t, \quad \mathbb{P}\textrm{-a.s.},
\end{aligned}
\end{split}
\end{equation}
for all $t \in(0, T\wedge\widetilde{\tau} _\epsilon)$, where
$$
\theta(s)= \frac{f(s)}{\chi(s)}, \quad \Psi(s)=  \int_{1}^{s} \frac{d \sigma}{\sqrt{\theta(\sigma)}} \quad \text { and } \quad \rho(s) =  \int_{1}^{s} \frac{d \sigma}{\theta(\sigma)}.
$$
\end{lemma}

\begin{Proof}
For almost all $\omega\in \Omega$, thanks to the fact of $\div u_\epsilon=0$, one can apply the chain rule to $\mathrm{d}(n_{\epsilon} \ln n_{\epsilon})$ and $\mathrm{d}\|\nabla \Psi (c_{\epsilon} ) \|_{L^2}^2$ associated with the $n_\epsilon$- and $c_\epsilon$-equation, respectively. Then the identity \eqref{(4.3)} can be obtained by performing a straightforward computation similar to \cite[Lemma 3.2]{winkler2012global}. We shall omit the details here.
\end{Proof}

\begin{remark}\label{remark4.2}
Employing the assumptions on $f,\chi$, and the fact that $t\mapsto \|c_{\epsilon}(\cdot,t)\|_{L^\infty}$ is nonincreasing. one can find two positive constants $C^-\leq C^+ $ such that
$C^- s \leq \theta (s) \leq C^+ s$  for all $s \in \left[0,\|c_0\|_{L^\infty}\right]$, and $
 \theta' (c_\epsilon) \geq \theta' (\|c_0\|_{L^\infty})>0$  for all  $(x,t)\in \mathcal {O}\times (0,\widetilde{\tau}_\epsilon)$.
\end{remark}


Let us recall the following lemma obtained by Mizoguchi and Souplet \cite{mizoguchi2014nondegeneracy}, which is crucial for us to derive a new stochastic version of the entropy-energy estimate without convex condition on the domain.

\begin{lemma} (\cite[Lemma 4.2 and Appendix A]{mizoguchi2014nondegeneracy})\label{lem-ms}
Let $\mathcal {O} \subset \mathbb{R}^{d}$, $d \geq 1$ be a bounded domain with a $\mathcal {C}^2$ boundary. Let $w \in \mathcal {C}^2(\overline{\mathcal {O}})$ such that $\frac{\partial w}{\partial \nu}=0$ on the boundary $\partial\mathcal {O}$. Then
$$
\frac{\partial|\nabla w|^{2}}{\partial \nu} \leq 2 \kappa|\nabla w|^{2} \quad \text { on } \partial\mathcal {O},
$$
for some constant $\kappa$ depending only on the domain $\mathcal {O}$.
\end{lemma}

\begin{lemma}\label{concervation}
Any solution $(n_{\epsilon},c_{\epsilon},u_{\epsilon},\widetilde{\tau}_\epsilon)$ to the initial-boundary value problem \eqref{SCNS-1} in the sense of Definition \ref{def-2} satisfies that, for all $T>0$
\begin{equation}\label{(4.4)}
\begin{split}
n_{\epsilon}(x,t)\geq 0,~~ c_{\epsilon}(x,t)\geq 0,~~\forall t\in(0,T\wedge\widetilde{\tau}_\epsilon),~~x\in \mathcal {O},~~\mathbb{P}\textrm{-a.s.,}
\end{split}
\end{equation}
\begin{equation}\label{(4.5)}
\begin{split}
\| n_\epsilon(\cdot,t)\|_{L^1}\equiv\| n_0\|_{L^1},~~ \|c_\epsilon(\cdot,t)\|_{L^\infty}\leq\|c_0\|_{L^\infty},~~\forall t\in(0,T\wedge\widetilde{\tau}_\epsilon),~~\mathbb{P}\textrm{-a.s.}
\end{split}
\end{equation}
\end{lemma}

\begin{Proof}
Integrating the first equation in \eqref{SCNS-1} and using the condition $\div u_\epsilon =0$, we obtain \eqref{(4.4)}. Since $f>0$ on $(0,\infty)$, $f(0)=0$ by the assumption ($\textbf{A}_2$) and $\textbf{h}_\epsilon (s)\geq 0$ for all $s>0$, an application of the comparison  principle to $c_\epsilon$-equation in \eqref{SCNS-1} gives \eqref{(4.5)}.
\end{Proof}

To derive several uniform bounds for approximations, we shall frequently use the following Gagliardo-Nirenberg (GN, for short) inequality in the rest proof of Theorem \ref{thm}.

\begin{lemma} (\cite[Lecture II]{nirenberg1959}) \label{nirenberg}
Let $\mathcal {O}\subset\mathbb {R} ^{n}$ be a measurable, bounded, open and connected domain satisfying the cone condition. Let $1\leq q\leq +\infty$, $1 \leq r \leq+\infty$, $p \geq 1$, $j<m$ be nonnegative integers and $\theta \in[0,1]$ such that $\frac{1}{p}=\frac{j}{n}+\theta(\frac{1}{r}-\frac{m}{n})+\frac{1-\theta}{q}$, $\frac{j}{m} \leq \theta \leq 1$.  Assume that $u \in L^{q}(\mathcal {O})$ such that $D^{m} u \in L^{r}(\mathcal {O})$ and $\sigma$ is arbitrary. Then,
$$
\left\|D^{j} u\right\|_{L^{p}} \leq C\left\|D^{m} u\right\|_{L^{r}}^{\theta}\|u\|_{L^{q}}^{1-\theta}+C\|u\|_{L^{\sigma}},
$$
where the constant $C>0$ is independent of $u$.
\end{lemma}

The Lyapunov functional inequality stated in the next result plays an important role in the stochastic compactness argument.

\begin{lemma}\label{lem4.5}
Let $T>0$ be arbitrary. Assume that conditions  $(\textbf{A}_1)-(\textbf{A}_4)$  hold, then for any solution $(n_{\epsilon},c_{\epsilon},u_{\epsilon},\widetilde{\tau}_\epsilon)$ to \eqref{SCNS-1} in the sense of Definition \ref{def-2}, we have
\begin{equation}\label{(4.6)}
\begin{split}
& \mathscr{E} [n_{\epsilon},c_{\epsilon},u_{\epsilon}](t) +  \int_0^t \mathscr{I}[n_{\epsilon},c_{\epsilon},u_{\epsilon}](s)\mathrm{d}s \leq \mathscr{E}[n_{\epsilon0},c_{\epsilon0},u_{\epsilon0}]
+\mathscr{L}[n_{\epsilon},c_{\epsilon},u_{\epsilon}](t),
\end{split}
\end{equation}
 where
\begin{align*}
\mathscr{E} [n_{\epsilon},c_{\epsilon},u_{\epsilon}](t)&\overset {\textrm{\emph{def}}} {=}  \int_\mathcal {O} \left(n_{\epsilon} \ln n_{\epsilon} + \frac{1}{2}|\nabla \Psi(c_{\epsilon})|^2+c^\dag |u_{\epsilon}|^2\right)\mathrm{d} x,\\
\mathscr{I} [n_{\epsilon},c_{\epsilon},u_{\epsilon}](t)&\overset {\textrm{\emph{def}}} {=}  \int_\mathcal {O}\left( \frac{1}{2} \frac{  |\nabla n_{\epsilon} |^{2}}{  n_{\epsilon}}  +  d_1\frac{|\nabla c_{\epsilon} |^4}{c_{\epsilon}^3} + d_2\frac{ |\Delta c_{\epsilon}|^2}{c_{\epsilon}}  +   |\nabla u_\epsilon|^{2}\right)\mathrm{d}x,
\end{align*}
and
{\wuhao\begin{equation*}
\begin{split}
\mathscr{L}[n_{\epsilon},c_{\epsilon},u_{\epsilon}](t)&\overset {\textrm{\emph{def}}} {=}  C  t +  C  \int_0^t \| \Psi (c_{\epsilon}) \|_{L^{2}}^2\mathrm{d}s + C  \int_0^t \| u_\epsilon(s)\|_{L^2}^2 \mathrm{d}s+2c^\dag \int_0^t \langle u_\epsilon, \mathcal {P}g(s,u_\epsilon) \mathrm{d}W_s  \rangle \\
&\quad+ c^\dag \int_0^t \int_{Z_0}  \left(\|\mathcal {P}K(u_\epsilon(s-),z)\|_{L^2} ^2+ 2\langle u_\epsilon, \mathcal {P}K(u_\epsilon(s-),z)\rangle \right) \widetilde{\pi}(\mathrm{d} s, \mathrm{d} z)\\
&\quad+c^\dag \int_0^t \int_{Z\backslash Z_0}  \left(\|\mathcal {P}G(u_\epsilon(s-),z)\|_{L^2} ^2+2\langle u_\epsilon, \mathcal {P}G(u_\epsilon(s-),z)\rangle\right) \pi (\mathrm{d} s, \mathrm{d} z),
\end{split}
\end{equation*}}\noindent
for all $t\in[0,T\wedge\widetilde{\tau}_\epsilon)$, with positive constants $d_1,d_2 ,c^\dagger$ depending only on $\|c_0\|_{L^\infty}$ and the functions $f$ and $\chi$.
\end{lemma}
\begin{Proof}
First, by virtue of the fact of  $\textbf{h}(u_\epsilon)\geq 0$ and the assumptions on $f,\chi$, we  infer that
\begin{equation*}
\begin{split}
\theta''(c_{\epsilon}) \leq 0\quad \textrm{and} \quad \frac{f\left(c_{\epsilon}\right) \theta^{\prime}\left(c_{\epsilon}\right)}{2 \theta^{2}\left(c_{\epsilon}\right)}
-\frac{f^{\prime}\left(c_{\epsilon}\right)}{\theta\left(c_{\epsilon}\right)} = - \frac{(\chi f)'(c_{\epsilon})}{2f(c_{\epsilon}) }\leq 0,
\end{split}
\end{equation*}
which imply that the first and third terms on the R.H.S. of \eqref{(4.3)} can be neglected in estimating. In order to control the second term, we get by integrating by parts that
\begin{equation}\label{(4.7)}
\begin{split}
 \begin{aligned}
 & -\frac{1}{2} \int_\mathcal {O}  \frac{\theta^{\prime} (c_{\epsilon})}{\theta^{2}(c_{\epsilon})}(u_{\epsilon} \cdot \nabla c_{\epsilon})|\nabla c_{\epsilon}|^{2} \mathrm{d}x\\
 &\quad=\frac{1}{2} \int_\mathcal {O} \div\left(   \frac{u_{\epsilon}}{\theta (c_{\epsilon})}\right) |\nabla c_{\epsilon}|^{2} \mathrm{d}x = -\sum_{i,j}  \int_\mathcal {O}  \frac{1}{\theta (c_{\epsilon})} u_{\epsilon}^i \partial_j c_{\epsilon} \partial_i \partial_j c_{\epsilon} \mathrm{d}x,
\end{aligned}
\end{split}
\end{equation}
and
\begin{equation}\label{(4.8)}
\begin{split}
 &  \int_\mathcal {O}  \frac{1}{\theta c_{\epsilon})}\Delta c_{\epsilon}  (u_{\epsilon} \cdot \nabla c_{\epsilon})\mathrm{d}x =  \sum_{i,j} \int_\mathcal {O} \frac{ \theta ' (c_{\epsilon})  }{\theta^2 (c_{\epsilon})} u_{\epsilon}^j  \partial_j c_{\epsilon} |\partial_i  c_{\epsilon} |^2 \mathrm{d}x\\
 &
 \quad-\sum_{i,j} \int_\mathcal {O}  \frac{1}{\theta (c_{\epsilon})}\partial_i  c_{\epsilon} \partial_i u_{\epsilon}^j  \partial_j c_{\epsilon}  \mathrm{d}x
 -\sum_{i,j} \int_\mathcal {O}  \frac{1}{\theta (c_{\epsilon})}\partial_i  c_{\epsilon}  u_{\epsilon}^j  \partial_i\partial_j c_{\epsilon}  \mathrm{d}x.
\end{split}
\end{equation}
Moreover, by applying the functional inequalities obtained in \cite[Lemma 3.2]{winkler2012global} and \cite[Lemma 3.3]{winkler2016global}, and then applying it to the function $\theta (c_{\epsilon})$, we  gain
\begin{equation}\label{(4.9)}
\begin{split}
 \int_\mathcal {O} \frac{\theta^{\prime}(c_{\epsilon})}{\theta^{3}(c_{\epsilon})}|\nabla c_{\epsilon}|^{4} \leq(2+\sqrt{3})^{2}  \int_\mathcal {O} \frac{\theta(c_{\epsilon})}{\theta^{\prime}(c_{\epsilon})}\left|D^{2} \rho(c_{\epsilon})\right|^{2}.
\end{split}
\end{equation}
It then follows from the identities \eqref{(4.7)}-\eqref{(4.9)} and the Young inequality that, for any $\delta >0$,
\begin{equation}\label{(4.10)}
\begin{split}
  &\left| \int_\mathcal {O} \left(\frac{1}{\theta (c_{\epsilon})}\Delta c_{\epsilon} -\frac{1}{2} \frac{\theta^{\prime} (c_{\epsilon})}{\theta^{2}(c_{\epsilon})}|\nabla c_{\epsilon}|^{2}\right)(u_{\epsilon} \cdot \nabla c_{\epsilon})\mathrm{d}x \right| =\left|\sum_{i,j}  \int_\mathcal {O}  \frac{1}{\theta (c_{\epsilon})}\partial_i  c_{\epsilon} \partial_i u_{\epsilon} \partial_j c_{\epsilon}  \mathrm{d}x\right| \\
   & \quad\leq \delta  \int_\mathcal {O} \frac{\theta^{\prime}(c_{\epsilon})}{\theta^{3}(c_{\epsilon})}|\nabla c_{\epsilon}|^{4} \mathrm{d}x+C_\delta  \int_\mathcal {O}  \frac{\theta (c_{\epsilon})}{\theta' (c_{\epsilon})}| \nabla u_{\epsilon}|^2 \mathrm{d}x\\
 & \quad\leq \delta (2+\sqrt{3})^{2}  \int_\mathcal {O} \frac{\theta(c_{\epsilon})}{\theta^{\prime}(c_{\epsilon})}\left|D^{2} \rho(c_{\epsilon})\right|^{2}+C   \int_\mathcal {O}  \frac{\theta (c_{\epsilon})}{\theta' (c_{\epsilon})}| \nabla u_{\epsilon}|^2 \mathrm{d}x.
\end{split}
\end{equation}
For the boundary integral term on the R.H.S. of \eqref{(4.3)}, we get from Lemma \ref{lem-ms} and the Sobolev Embedding Theorem that, for any $r \in (0,\frac{1}{2})$,
\begin{equation}\label{(4.11)}
\begin{split}
&\left|\frac{1}{2}  \int_{\partial\mathcal {O}} \frac{1}{\theta(c_{\epsilon})}  \frac{\partial|\nabla c_{\epsilon}|^{2}}{\partial \nu}\mathrm{d}\Sigma \right|\\
 &\quad\leq  \kappa  \int_{\partial\mathcal {O}} \left| \frac{\nabla c_{\epsilon}}{\sqrt{\theta(c_{\epsilon})}}\right|^{2} \mathrm{d}\Sigma\leq C  \left\|\frac{\nabla c_{\epsilon}}{\sqrt{\theta(c_{\epsilon})}}\right\|_{W^{r+\frac{1}{2},2}}^2=C  \|\nabla \Psi (c_{\epsilon})\|_{W^{r+\frac{1}{2},2}}^2 .
\end{split}
\end{equation}
By using an interpolation argument between $W^{r+\frac{3}{2},2}(\mathcal {O})$ and $L^2(\mathcal {O})$, and applying the  estimate $\| w\|_{W^{2,2}}\leq C( \|\Delta w\|_{L^{2}}+  \| w\|_{L^{2}})$, we get for any $\eta>0$
\begin{equation}\label{(4.12)}
\begin{split}
 \|\nabla \Psi (c_{\epsilon})\|_{W^{r+\frac{1}{2},2}}^2 &\leq C \| \Psi (c_{\epsilon})\|_{W^{2,2}}^{r+\frac{3}{2}}\| \Psi (c_{\epsilon})\|_{L^2}^{\frac{1}{2}-r}\\
 &\leq C \left(\|\Delta \Psi(c_{\epsilon})\|_{L^{2}}^2\right)^{\frac{r}{2}+\frac{3}{4}}\left(\| \Psi (c_{\epsilon})\|_{L^2}^2\right)^{\frac{1}{4}-\frac{r}{2}}+  C\| \Psi (c_{\epsilon})\|_{L^{2}}^{2} \\
 &\leq \eta \|\Delta \Psi (c_{\epsilon}) \|_{L^{2}}^2+ C \| \Psi (c_{\epsilon}) \|_{L^{2}}^2.
\end{split}
\end{equation}
A direct calculation for $\Delta \Psi (c_{\epsilon}) $ leads to $\Delta \Psi (c_{\epsilon})=\sqrt{\theta(c_{\epsilon})} \Delta \rho (c_{\epsilon})+\frac{1}{2} \frac{\theta^{\prime}(c_{\epsilon})}{\theta^{3/2}(c_{\epsilon})}|\nabla c_{\epsilon}|^{2}$, which implies that
\begin{equation*}
\begin{split}
 \|\Delta \Psi (c_{\epsilon}) \|_{L^{2}}^2\leq 6 \int_\mathcal {O} \theta(c_{\epsilon})|D^2 \rho (c_{\epsilon})|^2\mathrm{d}x +  \frac{1}{2}  \int_\mathcal {O} \frac{|\theta^{\prime}(c_{\epsilon})|^2}{\theta^{3}(c_{\epsilon})}|\nabla c_{\epsilon}|^{4}\mathrm{d}x,
\end{split}
\end{equation*}
where the first integral on the R.H.S. used the fact of  $|\Delta f|^2\leq 3|D^2f|^2$. This estimate together with \eqref{(4.11)} and \eqref{(4.12)} leads to
\begin{equation}\label{(4.13)}
\begin{split}
 &\left|\frac{1}{2}  \int_{\partial\mathcal {O}} \frac{1}{\theta(c_{\epsilon})}  \frac{\partial|\nabla c_{\epsilon}|^{2}}{\partial \nu}\mathrm{d}\Sigma \right|\\
 &\quad\leq C \| \Psi (c_{\epsilon}) \|_{L^{2}}^2+ \eta C   \int_\mathcal {O} \theta(c_{\epsilon})\left|D^2 \rho (c_{\epsilon})\right|^2\mathrm{d}x + \eta C    \int_\mathcal {O} \frac{\theta(c_{\epsilon})}{\theta^{\prime}(c_{\epsilon})}\left|D^{2} \rho(c_{\epsilon})\right|^{2}.
\end{split}
\end{equation}
Substituting the estimates \eqref{(4.10)} and \eqref{(4.13)} into \eqref{(4.3)}, in view of Remark 4.3, we get
\begin{equation}\label{(4.14)}
\begin{split}
&\mathrm{d} \int_\mathcal {O} \left(n_{\epsilon} \ln n_{\epsilon} +\frac{1}{2} |\nabla \Psi (c_{\epsilon} ) |^{2}\right)\mathrm{d} x + \int_\mathcal {O} \frac{ |\nabla n_{\epsilon} |^{2}}{n_{\epsilon}}\mathrm{d}x\mathrm{d}t\\
&\quad+\left(1-\eta C - \eta C \theta' (\|c_0\|_{L^\infty}) -\delta (2+\sqrt{3})^{2} \theta' (\|c_0\|_{L^\infty}) \right) \int_\mathcal {O} \theta (c_{\epsilon} ) |D^{2} \rho (c_{\epsilon} ) |^{2}\mathrm{d}x\mathrm{d}t \\
&\quad\leq C \| \Psi (c_{\epsilon}(t)) \|_{L^{2}}^2\mathrm{d}t +C \theta' (\|c_0\|_{L^\infty}) \|c_0\|_{L^\infty} \int_\mathcal {O}   | \nabla u_{\epsilon}|^2 \mathrm{d}x\mathrm{d}t.
\end{split}
\end{equation}
Choosing $\delta$, $\eta>0$  small enough such that $
\frac{1}{2}\leq 1-\eta C - \eta C \theta' (\|c_0\|_{L^\infty}) -\delta (2+\sqrt{3})^{2} \theta' (\|c_0\|_{L^\infty})  <1$,
and then integrating both sides of \eqref{(4.14)} over the interval $[0,t]$, we get
\begin{equation}\label{(4.15)}
\begin{split}
&  \int_\mathcal {O} \left(n_{\epsilon} \ln n_{\epsilon} +\frac{1}{2} |\nabla \Psi (c_{\epsilon} ) |^{2}\right)\mathrm{d} x + \int_0^t \int_\mathcal {O} \frac{ |\nabla n_{\epsilon} |^{2}}{n_{\epsilon}}\mathrm{d}x\mathrm{d}s + \frac{1}{2} \int_0^t  \int_\mathcal {O} \theta (c_{\epsilon} ) |D^{2} \rho (c_{\epsilon} ) |^{2}\mathrm{d}x\mathrm{d}s \\
&\quad\leq  \int_\mathcal {O} \left(n_{\epsilon0} \ln n_{\epsilon0} +\frac{1}{2}  |\nabla \Psi (c_{\epsilon0} ) |^{2}\right)\mathrm{d} x+ c^\dag  \int_0^t \| \nabla u_{\epsilon}\|_{L^{2}}^2 \mathrm{d}s +  C  \int_0^t \| \Psi (c_{\epsilon}) \|_{L^{2}}^2\mathrm{d}s,
\end{split}
\end{equation}
for any  $t\in[0,T \wedge \widetilde{\tau}_\epsilon)$, where
$c^\dag \overset{\textrm{def}}{=}  C \theta' (\|c_0\|_{L^\infty}) \|c_0\|_{L^\infty} $.  For the future  argument, one can magnify $C>0$ large enough such that $c^\dag> 1$.

Now let us make use of the It\^{o} formula to the functional $\|u_\epsilon(t)\|_{L^2} ^2$. After integrating by parts and rearranging the terms and multiplying  both sides by $c^\dag$, we infer that
\begin{equation}\label{(4.16)}
\begin{split}
c^\dag\|u_\epsilon(t)\|_{L^2} ^2 &\leq c^\dag\|u_{\epsilon0}\|_{L^2} ^2-2c^\dag \int_0^t \|\nabla u_\epsilon(s)\|_{L^2}^2 \mathrm{d} s+ \left|2c^\dag \int_0^t \langle u_\epsilon,  \mathcal {P}(n_\epsilon\nabla \Phi ) \rangle\mathrm{d}s\right|\\
&+ \left|2c^\dag \int_0^t \langle u_\epsilon,  \mathcal {P} h(s,u_\epsilon) \rangle\mathrm{d}s\right|+c^\dag \int_0^t\|\mathcal {P}g(s,u_\epsilon)\|_{\mathcal {L}_2(U;L^2)} ^2 \mathrm{d}s\\
&+  c^\dag  \int_0^t \int_{Z_0} \|\mathcal {P}K(u_\epsilon(s-),z)\|_{L^2} ^2 \mu(\mathrm{d} z)\mathrm{d}s+ 2c^\dag \int_0^t \langle u_\epsilon, \mathcal {P}g(s,u_\epsilon) \mathrm{d}W_s  \rangle \\
&+ c^\dag \int_0^t \int_{Z_0}  \left(\|\mathcal {P}K(u_\epsilon(s-),z)\|_{L^2} ^2+ 2\langle u_\epsilon, \mathcal {P}K(u_\epsilon(s-),z)\rangle \right) \widetilde{\pi}(\mathrm{d} s, \mathrm{d} z)\\
&+c^\dag \int_0^t \int_{Z\backslash Z_0}  \left(\|\mathcal {P}G(u_\epsilon(s-),z)\|_{L^2} ^2+2\langle u_\epsilon, \mathcal {P}G(u_\epsilon(s-),z)\rangle\right) \pi (\mathrm{d} s, \mathrm{d} z) ,
\end{split}
\end{equation}
for any $t\in[0,T \wedge \widetilde{\tau}_\epsilon)$.  Let us estimate the terms appearing on the R.H.S. of \eqref{(4.16)}. By applying the assumption on $\Phi$, the Sobolev embedding $W^{1,2}(\mathcal {O})\subset L^6(\mathcal {O})$ as well as the GN inequality (cf. Lemma \ref{nirenberg}), one has
\begin{equation*}\label{(4.17)}
\begin{split}
 \| n_\epsilon\| _{L^{\frac{6}{5}}}  \leq C \| \nabla\sqrt{n_\epsilon}\| _{L^{2}}^{\frac{1}{2}}\| \sqrt{n_\epsilon}\| _{L^2} ^{\frac{3}{2}}+C\| \sqrt{n_\epsilon}\| _{L^{2}}^2 \leq C \left (\| \nabla\sqrt{n_\epsilon}\| _{L^{2}}^{\frac{1}{2}}+1\right),
\end{split}
\end{equation*}
which is sufficient to derive
\begin{equation}\label{(4.18)}
\begin{split}
 \bigg|2c^\dag \int_0^t \langle u_\epsilon,  \mathcal {P}(n_\epsilon\nabla \Phi ) \rangle_{L^2}\mathrm{d}s \bigg|&\leq 2c^\dag \int_0^t \| u_\epsilon\|_{L^6}  \| n_\epsilon\| _{L^{\frac{6}{5}}}\|\nabla \Phi \| _{L^\infty} \mathrm{d}s\\
 &\leq C  \int_0^t \|\nabla u_\epsilon\|_{L^2} \left (\| \nabla\sqrt{n_\epsilon}\| _{L^{2}}^{\frac{1}{2}}+1\right) \mathrm{d}s\\
 &\leq (\frac{c^\dag}{2}-\frac{1}{2})  \int_0^t \|\nabla u_\epsilon\|_{L^2}^2 \mathrm{d}s+ 2 \int_0^t \| \nabla\sqrt{n_\epsilon}\| _{L^{2}}^2 \mathrm{d}s+C t,
\end{split}
\end{equation}
for all $t\in[0,T \wedge \widetilde{\tau}_\epsilon)$.  By using the Sobolev embedding $W^{1,2}(\mathcal {O})\subset L^6(\mathcal {O})$, the assumption on $h$ and the Young inequality, we get
\begin{equation}\label{(4.19)}
\begin{split}
 \left|2c^\dag \int_0^t \langle u_\epsilon,  \mathcal {P} h(s,u_\epsilon) \rangle\mathrm{d}s \right| &\leq 2c^\dag \int_0^t \| u_\epsilon\|_{L^6} \|\mathcal {P}h(s,u_\epsilon)\|_{L^{ \frac{6}{5}}}\mathrm{d}s\\
&\leq (\frac{c^\dag}{2}-\frac{1}{2}) \int_0^t \|\nabla u_\epsilon\|_{L^2}^2 \mathrm{d}s+C  \int_0^t \left(\| u_\epsilon\|_{L^2}^2 +1\right)\mathrm{d}s.
\end{split}
\end{equation}
Applying the assumptions on $g$ and $K$, for all $t\in[0,T \wedge \widetilde{\tau}_\epsilon)$, we get
 \begin{equation}\label{xx}
\begin{split}
c^\dag  \int_0^t \int_{Z_0} \|\mathcal {P}K(u_\epsilon ,z)\|_{L^2} ^2 \mu(\mathrm{d} z)\mathrm{d}s +c^\dag \int_0^t\|\mathcal {P}g(s,u_\epsilon)\|_{\mathcal {L}_2(U;L^2)}^2 \mathrm{d}s \leq C  \int_0^t \| u_\epsilon(s)\|_{L^2}^2 \mathrm{d}s + C  t.
\end{split}
\end{equation}
Putting \eqref{(4.15)}-\eqref{xx} together, we infer that
\begin{equation}\label{cc}
\begin{split}
&  \int_\mathcal {O} \left(n_{\epsilon} \ln n_{\epsilon} +\frac{1}{2} |\nabla \Psi (c_{\epsilon} ) |^{2}+c^\dag|u_\epsilon|^2\right)\mathrm{d} x +\frac{1}{2} \int_0^t \int_\mathcal {O} \frac{ |\nabla n_{\epsilon} |^{2}}{n_{\epsilon}}\mathrm{d}x\mathrm{d}s \\
&\quad+ \frac{1}{2} \int_0^t  \int_\mathcal {O} \theta (c_{\epsilon} ) |D^{2} \rho (c_{\epsilon} ) |^{2}\mathrm{d}x\mathrm{d}s +  \int_0^t \|\nabla u_\epsilon(s)\|_{L^2}^2 \mathrm{d} s \\
&\quad\leq \int_\mathcal {O} \left(n_{\epsilon0} \ln n_{\epsilon0} +\frac{1}{2}  |\nabla \Psi (c_{\epsilon0} ) |^{2}+c^\dag|u_{\epsilon0}|^2\right)\mathrm{d} x+  C  t+ C  \int_0^t \| \Psi (c_{\epsilon}) \|_{L^{2}}^2\mathrm{d}s \\
&\quad+ C  \int_0^t \| u_\epsilon(s)\|_{L^2}^2 \mathrm{d}s+2c^\dag \int_0^t \langle u_\epsilon, \mathcal {P}g(s,u_\epsilon) \mathrm{d}W_s  \rangle \\
&\quad+ c^\dag \int_0^t \int_{Z_0}  \left(\|\mathcal {P}K(u_\epsilon(s-),z)\|_{L^2} ^2+ 2\langle u_\epsilon, \mathcal {P}K(u_\epsilon(s-),z)\rangle \right) \widetilde{\pi}(\mathrm{d} s, \mathrm{d} z)\\
&\quad+c^\dag \int_0^t \int_{Z\backslash Z_0}  \left(\|\mathcal {P}G(u_\epsilon(s-),z)\|_{L^2} ^2+2\langle u_\epsilon, \mathcal {P}G(u_\epsilon(s-),z)\rangle\right) \pi (\mathrm{d} s, \mathrm{d} z).
\end{split}
\end{equation}
To obtain the desired inequality, one need to estimate the integral $ \int_\mathcal {O} \theta (c_{\epsilon} ) |D^{2} \rho (c_{\epsilon} ) |^{2}\mathrm{d}x$ from below. Indeed, we get by integrating by parts and using the H\"{o}lder inequality that
\begin{equation*}
\begin{split}
 \int_\mathcal {O} \theta (c_{\epsilon} ) |D^{2} \rho (c_{\epsilon} ) |^{2}\mathrm{d}x &\geq \frac{1}{2} \sum_{i,j} \int_\mathcal {O} \frac{\theta (c_{\epsilon} )  }{c_{\epsilon}^2}|\partial_{i}\partial_{j}c_{\epsilon}|^2\mathrm{d}x-\frac{1}{2} \sum_{i,j} \int_\mathcal {O} \frac{(\theta' (c_{\epsilon} ))^2}{\theta ^3(c_{\epsilon} )}|\partial_{i}c_{\epsilon}\partial_{j}c_{\epsilon}|^2\mathrm{d}x\\
 &  \geq  C'  \int_\mathcal {O} \frac{ |\Delta c_{\epsilon}|^2}{c_{\epsilon}}\mathrm{d}x-C'' \int_\mathcal {O}  \frac{|\nabla c_{\epsilon} |^4}{c_{\epsilon}^3}\mathrm{d}x,
\end{split}
\end{equation*}
for some constants $C',C''>0$. Using the estimate \eqref{(4.9)} and Remark \ref{remark4.2}, we gain
\begin{equation*}
\begin{split}
 \int_\mathcal {O} \theta (c_{\epsilon} ) |D^{2} \rho (c_{\epsilon} ) |^{2}\mathrm{d}x  \geq C''' \int_\mathcal {O}  \frac{|\nabla c_{\epsilon} |^4}{c_{\epsilon}^3}\mathrm{d}x,
\end{split}
\end{equation*}
for some constant $C'''>0$. From the last two inequalities, we have
\begin{equation*}
\begin{split}
 \int_\mathcal {O} \theta (c_{\epsilon} ) |D^{2} \rho (c_{\epsilon} ) |^{2}\mathrm{d}x  \geq C \left(  \int_\mathcal {O}  \frac{|\nabla c_{\epsilon} |^4}{c_{\epsilon}^3}\mathrm{d}x+   \int_\mathcal {O} \frac{ |\Delta c_{\epsilon}|^2}{c_{\epsilon}}\mathrm{d}x\right),
\end{split}
\end{equation*}
which combined with \eqref{cc} leads to the desired inequality. This completes the proof of Lemma \ref{lem4.5}.
\end{Proof}

As a consequence of the previous lemma, we have the following uniform estimate.

\begin{lemma}\label{lem4.6}
Assume that the assumptions $(\textbf{A}_1)-(\textbf{A}_4)$ hold, and  $(n_{\epsilon},c_{\epsilon},u_{\epsilon},\widetilde{\tau}_\epsilon)$ is a solution to \eqref{SCNS-1} in the sense of Definition \ref{def-2}. Then there exists a constant $C>0$ such that
\begin{equation}\label{(4.20)}
\begin{split}
&\mathbb{E} \left[\sup_{t\in [0,T \wedge \widetilde{\tau}_\epsilon)} \int_\mathcal {O} \Big(n_{\epsilon} \ln n_{\epsilon} + \frac{1}{2}|\nabla \Psi(c_{\epsilon})|^2+c^\dag |u_{\epsilon}|^2\Big)\mathrm{d} x\right]^p\\
&\quad  +\mathbb{E} \left[\frac{1}{2}  \int_0^{T \wedge \widetilde{\tau}_\epsilon} \int_\mathcal {O} \frac{ |\nabla n_{\epsilon} |^{2}}{n_{\epsilon}}\mathrm{d}x\mathrm{d} t+  d_1 \int_0 ^{T \wedge \widetilde{\tau}_\epsilon} \int_\mathcal {O}  \frac{|\nabla c_{\epsilon} |^4}{c_{\epsilon}^3}\mathrm{d}x\mathrm{d} t+d_2  \int_0^{T \wedge \widetilde{\tau}_\epsilon} \int_\mathcal {O} \frac{ |\Delta c_{\epsilon}|^2}{c_{\epsilon}}\mathrm{d}x\mathrm{d} t \right.\\
&\left.\quad\quad + \frac{c^\dag}{2} \int_0^{T \wedge \widetilde{\tau}_\epsilon}  \int_\mathcal {O}|\nabla u_\epsilon|^{2}\mathrm{d}x\mathrm{d} t\right]^p\\
 &\quad \leq C\left(\mathbb{E}\left[ \int_\mathcal {O} \left(n_{\epsilon0} \ln n_{\epsilon0} + \frac{1}{2}|\nabla \Psi(c_{\epsilon0})|^2+c^\dag |u_{\epsilon0}|^2\right)\mathrm{d} x\right]^p +1 \right),
 \end{split}
\end{equation}
for all $T>0$ and $1\leq p <\infty$, where the function $\Psi$ is defined in Lemma \ref{lem4.1}, and $C$ depends only on $p,\kappa,c_0,f$ and $\chi$.
\end{lemma}

\begin{Proof} The proof consists in explicitly estimating each integrals appearing on the R.H.S. of \eqref{(4.6)}. First, by using the  property for $\theta(\cdot)$ and \eqref{(4.5)}, we infer
$$
| \Psi (c_{\epsilon}) |=\bigg|   \int_{1}^{c_{\epsilon}} \frac{d \sigma}{\sqrt{\theta(\sigma)}}\bigg| \leq\frac{2}{\sqrt{C ^-  }}\bigg| \int_{1}^{c_{\epsilon}} \frac{d \sigma}{2\sqrt{ \sigma }}\bigg|\leq \frac{2}{\sqrt{C ^-  }}\left(1+\sqrt{\|c_0\|_{L^\infty}}\right),
$$
which implies that
$\mathbb{E}\sup_{t\in [0,{T \wedge \widetilde{\tau}_\epsilon})}    \int_0^t \| \Psi (c_{\epsilon}) \|_{L^{2}}^2\mathrm{d}s \leq C T$. Here, $C^-$ is the positive constant provided in Remark \ref{remark4.2}.
It then follows that
\begin{equation}\label{f1}
\begin{split}
&\mathbb{E} \sup_{t\in [0,T \wedge \widetilde{\tau}_\epsilon)}  \left|C t+ C  \int_0^t \| \Psi (c_{\epsilon}) \|_{L^{2}}^2\mathrm{d}s +C  \int_0^t \| u_\epsilon\|_{L^2}^2 \mathrm{d}s\right|^p \leq  C  \mathbb{E}\left( \int_0^T \| u_\epsilon\|_{L^2}^{2} \mathrm{d}s\right)^p+C T^{p}.
\end{split}
\end{equation}
By applying the BDG inequality and the condition on $g$, we get for any $\eta>0$
\begin{equation}
\begin{split}
&\mathbb{E} \sup_{t\in [0,T \wedge \widetilde{\tau}_\epsilon) }  \left|2c^\dag \int_0^t \langle u_\epsilon, \mathcal {P}g(s,u_\epsilon) \mathrm{d}W  \rangle_{L^2}\right|^p\\
&\quad\leq C  \mathbb{E} \left[\sup_{t\in [0,T \wedge \widetilde{\tau}_\epsilon)} \| u_\epsilon\|_{L^2}^2\left( \int_0^{T \wedge \widetilde{\tau}_\epsilon}\| u_\epsilon\|_{L^2}^2 \mathrm{d} t+T\right) \right]^{\frac{p}{2}}\\
&\quad\leq \eta \mathbb{E} \sup_{t\in [0,T \wedge \widetilde{\tau}_\epsilon)} \| u_\epsilon\|_{L^2}^{2p}+C \mathbb{E}\left( \int_0^{T \wedge \widetilde{\tau}_\epsilon }  \| u_\epsilon\|_{L^2}^2  \mathrm{d} t \right)^{p}+ C T^p.
\end{split}
\end{equation}
By virtue of the BDG inequality and the assumption on $K$, we deduce that
\begin{equation}
\begin{split}
&\mathbb{E} \sup_{t\in [0,T \wedge \widetilde{\tau}_\epsilon)} \left| 2 c^\dag  \int_0^t \int_{Z_0}  \langle u_\epsilon, \mathcal {P}K(u_\epsilon(s-),z)\rangle _{L^2}  \widetilde{\pi}(\mathrm{d} s, \mathrm{d} z)\right|^p\\
&\quad\leq  C \mathbb{E} \left( \sup_{t\in [0,T \wedge \widetilde{\tau}_\epsilon)} \| u_\epsilon(t)\|^2\int_0^{T \wedge \widetilde{\tau}_\epsilon} \int_{Z_0}\| \mathcal {P}K(u_\epsilon(s-),z)\|_{L^2}^2 \mu( \mathrm{d} z)\mathrm{d} s\right)^{\frac{p}{2}}\\
&\quad\leq \eta\mathbb{E} \sup_{t\in [0,T \wedge \widetilde{\tau}_\epsilon)}  \| u_\epsilon\|_{L^2}^{2p} +C \mathbb{E}\left( \int_0^{T \wedge \widetilde{\tau}_\epsilon}(\| u_\epsilon\|_{L^2}^2+1) \mathrm{d} s \right)^{p},\quad \forall\eta>0.
\end{split}
\end{equation}
Similarly, we get by making use of the BDG inequality that
\begin{equation}
\begin{split}
&\mathbb{E} \sup_{t\in [0,T \wedge \widetilde{\tau}_\epsilon)} \left|c^\dag \int_0^t \int_{Z_0}  \|\mathcal {P}K(u_\epsilon(s-),z)\|_{L^2} ^2  \widetilde{\pi}(\mathrm{d} s, \mathrm{d} z)\right|^p\\
&\quad\leq \eta\mathbb{E} \sup_{t\in [0,T \wedge \widetilde{\tau}_\epsilon)}  \| u_\epsilon\|_{L^2}^{2p} +C\mathbb{E}\left( \int_0^{T \wedge \widetilde{\tau} _\epsilon}  \| u_\epsilon\|_{L^2}^2  \mathrm{d} t \right)^{p}+ C T^p,\quad \forall\eta>0.
\end{split}
\end{equation}
Noting that the compensated Poisson random measure $\widetilde{\pi}(\mathrm{d} s, \mathrm{d} z)=\pi(\mathrm{d} s, \mathrm{d} z)-\mu( \mathrm{d} z)\mathrm{d} s$, it follows from the BDG inequality, Young inequality and assumption $(\textbf{A}_4)$ with $p=4$ that
\begin{equation}
\begin{split}
&\mathbb{E} \sup_{t\in [0,T \wedge \widetilde{\tau}_\epsilon)}  \left|c^\dag \int_0^t \int_{Z\backslash Z_0}  \| G(u_\epsilon(s-),z)\|_{L^2}^2  \pi(\mathrm{d} s, \mathrm{d} z)\right|^p\\
&\quad\leq C \left[\mathbb{E}  \left(  \int_0^{T \wedge \widetilde{\tau}_\epsilon} \int_{Z\backslash Z_0}  \| G(u_\epsilon(s-),z)\|_{L^2}^4   \mu( \mathrm{d} z)\mathrm{d} s  \right)^{\frac{p}{2}} + \mathbb{E} \left( \int_0^{T \wedge \widetilde{\tau}_\epsilon }(\| u_\epsilon\|_{L^2}^2+1)  \mathrm{d} s  \right)^p\right]\\
&\quad\leq \eta\mathbb{E} \sup_{t\in [0,T \wedge \widetilde{\tau}_\epsilon)} \left(\| u_\epsilon\|_{L^2}^{2p}+1\right)+C \mathbb{E}\left( \int_0^{T \wedge \widetilde{\tau}_\epsilon }  \| u_\epsilon\|_{L^2}^2  \mathrm{d} t \right)^{p}+ C  T^p,\quad \forall\eta>0.
\end{split}
\end{equation}
The last integral including the functional $G$ can be estimated by
\begin{equation}\label{f2}
\begin{split}
&\mathbb{E} \sup_{t\in [0,T \wedge \widetilde{\tau}_\epsilon)}  \left|c^\dag \int_0^t \int_{Z\backslash Z_0}  2\langle u_\epsilon, \mathcal {P}G(u_\epsilon(s-),z)\rangle   \pi(\mathrm{d} s, \mathrm{d} z)\right|^p\\
&\quad\leq C\mathbb{E} \left[ \int_0^{T \wedge \widetilde{\tau}_\epsilon } \left(\| u_\epsilon\|_{L^2}^2+\mu(Z\backslash Z_0)\int_{Z\backslash Z_0} \|\mathcal {P}G(u_\epsilon,z)\|_{L^2}^2  \mu( \mathrm{d} z)\right)\mathrm{d} s\right]^p\\
&\quad+C\mathbb{E} \left( \int_0^{T \wedge \widetilde{\tau}_\epsilon} \| u_\epsilon\|_{L^2}^2\int_{Z\backslash Z_0}  \|\mathcal {P}G(u_\epsilon,z)\|_{L^2}^2 \mu( \mathrm{d} z)\mathrm{d} s \right)^{\frac{p}{2}}\\
&\quad\leq \eta\mathbb{E} \sup_{t\in [0,T \wedge \widetilde{\tau}_\epsilon)} \left(\| u_\epsilon\|_{L^2}^2+1\right)^p+C \mathbb{E}\left( \int_0^{T \wedge \widetilde{\tau}_\epsilon }  \| u_\epsilon\|_{L^2}^2  \mathrm{d} t \right)^{p}+ C T^p,\quad \forall\eta>0.
\end{split}
\end{equation}
Collecting the estimates \eqref{(4.6)} and \eqref{f1}-\eqref{f2} together, taking the supremum over the interval $[0,T]$ and then the mathematical expectation, we get for any $\eta>0$
\begin{equation}\label{4.27}
\begin{split}
&\mathbb{E}\left(\sup_{t\in [0,T \wedge \widetilde{\tau}_\epsilon)}  \mathscr{E} [n_{\epsilon},c_{\epsilon},u_{\epsilon}](t) +   \int_0^{T \wedge \widetilde{\tau}_\epsilon} \mathscr{I}[n_{\epsilon},c_{\epsilon},u_{\epsilon}](s)\mathrm{d}s \right)^p\\
& \quad\leq5\eta\mathbb{E} \sup_{t\in [0,T \wedge \widetilde{\tau}_\epsilon)}\| u_\epsilon(t)\|_{ L^2 }^{2p}+C\mathbb{E}\left( \int_0^{T \wedge \widetilde{\tau}_\epsilon }  \| u_\epsilon\|_{L^2}^2  \mathrm{d} t \right)^{p}+C\mathbb{E}\left(\mathscr{E}[n_{\epsilon0},c_{\epsilon0},u_{\epsilon0}]^p\right)+C T^{p}.
\end{split}
\end{equation}
By choosing $0< \eta\leq \frac{c^\dagger}{10}$, we deduce from \eqref{4.27}, the basic inequality $a^p+b^p\leq (a+b)^p\leq 2^{p-1}(a^p+b^p)$ as well as the H\"{o}lder inequality that
\begin{equation*}
\begin{split}
\mathbb{E}\sup_{t\in [0,T \wedge \widetilde{\tau}_\epsilon)}\| u_\epsilon(t)\|_{ L^2 }^{2p} \leq CT^{p-1}\mathbb{E}  \int_0^ {T \wedge \widetilde{\tau}_\epsilon}\| u_\epsilon (t)\|_{L^2}^{2p} \mathrm{d}t +C\mathbb{E}\left(\mathscr{E}[n_{\epsilon0},c_{\epsilon0},u_{\epsilon0}]^p\right)+CT^{p},
\end{split}
\end{equation*}
which combined with the Gronwall inequality leads to
\begin{equation*}
\begin{split}
& \mathbb{E}  \sup_{t\in [0,T \wedge \widetilde{\tau}_\epsilon)}\| u_\epsilon(t)\|_{ L^2 }^{2p} \leq C e^{CT^{p-1}}\left(\mathbb{E}\left(\mathscr{E}[n_{\epsilon0},c_{\epsilon0},u_{\epsilon0}]^p\right)+ T^{p}\right).
\end{split}
\end{equation*}
Plugging the last estimate into \eqref{4.27} leads to
\begin{equation*}
\begin{split}
\mathbb{E}\left(\sup_{t\in [0,T \wedge \widetilde{\tau}_\epsilon)}  \mathscr{E} [n_{\epsilon},c_{\epsilon},u_{\epsilon}] +   \int_0^{T \wedge \widetilde{\tau}_\epsilon} \mathscr{I}[n_{\epsilon},c_{\epsilon},u_{\epsilon}]\mathrm{d}s \right)^p\leq C e^{CT^{p-1}}\left(\mathbb{E}\left(\mathscr{E}[n_{\epsilon0},c_{\epsilon0},u_{\epsilon0}]^p\right)+ T^{p}\right),
\end{split}
\end{equation*}
which implies the desired energy estimate. The proof of Lemma \ref{lem4.6} is completed.
 \end{Proof}

\subsection{Global existence of approximate solutions}

Based on the uniform bounds in Lemma \ref{lem4.6}, it is sufficient to prove that the approximate solution  is indeed global in time.
\begin{lemma}\label{lem4.7}
For each $\epsilon \in (0,1)$, the approximate solution $(n_{\epsilon},c_{\epsilon},u_{\epsilon},\widetilde{\tau}_\epsilon)$ ensured by Lemma \ref{lem1} is global in time, that is, $\mathbb{P}[\widetilde{\tau}_\epsilon=\infty]=1$.
\end{lemma}

\begin{Proof}
For any $l>0$, we define
\begin{equation*}
\tau_l\overset {\textrm{def}} {=}
\left\{
 \begin{array}{ll}
   \inf\left\{0\leq t<\infty ;~ \max\left\{ \int_0^t\|\nabla c_\epsilon^{\frac{1}{4}} \|_{L^4} ^4\mathrm{d}s ,\|u_\epsilon(t)\|_{L^2}, \int_0^t\|\nabla u_\epsilon(s)\|_{L^2} ^2\mathrm{d}s\right\} >l\right\},  \\
    +\infty,  \quad \textrm{if the above set $\{\cdot\cdot\cdot\}$ is empty.}
  \end{array}
\right.
\end{equation*}
Since $u_\epsilon$ is a c\`{a}dl\`{a}g process in $\mathscr{D}(\textrm{A})\subseteq L^2(\mathcal {O})$ and the triple $(n_\epsilon,c_\epsilon,u_\epsilon)$ is $\mathfrak{F}$-adapted, $(\tau_l)_{l>0}$ is a sequence of $\mathbb{P}$-a.s. nonnegative stopping times, and $\tau_l\rightarrow\infty$ as $l\rightarrow\infty$, $\mathbb{P}$-a.s. Thanks to the uniform bound in Lemma \ref{lem4.6}, we deduce from \eqref{(4.9)} in Lemma \ref{lem4.5} and the definition of $ \tau_l$ that
\begin{equation}\label{(4.31)}
\begin{split}
\mathbb{E}  \int_0^{T\wedge \tau_l}\|\nabla c_\epsilon(s)\|_{L^4} ^4\mathrm{d}s\leq \|c_0\|_{L^\infty}^3\mathbb{E}  \int_0^{T\wedge \tau_l}\left\|\nabla c_\epsilon^{\frac{1}{4}}(s)\right\|_{L^4} ^4\mathrm{d}s\leq C .
\end{split}
\end{equation}
Applying the chain rule to $\|n_\epsilon(t) \|_{L^p}^p$ with $p\in[2,4]$, integrating by parts and using the divergence-free condition, we infer that for all $t\in [0,T\wedge \tau_l]$
\begin{equation}\label{(4.32)}
\begin{split}
& \|n_\epsilon (t)\|_{L^p}^p+p(p-1) \int_0^t \int_\mathcal {O} n_\epsilon^{p-2} |\nabla n_\epsilon|^2\mathrm{d}x\mathrm{d} s\\
&\quad =\|n_{\epsilon0}\|_{L^p}^p +p(p-1) \int_0^t \int_\mathcal {O} n_\epsilon^{p-2}n_\epsilon \textbf{h}_\epsilon' (n_\epsilon)\chi(c_\epsilon)\nabla n_\epsilon\cdot\nabla c_\epsilon \mathrm{d}x\mathrm{d} s.
\end{split}
\end{equation}
Note that $2(p-2)\leq p$ for all $p\in[2,4]$, $0< n_\epsilon\textbf{h}_\epsilon' (n_\epsilon)\leq\frac{1}{\epsilon}$ and $\|\chi(c_\epsilon)\|_{L^\infty}\leq C$ by Remark \ref{remark4.2}. It follows  that  for $t\in [0,T\wedge \tau_l]$
\begin{equation*}
\begin{split}
&  \mathbb{E}\left|p(p-1) \int_\mathcal {O} n_\epsilon^{p-2}n_\epsilon \textbf{h}_\epsilon' (n_\epsilon)\chi(c_\epsilon)\nabla n_\epsilon\cdot\nabla c_\epsilon \mathrm{d}x \right|\\
&\quad\leq \frac{p(p-1)}{2}\mathbb{E} \int_\mathcal {O} n_\epsilon^{p-2} | \nabla n_\epsilon|^2 \mathrm{d}x+ C\left( \int_\mathcal {O} n_\epsilon^{2(p-2)}  \mathrm{d}x+ \mathbb{E} \int_\mathcal {O} |\nabla c_\epsilon|^4  \mathrm{d}x\right)\\
&\quad\leq \frac{p(p-1)}{2}\mathbb{E} \int_\mathcal {O} n_\epsilon^{p-2} | \nabla n_\epsilon|^2 \mathrm{d}x+ C\mathbb{E} \int_\mathcal {O} n_\epsilon^p  \mathrm{d}x+C,
\end{split}
\end{equation*}
which together with \eqref{(4.31)}-\eqref{(4.32)} and the Gronwall inequality yields that
\begin{equation}\label{(4.33)}
\begin{split}
 \mathbb{E} \sup_{t\in [0,T\wedge \tau_l]} \|n_\epsilon (t)\|_{L^p}^p+ \mathbb{E} \int_0^{T\wedge \tau_l} \int_\mathcal {O} n_\epsilon^{p-2} |\nabla n_\epsilon|^2\mathrm{d}x\mathrm{d} s \leq C\left(\mathbb{E}\|n_{ 0}\|_{L^p}^p+1\right) .
\end{split}
\end{equation}
By making use of the uniform bound \eqref{(4.20)}, we now prove the boundedness for $u_\epsilon$. First,  basic properties for $\textbf{L}_\epsilon$ imply that, for all $t \in [0,T\wedge \tau_l]$,
\begin{equation}\label{(4.34)}
\begin{split}
\|\mathcal {P} (\textbf{L}_\epsilon u_\epsilon\cdot \nabla) u_\epsilon \|_{L^2}^2\leq \|\textbf{L}_\epsilon u_\epsilon\|^2_{L^\infty}\|\nabla u_\epsilon \|_{L^2}^2 \leq C \|\nabla u_\epsilon \|_{L^2}^2 .
\end{split}
\end{equation}
Applying the It\^{o} formula to $\|\textrm{A}^{\frac{1}{2}}u_\epsilon\|_{L^2}^2$, after integrating by parts and using \eqref{(4.34)} and the Young inequality, we infer that for all $t \in [0,T\wedge \tau_l]$
\begin{equation}\label{(4.35)}
\begin{split}
& \left\|\textrm{A}^{\frac{1}{2}}u_\epsilon (t) \right\|_{L^2}^2+  \int_0^t\|\textrm{A} u_\epsilon \|_{L^2}^2\mathrm{d} s\\
 &\quad\leq  \left\|\textrm{A}^{\frac{1}{2}}u_{\epsilon0} \right\|_{L^2}^2+C_l \int_0^t\left(\|\textrm{A}^{\frac{1}{2}} u_\epsilon \|_{L^2}^2 +1 \right)\mathrm{d}s +2\left| \int_0^t\langle \textrm{A} ^{\frac{1}{2}}u_\epsilon,  \textrm{A}^{\frac{1}{2}}g(s,u_\epsilon) \mathrm{d}W_s\rangle_{L^2}\right|\\
&\quad+\left| \int_0^t \int_{Z_0}\left(\| \textrm{A}^{\frac{1}{2}}K(u_\epsilon(s-),z) \|_{ L^2}^2-2\langle \textrm{A}^{\frac{1}{2}} u_\epsilon,  \textrm{A}^{\frac{1}{2}} K(u_\epsilon(s-),z)   \rangle_{L^2}\right)       \widetilde{\pi}(\mathrm{d}s,\mathrm{d}z)\right|\\
&\quad+\left| \int_0^t \int_{Z\backslash Z_0}\left(\| \textrm{A}^{\frac{1}{2}} G(u_\epsilon(s-),z) \|_{L^2}^2-2\langle \textrm{A}^{\frac{1}{2}} u_\epsilon,  \textrm{A}^{\frac{1}{2}} G(u_\epsilon(s-),z)   \rangle_{L^2}\right) \pi(\mathrm{d}s,\mathrm{d}z)\right|,
\end{split}
\end{equation}
where the estimate \eqref{(4.20)} with $p=2$ and the assumption $(\textbf{A}_4)$ with $\alpha=\frac{1}{2}$ are  applied. By taking  the supremum in time over $[0, T\wedge \tau_l]$ and then the mathematical expectation, we get by applying the BDG inequality and the assumptions on $g,K,G$ and $h$ that
\begin{equation*}
\begin{split}
& \frac{1}{2}\mathbb{E} \sup_{t\in [0,T\wedge \tau_l]}\|\textrm{A}^{\frac{1}{2}}u_\epsilon (t) \|_{L^2}^2+ \mathbb{E}  \int_0^t\|\textrm{A} u_\epsilon (s)\|_{L^2}^2\mathrm{d} s\\
 &\quad\leq  \mathbb{E}\|\textrm{A}^{\frac{1}{2}}u_{\epsilon0} \|_{L^2}^2+C_l\mathbb{E} \int_0^{T\wedge \tau_l}\left(\|\textrm{A}^{\frac{1}{2}} u_\epsilon (s) \|_{L^2}^2 +1 \right)\mathrm{d}s.
\end{split}
\end{equation*}
Thereby, by utilizing the equivalence between the norms $\|\textrm{A}^{\alpha}u_\epsilon \|_{L^2}$ and $\|(-\Delta)^\alpha u_\epsilon  \|_{L^2}$  and then the Gronwall inequality, we gain
\begin{equation}\label{(4.36)}
\begin{split}
&  \mathbb{E} \sup_{t\in [0,T\wedge \tau_l]}\|\nabla u_\epsilon (t) \|_{L^2}^2+ \mathbb{E} \int_0^{T\wedge \tau_l}\|\Delta u_\epsilon (s)\|_{L^2}^2\mathrm{d} s\leq C.
\end{split}
\end{equation}

Next, we need to explore the evolution of the stochastic process $\textrm{A}^\alpha u_{\epsilon}(t)$, $\alpha\in (\frac{3}{4},1)$,  by making good use of the $u_\epsilon$-equation in the mild form. After taking the $L^2$-norm on both sides of the formula of variation of constants, we get for all $t \in [0,T\wedge \tau_l]$
\begin{equation}\label{(4.37)}
\begin{split}
  \|\textrm{A}^\alpha u_\epsilon (t)\|_{L^2}&\leq   \|\textrm{A}^\alpha e^{-t \textrm{A}} u_{\epsilon0} \|_{L^2} + \int_0^t  \|\textrm{A}^\alpha e^{-(t-s) \textrm{A}}\mathcal {P} (\textbf{L}_\epsilon u_\epsilon\cdot \nabla) u_\epsilon \| _{L^2} \mathrm{d}s\\
  &+ \int_0^t  \left\|\textrm{A}^\alpha e^{-(t-s) \textrm{A}} \mathcal {P}(n_\epsilon\nabla \Phi) \right\| _{L^2}\mathrm{d}s  + \int_0^t  \left\|\textrm{A}^\alpha e^{-(t-s) \textrm{A}} \mathcal {P}h(s,u_\epsilon)  \right\|_{L^2}\mathrm{d}s\\
  &+\left \| \int_0^t \textrm{A}^\alpha e^{-(t-s) \textrm{A}} \mathcal {P} g(s,u_\epsilon) \mathrm{d}W_s\right\|_{L^2}  \\
 & +\left\| \int _0^t   \int_{Z_0} \textrm{A}^\alpha e^{-(t-s) \textrm{A}} \mathcal {P}  K(u_\epsilon(s-),z)\widetilde{\pi}(\mathrm{d}s,\mathrm{d}z)\right\|_{L^2}\\
 & +\left\| \int_0^t    \int_{Z\backslash Z_0} \textrm{A}^\alpha e^{-(t-s) \textrm{A}}\mathcal {P} G(u_\epsilon(s-),z)\pi(\mathrm{d}s,\mathrm{d}z)\right \|_{L^2}\\
 & \overset{\textrm{def}}{=} U_1(t)+\cdot\cdot\cdot+ U_4(t) +\|U_5(t)\|_{L^2}+\cdot\cdot\cdot+\|U_7(t)\|_{L^2}.
\end{split}
\end{equation}
For $U_1(t)$, we have for any $\zeta\in (0,T\wedge \tau_l)$
\begin{equation*}
\begin{split}
U_1(t)\leq Ct^{-\alpha}\| u_{\epsilon0} \|_{L^2}\leq C \| u_{ 0} \|_{L^2}, \quad \textrm{for all}~t \in [\zeta,T\wedge \tau_l].
\end{split}
\end{equation*}
For $U_2(t)$, by applying the smoothing estimate of Stokes semigroup (cf. \cite{giga1986solutions}),  we deduce from \eqref{(4.36)} that
\begin{equation*}
\begin{split}
\mathbb{E}\sup_{t\in [0,T\wedge \tau_l]}U_2^2(t)
&\leq C\mathbb{E}\left(\sup_{t\in [0,T\wedge \tau_l]}\| \nabla u_\epsilon \| _{L^2} \int_0^{T\wedge \tau_l}  (T\wedge \tau_l-s)^{-\alpha} \mathrm{d}s\right)^2\\
&\leq  C\frac{T^{1-\alpha}}{1-\alpha}  \mathbb{E}\sup_{t\in [0,T\wedge \tau_l]}\| \nabla u_\epsilon \| _{L^2}^2\leq C.
\end{split}
\end{equation*}
For $U_3(t)$ and $U_4(t)$, it follows from \eqref{(4.20)} with $p=2$ and the condition on $h$ that
\begin{equation*}
\begin{split}
&\mathbb{E}\sup_{t\in [0,T\wedge \tau_l]}(U_3+U_4 )^2(t)\leq C\mathbb{E}\left( \int_0^{T\wedge \tau_l}  (T\wedge \tau_l-s)^{-\alpha} (\| n_\epsilon \| _{L^2}+\|u_\epsilon\|_{L^2}+1) \mathrm{d}s\right)^2 \leq C.
 \end{split}
\end{equation*}
The term $U_5(t)$ can be treated similar to \eqref{(3.8)}, and one can obtain for any $t\in (0,T\wedge \tau_l]$
\begin{equation*}
\begin{split}
&\mathbb{E}\sup_{t\in [0,T\wedge \tau_l]} (U_5(t))^2\leq CT+C \mathbb{E} \int_0^t \|\textrm{A}^\alpha u_\epsilon (s)\|_{L^2}^2 \mathrm{d}s.
 \end{split}
\end{equation*}
To estimate $U_6(t)$, we observe that $
\mathrm{d} U_6(t)= - \textrm{A} U_6(t)\mathrm{d}t+ \int_{Z_0} \textrm{A}^\alpha \mathcal {P}  K(u_\epsilon(s-),z)\widetilde{\pi}(\mathrm{d}t,\mathrm{d}z)$ with $U_6(0)=0$. By applying the It\^{o} formula to $\|U_6(t)\|_{L^2}^2$ and then integrating by parts, it leads to
\begin{equation*}
\begin{split}
 \|U_6(t)\|_{L^2}^2 +2 \int_0^t\|\nabla U_6\|_{L^2}^2 \mathrm{d}s&\leq 2\left| \int_0^t \int_{Z_0}\langle U_6, \textrm{A}^\alpha \mathcal {P}K(u_\epsilon(s-),z)\rangle_{L^2} \widetilde{\pi}(\mathrm{d}s,\mathrm{d}z)\right| \\
&+ \int_0^t \int_{Z_0}  \|\textrm{A}^\alpha \mathcal {P}K(u_\epsilon(s-),z) \|_{L^2} ^2 \mu(\mathrm{d}z)\mathrm{d}s,
 \end{split}
\end{equation*}
for all $t\in [0,T\wedge \tau_l]$. By taking the supremum in time over $[0,T\wedge\tau_l]$ and then mathematical expectation, it follows from the assumption on $K$ that, for any $\eta>0$,
\begin{equation*}
\begin{split}
&\mathbb{E}\sup_{t\in [0,T\wedge \tau_l]} \|U_6(t)\|_{L^2}^2 +2   \int_0^{T\wedge \tau_l}\|\nabla U_6(s)\|_{L^2}^2 \mathrm{d}s\\
&\quad\leq  C\mathbb{E} \left[ \int_0^{T\wedge \tau_l} \int_{Z_0}\left|\langle U_6,\textrm{A}^\alpha \mathcal {P} K(u_\epsilon(s-),z)\rangle_{L^2} \right|^2\mu(\mathrm{d}z)\mathrm{d}s \right]^{\frac{1}{2}} \\
&\quad+\mathbb{E} \int_0^{T\wedge \tau_l} \int_{Z_0}  \|\textrm{A}^\alpha \mathcal {P}K(u_\epsilon(s-),z) \|_{L^2} ^2 \mu(\mathrm{d}z)\mathrm{d}s\\
&\quad\leq  \eta \mathbb{E}\sup_{t\in [0,T\wedge \tau_l]} \|U_6(t)\|_{L^2}^2+ C \mathbb{E} \int_0^{T\wedge \tau_l}\left(\|\textrm{A}^\alpha u_\epsilon (s)\|_{L^2}^2+1\right)\mathrm{d}s.
 \end{split}
\end{equation*}
By choosing $\eta>0$ small enough in the last estimate, it leads to
\begin{equation*}
\begin{split}
 \mathbb{E}\sup_{t\in [0,T\wedge \tau_l]} \|U_6(t)\|_{L^2}^2 \leq  C \mathbb{E} \int_0^{T\wedge \tau_l}\left(\| \textrm{A}^\alpha u_\epsilon (s)\|_{L^2}^2+1\right)\mathrm{d}s.
 \end{split}
\end{equation*}
Using the assumption on $G$ and the BDG inequality, the stochastic integral $U_7(t)$ can be estimated similar to $U_6(t)$, and we infer that
\begin{equation*}
\begin{split}
 \mathbb{E}\sup_{t\in [0,T\wedge \tau_l]} \|U_7(t)\|_{L^2}^2 \leq  C \mathbb{E} \int_0^{T\wedge \tau_l}\left(\| \textrm{A}^\alpha u_\epsilon (s)\|_{L^2}^2+1\right)\mathrm{d}s.
 \end{split}
\end{equation*}
To obtain the desired estimate, we take the $2$-th power on both sides of \eqref{(4.37)}, the supremum in time over $[0,T\wedge \tau_l]$, and apply expectations. Summarizing the previous estimates for $U_1(t)\sim U_7(t)$, we obtain from the Gronwall inequality that
\begin{equation}\label{(4.38)}
\begin{split}
 \mathbb{E}\sup_{t\in [0,T\wedge \tau_l]} \|\textrm{A}^\alpha u_\epsilon (t)\|_{L^2}^2\leq C.
\end{split}
\end{equation}
In view of the embedding $\mathscr{D}(\textrm{A}^\alpha) \subset L^\infty(\mathcal {O})$ with $\frac{3}{4}< \alpha <1$, \eqref{(4.38)} implies that
\begin{equation}
\begin{split}
 \| u_\epsilon (t)\|_{L^\infty} \leq  C\| u_\epsilon (t)\|_{\mathscr{D}(\textrm{A}^\alpha)} <\infty,\quad \mathbb{P}\textrm{-a.s.},
\end{split}
\end{equation}
for all $t\in [\zeta,T\wedge \tau_l]$ and any $\zeta\in (0,\widetilde{\tau}_\epsilon)$.

Now we need to explore the evolution of the quantity $\|\nabla c_\epsilon\|_{L^q}$ with $q>3$. To this end, we employ the operator $\nabla$ to both sides of the variation-of-constants formula to obtain
\begin{equation*}
\begin{split}
&\nabla c_\epsilon(t)=  \nabla  e^{(t-\frac{\zeta}{2})\Delta}  c_{\epsilon } (\frac{\zeta}{2})- \int_{\frac{\zeta}{2}}^t \nabla  e^{(t-s)\Delta} \left( u_\epsilon\cdot \nabla c _\epsilon + \textbf{h}_\epsilon(n_\epsilon) f(c_\epsilon)  \right)\mathrm{d}s,\quad t\in (\frac{\zeta}{2},T\wedge \tau_l].
 \end{split}
\end{equation*}
Taking advantage of the smooth effect of Neaumann heat semigroup (cf. \cite{winkler2015boundedness}), the bound \eqref{(4.35)} and \eqref{(4.33)}, we get from the Duhamel formula of $c_\epsilon$-equation that
\begin{equation*}
\begin{split}
 &\mathbb{E}\sup_{t\in [ \varsigma,T\wedge \tau_l]}\|\nabla c_\epsilon(t)\|_{L^4}\\
& \quad\leq C\zeta^{-\frac{1}{2}}
+ C\mathbb{E} \left(\sup_{t\in [ \frac{\varsigma}{2},T\wedge \tau_l]}\|n_\epsilon(t)\|_{L^4} \sup_{t\in [\varsigma,T\wedge \tau_l]} \int_{\frac{\zeta}{2}}^t (t-s)^{-\frac{1}{2}} \mathrm{d}s\right)\\
&\quad+ C\mathbb{E}\left(\sup_{t\in [ \frac{\varsigma}{2},T\wedge \tau_l]}\|u_\epsilon\|_{L^\infty}\sup_{t\in [ \varsigma,T\wedge \tau_l]} \int_{\frac{\zeta}{2}}^t (t-s)^{-\frac{1}{2}} \|\nabla c_\epsilon(s)\|_{L^4}\mathrm{d}s\right)\\
& \quad\leq C\sup_{t\in [ \varsigma,T ]}\left( \int_{\frac{\zeta}{2}}^t (t-s)^{-\frac{2}{3}} \mathrm{d}s\right)^{3}\mathbb{E} \int_{\frac{\zeta}{2}}^{T\wedge \tau_l} \|\nabla c_\epsilon(s)\|_{L^4}^4\mathrm{d}s+ C\zeta^{-\frac{1}{2}}
+ CT +C,
 \end{split}
\end{equation*}
which combined with \eqref{(4.31)} implies that
\begin{equation}\label{4.37}
\begin{split}
 \mathbb{E}\sup_{t\in [\zeta,T\wedge \tau_l]} \|\nabla c_\epsilon(t)\|_{L^4}  \leq  C.
 \end{split}
\end{equation}
Now we assert that it is possible to obtain an enhanced estimate than \eqref{(4.31)}. Setting the sectorial operator $B=I-\Delta$ with homogeneous Neamann boundary condition, and rewriting the diffusion term $\Delta c_\epsilon$ as $c_\epsilon -B c_\epsilon $. It then follows from the $L^p$-$L^q$ estimates (cf. \cite{winkler2010aggregation}) and the embedding $W^{2\beta,4}(\mathcal {O})\subset W^{1,\infty}(\mathcal {O})$ with $\beta\in (\frac{7}{8},1)$ that
\begin{equation*}
\begin{split}
\|\nabla c_\epsilon(t)\|_{L^\infty}
&\leq \|B^\beta e^{-t B} c_{\epsilon 0}\|_{L^4}+ \int_0^t \left\|B^\beta e^{-(t-s) B} \left(u_\epsilon\cdot \nabla c _\epsilon -c_\epsilon +\textbf{h}_\epsilon(n_\epsilon) f(c_\epsilon)\right) \right\|_{L^4}\mathrm{d}s\\
&\leq Ct^{-\beta}\| c_{\epsilon 0}\|_{L^4}+ \int_0^t (t-s)^{-\beta} \left(\| \nabla c _\epsilon\|_{L^4}+ \|c_\epsilon\|_{L^4} +\| n_\epsilon \|_{L^4}\right) \mathrm{d}s,
 \end{split}
\end{equation*}
for all $t \in[\zeta,T\wedge \tau_l]$, where we used the fact  of $\textbf{h}_\epsilon(n_\epsilon)\leq n_\epsilon$ and the uniform bound \eqref{(4.33)}.  Then by taking the supremum and expectation, we deduce from \eqref{(4.33)} and \eqref{4.37} that
\begin{equation}\label{(4.41)}
\begin{split}
 \mathbb{E}\sup_{t\in [\zeta,T\wedge \tau_l]} \|\nabla c_\epsilon(t)\|_{L^\infty}  \leq  C.
 \end{split}
\end{equation}
In view of the boundedness of $\textbf{h}_\epsilon' (n_\epsilon)$ and $\chi(c_\epsilon)$, and making use of the Duhamel formula for the $n_\epsilon$-equation, we deduce from \eqref{(4.33)}, \eqref{(4.38)} and \eqref{(4.41)} that
\begin{equation}\label{4.42}
\begin{split}
&\mathbb{E}\sup_{t\in [\zeta,T\wedge \tau_l ]} \|n_\epsilon  (t ) \|_{L^\infty}
 \leq \mathbb{E} \left(\sup_{t\in [\zeta,T\wedge \tau_l]}(t-\frac{\zeta}{2})^{-\frac{3}{2}} \|n_{ \epsilon0}\|_{L^1}   \right)\\
&\quad+\mathbb{E}\sup_{t\in [\zeta,T\wedge \tau_l ]}  \int_{\frac{\zeta}{2}}^{t } (t-\frac{\zeta}{2})^{-\frac{p+3}{2p}} \left \|  u_\epsilon n_\epsilon+ n_\epsilon \textbf{h}_\epsilon' (n_\epsilon)\chi(c_\epsilon)\nabla c_\epsilon \right\|_{L^p}\mathrm{d}s\\
&\quad\leq C \|n_0\|_{L^1} +C\mathbb{E}  \int_{\frac{\zeta}{2}}^{T } (t-\frac{\zeta}{2})^{-\frac{p+3}{2p}} \left (\|  u_\epsilon\|_{L^\infty}+  \|\nabla c_\epsilon\|_{L^\infty}\right)\| n_\epsilon\|_{L^p}\mathrm{d}s \leq C.
\end{split}
\end{equation}
In view of \eqref{(4.38)}, \eqref{(4.41)} and \eqref{4.42}, we have $
\sup_{t\in [\zeta,T\wedge \tau_l ]} \|(n_\epsilon(t),c_\epsilon(t),u_\epsilon(t))\|_{L^\infty\times W^{1,q}  \times \mathscr{D}(\textrm{A}^\alpha)}<\infty
$ $\mathbb{P}$-a.s., which implies that the solution $(n_\epsilon,c_\epsilon,u_\epsilon)$ does not blow up at time $t=T\wedge \tau_l$, for any $T>0$. Moreover, from the definition of maximal existence time $\widetilde{\tau}_\epsilon$, we have
$\widetilde{\tau}_\epsilon\geq T\wedge \tau_l,$ for any $T>0$ and $l>0$. Therefore, by taking the limit $l\rightarrow\infty$, it shows that $\mathbb{P}[\widetilde{\tau}_\epsilon=\infty]=1$. The proof of Lemma \ref{lem4.7} is completed.
\end{Proof}

\section{Identification of the limits}\label{sec5}


\subsection{Further spatio-temporal regularity}
\begin{lemma}\label{lem5.1}
Assume that $T>0$  and $p\geq 1$. For any $\epsilon \in (0,1)$, let $(n_\epsilon,c_\epsilon,u_\epsilon)$ be  solutions to system \eqref{SCNS-1}, then there exists a positive constant $C$ independent of $\epsilon$ such that
\begin{subequations}
\begin{align}
&\|\nabla \sqrt{n_\epsilon} \|_{L^p(\Omega;L^2(0,T;L^2)) } \leq  C,\label{5.1a}\\
&\| n_\epsilon  \|_{L^p(\Omega;L^{\frac{5}{4}}(0,T;W^{1,\frac{5}{4}})) } \leq  C, \label{5.1b}\\
&\|\nabla \sqrt[4]{c_\epsilon}  \|_{L^p(\Omega;L^4(0,T;L^4)) } \leq  C,\label{5.1d}\\
&\|c_\epsilon\|_{L^p(\Omega;L^2(0,T;W^{2,2})) } \leq  C,\label{5.1e}\\
& \|u_\epsilon \|_{L^p(\Omega;L^\infty(0,T;L^2_\sigma))\cap L^p(\Omega;L^2(0,T;W^{1,2}_\sigma))} \leq  C,\label{5.1f}\\
&\|u_\epsilon  \|_{L^p(\Omega;L^{\frac{10}{3}}(0,T;L^{\frac{10}{3}})) } \leq  C.\label{5.1g}
\end{align}
\end{subequations}
\end{lemma}

\begin{proof}[\textbf{\emph{Proof.}}]
The uniform bounds \eqref{5.1a}, \eqref{5.1d}, \eqref{5.1e} and \eqref{5.1f} are direct consequence of the uniform boundedness estimate in Lemma \ref{lem4.6}. The estimates \eqref{5.1b} and \eqref{5.1g} can be verified by utilizing the conservation of mass and properly making use of the GN inequality, and we omit the details here.
\end{proof}

\begin{lemma}\label{lem5.2}
For any $T>0$ and $\epsilon \in (0,1)$, there exists $C>0$ independent of $\epsilon$ such that for all $p\geq 1$, there hold
\begin{align}
&\mathbb{E} \left(\int_0^T\|\partial_t n_\epsilon \|_{(W^{1, 11})^*}^{\frac{11}{10}} \mathrm{d} t\right)^p \leq  C, \quad \textrm{for some} ~ q>3,\label{5.2}\\
&\mathbb{E}\left( \int_0^T\|\partial_t \sqrt{c_\epsilon}  \|_{(W^{1,\frac{5}{2}})^*}^{\frac{5}{3}}\mathrm{d} t\right)^p \leq  C.\label{5.3}
\end{align}
\end{lemma}

\begin{proof}[\textbf{\emph{Proof.}}]
For any $\varphi \in W^{1,11}(\mathcal {O})$, multiplying the $n_\epsilon$-equation by $\varphi$ and then integrating with respect to $x $ on $\mathcal {O}$, it follows from the H\"{o}lder inequality that
\begin{equation*}
\begin{split}
\left| \int_\mathcal {O}\partial_t n_\epsilon \varphi \mathrm{d} x\right| &\leq \left|  \int_\mathcal {O}\left(u_\epsilon n_\epsilon   -  \nabla n_\epsilon + n_\epsilon \textbf{h}_\epsilon' (n_\epsilon)\chi(c_\epsilon)\nabla c_\epsilon \right) \cdot \nabla\varphi \mathrm{d} x \right| \\
 & \leq C \left(\|u_\epsilon n_\epsilon\|_{L^\frac{11}{10}}+ \|\nabla n_\epsilon\|_{L^\frac{11}{10}}+\|n_\epsilon \nabla c_\epsilon \|_{L^\frac{11}{10}}  \right)\| \nabla\varphi\|_{L^{11}}.
\end{split}
\end{equation*}
It then follows from the definition of the norm in $(W^{1,11}(\mathcal {O}))^*$, the basic inequality $(a+b)^p\leq C (a^p+ b^p)$ and the Young inequality that, for all $p\geq 1$,
\begin{equation}\label{cc1}
\begin{split}
\mathbb{E} \left(\int_0^T\|\partial_t n_\epsilon \|_{(W^{1, 11})^*}^{\frac{11}{10}} \mathrm{d} t\right)^p&\leq C\mathbb{E}  \bigg(\|n_\epsilon\|_{L^\frac{5}{3}(0,T;L^\frac{5}{3})}^\frac{5p}{3}+ \|\nabla n_\epsilon\|_{L^\frac{5}{4}(0,T;L^\frac{5}{4})}^\frac{5p}{4} +\left\| \frac{\nabla c_\epsilon}{c_\epsilon^{3/4}}  \right\|_{L^4(0,T;L^4)} ^{4p} \\
 &\quad+\|u_\epsilon\|_{L^\frac{10}{3}(0,T;L^\frac{10}{3})}^\frac{10p}{3}+1 \bigg).
\end{split}
\end{equation}
Concerning the term $\mathbb{E}\| n_\epsilon  \|_{L^{5/3}(0,T;L^{5/3})}^{5p/3}$ on the R.H.S. of \eqref{cc1}, we get by Lemma \ref{concervation}, Lemma \ref{lem4.6} and the GN inequality (cf. Lemma \ref{nirenberg}) that for all $r\geq \frac{5}{3}$
\begin{equation}\label{+1}
\begin{split}
\mathbb{E} \left(\| n_\epsilon  \|_{L^{\frac{5}{3}}(0,T;L^{\frac{5}{3}})} ^r\right) &=\mathbb{E}\left(\int_0^T\|\sqrt{n_\epsilon} \|_{L^{\frac{10}{3}}}^{\frac{10}{3}} \textrm{d}t\right)^\frac{3r}{5}\\
& \leq C \mathbb{E}\left[\int_0^T\left(\|\nabla \sqrt{n_\epsilon}\|_{L^2}^2 \|\sqrt{n_\epsilon}\|_{L^2}^{\frac{4}{3}}+\|\sqrt{n_\epsilon}\|_{L^2}^{\frac{10}{3}}\right) \textrm{d}t\right]^\frac{3r}{5}\\
&\leq C \left[ \left\|n_0\right\|_{L^1}^{\frac{2r}{5}}\mathbb{E}\left(\int_0^T\int_\mathcal {O}\frac{|\nabla n_{\epsilon}|^2}{n_\epsilon}\textrm{d}x \textrm{d}t\right)^\frac{3r}{5}+\left\|n_0\right\|_{L^1}^{q} T^{\frac{3r}{5}}\right]\leq C,
\end{split}
\end{equation}
which combined with \eqref{cc1}, \eqref{5.1b}, \eqref{5.1g} and \eqref{(4.20)} in Lemma \ref{lem4.6} yields \eqref{5.2}.

To show \eqref{5.3}, we first rewrite the $c_\epsilon$-equation in the following form
\begin{equation*}
\begin{split}
 \mathrm{d} \sqrt{c_\epsilon} =-  u_\epsilon\cdot \nabla \sqrt{c _\epsilon} \mathrm{d}t+\frac{\Delta c_\epsilon}{2\sqrt{c_\epsilon} }\mathrm{d}t-\frac{\textbf{h}_\epsilon(n_\epsilon) f(c_\epsilon)}{2\sqrt{c_\epsilon} }\mathrm{d}t.
\end{split}
\end{equation*}
Then, multiplying both sides of above equation by any $\varphi\in W^{1, \frac{5}{2}}(\mathcal {O})$ and integrating by parts on $\mathcal {O}$, we deduce from the H\"{o}lder inequality that
\begin{equation}\label{5.4}
\begin{split}
 \left| \int_\mathcal {O}\partial_t \sqrt{c_\epsilon} \varphi \mathrm{d} x\right| &\leq \left|  \int_\mathcal {O} \sqrt{c _\epsilon}u_\epsilon \cdot \nabla\varphi \mathrm{d} x \right|+ \left|  \int_\mathcal {O}  \nabla \sqrt{c_\epsilon}\cdot \nabla\varphi   \mathrm{d} x \right|+ \left|  \int_\mathcal {O} |\nabla \sqrt[4]{c_\epsilon}|^2 \varphi \mathrm{d} x \right|\\
 &+\left|  \int_\mathcal {O} \frac{\textbf{h}_\epsilon(n_\epsilon) f(c_\epsilon)}{2\sqrt{c_\epsilon} } \varphi \mathrm{d} x \right|\\
 &\leq C\left(\| \sqrt{c _\epsilon}u_\epsilon\|_{L^{\frac{5}{3}}} +\| \nabla \sqrt{c_\epsilon} \|_{L^{\frac{5}{3}}} + \left\||\nabla \sqrt[4]{c_\epsilon}|^2\right\|_{L^{\frac{5}{3}}} + \left\|n_\epsilon\right\|_{L^{\frac{5}{3}}} \right) \|\varphi\|_{W^{1, \frac{5}{2}}} \\
 & \leq C\left(\| u_\epsilon\|_{L^{\frac{10}{3}}}^{2} +\| \nabla \sqrt[4]{c_\epsilon}  \|_{L^{4}} ^2+ \|n_\epsilon\|_{L^{\frac{5}{3}}} +1\right) \|\varphi\|_{W^{1, \frac{5}{2}}},
\end{split}
\end{equation}
where we used the following estimates
$$
\|\nabla \sqrt{c_\epsilon}\|_{L^{\frac{5}{3}}}\leq C\|\nabla \sqrt[4]{c_\epsilon} \|_{L^{4}}\|\sqrt[4]{c_\epsilon}\|_{L^{\frac{20}{7}}}\leq C\|\nabla \sqrt[4]{c_\epsilon} \|_{L^{4}},
$$
and for some $\beta \in (0,1)$
\begin{equation*}
\begin{split}
 \left|  \int_\mathcal {O} \frac{\textbf{h}_\epsilon(n_\epsilon) f(c_\epsilon)}{2\sqrt{c_\epsilon} } \varphi \mathrm{d} x \right| &\leq C\int_\mathcal {O}| n_\epsilon  f'(\beta c_\epsilon)\sqrt{c_\epsilon} \ \varphi| \mathrm{d} x \\
 &\leq C\sqrt{\| c_\epsilon\|_{L^\infty}}\sup_{r\in[0,\beta\| c_\epsilon\|_{L^\infty}]}|f'(r)|\int_\mathcal {O}|n_\epsilon||\varphi|\mathrm{d} x\\
 &\leq C\|n_\epsilon\|_{L^{\frac{5}{3}}}\|\varphi\|_{W^{1,\frac{5}{2}}}.
\end{split}
\end{equation*}
It follows from \eqref{5.4} and the Young inequality that
\begin{equation*}
\begin{split}
 \|\partial_t \sqrt{c_\epsilon} \|_{(W^{1, \frac{5}{2}})^*}^{\frac{5}{3}} &= \sup_{ \|\varphi\|_{W^{1, \frac{5}{2}}}=1}\left| \int_\mathcal {O}\partial_t \sqrt{c_\epsilon} \ \varphi \mathrm{d} x\right|^{\frac{5}{3}} \leq C\left(\| u_\epsilon\|_{L^{\frac{10}{3}}}^{\frac{10}{3}} +\| \nabla \sqrt[4]{c_\epsilon}  \|_{L^{4}}^4 + \|n_\epsilon\|_{L^{\frac{5}{3}}}^{\frac{5}{3}}+1\right),
\end{split}
\end{equation*}
which implies that
\begin{equation}\label{ad1}
\begin{split}
& \mathbb{E} \left( \int_0^T\|\partial_t \sqrt{c_\epsilon}\|_{(W^{1, \frac{5}{2}} )^*}^{\frac{5}{3}} \mathrm{d} s\right)^p 
  \\
  &\quad \leq C \left(\mathbb{E}\|u_\epsilon  \|_{L^{\frac{10}{3}}(0,T;L^{\frac{10}{3}})}^{\frac{10p}{3}}+ \mathbb{E}\|\nabla \sqrt[4]{c_\epsilon}  \|_{L^4(0,T;L^4))}^{4p} + \mathbb{E}\| n_\epsilon  \|_{L^{\frac{5}{3}}(0,T;L^{\frac{5}{3}})}^{\frac{5p}{3}}+T^p\right) .
\end{split}
\end{equation}
Putting the estimates \eqref{ad1}, \eqref{+1} and  the uniform bounds \eqref{5.1d}, \eqref{5.1g} in Lemma \ref{lem5.1} together yields the uniform bound \eqref{5.3}. The proof of Lemma \ref{lem5.2} is finished.
\end{proof}

In order to obtain the compactness result for the $u_\epsilon$-equation whose solution are not first-order differentiable, we use the fractional Sobolev spaces $W^{\alpha, p}(0,T; H)$: Let $p>1$, $\alpha \in(0,1)$ and $H $ be a separable Hilbert space. The space $W^{\alpha, p}(0, T ; H)$ consists of all of the measurable functions $v \in L^{p}(0,T; H)$ endowed with the norm
$$
\|v\|_{W^{\alpha, p}(0, T ; H)}^{p}\overset{\textrm{def}}{=}\|v\|_{L^{p}(0, T ; H)}^{p}+ \int_{0}^{T}  \int_{0}^{T} \frac{|v(t)-v(s)|^{p}}{|t-s|^{1+\alpha m}} \mathrm{d} t \mathrm{d} s.
$$

\begin{lemma}\label{lem5.3}
For any given $T>0$ and $\alpha \in (0,\frac{1}{2})$, there exists a positive constant $C$ independent of $\epsilon$ such that
\begin{equation}\label{(5.5)}
\begin{split}
 \mathbb{E}\left(\|u_\epsilon\|_{W^{\alpha,2}(0,T; (\mathscr{D}(\textrm{A}))^*)}^2\right) \leq C.
\end{split}
\end{equation}
\end{lemma}

\begin{proof}[\textbf{\emph{Proof.}}]
By making use of the Lemma 2.1 in \cite{flandoli1995martingale} and \eqref{5.1f}, one can deduce that for $p\geq 2$
\begin{equation}\label{(5.6)}
\begin{split}
&\mathbb{E}\left[\left\| \int_0^\cdot\mathcal {P}g(s,u_\epsilon(s)) \mathrm{d}W\right\|_{W^{\alpha,p}(0,T;L^{2}_\sigma)}^p \right]\leq C.\\
\end{split}
\end{equation}
Next we prove
\begin{equation}\label{(5.7)}
\begin{split}
\mathbb{E}\left[\left\| \int_0^\cdot \int_{Z_0} \mathcal {P}K(u_\epsilon(x,s-),z)\widetilde{\pi}(\mathrm{d}s,\mathrm{d}z) \right\|_{W^{\alpha,p}(0,T;L^2_\sigma)}^p \right] \leq C.
\end{split}
\end{equation}
To this end, setting $\mathcal {J}(t) \overset{\textrm{def}}{=}  \int_0^t  \int_{Z_0} \mathcal {P}K(u_\epsilon(s-),z)\widetilde{\pi}(\mathrm{d}s,\mathrm{d}z) $, it follows that
\begin{equation*}
\begin{split}
\mathbb{E}\left(\|\mathcal {J}(t)\|_{W^{\alpha,p}(0,T;L^2_\sigma)}^p\right)&= \mathbb{E}\left(\|\mathcal {J}(t)\|_{L^p(0,T;L^2_\sigma)}^p\right)+
\mathbb{E} \int_0^T  \int_0^T \frac{|\mathcal {J}(t)-\mathcal {J}(s)|^p}{|t-s|^{1+\alpha p}} \mathrm{d}s\mathrm{d}t \\
& \overset{\textrm{def}}{=} I_1+I_2.
\end{split}
\end{equation*}
For $I_1$, by using the BDG inequality, the assumption on $K$ and  the bound \eqref{5.1f}, we have
\begin{equation*}
\begin{split}
 I_1&\leq C \mathbb{E}\left[\left( \int_0^T \int_{Z_0} \|\mathcal {P}K(u_\epsilon(s-),z)\|_{L^2}^2\mu(\mathrm{d}z)\mathrm{d}s\right)^{\frac{p}{2}}\right] \leq C \mathbb{E}\sup_{t\in [0,T]}\left(\|u_\epsilon(s) \|_{L^2}^p+1\right)\leq C.
\end{split}
\end{equation*}
For $I_2$, by virtue of the BDG inequality, we deduce that
\begin{equation*}
\begin{split}
 I_2&=\mathbb{E} \int_0^T  \int_0^T\frac{| \int _{s\wedge t}^{ s\vee t} \int _ {Z_0} \mathcal {P}K(u_\epsilon(\sigma-),z)\widetilde{\pi}(\mathrm{d}\sigma,\mathrm{d}z)|^p}{|t-s|^{1+\alpha p}} \mathrm{d}s\mathrm{d}t\\
 &\leq C\mathbb{E} \int_0^T  \int_0^T \frac{( \int _{s\wedge t}^{ s\vee t}  \left(\|u_\epsilon(\sigma)\|_{L^2}^2+ 1\right)\mathrm{d}\sigma)^{\frac{p}{2}}}{|t-s|^{1+\alpha p}} \mathrm{d}s\mathrm{d}t.
\end{split}
\end{equation*}

$\bullet$ If $p=2$, then it follows from the Stochastic Fubini Theorem that
\begin{equation*}
\begin{split}
 I_2 &\leq 2 \mathbb{E} \int_0^T  \int_s^T  \int_s^t  \frac{\mathbb{E}( \|u_\epsilon(\sigma)\|_{L^2}^2 +1 )}{|t-s|^{1+\alpha p}}\mathrm{d}\sigma \mathrm{d}s\mathrm{d}t  \leq C \mathbb{E}  \int_0^T \left(\|u_\epsilon(t)\|_{L^2}^2 +1\right)\mathrm{d}t \leq C.
\end{split}
\end{equation*}

$\bullet$ If $p>2$, by utilizing the embedding $W^{1, \frac{p}{2}}(0,T;\mathbb{R})\subset W^{2\alpha, \frac{p}{2}}(0,T;\mathbb{R})$, we then get
\begin{equation*}
\begin{split}
 I_2 &= \mathbb{E}\left[\left\| \int_0^\cdot  \int_0^T\left(\|u_\epsilon \|_{L^2}^{2}+1\right)\mathrm{d}s\right\|_{W^{2\alpha, \frac{p}{2}}(0,T;\mathbb{R})}^\frac{p}{2}\right]\\
 &\leq C\mathbb{E}   \int_0^T\left(\|u_\epsilon \|_{L^2}^{2}+1\right)^\frac{p}{2}\mathrm{d}s + C\mathbb{E}  \int_0^T \left( \int_0^t \left( \|u_\epsilon \|_{L^2}^{2}+1\right)^p\mathrm{d}s\right)^\frac{p}{2}\mathrm{d} t\leq C.
\end{split}
\end{equation*}
In either case, one obtain the uniform bound
\begin{equation}\label{(5.8)}
\begin{split}
\mathbb{E}\left(\|\mathcal {J}(t)\|_{W^{\alpha,p}(0,T;L^2_\sigma)}^p\right)\leq C.
\end{split}
\end{equation}
Moreover, making use of the decomposition $\widetilde{\pi}(\mathrm{d}s,\mathrm{d}z)=\pi(\mathrm{d}s,\mathrm{d}z)-\mu(\mathrm{d}z)\mathrm{d}s$ and assumptions on $G$, similar to estimates for $I(t)$, we get
\begin{equation}\label{(5.9)}
\begin{split}
\mathbb{E}\left[\left\| \int_0^\cdot \int_{Z\backslash Z_0} \mathcal {P}G(u_\epsilon(s-),z)\pi(\mathrm{d}s,\mathrm{d}z) \right\|_{W^{\alpha,p}(0,T;L^2_\sigma)}^p\right] \leq C.
\end{split}
\end{equation}
To prove \eqref{(5.5)}, in view of \eqref{(5.6)}-\eqref{(5.9)} and the $u_\epsilon$-equation, it is sufficient to estimate the term
$
\mathcal {S}_\epsilon(t)\overset{\textrm{def}}{=}
u_{\epsilon0}-  \int_0^t[\mathcal {P} (\textbf{L}_\epsilon u_\epsilon\cdot \nabla) u_\epsilon - \textrm{A} u_\epsilon]\mathrm{d}s +  \int_0^t\mathcal {P}(n_\epsilon\nabla \Phi+h(s,u_\epsilon) )\mathrm{d}s
$
in  $L^2(\Omega;W^{1,2}(0,T;(\mathscr{D}(\textrm{A}))^*)) $.  Indeed, from the Sobolev embedding $L^2_\sigma (\mathcal {O}) \subset (\mathscr{D}(\textrm{A}))^*$ and $W^{1,2}(0,T)\subset W^{\alpha,2}(0,T)$, we have
\begin{equation}\label{(5.10)}
\begin{split}
&\mathbb{E}\left(\|\mathcal {S}_\epsilon \|_{W^{1,2}(0,T;(\mathscr{D}(\textrm{A}))^*) }^2\right) \leq C\left(\|u_{\epsilon0}\|+ \mathbb{E} \int_0^T\|\mathcal {P} (\textbf{L}_\epsilon u_\epsilon\cdot \nabla) u_\epsilon \|_{(\mathscr{D}(\textrm{A}))^*}^2 \mathrm{d}t \right. \\
&\quad\left.+ \mathbb{E} \int_0^T\|\textrm{A} u_\epsilon \|_{(\mathscr{D}(\textrm{A}))^*}^2 \mathrm{d}t+ \mathbb{E} \int_0^T \|\mathcal {P}(n_\epsilon\nabla \Phi+h(t,u_\epsilon) )\|_{(\mathscr{D}(\textrm{A}))^*}^2\mathrm{d}t\right).
\end{split}
\end{equation}
For the second term on the R.H.S. of \eqref{(5.10)}, we have
\begin{equation*}
\begin{split}
\mathbb{E} \int_0^T\|\mathcal {P} (\textbf{L}_\epsilon u_\epsilon\cdot \nabla) u_\epsilon \|_{(\mathscr{D}(\textrm{A}))^*}^2 \mathrm{d}t \leq C\mathbb{E}\sup_{t\in [0,T]}\|u_\epsilon\|_{L^2}^4+C\mathbb{E}\left[\left( \int_0^T\| \nabla u_\epsilon\|_{L^2}^2 \mathrm{d}t\right)^2\right]\leq C.
\end{split}
\end{equation*}
The third term on the R.H.S. of \eqref{(5.10)} can be estimated as
\begin{equation*}
\begin{split}
\mathbb{E} \int_0^T\|\textrm{A} u_\epsilon \|_{(\mathscr{D}(\textrm{A}))^*}^2 \mathrm{d}t
\leq C \mathbb{E} \int_0^T\|\nabla  u_\epsilon \|_{L^2}^2 \mathrm{d}t \leq C.
\end{split}
\end{equation*}
For the forth term, in view of the continuously embedding $ L^2(\mathcal {O})\subset L^1(\mathcal {O})\subset (\mathscr{D}(\textrm{A}))^*$, the conservation property \eqref{(4.5)} as well as the assumptions on $h$ and $\Phi$, we gain
\begin{equation*}
\begin{split}
\mathbb{E} \int_0^T \|\mathcal {P}(n_\epsilon\nabla \Phi+h(t,u_\epsilon) )\|_{(\mathscr{D}(\textrm{A}))^*}^2\mathrm{d}t\leq C \mathbb{E} \int_0^T \left(\| n_\epsilon \|_{L^1}^2 +\|u_\epsilon\|_{L^2}^2+1\right)\mathrm{d}t \leq C.
\end{split}
\end{equation*}
Plugging the last three estimates into \eqref{(5.10)} leads to $\mathbb{E}(\|\mathcal {S}_\epsilon \|_{W^{1,2}(0,T;(\mathscr{D}(\textrm{A}))^*) }^2) \leq C$, which implies the desired uniform bound \eqref{(5.5)}.
\end{proof}

\subsection{Reconstruction of approximate solutions}

\subsubsection{Tightness}
In our next result, we demonstrate that the distributions induced by the family of $\{(W_{\epsilon}, \pi_{\epsilon}, n_{\epsilon}, c_{\epsilon}, u_{\epsilon})\}_{\epsilon>0}$ is tight on a phase space $\mathcal {X}$. Recall that a family $\Lambda$ of probability measures on $(E, \mathscr{B}(E))$ is said to be \emph{tight} if and only if for any $\varepsilon>0$, there exists a compact set $K_{\varepsilon} \subset E$ such that $\mu\left(K_{\varepsilon}\right) \geq 1-\varepsilon$, for all $\mu \in \Lambda$.

To this end, we introduce the phase space
\begin{align*}
&\mathcal {X}\overset {\textrm{def}} {=}\mathcal {X}_{W } \times \mathcal {X}_{\pi } \times \mathcal {X}_{n } \times \mathcal {X}_{c } \times \mathcal {X}_{u },
\end{align*}
where
\begin{align*}
\mathcal {X}_{W }\overset {\textrm{def}} {=}&\mathcal {C}_{loc}\left([0,\infty);U\right),~~\mathcal {X}_{\pi }=\mathcal {P}_{\mathbb{N}}\left(Z\times[0,T]\right), \\
\mathcal {X}_{n }\overset {\textrm{def}} {=}& \mathcal {C}_{loc} \left([0,\infty;(W^{\frac{13}{11},11}(\mathcal {O}))^*\right)\cap\left(L_{loc}^{\frac{5}{4}}(0,\infty;W^{1,\frac{5}{4}}(\mathcal {O})),\textrm{weak}\right)\cap L_{loc}^{\frac{5}{4}} \left(0,\infty;L^{ \frac{5}{4}}(\mathcal {O})\right)\\
&\cap  \left(L_{loc}^{\frac{5}{3}}(0,\infty;L^{ \frac{5}{3}}(\mathcal {O}) ),\textrm{weak}\right) , \\
\mathcal {X}_{c }\overset {\textrm{def}} {=}& \mathcal {C}_{loc} \left([0,\infty);(W^{2,\frac{5}{2} }(\mathcal {O}))^*\right)\cap L^2_{loc}\left(0,\infty;W^{1,2}(\mathcal {O})\right)\cap \left(L^\infty_{loc}(0,\infty;L^\infty(\mathcal {O})),\textrm{weak-star}\right),\\
\mathcal {X}_{u }\overset {\textrm{def}} {=}&\mathcal {D}_{loc} \left([0,\infty);(W^{1, 5}_{0,\sigma}(\mathcal {O}))^*\right)\cap L^2_{loc}\left(0,\infty;L^2_\sigma(\mathcal {O}) \right) \cap \left(L^2_{loc}(0,\infty;W^{1,2}(\mathcal {O})),\textrm{weak}\right).
\end{align*}
Here $(X,\textrm{weak})$ ($(X,\textrm{weak-star})$, resp.) stands for a Banach space $X$ equipped  with the weak topology (weak-star topology, resp.).

In the following, the c\`{a}dl\`{a}g space $\mathcal {D} ([0,T];(W^{1, 5}_{0,\sigma}(\mathcal {O}))^*)$ is endowed with the Skorokhod topology, which makes it  separable and metrizable by a complete metric  (cf. \cite[Section 12]{billingsley2013convergence} and \cite[Section 5]{brzezniak2019weak}).

For any $\epsilon \in (0,1)$, let $(n_\epsilon,c_\epsilon,u_\epsilon)$ be the global approximate solutions to the regularized system \eqref{SCNS-1} guaranteed by Lemma \ref{lem4.6}, then one can prove the following result.

\begin{lemma}\label{lem5.4}
Denote by $\Pi_\epsilon (\cdot)$ the joint distribution  of $(W_\epsilon,\pi_\epsilon,n_\epsilon,c_\epsilon,u_\epsilon)$ on $\mathcal {X}$ with respect to the measure $\mathbb{P}$, that is,
$$
\Pi_\epsilon(B)\overset {\textrm{def}} {=}\mathbb{P}\left[\omega \in \Omega:\left(W_{\epsilon}(\omega, \cdot), \pi_{\epsilon}(\omega, \cdot), n_{\epsilon}(\omega, \cdot), c_{\epsilon}(\omega, \cdot), u_{\epsilon}(\omega, \cdot)\right) \in B\right] ,\quad \forall B \subseteq \mathcal {X}.
$$
Then, the sequence $(\Pi_\epsilon)_{\epsilon> 0}$ is tight on $\mathcal {X}$.
\end{lemma}

The following result provides an effective method to prove the tightness of processes on a c\`{a}dl\`{a}g function space endowed with the Skorohod topology.
\begin{lemma} [\cite{aldous1978stopping,nguyen2021nonlinear,aldous1989stopping}] \label{aldous}
Let $E$ be a separable Banach space, and $\left(Y_{k}\right)_{k \geq 1}$ be a sequence of $E$-valued random variables. Assume that for each stopping time sequence $ (\tau_{k} )_{k\geq 1}$ with $\tau_{k} \leq T$ and for each $\theta \geq 0$ the following condition holds
$$
\sup_{k\geq 1} \mathbb{E}\left(\left\|Y_{k}\left(\tau_{k}+\theta\right)-Y_{k}\left(\tau_{k}\right)\right\|_{E}^{\alpha}\right) \leq C \theta^{\epsilon} ,
$$
for some $\alpha, \epsilon>0$ and a constant $C>0$. Then the distributions of
 $(Y_{k})_{k\geq1}$ form a tight sequence on $\mathcal {D}([0, T] ; E)$ endowed with the Skorohod topology.
\end{lemma}

\begin{proof}[\textbf{\emph{Proof of Lemma \ref{lem5.4}.}}] The proof is divided into five steps.

\textsc{Step 1.}
Recall that every Borel probability measure on a complete separable metric space is tight (cf. \cite[Theorem 1.3]{billingsley2013convergence}). On the one hand, it is an established fact that the space  $\mathcal {C}([0, T] ; U)$ is separable and metrizable by a complete metric, which implies that the sequence of probability measures, defined by $\mu_{W_\epsilon} (\cdot) =\mu_{W} (\cdot) \overset{\textrm{def}}{=} \mathbb{P}[W_\epsilon \in \cdot]$, is constantly equal to single element and is tight on $\mathcal {C}([0, T] ; U)$. On the other hand, the Poisson random measure $\mu_{\pi_\epsilon} (\cdot)=\mu_{\pi} (\cdot) \overset{\textrm{def}}{=} \mathbb{P}[\pi_\epsilon \in \cdot]$ takes values in the space $\mathcal {P}_{\mathbb{N}}\left(Z\times[0,T]\right)$ of counting measures on $Z\times[0,T]$ that are finite on bounded sets. Since $Z$ is a Polish space,  $\mathcal {P}_{\mathbb{N}}\left(Z\times[0,T]\right)$ endowed with Prohorov's metric is a complete separable metric space, one can get that the family of the laws $\{\mu_{\pi_\epsilon}\}_{\epsilon>0}$ is tight on $ \mathcal {P}_{\mathbb{N}}\left(Z\times[0,T]\right)$.

\textsc{Step 2.} We first verify the tightness of $\mu_{n_\epsilon}$ on $\mathcal {C}([0,T];(W^{\frac{13}{11},11}(\mathcal {O}))^*)$. Indeed, it follows from the uniform bounds \eqref{5.1b} and \eqref{5.2} that
\begin{equation}\label{ntight}
\begin{split}
\mathbb{E}\left(\|n_\epsilon\|_{W^{1,\frac{11}{10}}(0,T;(W^{1,11}(\mathcal {O}))^*)}^2\right)\leq C.
\end{split}
\end{equation}
Since the embedding from $(W^{1,11}(\mathcal {O}))^*$ into $ (W^{\frac{13}{11},11}(\mathcal {O}))^*$ is compact, it follows from the Theorem 2.2 in \cite{flandoli1995martingale} that the space $W^{1,\frac{11}{10}}(0,T;(W^{1,11}(\mathcal {O}))^*)$ is compactly embedded into $\mathcal {C}([0,T];(W^{\frac{13}{11},11}(\mathcal {O}))^*)$. Let us choose $B_0(r)\subseteq W^{1,\frac{11}{10}}(0,T;(W^{1,11}(\mathcal {O}))^*)$ be a closed ball centered at $0$ with radius $r$, then $B_0(r)$ is a compact set in $\mathcal {C}([0,T];(W^{\frac{13}{11},11}(\mathcal {O}))^*)$. By \eqref{ntight}, we deduce from the Chebyshev inequality that
$$
\mathbb{P}[n_\epsilon \notin B_0(r)]\leq \mathbb{P}\left[\|n_\epsilon\|_{W^{1,\frac{11}{10}}(0,T;(W^{1,11}(\mathcal {O}))^*)}> r\right]\leq \frac{C}{r^2}.
$$
For any $\eta>0$, by choosing $r>0$ such that $\eta r^2=C$, we get $\inf_{\epsilon> 0}\mathbb{P}\{w;~n_\epsilon(\omega,\cdot) \in  B_0(r)\} > 1-\eta$, which implies the desired result.

To prove that the sequence $\mu_{n_\epsilon} (\cdot) \overset{\textrm{def}}{=} \mathbb{P}[n_\epsilon \in \cdot]$ is tight on $L^\frac{5}{4}(0,T;L^\frac{5}{4}(\mathcal {O}))$, we introduce the Banach space $Y_1$ with the norm
$$
\|y\|_{Y_1} \overset {\textrm{def}} {=} \|y\|_{L^{\frac{5}{4}}(0,T;W^{1,\frac{5}{4}})}+\left\| \frac{ \mathrm{d}y}{\mathrm{d} t}\right\|_{L^1(0,T;(W^{2,q})^*)},\quad q>3.
$$
We choose $B_1(r)$ as a closed ball with radius $r$ in $L^{\frac{5}{4}}(0,T;W^{1,\frac{5}{4}}(\mathcal {O}))\cap W^{1,1}(0,T;(W^{2,q}(\mathcal {O}))^*)$. Then it follows from a generalized Aubin-Lions Lemma (cf. \cite[Lemma 4.3]{nguyen2021nonlinear}) that the ball $B_1(r)$ is compact in $L^{\frac{5}{4}}(0,T;L^\frac{5}{4}(\mathcal {O}))$. In view of the uniform bounds \eqref{5.1b} and \eqref{5.2}, we infer that
\begin{equation*}
\begin{split}
\mathbb{P}[n_\epsilon \not\in  B_1(r) ] &\leq \mathbb{P}\left[ \|n_\epsilon\|_{L^{\frac{5}{4}}(0,T;W^{1,\frac{5}{4}})} > \frac{r}{2}\right]+\mathbb{P}\left[ \|n_\epsilon\|_{W^{1,1}(0,T;(W^{2,q})^*)} > \frac{r}{2}\right]\\
&\leq   \frac{4}{r^2} \mathbb{E} \left(\|n_\epsilon\|_{L^{\frac{5}{4}}(0,T;W^{1,\frac{5}{4}})}^2+\|n_\epsilon\|_{W^{1,1}(0,T;(W^{2,q})^*)}^2\right) \leq \frac{  C }{r^2}.
\end{split}
\end{equation*}
By choosing $\eta r^2=C$, it leads to $\inf_{\epsilon> 0}\mathbb{P}\{w;~n_\epsilon(\omega,\cdot) \in B_1(r) \} > 1-\eta$, which implies the tightness of $\mu_{n_\epsilon} $ on $L^{\frac{5}{4}}(0,T;L^\frac{5}{4}(\mathcal {O}))$. The tightness of $\mu_{n_\epsilon} $ on $(L_{loc}^{\frac{5}{4}}(0,\infty;W^{1,\frac{5}{4}}(\mathcal {O})),\textrm{weak})$ is a direct consequence of the uniform bound \eqref{5.1b}.

To show the tightness of $\mu_{n_\epsilon}$ on $(L^{\frac{5}{3}}(0,T;L^{ \frac{5}{3}}(\mathcal {O}) ),\textrm{weak})$, for any $r>0$, we set
$$
\widetilde{B}_1(r)=\left\{n_\epsilon \in L^{\frac{5}{3}}(0,T;L^{ \frac{5}{3}}(\mathcal {O})) ;~ \|n_\epsilon\|_{L^{\frac{5}{3}}(0,T;L^{ \frac{5}{3}})}\leq r\right\}.
$$
By invoking the GN inequality, we gain from \eqref{5.1a} that, for any $p\geq1$,
\begin{equation}\label{5.11}
\begin{split}
\mathbb{E}\left(\|n_\epsilon\|_{L^{\frac{5}{3}}(0,T;L^{ \frac{5}{3}} )}^p\right) \leq C\mathbb{E}\left[ \int_0^T \left(\|\nabla \sqrt{n_\epsilon}\|_{ L^{2}}^2  +1\right) \mathrm{d}t\right]^p\leq C.
\end{split}
\end{equation}
Thereby, we get from \eqref{5.11} with $p=2$ that
\begin{equation*}
\begin{split}
\mathbb{P}\left[n_\epsilon \not\in  \widetilde{B}_1 (r) \right] =\mathbb{P}\left[\|n_\epsilon\|_{L^{\frac{5}{3}}(0,T;L^{ \frac{5}{3}})}> r \right] \leq \frac{\mathbb{E}  \left(\|n_\epsilon\|_{L^{\frac{5}{3}}(0,T;L^{ \frac{5}{3}})}^2 \right)}{r^2} \leq \frac{C}{r^2},
\end{split}
\end{equation*}
which implies the desired result.

\textsc{Step 3.} Let us first show the tightness of $(\mu_{c_\epsilon})_{\epsilon>0}$ on $\mathcal {C}([0,T];(W^{2,\frac{5}{2} }(\mathcal {O}))^*)$. Note that
$$
\left|\int_\mathcal {O} \partial_t c_\epsilon \cdot\varphi \textrm{d}x\right|\leq 2\left|\int_\mathcal {O} \sqrt{c_\epsilon} \partial_t \sqrt{c_\epsilon}  \cdot\varphi \textrm{d}x\right|\leq 2\sqrt{\|c_0\|_{L^\infty}}\left|\int_\mathcal {O} \partial_t \sqrt{c_\epsilon}  \cdot\varphi \textrm{d}x\right|,
$$
for any $\varphi \in \mathcal {C} ^\infty(\overline{\mathcal {O}})$, which by \eqref{5.4} implies that $
\|\partial_t c_\epsilon  \|_{(W^{1,5/2})^*}\leq C\|\partial_t \sqrt{c_\epsilon}  \|_{(W^{1,5/2})^*}$. Therefore, $\partial_t c_\epsilon$ is uniformly bounded in $L^2(\Omega;L^{\frac{5}{3}}(0,T;(W^{1,\frac{5}{2}}(\mathcal {O}))^* ))$. This fact and the uniform bound in \eqref{5.1e} imply that $c_\epsilon$ is uniformly bounded in $L^2(\Omega;W^{1,
\frac{5}{3}}(0,T;(W^{1,\frac{5}{2}}(\mathcal {O}))^*))$. Since $(W^{1,\frac{5}{2}}(\mathcal {O}))^* $ is compactly embedded into $(W^{2,\frac{5}{2}}(\mathcal {O}))^* $, it follows from \cite[Theorem 2.2]{flandoli1995martingale} that the embedding $ W^{1,
\frac{5}{3}}(0,T;(W^{1,\frac{5}{2}}(\mathcal {O}))^*) \subset \mathcal {C}([0,T];(W^{2,\frac{5}{2}}(\mathcal {O}))^*) $ is compact. Following the same argument as that at the beginning of Step 2,  we obtain the desired result.

To find a compact subset $B_2(r)\in L^2(0,T;W^{1,2}(\mathcal {O}))$ such that $\mu_{c_\epsilon} (B_2(r)^c) \overset{\textrm{def}}{=} \mathbb{P}[c_\epsilon \not\in  B_2(r) ] <  \eta$. Let us consider the Banach space $Y_2$ endowed with the norm
$$
\|y\|_{Y_2} \overset {\textrm{def}} {=} \|y\|_{L^{2}(0,T;W^{2,2})}+\left\| \frac{ \mathrm{d}y}{\mathrm{d} t}\right\|_{L   ^{ \frac{5}{3}}(0,T;(W^{1,\frac{5}{2}})^*)}.
$$
Choosing a closed ball $B_2(r)$ in $L^{2}(0,T;W^{2,2}(\mathcal {O}))\cap W^{1,\frac{5}{3}}(0,T;(W^{1,\frac{5}{2}}(\mathcal {O}))^*)$ with radius $r>0$ centered at $0$. The Aubin-Lions Lemma infers that the set $B_2(r) \subset L^2(0,T;W^{1,2}(\mathcal {O}))$ is compact. Moreover, we get from \eqref{5.1d}, \eqref{5.3} and the Chebyshev inequality that
\begin{equation*}
\begin{split}
\mathbb{P}[c_\epsilon \not\in  B_2(r) ] \leq  \frac{4}{r^2}\mathbb{E} \left( \|c_\epsilon\|_{L^{2}(0,T;W^{2,2})\cap W^{1,\frac{5}{3}}(0,T;(W^{1,\frac{5}{2}})^*)}^2\right) \leq \frac{  C }{r^2}.
\end{split}
\end{equation*}
Choosing $r$ large enough leads to $\inf_{\epsilon> 0}\mathbb{P}\{w;~c_\epsilon(\omega,\cdot) \in  B_2(r) \}  > 1-\epsilon$, and hence $(\mu_{c_\epsilon})_{\epsilon>0} $ is tight on $L^2(0,T;W^{1,2}(\mathcal {O}))$. Moreover, by \eqref{(4.6)}, any   ball with finite radius in $L^\infty(\mathcal {O}\times [0,T])$ is relatively compact in  the weak-star topology. This proves the tightness of $(\mu_{c_\epsilon})_{\epsilon>0}$ on $\mathcal {X}_{c_\epsilon}$.

\textsc{Step 4.} To prove that the sequence $\mu_{u_\epsilon} (\cdot) \overset{\textrm{def}}{=} \mathbb{P}[u_\epsilon \in \cdot]$ is tight on $\mathcal {D} ([0,T];(W^{1, 5}_{0,\sigma}(\mathcal {O}))^*)$, according to the Lemma \ref{aldous}, it suffices to prove that for any stopping times $\left(\tau_{k}\right)_{k \geq 1}$ satisfying $0 \leq \tau_{k} \leq T$, there exists a constant $C>0$ independent of $\epsilon$ such that
\begin{align}\label{5.12}
 \mathbb{E} \left(\|u_\epsilon(\tau_{k}+\theta)-u_\epsilon(\tau_{k})\|_{(W^{1, 5}_{0,\sigma} )^*} ^2\right) \leq C \theta^{\frac{1}{2}}.
\end{align}
Indeed, by integrating the $u_\epsilon$-equation from $0$ to $t$ for any $t \in (0, T]$, we obtain
\begin{equation}\label{5.13}
\begin{split}
&u_\epsilon(\tau_{k}+\theta)-u_\epsilon(\tau_{k})= - \int_{\tau_{k}}^{\tau_{k}+\theta} \mathcal {P} (\textbf{L}_\epsilon u_\epsilon\cdot \nabla) u_\epsilon  \mathrm{d}t - \int_{\tau_{k}}^{\tau_{k}+\theta}  \textrm{A} u_\epsilon\mathrm{d}t \\
 & + \int_{\tau_{k}}^{\tau_{k}+\theta}  \mathcal {P}(n_\epsilon\nabla \Phi +h(t,u_\epsilon))\mathrm{d}t +  \int_{\tau_{k}}^{\tau_{k}+\theta}  \mathcal {P}g(t,u_\epsilon) \mathrm{d}W_t + \int_{\tau_{k}}^{\tau_{k}+\theta} \int_{Z_0} \mathcal {P}K(u_\epsilon(t-),z)\widetilde{\pi}(\mathrm{d}t,\mathrm{d}z) \\
& + \int_{\tau_{k}}^{\tau_{k}+\theta} \int_{Z\backslash Z_0} \mathcal {P}G(u_\epsilon(t-),z)\pi(\mathrm{d}t,\mathrm{d}z)  \overset{\textrm{def}}{=}  D_1+\cdot\cdot\cdot+D_6 .
\end{split}
\end{equation}
For $D_1$, by using the divergence-free condition and integrating by parts, we get form the H\"{o}lder inequality $\|ab\|_{L^{\frac{5}{4}}}\leq\|a\|_{L^2}\|b\|_{L^{\frac{10}{3}}}$ that
\begin{equation*}
\begin{split}
  \mathbb{E}\left(\| D_1\|_{(W^{1, 5}_{0,\sigma})^* } \right)&=  \mathbb{E} \sup_{ \|\psi\|_{W^{1, 5}_{0,\sigma} }=1}  \int_{\tau_{k}}^{\tau_{k}+\theta}  \int_\mathcal {O} \mathcal {P} (\textbf{L}_\epsilon u_\epsilon \otimes u_\epsilon)\cdot \nabla\psi(x) \mathrm{d} t\mathrm{d} x \\
  & \leq C \theta^{\frac{1}{2}} \left( \mathbb{E} \sup_{t\in [0,T]}\| u_\epsilon (t) \|_{L^{2}}^2+ \mathbb{E}  \int_0^{T }  \|u_\epsilon(t)\|_{L^{\frac{10}{3}}}^{\frac{10}{3}}  \mathrm{d} t \right)\leq C \theta^{\frac{1}{2}},
 \end{split}
\end{equation*}
where we used the uniform bounds \eqref{5.1f}, \eqref{5.1g} and the fact of $\|\textbf{L}_\epsilon u_\epsilon  \|_{L^{2}}\leq C \| u_\epsilon  \|_{L^{2}}$. Similarly, by using the embedding $L^\infty(\mathcal {O})\subset L^2(\mathcal {O})\subset(W^{1, 5}_{0,\sigma}(\mathcal {O}))^* $, we get
$$
\mathbb{E}\left(\| D_2\|_{(W^{1, 5}_{0,\sigma})^*}\right)
\leq C \theta^{\frac{1}{2}}  \mathbb{E} \left( \int_0^{T } \|\nabla u_\epsilon  \|_{L^{2}}^2 \mathrm{d}t\right)^{\frac{1}{2}}\leq C \theta^{\frac{1}{2}},
$$
and
$$\mathbb{E}\left(\| D_3\|_{(W^{1, 5}_{0,\sigma})^* }\right) \leq C \theta  \mathbb{E} \int_{\tau_{k}}^{\tau_{k}+\theta}   \left(\|n_\epsilon \nabla\Phi  \|_{L^{1}}  + \|u_\epsilon\|_{L^2}  +1  \right)\mathrm{d}t  \leq C \theta.
$$
For $D_4$, by using the BDG inequality, the H\"{o}lder inequality and the assumption on $g$, we obtain
\begin{equation*}
\begin{split}
\mathbb{E}\left(\| D_4\|_{(W^{1, 5}_{0,\sigma})^* } \right) &\leq C \mathbb{E} \left[ \int_{\tau_{k}}^{\tau_{k}+\theta} \left(\sup_{ \|\psi\|_{W^{1, 5}_{0,\sigma} }=1} \int_\mathcal {O}\mathcal {P}g(t,u_\epsilon) \psi  \mathrm{d} x \right) ^2 \mathrm{d}t\right]^{\frac{1}{2}} \\
   &\leq C \mathbb{E} \left( \int_{\tau_{k}}^{\tau_{k}+\theta} (\| u_\epsilon (t)\|_{L^{2}}^2 +1) \mathrm{d}t\right)^{\frac{1}{2}} \leq C \theta^{\frac{1}{2}}\left(\mathbb{E}\sup_{t\in [0,T]}\|u_\epsilon\|_{L^{2}}^2 +1\right)\leq C \theta^{\frac{1}{2}}.
 \end{split}
\end{equation*}
For $D_5$, in virtue of the BDG inequality, the Stochastic Fubini Theorem and the assumption on $K$, we obtain
\begin{equation*}
\begin{split}
 \mathbb{E}\left(\| D_5\|_{(W^{1, 5}_{0,\sigma})^* }\right)
 &\leq C\mathbb{E}\left( \int_{\tau_{k}}^{\tau_{k}+\theta} \int_{Z_0}\|\mathcal {P}K(u_\epsilon(t-),z) \|_{L^2}^2 \mu(\mathrm{d}z)\mathrm{d}t\right)^{\frac{1}{2}}\\
 &\leq C\theta^{\frac{1}{2}}\mathbb{E}\left(\sup_{t \in [0,T]}\|u_\epsilon(t)\|_{L^{2}}^2 +1\right)^{\frac{1}{2}} \leq C\theta^{\frac{1}{2}}.
 \end{split}
\end{equation*}
For $D_6$, recalling that $\pi(\mathrm{d}t,\mathrm{d}z)=\widetilde{\pi}(\mathrm{d}t,\mathrm{d}z)+\mu(\mathrm{d}z)\mathrm{d}t$ is a martingale Poisson random measure, it follows that
\begin{equation*}
\begin{split}
 \mathbb{E}\left(\| D_6\|_{(W^{1, 5}_{0,\sigma})^* }\right)
 &\leq\mathbb{E} \int_{\tau_{k}}^{\tau_{k}+\theta} \int_{Z\backslash Z_0}  \sup_{ \|\psi\|_{W^{1, 5}_{0,\sigma} }=1} \int_\mathcal {O}\mathcal {P}G(u_\epsilon(t-),z)\psi\mathrm{d} x \mu(\mathrm{d}z)\mathrm{d}t \\
 &\leq \mathbb{E} \int_{\tau_{k}}^{\tau_{k}+\theta}(\| u_\epsilon (t)\|_{L^{2}}  +1) \mathrm{d}t \leq C\theta.
 \end{split}
\end{equation*}
Substituting the estimates of $D_1\sim D_6$ into \eqref{5.13} gives \eqref{5.12}. Hence $(u_\epsilon)_{\epsilon> 0}$ satisfy the Aldous condition in $(W^{1, 5}_{0,\sigma}(\mathcal {O}))^*$ (cf. \cite{aldous1978stopping,aldous1989stopping}), and Lemma \ref{aldous} implies the tightness of $(u_\epsilon)_{\epsilon> 0}$ on $\mathcal {D} \left([0,T];(W^{1, 5}_{0,\sigma}(\mathcal {O}))^*\right)$.

\textsc{Step 5.} We show that the sequence $(\mu_{u_\epsilon})_{\epsilon> 0} $ is also tight on $L^2([0,T];L^2_\sigma(\mathcal {O}) )$. This can be done by introducing the Banach space $Y_3$ with the norm
$$
\|y\|_{Y_3} \overset {\textrm{def}} {=} \|y\|_{L^2(0,T;W^{1,2}_{0,\sigma})}+\left\|y\right\|_{W^{\alpha,2}(0,T; (\mathscr{D}(\textrm{A}))^*)}.
$$
In terms of the Aubin-Lions Lemma (cf. \cite{nguyen2021nonlinear}), we infer that any bounded closed ball of $L^2(0,T;W^{1,2}_{0,\sigma}(\mathcal {O}))\cap W^{\alpha,2}(0,T; (\mathscr{D}(\textrm{A}))^*) $ is compact in $L^2(0,T;L^{2}_{\sigma}(\mathcal {O}))$. Therefore, by choosing a closed ball $B_3(r)$ with radius $r>0$ centered at 0 in $Y_3$, we find
\begin{equation*}
\begin{split}
\mathbb{P}[u_\epsilon \not\in  B_3(r) ]\leq  \frac{4}{r^2} \mathbb{E} \left( \|u_\epsilon\|_{L^2(0,T;W^{1,2}_{0,\sigma})\cap W^{\alpha,2}(0,T; (\mathscr{D}(\textrm{A}))^*)}^2\right) \leq \frac{  C }{r^2}.
\end{split}
\end{equation*}
Then one can finish the proof by choosing $r>0$ large enough such that $\mathbb{P}[u_\epsilon \not\in  B_3(r)] <\eta$. The tightness of $(\mu_{u_\epsilon})_{\epsilon >0} $ on $\left(L^2_{loc}(0,\infty;W^{1,2}(\mathcal {O})),\textrm{weak}\right)$ is obvious due to the bound \eqref{5.1f}.

In conclusion, the joint distribution of the processes $\left(W_\epsilon,\pi_\epsilon,n_\epsilon,c_\epsilon,u_\epsilon\right)_{\epsilon>0}$ is tight on $\mathcal {X}$. The proof of Lemma \ref{lem5.4} is completed.
\end{proof}

From Lemma \ref{5.4},  the Prohorov Theorem (cf. Theorem 5.1 in \cite{billingsley2013convergence}) tells us that $(\Pi_\epsilon)_{\epsilon > 0}$ is weakly compact in probability measure space, which implies that there exists  a measure $\Pi$ on $\mathcal {X}$ such that, for a subsequence $(\epsilon_j)_{j\geq 1}$,
\begin{equation}\label{5.14}
\begin{split}
 \Pi_{\epsilon_j} \overset {\textrm{def}} {=} \mathbb{P}\circ \left(W_{\epsilon_j} ,\pi_{\epsilon_j} ,n_\epsilon,c_{\epsilon_j} ,u_{\epsilon_j} \right)^{-1}~~\rightarrow ~~\Pi \quad\textrm{weaky as} \quad j\rightarrow\infty.
\end{split}
\end{equation}
However, the weak convergence of $u_\epsilon$ to $u$ in distribution is too weak to construct solutions due to the missing of topology structure in random variable $\omega$. Instead, we shall apply the Skorokhod Representation Theorem to upgrade the weak convergence \eqref{5.14} to almost surely  convergence in $\mathcal {X}$ for newly-found random variables but defined on another probability space as a price. To cater to the Banach spaces with weak topology, the generalization of the classical Skorokhod Representation Theorem, called the Jakubowski-Skorokhod Theorem (cf. \cite[Theorem 2]{jakubowski1998almost}),  will be applied in current problem.

\begin{lemma}\label{lem5.5}
There exists a subsequence $(\epsilon_j)_{j\geq 1}$ whose limit is zero  as $j\rightarrow\infty$, a new probability space $(\widehat{\Omega}, \widehat{\mathcal {F}}, \widehat{\mathbb{P}})$, with the associated expectation denoted by $\widehat{\mathbb{E}}$, and $\mathcal {X}$-valued random variables $(\widehat{W}_{\epsilon_j},\widehat{\pi}_{\epsilon_j},\widehat{n}_{\epsilon_j},
\widehat{c}_{\epsilon_j},\widehat{u}_{\epsilon_j})_{j\geq 1}$, $(\widehat{W} ,\widehat{\pi} ,\widehat{n} ,\widehat{c },\widehat{u} )$ such that
\begin{itemize}[leftmargin=0.9cm]
\item [$a)$]  The laws of $ (\widehat{W}_{\epsilon_j},\widehat{\pi}_{\epsilon_j},\widehat{n}_{\epsilon_j},
    \widehat{c}_{\epsilon_j},\widehat{u}_{\epsilon_j})$ and $(\widehat{W} ,\widehat{\pi} ,\widehat{n} ,\widehat{c },\widehat{u} )$ are $\Pi_{\epsilon_j}$ and $\Pi$, respectively;

\item [$b)$] $(\widehat{W}_{\epsilon_j},\widehat{\pi}_{\epsilon_j})=(\widehat{W} ,\widehat{\pi})$ everywhere on $\widehat{\Omega}$;

\item [$c)$]   $(\widehat{W}_{\epsilon_j},\widehat{\pi}_{\epsilon_j},\widehat{n}_{\epsilon_j},
    \widehat{c}_{\epsilon_j},\widehat{u}_{\epsilon_j})$ converges $\widehat{\mathbb{P}}$-a.s. to $(\widehat{W} ,\widehat{\pi} ,\widehat{n} ,\widehat{c },\widehat{u} )$ in the topology of $\mathcal {X}$, that is, $\widehat{W}_{\epsilon_j}=\widehat{W}$, $ \widehat{\pi}_{\epsilon_j}= \widehat{\pi}$, and
\begin{align*}
\widehat{n}_{\epsilon_j}\rightarrow \widehat{n}~~ \textrm{strongly in}~~\mathcal {X}_{n }, \ \ \widehat{c}_{\epsilon_j}\rightarrow \widehat{c}~~\textrm{strongly in}~~\mathcal {X}_{c }, \quad \widehat{u}_{\epsilon_j}\rightarrow \widehat{u} ~~\textrm{strongly in}~ \mathcal {X}_{u },  \ \ \widehat{\mathbb{P}}\textrm{-a.s.};
\end{align*}
\item [$d)$]   $\widehat{W}$ is a cylindrical Wiener process;
$\widehat{\pi}$ is a  time homogeneous Poisson random measure on $\mathscr{B}(Z)\times \mathscr{B}([0,\infty))$ over $(\widehat{\Omega},  \widehat{\mathcal{F}} , \widehat{\mathbb{P}})$ with intensity measure $\textrm{d}\mu \otimes \textrm{d}t$;

\item [$e)$] the quantity $((\widehat{\Omega}, \widehat{\mathcal {F}}, \widehat{\mathfrak{F}}, \widehat{\mathbb{P}}),\widehat{W}_{\epsilon_j},\widehat{\pi}_{\epsilon_j},
    \widehat{n}_{\epsilon_j},\widehat{c}_{\epsilon_j},\widehat{u}_{\epsilon_j})$ with $\widehat{\mathfrak{F}}=(\widehat{\mathcal {F}}_t)_{t\geq 0}$ is a martingale weak solution to the regularized system \eqref{SCNS-1}, which satisfies $\widehat{\mathbb{P}}$-a.s.
{\wuhao
\begin{subequations}
\begin{align}
&\langle \widehat{n}_{\epsilon_j} (t),\varphi\rangle =\langle \widehat{n}_{\epsilon_j} (0),\varphi|_{t=0}\rangle+ \int_0^t \langle \widehat{u}_{\epsilon_j} \widehat{n}_{\epsilon_j}-\nabla \widehat{n}_{\epsilon_j}+ \widehat{n}_{\epsilon_j}\textbf{h}_{\epsilon_j}' (\widehat{n}_{\epsilon_j})\chi(\widehat{c}_{\epsilon_j})\nabla \widehat{c}_{\epsilon_j} ,\nabla\varphi\rangle\mathrm{d}s,\label{5.15a}\\
&\langle \widehat{c}_{\epsilon_j}(t),\varphi\rangle =\langle \widehat{c}_{\epsilon_j}(0),\varphi|_{t=0}\rangle+  \int_0^t \langle \widehat{u}_{\epsilon_j} \widehat{c}_{\epsilon_j}-\nabla \widehat{c}_{\epsilon_j},\nabla\varphi\rangle\mathrm{d}s- \int_0^t\langle\textbf{h}_{\epsilon_j} (\widehat{n}_{\epsilon_j}) f(\widehat{c}_{\epsilon_j}),\varphi\rangle\mathrm{d}s,\label{5.15b}\\
&\langle \widehat{u}_{\epsilon_j}(t),\psi\rangle =\langle \widehat{u}_{\epsilon_j}(0),\psi|_{t=0}\rangle+ \int_0^t\langle  \mathcal {P} (\textbf{Y}_{\epsilon_j } \widehat{u}_{\epsilon_j}\otimes \widehat{u}_{\epsilon_j})-\nabla \widehat{u}_{\epsilon_j} ,\nabla\psi\rangle\mathrm{d}s\label{5.15c}\\
& \quad  + \int_0^t\langle \mathcal {P}(\widehat{n}_{\epsilon_j}\nabla \Phi )+\mathcal {P}h(s,\widehat{u}_{\epsilon_j}),\psi\rangle\mathrm{d}s+ \int_0^t \langle\mathcal {P}g(s,\widehat{u}_{\epsilon_j}) \mathrm{d}\widehat{W}_{\epsilon_j}(s),\psi\rangle \\
& \quad + \int_0^t\int_{Z_0} \langle\mathcal {P}K(\widehat{u}_{\epsilon_j}(s-),z)\widetilde{\widehat{\pi}}_{\epsilon_j}(\mathrm{d}s,\mathrm{d}z),\psi\rangle \notag +  \int_0^t\int_{Z\backslash Z_0} \langle\mathcal {P}G(\widehat{u}_{\epsilon_j}(s-),z) \widehat{\pi}_{\epsilon_j} (\mathrm{d}s,\mathrm{d}z),\psi\rangle.\nonumber
\end{align}
\end{subequations}}
for any $\varphi\in  \mathcal {C}^\infty_{0}(\mathcal {O}\times [0,\infty);\mathbb{R})$, and $\psi\in  \mathcal {C}^\infty_{0}(\mathcal {O}\times [0,\infty);\mathbb{R}^3)$ with $\div \psi=0$.
\end{itemize}
\end{lemma}

\begin{remark}\label{remark5.6}
By Lemma \ref{lem5.5}, since $ (\widehat{n}_{\epsilon_j}, \widehat{c}_{\epsilon_j},\widehat{u}_{\epsilon_j})$ in the new probability space has the same distributions as for $(n_{\epsilon_j}, c_{\epsilon_j},u_{\epsilon_j})$,  it is straightforward to show that  $ (\widehat{n}_{\epsilon_j}, \widehat{c}_{\epsilon_j},\widehat{u}_{\epsilon_j})$ shares  the same estimates in Lemmas \ref{lem5.1}-\ref{lem5.3}, under the expectation $\widehat{\mathbb{E}} $.
\end{remark}

\begin{proof} [\emph{\textbf{Proof of Lemma \ref{lem5.5}.}}]
 The proof of $a)-d)$ is standard (see e.g., \cite{chen2019martingale,nguyen2021nonlinear,brzezniak2019weak}). For $e)$, the verification of the first two equations in \eqref{SCNS-1} follows directly from the $c)$ of Lemma \ref{lem5.5}, it remains to prove the $u_\epsilon$-equation. For all $t\in[0,T]$ and $\psi\in  \mathcal {C}^\infty_{0}(\mathcal {O}\times [0,\infty);\mathbb{R}^3)$, we define
$$
\mathcal {L}_{3}(W,\pi,n,c,u)_t\overset {\textrm{def}} {=}\mathcal {L}_{31}(n,c,u)_t-\mathcal {L}_{32}(W ,u)_t-\mathcal {L}_{33}( \pi, u)_t,
$$
where
\begin{equation*}
\begin{split}
\mathcal {L}_{31}\left(n,c,u\right)_t&\overset {\textrm{def}} {=} \langle u(t),\psi\rangle-\langle u(0),\psi\rangle+ \int_0^t\langle  \mathcal {P} (\textbf{Y}_{\epsilon_j } u\otimes u) , \nabla \psi\rangle\mathrm{d}s- \int_0^t\langle  \nabla u ,\nabla\psi\rangle\mathrm{d}s\\
& + \int_0^t\langle \mathcal {P}(n\nabla \Phi ),\psi\rangle\mathrm{d}s- \int_0^t\langle \mathcal {P}h(t,u),\psi\rangle \mathrm{d}s, \\
\mathcal {L}_{32}\left(W,u\right)_t&\overset {\textrm{def}} {=} \int_0^t \langle\mathcal {P}g(s,u) \mathrm{d}W_s,\psi\rangle ,\\
\mathcal {L}_{33}\left(\pi,u\right)_t&\overset {\textrm{def}} {=}  \int_0^t  \int_{Z_0} \langle\mathcal {P}K(u(s-),z),\psi\rangle  \widetilde{\pi}(\mathrm{d}s,\mathrm{d}z)+ \int_0^t \int_{Z\backslash Z_0} \langle \mathcal {P}G(u(s-),z),\psi\rangle \pi (\mathrm{d}z)\mathrm{d}s.
\end{split}
\end{equation*}
Since the quantity $\left(W_{\epsilon_j},\pi_{\epsilon_j},n_{\epsilon_j},c_{\epsilon_j},u_{\epsilon_j}\right)$ satisfies \eqref{SCNS-1} in the sense of distribution, it is clear that $
\mathcal {L}_3\left(W_{\epsilon_j},\pi_{\epsilon_j},n_{\epsilon_j},c_{\epsilon_j},
u_{\epsilon_j}\right)_t=0,$ $\mathbb{P}$-a.s.  To complete the proof, it suffices to prove that
$$
\widehat{\mathbb{E}}\left|\mathcal {L}_3 (\widehat{W}_{\epsilon_j},\widehat{\pi}_{\epsilon_j},\widehat{n}_{\epsilon_j},\widehat{c}_{\epsilon_j},
\widehat{u}_{\epsilon_j} )_t\right|=\mathbb{E}\left|\mathcal {L}_3 (W_{\epsilon_j},\pi_{\epsilon_j},n_{\epsilon_j},c_{\epsilon_j},
u_{\epsilon_j} )_t\right|.
$$

The argument will be divided into three parts.

($\textbf{P}_1$)  The absolutely continuous part $\mathcal {L}_{31}$ can be treated as before, and we obtain
\begin{equation}\label{5.16}
\begin{split}
\widehat{\mathbb{E}}\left[\mathcal {L}_{31}(\widehat{n}_{\epsilon_j},\widehat{c}_{\epsilon_j},\widehat{u}_{\epsilon_j})_t\right] = \mathbb{E}\left[\mathcal {L}_{31}(n_{\epsilon_j},c_{\epsilon_j},u_{\epsilon_j})_t\right].
\end{split}
\end{equation}

($\textbf{P}_2$) Since $W_{\epsilon_j}$ is a $U$-valued $\mathfrak{F}$-cylindrical Wiener process, we have the expansion $W_{\epsilon_j}=\sum_{\ell\geq 1} W_{\epsilon_j}^\ell e_\ell$, where $(W_{\epsilon_j}^\ell)_{\ell\geq 1}$ is a family of mutually independent real valued Wiener processes and $(e_\ell)_{\ell\geq1}$ is a complete orthogonal sequence in $U$. For each $n\in\mathbb{N}$, let us denote by $\mathscr{P}_n$ the orthogonal projector from $U$ onto $U_n\overset{\textrm{def}}{=}\textrm{span}\{e_1,e_2,\cdots,e_n\}$. Then, one can define a Wiener process $W_{\epsilon_j}^n\overset{\textrm{def}}{=} \mathscr{P}_n W_{\epsilon_j}= \sum_{1\leq\ell\leq n } W_{\epsilon_j}^\ell e_\ell$ on the finite-dimensional Hilbert space $U_n$.

     Let $\rho_k(t)$  be a standard mollifier, then we define a time smoothing function
\begin{equation*}
\begin{split}
[\textbf{g}_k(u_{\epsilon_j})](t)\overset{\textrm{def}}{=}  \int_0^t \rho_k(t-s) \mathcal {P}g(s,u_{\epsilon_j}(s)) \mathscr{P}_k\mathrm{d}s,\quad \forall k\in \mathbb{N}.
\end{split}
\end{equation*}
Apparently, $\textbf{g}_k(u_{\epsilon_j})$ belongs to $L^2(\Omega; L^2(0,T;\mathcal {L}_2(U; L^2_\sigma(\mathcal {O}))))$, which is of course a bounded variation function for  $t\in [0,T]$. Moreover, the properties of mollifiers imply that
\begin{equation}\label{5.17}
\begin{split}
  \|\textbf{g}_k(u_{\epsilon_j})\|_{L^2(\Omega; L^2(0,T;\mathcal {L}_2(U; L^2_\sigma)))} \leq C  \|\mathcal {P}g(\cdot,u_{\epsilon_j}(\cdot))\|_{L^2(\Omega; L^2(0,T;\mathcal {L}_2(U; L^2_\sigma)))},
\end{split}
\end{equation}
and
\begin{equation}\label{5.18}
\begin{split}
 \textbf{g}_k(u_{\epsilon_j})\rightarrow  \mathcal {P}g(s,u_{\epsilon_j}(s))~~ \textrm{in} ~~ L^2\left(\Omega; L^2\left(0,T;\mathcal {L}_2(U; L^2_\sigma(\mathcal {O}))\right)\right),~~\textrm{as} ~k\rightarrow \infty.
\end{split}
\end{equation}
Note that the finite-dimensional Wiener process $W_{\epsilon_j}^k$ is continuous in time, by making use of the integrating by parts for the Riemann-Stieltjes integral, we gain
\begin{equation}\label{5.19}
\begin{split}
\mathcal {L}_{32}^k\left(W_{\epsilon_j}^k,u_{\epsilon_j}\right)_t& \overset{\textrm{def}}{=} \int_0^t \langle[ \textbf{g}_k(u_{\epsilon_j})](s) \mathrm{d}W_{\epsilon_j}^k(s),\psi\rangle\\
& =\langle[ \textbf{g}_k(u_{\epsilon_j})](t) W_{\epsilon_j}^k(t),\psi\rangle - \int_0^t \langle[ \textbf{g}_k(u_{\epsilon_j})]'(s) W_{\epsilon_j}^k(s)\mathrm{d} s,\psi\rangle,
\end{split}
\end{equation}
which can be viewed as a deterministic functional of $(W_{\epsilon_j}^k,u_{\epsilon_j})$. On the other hand, note that $\int_0^t \langle[ \textbf{g}_k(u_{\epsilon_j})](s) \mathrm{d}W_{\epsilon_j}^k(s),\psi\rangle=\int_0^t \langle[ \textbf{g}_k(u_{\epsilon_j})](s) \mathrm{d}W_{\epsilon_j}(s),\psi\rangle$, by virtue of the BDG inequality, the properties   \eqref{5.17}-\eqref{5.18} and the Dominated Convergence Theorem,  we have
\begin{equation}\label{5.20}
\begin{split}
&\mathbb{E}\left|\mathcal {L}_{32}\left(W_{\epsilon_j},u_{\epsilon_j}\right)_t-\mathcal {L}_{32}^k\left(W_{\epsilon_j}^k,u_{\epsilon_j}\right)_t\right|\\
&\quad \leq C \mathbb{E}\left( \int_0^t \|\mathcal {P}g(s,u_{\epsilon_j}(s))-[ \textbf{g}_k(u_{\epsilon_j})](s)\|_{\mathcal {L}_2(U; L^2_\sigma)} ^2 \mathrm{d} s\right)^{\frac{1}{2}}\rightarrow 0, \quad \textrm{as} ~k\rightarrow \infty.
\end{split}
\end{equation}

($\textbf{P}_3$) Now let us identify the stochastic integrals with jumps by borrowing some ideas from Cyr et al. \cite{cyr2020review,cyr2018euler}. Note that by employing the well-known piecing out argument (or interlacing procedure) (cf. \cite{ikea1966construction,applebaum2009levy,cyr2018euler}), it suffices to consider the L\'{e}vy process possesses with jumps of small size, that is,
$G\equiv 0$ in $\mathcal {L}_{33}\left(\pi_{\epsilon_j},u_{\epsilon_j}\right)_t$. To this end, we choose subsets $\left(Z_{k}\right)_{k\geq1} \subset Z_{0} \overset{\textrm{def}}{=} \{z \in Z;~Z_0\}$ such that $Z_{k} \uparrow Z_{0}$ and $\mu\left(Z_{k}\right)<\infty$, for all $ k\in \mathbb{N}$. For each $k \in \mathbb{N}$, $\pi_{k} \overset{\textrm{def}}{=} \left.\pi\right|_{[0, \infty) \times Z_{k}}$ is a Poisson random measure on $Z$ corresponding to a $\left(\mathcal{F}_{t}\right)$-Poisson point process, and the intensity measure of $\pi_{k}$ is given by $\textrm{d}\mu_{k} \otimes \textrm{d}t$ with $\mu_{k} \overset{\textrm{def}}{=} \left.\mu\right|_{Z_{k}}$ which is a finite measure. Define
\begin{equation*}
\begin{split}
&\mathcal {L}_{33}^k\left(\pi_{k,\epsilon_j},u_{\epsilon_j}\right)_t \overset{\textrm{def}}{=}   \int_0^t  \int_{Z_{0}}\langle\mathcal {P}K(u_{\epsilon_j}(s-),z),\psi\rangle \textbf{1}_{Z_k} (z)\widetilde{\pi}_{k,\epsilon_j} (\mathrm{d}s,\mathrm{d}z) \\
&\quad=  \int_0^t  \int_{Z_k} \langle\mathcal {P}K(u_{\epsilon_j}(s-),z),\psi\rangle \pi_{k,\epsilon_j} (\mathrm{d}s,\mathrm{d}z)- \int_0^t  \int_{Z_k} \langle\mathcal {P}K(u_{\epsilon_j}(s-),z),\psi\rangle \mu_k  (\mathrm{d}z)\mathrm{d}s,
\end{split}
\end{equation*}
which can be viewed as a deterministic functional of $(\pi_{k,\epsilon_j},u_{\epsilon_j})$. Moreover, in view of the facts of  $Z_{k} \uparrow Z_{0}$ and $\pi_{k,\epsilon_j}=\pi_{\epsilon_j}|_{Z_k}$, we deduce from the BDG inequality and the Dominated Convergence Theorem that
\begin{equation}\label{5.21}
\begin{split}
&\mathbb{E} \left|\mathcal {L}_{33}\left(\pi_{\epsilon_j},u_{\epsilon_j}\right)_t-\mathcal {L}_{33}^k\left(\pi_{k,\epsilon_j},u_{\epsilon_j}\right)_t\right|\\
&\quad=\mathbb{E} \left| \int_0^t  \int_{Z_{0}} \textbf{1}_{Z_0\backslash Z_k}(z)\langle\mathcal {P}K(u_{\epsilon_j}(s-),z),\psi\rangle  \widetilde{\pi}_{\epsilon_j}(\mathrm{d}s,\mathrm{d}z) \right|\\
&\quad\leq C \mathbb{E} \left( \int_0^t  \int_{Z_0\backslash Z_k}  \|\mathcal {P}K(u_{\epsilon_j}(s-),z) \|_{L^2}^2 \mu(\mathrm{d}z)\mathrm{d}s \right)^{\frac{1}{2}}\rightarrow 0, \quad \textrm{as} ~k\rightarrow 0.
\end{split}
\end{equation}
From $(\textbf{P}_1)$-$(\textbf{P}_3)$, for each $k\in \mathbb{N}$, one can introduce
\begin{equation}\label{sss}
\begin{split}
&\mathcal {L}_{3}^k\left(W_{\epsilon_j}^k,\pi_{k,\epsilon_j},n_{\epsilon_j},c_{\epsilon_j},u_{\epsilon_j}\right)_t\\
&\quad \overset{\textrm{def}}{=} \mathcal {L}_{31}\left( n_{\epsilon_j},c_{\epsilon_j},u_{\epsilon_j}\right)_t-\mathcal {L}_{32}^k\left(W _{\epsilon_j}^k,u_{\epsilon_j}\right)_t-\mathcal {L}_{33}^k\left( \pi_{k,\epsilon_j}, u_{\epsilon_j}\right)_t \\
&\quad= \mathcal {L}_{31}\left(n_{\epsilon_j},c_{\epsilon_j},u_{\epsilon_j}\right)_t-\langle[ \textbf{g}_k(u_{\epsilon_j})](t) W_{\epsilon_j}^k(t),\psi\rangle - \int_0^t\langle [ \textbf{g}_k(u_{\epsilon_j})]'(s) W_{\epsilon_j}^k(s),\psi\rangle\mathrm{d} s\\
&\quad+ \int_0^t  \int_{Z_k} \langle\mathcal {P}K(u_{\epsilon_j}(s-),z),\psi_{\epsilon_j}\rangle \pi_{k,\epsilon_j}  (\mathrm{d}s,\mathrm{d}z)- \int_0^t  \int_{Z_k} \langle\mathcal {P}K(u_{\epsilon_j}(s-),z),\psi\rangle \mu_k (\mathrm{d}z)\mathrm{d}s,
\end{split}
\end{equation}
which can be viewed as a deterministic functional defined on $\mathcal {X}$. Moreover, it follows from the estimates  \eqref{5.16}, \eqref{5.20} and \eqref{5.21} that
\begin{equation}\label{5.22}
\begin{split}
&\mathbb{E} \left|\mathcal {L}_{3}\left(W_{\epsilon_j},\pi_{\epsilon_j},n_{\epsilon_j},c_{\epsilon_j},u_{\epsilon_j}\right)_t-\mathcal {L}_{3}^k\left(W_{\epsilon_j}^k,\pi_{k,\epsilon_j},n_{\epsilon_j},c_{\epsilon_j},u_{\epsilon_j}\right)_t\right|\\
\quad&\leq \mathbb{E} \left|\mathcal {L}_{32}(W_{\epsilon_j}, u_{\epsilon_j})_t-\mathcal {L}_{32}^k(W^k_{\epsilon_j}, u_{\epsilon_j})_t\right|+\mathbb{E} \left|\mathcal {L}_{33}(\pi_{\epsilon_j},u_{\epsilon_j})_t-\mathcal {L}_{33}^k(\pi_{k,\epsilon_j},u_{\epsilon_j})_t\right|\\
 &\rightarrow 0, \quad \textrm{as} ~k\rightarrow 0.
\end{split}
\end{equation}
By the same reasoning, the random variable
$$
\mathcal {L}_{3}^k  \left(\widehat{W}_{\epsilon_j}^k,\widehat{\pi}_{k,\epsilon_j},\widehat{n}_{\epsilon_j},\widehat{c}_{\epsilon_j},
\widehat{u}_{\epsilon_j}\right)_t \overset{\textrm{def}}{=} \mathcal {L}_{31}\left( \widehat{n}_{\epsilon_j},\widehat{c}_{\epsilon_j},\widehat{u}_{\epsilon_j}\right)_t-\mathcal {L}_{32}^k\left(\widehat{W}_{\epsilon_j}^k ,\widehat{u}_{\epsilon_j}\right)_t-\mathcal {L}_{33}^k\left( \widehat{\pi}_{k,\epsilon_j}, \widehat{u}_{\epsilon_j}\right)_t,
$$
where
$\mathcal {L}_{3}^k  (\cdot)$ provided in \eqref{sss} is also a deterministic function on  $\mathcal {X}$ such that
\begin{equation}\label{5.23}
\begin{split}
\widehat{\mathbb{E}}\left|\mathcal {L}_{3}   \left(\widehat{W}_{\epsilon_j},\widehat{\pi}_{\epsilon_j},\widehat{n}_{\epsilon_j},\widehat{c}_{\epsilon_j},
\widehat{u}_{\epsilon_j}\right)_t-\mathcal {L}_{3}^k  \left(\widehat{W}_{\epsilon_j}^k,\widehat{\pi}_{k,\epsilon_j},\widehat{n}_{\epsilon_j},\widehat{c}_{\epsilon_j},
\widehat{u}_{\epsilon_j}\right)_t\right| \rightarrow 0, \quad \textrm{as} ~k\rightarrow 0.
\end{split}
\end{equation}
Since $(\widehat{W}_{\epsilon_j},\widehat{\pi}_{\epsilon_j},\widehat{n}_{\epsilon_j},\widehat{c}_{\epsilon_j},
\widehat{u}_{\epsilon_j})$ and $(W_{\epsilon_j},\pi_{\epsilon_j},n_{\epsilon_j},c_{\epsilon_j},u_{\epsilon_j})$ has the same distribution, it follows from the definition of $\mathcal {L}_{3} ^k(\cdot)$ that
\begin{equation*}
\begin{split}
\widehat{\mathbb{E}} \left[\mathcal {L}_{3} ^k  \left(\widehat{W}_{\epsilon_j}^k,\widehat{\pi}_{k,\epsilon_j},\widehat{n}_{\epsilon_j},\widehat{c}_{\epsilon_j},
\widehat{u}_{\epsilon_j}\right)_t\right]=  \mathbb{E}\left[ \mathcal {L}_{3} ^k \left(W_{\epsilon_j}^k,\pi_{k,\epsilon_j},n_{\epsilon_j},c_{\epsilon_j},u_{\epsilon_j}\right)_t\right],
\end{split}
\end{equation*}
which combined with the convergence \eqref{5.22} and \eqref{5.23} leads to the desired equality \eqref{5.15c}. Therefore, the quantity $(\widehat{W}_{\epsilon_j},\widehat{\pi}_{\epsilon_j},\widehat{n}_{\epsilon_j},\widehat{c}_{\epsilon_j},
\widehat{u}_{\epsilon_j})$ satisfies the $u_\epsilon$-equation in \eqref{SCNS-1}. The proof of Lemma \ref{lem5.5} is now completed.
\end{proof}

\subsection{Recovering original SPDEs}

Now we have all in hand to complete the proof of Theorem \ref{thm}. We begin by turning the probability space $(\widehat{\Omega},\widehat{\mathcal {F}} ,\widehat{\mathbb{P}})$ into a stochastic basis
$(\widehat{\Omega},\widehat{\mathcal {F}},\widehat{\mathfrak{F}},\widehat{\mathbb{P}},\widehat{W} ,\widehat{\pi})$ with natural filtration $\widehat{\mathfrak{F}}=\{\widehat{\mathcal {F}}_t\}_{t\geq0}$, that is, all the relevant processes with respect to the smallest filtration  are adapted,
$$
\widehat{\mathcal {F}}_t\overset{\textrm{def}}{=}\sigma \left ( \sigma \left (\widehat{n}|_{[0,t]},\widehat{c}|_{[0,t]},\widehat{u}|_{[0,t]},\widehat{W}|_{[0,t]},\widehat{\pi}|_{[0,t]} \right)\bigcup \left\{N\in \widehat{\mathcal {F}};~ \widehat{\mathbb{P}}(N)=0\right\}\right)   ,\quad t \in [0,T].
$$

\begin{proof} [\emph{\textbf{Proof of Theorem \ref{thm}.}}] The proof is long and will be divided into several steps.

\textsc{Step 1 (Identification of $n$-equation).} In the first step, let us identify the limit for $n$-equation. To  this end, we consider functionals
\begin{equation}\label{s11}
\begin{split}
\textbf{m}_{\epsilon_j}(\widehat{n}_{\epsilon_j},\widehat{c}_{\epsilon_j},\widehat{u}_{\epsilon_j})_t
&\overset {\textrm{def}} {=} \langle\widehat{n}_{\epsilon_j} (0),\psi\rangle- \int_0^t \langle\widehat{u}_{\epsilon_j}\widehat{n}_{\epsilon_j},\nabla\psi\rangle \mathrm{d}s\\
&  - \int_0^t \langle\nabla \widehat{n}_{\epsilon_j} ,\nabla\psi\rangle\mathrm{d}s-  \int_0^t\left\langle \widehat{n}_{\epsilon_j}\textbf{h}_{\epsilon_j}' (\widehat{n}_{\epsilon_j})\chi(\widehat{c}_{\epsilon_j})\nabla \widehat{c}_{\epsilon_j} ,\nabla \psi\right\rangle\mathrm{d}s ,
\end{split}
\end{equation}
and
\begin{equation}\label{s12}
\begin{split}
 &\textbf{m}(\widehat{n} ,\widehat{c} ,\widehat{u} )_t
\overset {\textrm{def}} {=}\langle\widehat{n}  (0),\psi\rangle- \int_0^t \langle\widehat{u} \widehat{n} ,\nabla\psi\rangle \mathrm{d}s  - \int_0^t \langle\nabla \widehat{n}  ,\nabla\psi\rangle\mathrm{d}s- \int_0^t\left\langle \widehat{n}  \chi(\widehat{c} )\nabla \widehat{c} ,\nabla \psi\right\rangle\mathrm{d}s,
\end{split}
\end{equation}
where $\left(\widehat{n}_{\epsilon_j},\widehat{c}_{\epsilon_j},\widehat{u}_{\epsilon_j}\right)$ and $\left(\widehat{n} ,\widehat{c} ,\widehat{u} \right)$ are the processes constructed in Lemma \ref{lem5.5}.

We \textbf{Claim} that, for all $\psi\in   \mathcal {C}^\infty_{0}(\mathcal {O}\times [0,\infty);\mathbb{R})$,
\begin{subequations}
\begin{align}
 &\lim\limits_{j\rightarrow\infty}\widehat{\mathbb{E}}\left[ \int_0^T\left|\langle\widehat{n}_{\epsilon_j}(t)-\widehat{n}  (t),\psi\rangle\right|     \mathrm{d}t\right]=0, \label{5.24a} \\
 & \lim\limits_{j\rightarrow\infty}\widehat{\mathbb{E}}\left[ \int_0^T
 \left|\textbf{m}_{\epsilon_j}(\widehat{n}_{\epsilon_j},\widehat{c}_{\epsilon_j},\widehat{u}_{\epsilon_j})_t
 -\textbf{m}(\widehat{n} ,\widehat{c} ,\widehat{u} )_t,\psi\rangle\right|    \mathrm{d}t\right]=0.\label{5.24b}
\end{align}
\end{subequations}
\textsc{Proof of \eqref{5.24a}.} Observing that, by  $c)$ of Lemma \ref{lem5.5}, $\widehat{n}_{\epsilon_j}\rightarrow \widehat{n}$ weakly in $L_{loc}^{\frac{5}{3}}(0,\infty;L^{ \frac{5}{3}}(\mathcal {O}))$ $\widehat{\mathbb{P}}$-a.s. For each $T>0$, we have
\begin{equation*}
\begin{split}
& \int_0^T \int_\mathcal {O} \widehat{n}_{\epsilon_j}^{\frac{5}{3}}(x,t) \mathrm{d} x \mathrm{d} t=  \int_0^\infty \int_\mathcal {O} 1_{\mathcal {O}\times (0,T)} \widehat{n}_{\epsilon_j}^{\frac{5}{3}}(x,t) \mathrm{d} x \mathrm{d} t\\
&\quad \rightarrow  \int_0^\infty \int_\mathcal {O} 1_{\mathcal {O}\times (0,T)} \widehat{n} ^{\frac{5}{3}}(x,t) \mathrm{d} x \mathrm{d} t= \int_0^T \int_\mathcal {O} \widehat{n} ^{\frac{5}{3}}(x,t) \mathrm{d} x \mathrm{d} t \quad \textrm{as}~ j\rightarrow\infty,~~\widehat{\mathbb{P}} \textrm{-a.s},
\end{split}
\end{equation*}
which implies that
\begin{equation}\label{5.25}
\begin{split}
\widehat{n}_{\epsilon_j}\rightarrow \widehat{n} \quad \textrm{strongly in} \quad L_{loc}^{\frac{5}{3}}(0,\infty;L^{ \frac{5}{3}}(\mathcal {O})) \quad \textrm{as} ~ j\rightarrow\infty,~~\widehat{\mathbb{P}} \textrm{-a.s}.
\end{split}
\end{equation}
By utilizing \eqref{5.25} and the Young inequality, we gain
\begin{equation*}
\begin{split}
 &  \int_0^T\left|\langle\widehat{n}_{\epsilon_j}(t)-\widehat{n}  (t),\psi\rangle\right|^{\frac{5}{3}}     \mathrm{d}t \leq \|\psi\|_{L^\infty(0,T;L^{\frac{5}{2}})}^{\frac{5}{3}} \int_0^T\|\widehat{n}_{\epsilon_j}(t)-\widehat{n}  (t)\|_{L^{\frac{5}{3}}}^{\frac{5}{3}}      \mathrm{d}t\rightarrow 0 \quad \textrm{as} ~ j\rightarrow\infty,~~\widehat{\mathbb{P}} \textrm{-a.s}.
\end{split}
\end{equation*}
Moreover, since $\widehat{n}_{\epsilon_j}$ satisfies the uniform bound \eqref{5.11}, we have for any $p\geq1$ that
\begin{equation*}
\begin{split}
 & \widehat{\mathbb{E}}\left( \int_0^T\left|\langle\widehat{n}_{\epsilon_j}(t)-\widehat{n}  (t),\psi\rangle\right|^{\frac{5}{3}}     \mathrm{d}t\right)^p\\
  & \quad \leq  \|\psi\|_{L^\infty(0,T;L^{\frac{5}{2}})}^{\frac{5p}{3}} \left[ \widehat{\mathbb{E}}\left( \int_0^T\| \widehat{n}_{\epsilon_j}(t)\|_{L^\frac{5}{3}}^{\frac{5}{3}}   \mathrm{d}t\right)^p +\widehat{\mathbb{E}}\left( \int_0^T\|\widehat{n}  (t)\|_{L^\frac{5}{3}}^{\frac{5}{3}} \mathrm{d}t\right)^p\right] \leq C,
\end{split}
\end{equation*}
which implies by \eqref{5.11} that the sequence $ \int_0^T|\langle\widehat{n}_{\epsilon_j}(t)-\widehat{n}  (t),\psi\rangle|^{\frac{5}{3}}     \mathrm{d}t$ is uniformly integrable. Hence by applying the following Vitali Convergence Theorem that will be frequently used in the sequel,  we infer that \eqref{5.24a} holds.
\begin{lemma}(\cite[Theorem 5.12]{kallenberg1997foundations})
Let $\left(f_n\right)_{n\geq 1}$ be a sequence of integrable functions defined on the probability space $(\Omega, \mathcal {F}, \mathbb{P})$, such that $f_n \rightarrow f$  a.e. as $n \rightarrow \infty$ (or $f_n \rightarrow f$ in probability) for some integrable function $f$. Assume that there exist a $r>1$ and a constant $C>0$ such that $\mathbb{E}\left|f_n\right|^{r} \leq C$ for all $n \in \mathbb{N}$. Then $\mathbb{E}\left|f_n\right| \rightarrow \mathbb{E} \left|f\right|$ as $n \rightarrow \infty$.
\end{lemma}

\noindent
\textsc{Proof of \eqref{5.24b}.} Note that, we get by using the Stochastic Fubini Theorem
\begin{equation*}
\begin{split}
 &\widehat{\mathbb{E}}  \int_0^T
 \left|\textbf{m}_{\epsilon_j}(\widehat{n}_{\epsilon_j},\widehat{c}_{\epsilon_j},\widehat{u}_{\epsilon_j})_t
 -\textbf{m}(\widehat{n} ,\widehat{c} ,\widehat{u} )_t,\psi\rangle\right|      \mathrm{d}t \\
 &\quad = \int_0^T\widehat{\mathbb{E}}
  \left|\textbf{m}_{\epsilon_j}(\widehat{n}_{\epsilon_j},\widehat{c}_{\epsilon_j},\widehat{u}_{\epsilon_j})_t
 -\textbf{m}(\widehat{n} ,\widehat{c} ,\widehat{u} )_t,\psi\rangle\right|    \mathrm{d}t.
\end{split}
\end{equation*}
It  suffices to prove that each term on the R.H.S. of \eqref{s11} converges to the corresponding term on the R.H.S. of \eqref{s12} in $L^{1}( \Omega \times (0,T))$.
Similar to \eqref{5.24a}, noting that $\widehat{n}$ is right-continuous at $t=0$, we first have
\begin{equation}\label{5.28}
\begin{split}
   \widehat{\mathbb{E}} \left[|\langle\widehat{n}_{\epsilon_j} (0)- \widehat{n} (0),\psi|_{t=0}\rangle| \right] =0.
\end{split}
\end{equation}
By $c)$ of Lemma \ref{lem5.5} we have $\nabla\widehat{ n}_{\epsilon_j}\rightarrow \nabla\widehat{ n}$ weakly  in $L_{loc}^{\frac{5}{4}}(0,\infty;L^{ \frac{5}{4}}(\mathcal {O}))$, $\mathbb{P}$-a.s., which implies
\begin{equation}\label{5.29}
\begin{split}
 \int_0^T \langle\nabla \widehat{n}_{\epsilon_j} ,\nabla\psi\rangle\mathrm{d}s\rightarrow \int_0^T \langle\nabla \widehat{n}  ,\nabla\psi\rangle\mathrm{d}s \quad \textrm{as} ~ j\rightarrow\infty,~~\widehat{\mathbb{P}} \textrm{-a.s}.
\end{split}
\end{equation}
By uniform bounds \eqref{5.1f}, \eqref{5.1g} and the GN inequality in Lemma \ref{nirenberg}, we have
\begin{equation*}
\begin{split}
\widehat{\mathbb{E}}\left( \int_0^T \|\widehat{u}_{\epsilon_j}(t) \|_{L^{\frac{10}{3}}}^{\frac{10}{3}}     \mathrm{d}t\right)^p &\leq C\widehat{\mathbb{E}} \left[ \int_0^T \left(\|\nabla \widehat{u}_{\epsilon_j}(t) \|_{L^2}^2 \| \widehat{u}_{\epsilon_j}(t) \|_{L^2}^{\frac{4}{3}}+ \| \widehat{u}_{\epsilon_j}(t) \|_{L^2}^{\frac{10}{3}}\right )    \mathrm{d}t\right]^p\\
 &\leq C\widehat{\mathbb{E}} \left( \sup_{t\in [0,T]}\| \widehat{u}_{\epsilon_j}(t) \|_{L^2}^{\frac{10p}{3}}+1\right) +C\widehat{\mathbb{E}} \left(  \int_0^T \|\nabla \widehat{u}_{\epsilon_j}(t) \|_{L^2}^2     \mathrm{d}t\right)^{2p} \leq C,
\end{split}
\end{equation*}
which implies that $\widehat{u}_{\epsilon_j}$ is uniformly bounded in $L^p( \Omega; L^{\frac{10}{3}}(0,T;L^{\frac{10}{3}}(\mathcal {O}) ))$, and so it follows from the Banach-Alaoglu Theorem that there exists a subsequence of $\{\widehat{u}_{\epsilon_j}\}_{j\geq1}$ (still denoted by itself) such that
\begin{equation}\label{5.30}
\begin{split}
 \widehat{u}_{\epsilon_j} \rightarrow \widehat{u} \quad \textrm{weakly in} \quad L^p\left( \Omega; L^{\frac{10}{3}}(0,T;L^{\frac{10}{3}}(\mathcal {O}) )\right),~~ j\rightarrow\infty.
\end{split}
\end{equation}
In virtue of \eqref{5.30}, \eqref{5.1g} and \eqref{5.11}, it follows from the H\"{o}lder inequality that
$$
\widehat{\mathbb{E}}\left(\|\widehat{u}_{\epsilon_j} \widehat{n }_{\epsilon_j}\|_{ L^{1}(0,T;L^{1}  )}^p\right) \leq C\widehat{\mathbb{E}}\left(\|\widehat{u}_{\epsilon_j} \|_{ L^{\frac{10}{3}}(0,T;L^{\frac{10}{3}}  )}^{2p}\right)+C\widehat{\mathbb{E}}\left(\| \widehat{n }_{\epsilon_j}\|_{ L^{\frac{5}{3}}(0,T;L^{\frac{5}{3}}  )}^{2p}\right)\leq C,
$$
which shows that $\widehat{u}_{\epsilon_j} \widehat{n }_{\epsilon_j}$ is uniformly bounded in $L^p\left( \Omega; L^{1}(0,T;L^{1}(\mathcal {O}) )\right)$. By applying the Banach-Alaoglu Theorem again, one can extract a
subsequence of $(\widehat{u}_{\epsilon_j} \widehat{n }_{\epsilon_j})_{j\geq1}$ (still denoted by itself) such that
\begin{equation}\label{5.31}
\begin{split}
 \widehat{u}_{\epsilon_j} \widehat{n }_{\epsilon_j}\rightarrow \widehat{u}\widehat{n}\quad \textrm{weakly in} \quad L^p\left( \Omega; L^{1}(0,T;L^{1}(\mathcal {O}) )\right) ~~ \textrm{as} ~~ j\rightarrow\infty,
\end{split}
\end{equation}
which infers that $\widehat{u}\widehat{n} \in L^{1}(0,T;L^{1}(\mathcal {O}))$ $\widehat{\mathbb{P}}$-a.s., and
$
\widehat{\mathbb{E}}| \int_0^t \langle\widehat{u}_{\epsilon_j}\widehat{n}_{\epsilon_j}-\widehat{u}\widehat{n},\nabla\psi\rangle \mathrm{d}s| \rightarrow 0$ as $j\rightarrow\infty$.
By \eqref{5.31} and the Dominated Convergence Theorem, we gain
\begin{equation}\label{5.32}
\begin{split}
\lim\limits_{j\rightarrow\infty} \int_0^T\widehat{\mathbb{E}}\left| \int_0^t \langle\widehat{u}_{\epsilon_j}\widehat{n}_{\epsilon_j}-\widehat{u}\widehat{n},\nabla\psi\rangle \mathrm{d}s\right|  \mathrm{d}t=0.
\end{split}
\end{equation}
Now we show that
\begin{equation}\label{5.33}
\begin{split}
 \widehat{\mathbb{E}}\left| \int_0^t \langle\widehat{n}_{\epsilon_j}\textbf{h}_{\epsilon_j}' (\widehat{n}_{\epsilon_j})\chi(\widehat{c}_{\epsilon_j}) \nabla\widehat{c}_{\epsilon_j} -  \widehat{n} \chi(\widehat{c} ) \nabla\widehat{c} ,\nabla\psi \rangle \mathrm{d}s \right| \rightarrow 0~~ \textrm{as} ~ j\rightarrow\infty.
\end{split}
\end{equation}
To do that, we make a decomposition
\begin{equation*}
\begin{split}
&\textbf{h}_{\epsilon_j}' (\widehat{n}_{\epsilon_j})\chi(\widehat{c}_{\epsilon_j}) \widehat{c}_{\epsilon_j}^{\frac{3}{4}}- \chi(\widehat{c} ) \widehat{c} ^{\frac{3}{4}}= \left(\textbf{h}_{\epsilon_j}' (\widehat{n}_{\epsilon_j})-1\right) \chi(\widehat{c}_{\epsilon_j}) \widehat{c}_{\epsilon_j}^{\frac{3}{4}}
\\
&\quad +  \left( \widehat{c}_{\epsilon_j}-\widehat{c} \right) \chi'\left(\theta\widehat{c}_{\epsilon_j}+(1-\theta) \widehat{c}\right)\widehat{c}_{\epsilon_j}^{\frac{3}{4}}
+  \left( \widehat{c}_{\epsilon_j}^{\frac{3}{4}}-\widehat{c} ^{\frac{3}{4}} \right) \chi(\widehat{c} ) ,
\end{split}
\end{equation*}
where the Mean Value Theorem is applied for some $\theta\in (0,1)$. To verify \eqref{5.33}, let us first show the convergence
\begin{equation}\label{5.36}
\begin{split}
\textbf{h}_{\epsilon_j}' (\widehat{n}_{\epsilon_j})\chi(\widehat{c}_{\epsilon_j}) \widehat{c}_{\epsilon_j}^{\frac{3}{4}} \rightarrow \chi(\widehat{c} ) \widehat{c} ^{\frac{3}{4}}\quad \textrm{strongly in}\quad L^{\frac{20}{3}}\left(0,T; L^{\frac{20}{3}}(\mathcal {O})\right)~~ \textrm{as} ~ j\rightarrow\infty,~~\widehat{\mathbb{P}}\textrm{-a.s.}
\end{split}
\end{equation}
The proof is divided into three steps:
\begin{itemize}[leftmargin=0.9cm]
\item [$\bullet$] From the definition of $\textbf{h}_{\epsilon_j} $ and $\widehat{n}_{\epsilon_j}\in L^\infty(0,T;L^\infty(\mathcal {O}))$ $\widehat{\mathbb{P}}$-a.a., we infer  that $\textbf{h}_{\epsilon_j}' (\widehat{n}_{\epsilon_j})-1=- \frac{\epsilon_j \widehat{n}_{\epsilon_j}}{1+\epsilon_j \widehat{n}_{\epsilon_j}}\rightarrow 0$ as $j\rightarrow\infty$ $\widehat{\mathbb{P}}\otimes \mathrm{d}t\otimes\mathrm{d}x$-a.s., which together with $ \|\chi(\widehat{c}_{\epsilon_j}) \widehat{c}_{\epsilon_j}^{\frac{3}{4}}\|_{L^\infty(0,T;L^\infty )} \leq C$ $\widehat{\mathbb{P}}$-a.s. (see \eqref{(4.5)})   and the Dominated Convergence Theorem yields that
\begin{equation*}
\begin{split}
 \left\|\left(\textbf{h}_{\epsilon_j}' (\widehat{n}_{\epsilon_j})-1\right) \chi(\widehat{c}_{\epsilon_j}) \widehat{c}_{\epsilon_j}^{\frac{3}{4}} \right\|_{L^{\frac{20}{3}}(0,T; L^{\frac{20}{3}} )} \rightarrow 0~~\textrm{ as} ~~j\rightarrow\infty,~~ \widehat{\mathbb{P}} \textrm{-a.s.}
\end{split}
\end{equation*}

\item [$\bullet$] Using the boundedness $\|\chi'(\theta\widehat{c}_{\epsilon_j}+(1-\theta) \widehat{c})\widehat{c}_{\epsilon_j}^{\frac{3}{4}} \| _{ L^\infty([0,T],L^\infty)  }\leq C$ $\widehat{\mathbb{P}}$-a.s. (see \eqref{(4.5)}),  we have
\begin{equation*}
\begin{split}
 \left\|\left( \widehat{c}_{\epsilon_j}-\widehat{c} \right) \chi'\left(\theta\widehat{c}_{\epsilon_j}+(1-\theta) \widehat{c}\right)\widehat{c}_{\epsilon_j}^{\frac{3}{4}}  \right\|_{L^{\frac{20}{3}}(0,T; L^{\frac{20}{3}} )} \rightarrow 0~~\textrm{ as} ~~j\rightarrow\infty,~~ \widehat{\mathbb{P}} \textrm{-a.s.}
\end{split}
\end{equation*}

\item [$\bullet$] Since by Lemma \ref{lem5.5}  $\widehat{c}_{\epsilon_j}\rightarrow\widehat{ c}$ strongly  in  $L^2_{loc}(0,\infty;L^{2}(\mathcal {O}))$ $\widehat{\mathbb{P}}$-a.s., we have $\widehat{c}_{\epsilon_j}\rightarrow\widehat{ c}$ strongly  in  $L^2\left(\Omega;L^2_{loc}(0,\infty;L^{2}(\mathcal {O}))\right)$, due to the uniform bound in Lemma \ref{lem5.1} and the Vitali Convergence Theorem. Hence, there exists a subsequence denoted by itself such that
\begin{equation}\label{5.34}
\begin{split}
\widehat{c}_{\epsilon_j}\rightarrow\widehat{ c}\quad \widehat{\mathbb{P}}\otimes \mathrm{d}t\otimes\mathrm{d}x \textrm{-a.a.},~~  \textrm{as}~~ j\rightarrow\infty.
\end{split}
\end{equation}
In terms of \eqref{5.34} and the continuity of $\chi(\cdot)$, we have $|\widehat{c}_{\epsilon_j}^{\frac{3}{4}}\chi(\widehat{c}_{\epsilon_j})|^{\frac{20}{3}}\rightarrow |\widehat{c} ^{\frac{3}{4}}\chi(\widehat{c} )|^{\frac{20}{3}} $ $\widehat{\mathbb{P}}\otimes \mathrm{d}t\otimes\mathrm{d}x$-a.a. as $j\rightarrow\infty$. It follows from the boundedness of $\widehat{c}$ and $ \widehat{c}_{\epsilon_j}$ as well as the Dominated Convergence Theorem that
\begin{equation}\label{5.35}
\begin{split}
 \int_0^T \int_\mathcal {O}|\widehat{c}_{\epsilon_j}^{\frac{3}{4}}\chi(\widehat{c}_{\epsilon_j})|^{\frac{20}{3}}\mathrm{d}x \mathrm{d} t \rightarrow  \int_0^T \int_\mathcal {O} |\widehat{c} ^{\frac{3}{4}}\chi(\widehat{c} )|^{\frac{20}{3}} \mathrm{d}x \mathrm{d}t ~~\textrm{ as} ~~j\rightarrow\infty,~~ \widehat{\mathbb{P}} \textrm{-a.s.}
\end{split}
\end{equation}
Moreover, we conclude from $c)$ of Lemma \ref{lem5.5} that
$$
\widehat{c}_{\epsilon_j}^{\frac{3}{4}}\chi(\widehat{c}_{\epsilon_j})\rightarrow \widehat{c} ^{\frac{3}{4}}\chi(\widehat{c} )\quad  \textrm{weakly-star in} \quad L^\infty_{loc}(0,\infty;L^\infty(\mathcal {O}))~~\textrm{ as} ~~j\rightarrow\infty,~~ \widehat{\mathbb{P}} \textrm{-a.s.},
$$
and hence weakly in $L^{\frac{20}{3}} (0,T; L^{\frac{20}{3}}(\mathcal {O})  )$ $\widehat{\mathbb{P}}$-a.s.,  which together with \eqref{5.35} implies
$$
\left\|\left( \widehat{c}_{\epsilon_j}^{\frac{3}{4}}-\widehat{c} ^{\frac{3}{4}} \right) \chi(\widehat{c} ) \right\|_{L^{\frac{20}{3}}(0,T; L^{\frac{20}{3}} )} \rightarrow 0~~\textrm{ as} ~~j\rightarrow\infty,~~ \widehat{\mathbb{P}} \textrm{-a.s.}
$$
This proves \eqref{5.36}.
\end{itemize}
Note that
$$
\widehat{n}_{\epsilon_j}\textbf{h}_{\epsilon_j}' (\widehat{n}_{\epsilon_j})\chi(\widehat{c}_{\epsilon_j})\nabla \widehat{c}_{\epsilon_j}=4\widehat{n}_{\epsilon_j}\left(\textbf{h}_{\epsilon_j}' (\widehat{n}_{\epsilon_j})\chi(\widehat{c}_{\epsilon_j}) \widehat{c}_{\epsilon_j}^{\frac{3}{4}}\right) \nabla \widehat{c}_{\epsilon_j}^{\frac{1}{4}},
$$
we get from \eqref{5.25}, \eqref{5.36}  and \eqref{5.1d} that
\begin{equation}\label{5.37}
\begin{split}
\widehat{n}_{\epsilon_j}\textbf{h}_{\epsilon_j}' (\widehat{n}_{\epsilon_j})\chi(\widehat{c}_{\epsilon_j})\nabla \widehat{c}_{\epsilon_j} \rightarrow \widehat{n} \chi(\widehat{c} ) \nabla\widehat{c} \quad \textrm{weakly in}\quad L^{1}(0,T; L^{1}(\mathcal {O}))~~ \textrm{as} ~ j\rightarrow\infty,~~\widehat{\mathbb{P}}\textrm{-a.s.}
\end{split}
\end{equation}
Moreover, for any $p\geq1$, Lemma \ref{lem5.1} implies that
\begin{equation*}
\begin{split}
\widehat{\mathbb{E}}\left[ \int_0^T \langle\widehat{n}_{\epsilon_j}\textbf{h}_{\epsilon_j}' (\widehat{n}_{\epsilon_j})\chi(\widehat{c}_{\epsilon_j})\nabla \widehat{c}_{\epsilon_j} ,\nabla\psi \rangle\mathrm{d}t\right]^p &\leq C \widehat{\mathbb{E}}\left[ \int_0^T \|\widehat{n}_{\epsilon_j} \|_{L^1} \mathrm{d}t\right]^p\\
&\leq C \widehat{\mathbb{E}}\left[ \int_0^T \|\widehat{n}_{\epsilon_j} \|^{\frac{5}{3}}_{L^{\frac{5}{3}}} \mathrm{d}t\right]^{\frac{3p}{5}}\leq C,
\end{split}
\end{equation*}
which together with \eqref{5.37} and the Vitali Convergence Theorem leads to \eqref{5.33}. Hence we get
\begin{equation}\label{5.38}
\begin{split}
\lim\limits_{j\rightarrow\infty}  \int_0^T\widehat{\mathbb{E}}\left| \int_0^t \left\langle\widehat{n}_{\epsilon_j}\textbf{h}_{\epsilon_j}' (\widehat{n}_{\epsilon_j})\chi(\widehat{c}_{\epsilon_j}) \nabla\widehat{c}_{\epsilon_j} -  \widehat{n} \chi(\widehat{c} ) \nabla\widehat{c} ,\nabla\psi\right\rangle \mathrm{d}s \right|=0.
\end{split}
\end{equation}

According to  \eqref{5.28}, \eqref{5.29}, \eqref{5.32}, \eqref{5.38}, and noting that $(\widehat{n}_{\epsilon_j},\widehat{c}_{\epsilon_j},\widehat{u}_{\epsilon_j})$ is a martingale weak solution to \eqref{SCNS-1},  we gain
\begin{equation*}
\begin{split}
 \widehat{\mathbb{E}}  \int_0^T\left|\langle\widehat{n} (t)
  ,\psi\rangle-\textbf{m} (\widehat{n} ,\widehat{c} ,\widehat{u} )_t\right|     \mathrm{d}t =\lim\limits_{j\rightarrow\infty}\widehat{\mathbb{E}}  \int_0^T\left|\langle\widehat{n}_{\epsilon_j}(t)
  ,\psi\rangle-\textbf{m}_{\epsilon_j}(\widehat{n}_{\epsilon_j},\widehat{c}_{\epsilon_j},\widehat{u}_{\epsilon_j})_t\right|     \mathrm{d}t  =0,
\end{split}
\end{equation*}
which implies that
\begin{equation*}
\begin{split}
\langle\widehat{n} (t)
  ,\psi\rangle=\textbf{m} (\widehat{n} ,\widehat{c} ,\widehat{u} )_t,\quad \textrm{for a.e.}~ t\in [0,T] ,~~ \widehat{\mathbb{P}}\textrm{ -a.s.}
\end{split}
\end{equation*}
Note  the fact that if two c\`{a}dl\`{a}g functions equal for a.e. $t \in [0, T ]$, they must be equal for all $t\in[0, T ]$. Thus  the $n$-equation  holds for all $t \in [0, T ]$, $\widehat{\mathbb{P}}$-a.s. The proof for $n$-equation is now completed.

\textsc{Step 2 (Identification of $c$-equation).}  The identification of limit for the $c$-equation can be treated in a same manner as that for $n$-equation, where the main difficulty lies upon proving that $\widehat{u}_{\epsilon_j}\widehat{c}_{\epsilon_j} \rightarrow \widehat{u}\widehat{c}$ strongly in $ L^1(0,T;L^1(\mathcal {O})) $  and
$\textbf{h}_{\epsilon_j}(\widehat{n}_{\epsilon_j}) f(\widehat{c}_{\epsilon_j})\rightarrow \widehat{n} f(\widehat{c} )$ strongly in $L^1(0,T;L^1(\mathcal {O})) $ as $j\rightarrow\infty$, $\widehat{\mathbb{P}}$-a.s. Here we only verify the later one for saving the space.

Indeed, we first get by continuity of $f(\cdot)$ and Lemma \ref{lem5.5} that
\begin{equation}\label{5.39}
\begin{split}
f(\widehat{c}_{\epsilon_j})\rightarrow f(\widehat{c} )\quad \textrm{weakly-star}\quad L^\infty(0,T;L^\infty(\mathcal {O}))~~ \textrm{as} ~ j\rightarrow\infty,~~\widehat{\mathbb{P}}\textrm{-a.s.}
\end{split}
\end{equation}
It follows from the Lipschitz continuity property of $\textbf{h}_{\epsilon_j}(\cdot)$ that
\begin{equation*}
\begin{split}
 \|\textbf{h}_{\epsilon_j}(\widehat{n}_{\epsilon_j})- \widehat{n} \|_{L^{\frac{5}{4}}(0,T;L^{\frac{5}{4}})} & \leq  \|\textbf{h}_{\epsilon_j}(\widehat{n}_{\epsilon_j})- \textbf{h}_{\epsilon_j}(\widehat{n} ) \|_{L^{\frac{5}{4}}(0,T;L^{\frac{5}{4}})}  + \|\textbf{h}_{\epsilon_j}(\widehat{n} )- \widehat{n} \|_{L^{\frac{5}{4}}(0,T;L^{\frac{5}{4}})} \\
& \leq  \| \widehat{n}_{\epsilon_j} -  \widehat{n} \|_{L^{\frac{5}{4}}(0,T;L^{\frac{5}{4}})}  + \|\textbf{h}_{\epsilon_j}(\widehat{n} )- \widehat{n} \|_{L^{\frac{5}{4}}(0,T;L^{\frac{5}{4}})}.
\end{split}
\end{equation*}
For the first term,  we get by $c)$ of Lemma \ref{lem5.5} that $\| \widehat{n}_{\epsilon_j} -  \widehat{n} \|_{L^{\frac{5}{4}}(0,T;L^{\frac{5}{4}})} \rightarrow 0$ as $j\rightarrow\infty$ $\widehat{\mathbb{P}}$-a.s. For the second term, we first use the property $ \textbf{h}_{\epsilon_j}(s)\rightarrow s$ as $j\rightarrow\infty$ to get $\textbf{h}_{\epsilon_j}(\widehat{n} )\rightarrow \widehat{n} $ $\widehat{\mathbb{P}}\otimes \mathrm{d}t\otimes\mathrm{d}x$-a.a.  as $j\rightarrow\infty$. Moreover, we get by the uniform bound \eqref{5.1d} that $\|\textbf{h}_{\epsilon_j}(\widehat{n} ) \|_{L^{\frac{5}{4}}(0,T;L^{\frac{5}{4}})}\leq \| \widehat{n}   \|_{L^{\frac{5}{4}}(0,T;L^{\frac{5}{4}})} <\infty$ $\widehat{\mathbb{P}}$-a.s. Therefore, we get by Dominated Convergence Theorem that
\begin{equation}\label{5.40}
\begin{split}
\textbf{h}_{\epsilon_j}(\widehat{n}_{\epsilon_j})\rightarrow\widehat{n} \quad\textrm{strongly in} \quad L^{\frac{5}{4}}(0,T;L^{\frac{5}{4}}(\mathcal {O}))~~\textrm{ as} ~~j\rightarrow\infty,~~ \widehat{\mathbb{P}} \textrm{-a.s.}
\end{split}
\end{equation}
Note that \eqref{5.39} also infer that $f(\widehat{c}_{\epsilon_j})\rightarrow f(\widehat{c} )$ weakly in $ L^5(0,T;L^5(\mathcal {O}))$ as $j\rightarrow\infty$, $\widehat{\mathbb{P}}$-a.s., which combined with \eqref{5.40} leads to the desired result.

\textsc{Step 3 (Identification of $u$-equation).} For any $\psi\in  \mathcal {C}^\infty_{0}(\mathcal {O}\times [0,\infty);\mathbb{R}^3)$ with $\div \psi=0$, we define
\begin{equation*}
\begin{split}
&\textbf{k}_{\epsilon_j}(\widehat{n}_{\epsilon_j},\widehat{c}_{\epsilon_j},\widehat{u}_{\epsilon_j})_t
\overset {\textrm{def}} {=}\langle\widehat{u}_{\epsilon_j}(0),\psi\rangle
+ \int_0^t\langle\mathcal {P} (\textbf{L}_{\epsilon_j} \widehat{u}_{\epsilon_j}\otimes \widehat{u}_{\epsilon_j}),\nabla\psi\rangle\mathrm{d}s- \int_0^t\langle \nabla \widehat{u}_{\epsilon_j},\nabla\psi\rangle\mathrm{d}s\\
&\quad+ \int_0^t\langle\mathcal {P}(\widehat{n}_{\epsilon_j}\nabla \Phi ),\psi\rangle\mathrm{d}s+ \int_0^t\langle\mathcal {P}h(s,\widehat{u}_{\epsilon_j}) ,\psi\rangle\mathrm{d}s+ \int_0^t\langle\mathcal {P}g(s,\widehat{u}_{\epsilon_j}) \mathrm{d}\widehat{W}_{\epsilon_j}(s),\psi\rangle\\
&\quad + \int_0^t \int_{Z_0} \langle\mathcal {P}K(\widehat{u}_\epsilon(s-),z) ,\psi\rangle \widetilde{\widehat{\pi}}_{\epsilon_j}(\mathrm{d}s,\mathrm{d}z) + \int_0^t \int_{Z\backslash Z_0} \langle\mathcal {P}G(\widehat{u}_\epsilon(s-),z) ,\psi\rangle\widehat{\pi}_{\epsilon_j}(\mathrm{d}z)\mathrm{d}s,
\end{split}
\end{equation*}
and
\begin{equation*}
\begin{split}
&\textbf{k}(\widehat{n},\widehat{c},\widehat{u})_t
\overset {\textrm{def}} {=}\langle\widehat{u}(0),\psi\rangle
+ \int_0^t\langle\mathcal {P} \widehat{u}\otimes\widehat{u},\nabla\psi\rangle\mathrm{d}s- \int_0^t\langle \nabla \widehat{u},\nabla\psi\rangle\mathrm{d}s+ \int_0^t\langle\mathcal {P}(\widehat{n}\nabla \Phi ),\psi\rangle\mathrm{d}s\\
&\quad+ \int_0^t\langle\mathcal {P}h(s,\widehat{u}) ,\psi\rangle\mathrm{d}s+ \langle \int_0^t\mathcal {P}g(s,\widehat{u}) \mathrm{d}\widehat{W}_s,\psi\rangle + \int_0^t \int_{Z_0} \langle\mathcal {P}K(\widehat{u}(s-),z) ,\psi\rangle \widetilde{\widehat{\pi}}(\mathrm{d}s,\mathrm{d}z)\\
& \quad+ \int_0^t \int_{Z\backslash Z_0} \langle\mathcal {P}G(\widehat{u}(s-),z) ,\psi\rangle\widehat{\pi}(\mathrm{d}z)\mathrm{d}s.
\end{split}
\end{equation*}
Apparently, since $(\widehat{n}_{\epsilon_j},\widehat{c}_{\epsilon_j},\widehat{u}_{\epsilon_j})$ is a martingale weak solution to the system \eqref{SCNS-1},  there holds $\langle\widehat{u}_{\epsilon_j}(t),\psi\rangle=\textbf{k}_{\epsilon_j}(\widehat{n}_{\epsilon_j},
\widehat{c}_{\epsilon_j},\widehat{u}_{\epsilon_j})_t$ for all $t\in [0,T]$ $\widehat{\mathbb{P}}$-a.s. In particular,
\begin{equation}\label{5.41}
\begin{split}
\widehat{\mathbb{E}}  \int_0^T\left|\langle\widehat{u}_{\epsilon_j}(t)
  ,\psi\rangle-\textbf{k}_{\epsilon_j}(\widehat{n}_{\epsilon_j},
  \widehat{c}_{\epsilon_j},\widehat{u}_{\epsilon_j})_t\right| ^2    \mathrm{d}t  =0.
\end{split}
\end{equation}
The \textbf{Key points} is to prove that, for any $\psi\in \mathcal {C}_0^\infty(\mathcal {O}\times (0,T);\mathbb{R}^3)$ with $\div \psi =0$,
\begin{subequations}
\begin{align}
 &\lim\limits_{j\rightarrow\infty}\widehat{\mathbb{E}}  \int_0^T\left|\langle\widehat{u}_{\epsilon_j}(t)-\widehat{u}  (t),\psi\rangle\right|^2     \mathrm{d}t =0,  \label{5.42a}\\
 & \lim\limits_{j\rightarrow\infty}\widehat{\mathbb{E}}  \int_0^T
 \left|\langle\textbf{k}_{\epsilon_j}(\widehat{n}_{\epsilon_j},\widehat{c}_{\epsilon_j},\widehat{u}_{\epsilon_j})_t
 -\textbf{k}(\widehat{n} ,\widehat{c} ,\widehat{u} )_t,\psi\rangle\right|^2    \mathrm{d}t =0.\label{5.42b}
\end{align}
\end{subequations}
Indeed, if \eqref{5.42a}-\eqref{5.42b}  hold, then we infer from \eqref{5.41} that
\begin{equation*}
\begin{split}
 \widehat{\mathbb{E}}  \int_0^T\left|\langle\widehat{u} (t)
  ,\psi\rangle-\textbf{k} (\widehat{n} ,\widehat{c} ,\widehat{u} )_t\right|^2     \mathrm{d}t =\lim\limits_{j\rightarrow\infty}\widehat{\mathbb{E}}  \int_0^T\left|\langle\widehat{u}_{\epsilon_j}(t)
  ,\psi\rangle-\textbf{k}_{\epsilon_j}(\widehat{n}_{\epsilon_j},\widehat{c}_{\epsilon_j},
  \widehat{u}_{\epsilon_j})_t\right|^2      \mathrm{d}t =0,
\end{split}
\end{equation*}
which indicates that $\langle\widehat{u} (t)
,\psi\rangle=\textbf{k} (\widehat{n} ,\widehat{c} ,\widehat{u} )_t$ holds for a.e. $t\in [0,T]$ $\widehat{\mathbb{P}}$-a.s. Since two c\`{a}dl\`{a}g functions  must be equal for all $t\in[0, T ]$ if they are equal for a.e.  $t \in [0, T ]$, we see that $\widehat{u} (\cdot)$ is a martingale solution for the $u$-equation for all $t \in [0, T ]$, $\widehat{\mathbb{P}}$-a.s.
 In the sequel, let us prove \eqref{5.42a} and \eqref{5.42b} respectively.

\textsc{Proof of \eqref{5.42a}.}  By $c)$ of Lemma \ref{lem5.5}, we have
\begin{equation*}
\begin{split}
 \int_0^T \left|\langle\widehat{u}_{\epsilon_j}(t)-\widehat{u}  (t),\psi\rangle\right|^2  \mathrm{d}t
\leq  \|\psi\|_{L^\infty(0,T;L^2)}^2  \int_0^T\| \widehat{u}_{\epsilon_j}(t)-\widehat{u}  (t)\|_{ L^2 }^2 \mathrm{d}t  \rightarrow 0 \quad \textrm{as} ~ j\rightarrow\infty,~~\widehat{\mathbb{P}}\textrm{-a.s}.
\end{split}
\end{equation*}
Moreover, it follows from \eqref{5.1f} that
$$
\widehat{\mathbb{E}}\left( \int_0^T \left|\langle\widehat{u}_{\epsilon_j}(t)-\widehat{u}  (t),\psi\rangle\right|^2  \mathrm{d}t\right)^p \leq C,\quad \textrm{for any}~ p\geq1,
$$
which indicates the uniform integrability of $ \int_0^T \left|\langle\widehat{u}_{\epsilon_j}(t)-\widehat{u}  (t),\psi\rangle\right|^2  \mathrm{d}t$, and hence by Vitali Convergence Theorem we get \eqref{5.42a}.

\textsc{Proof of  \eqref{5.42b}.}  For the linear terms involved in $\textbf{q} (\widehat{n} ,\widehat{c} ,\widehat{u} )_t$, similar to Step 1, by using the Lemma \ref{lem5.5} and Lemma \ref{lem5.1}, one can prove that
\begin{align}
 &\lim\limits_{j\rightarrow\infty} \widehat{\mathbb{E}} \left[|\langle\widehat{u}_{\epsilon_j} (0)- \widehat{n} (0),\psi|_{t=0}\rangle|^2 \right] =0,\label{5.43}\\
&\lim\limits_{j\rightarrow\infty} \int_0^T\widehat{\mathbb{E}}\left| \int_0^t \langle\nabla \widehat{u}_{\epsilon_j}-\nabla \widehat{u} ,\nabla\psi\rangle \mathrm{d}s\right|^2  \mathrm{d}t=0,\label{5.44}\\
&\lim\limits_{j\rightarrow\infty} \int_0^T\widehat{\mathbb{E}}\left| \int_0^t \langle\mathcal {P}(\widehat{n}_{\epsilon_j}\nabla \Phi )-\mathcal {P}(\widehat{n} \nabla \Phi ) , \psi\rangle \mathrm{d}s\right|^2  \mathrm{d}t=0.\label{5.45}
\end{align}
It remains to deal with the remaining nonlinear terms and the stochastic integrals involved in $\textbf{q} (\widehat{n} ,\widehat{c} ,\widehat{u} )_t$. The proof will be divided into two parts.

\textsc{Part 1}: We show the convergence of nonlinear terms, that is,
\begin{align}
&\lim\limits_{j\rightarrow\infty} \int_0^T\widehat{\mathbb{E}}\left| \int_0^t\left\langle\textbf{L}_{\epsilon_j} \widehat{u}_{\epsilon_j}\otimes \widehat{u}_{\epsilon_j}-\widehat{u} \otimes \widehat{u}, \nabla\psi \right\rangle\mathrm{d}s  \right|^2\mathrm{d} t=0,\label{5.49}\\
& \lim\limits_{j\rightarrow\infty} \int_0^T\widehat{\mathbb{E}}\left| \int_0^t\left\langle h(t,\widehat{u}_{\epsilon_j})-h(t,\widehat{u} ), \nabla\psi \right\rangle\mathrm{d}s  \right|^2\mathrm{d} t=0.\label{5.51}
\end{align}

The convergence of the convection term is based on the following fact:
\begin{equation}\label{5.46}
\begin{split}
 \textbf{L}_{\epsilon_j} \widehat{u}_{\epsilon_j}\otimes \widehat{u}_\epsilon\rightarrow \widehat{u} \otimes \widehat{u}\quad \textrm{strongly in}\quad L^1(0,T;L^1(\mathcal {O}))~~ \textrm{as} ~ j\rightarrow\infty,~~\widehat{\mathbb{P}}\textrm{-a.s.}
\end{split}
\end{equation}
Indeed, it suffices to prove that $\textbf{L}_{\epsilon_j} \widehat{u}_{\epsilon_j} \rightarrow \widehat{u}$ strongly in $L^2(0,T;L^2(\mathcal {O}))$ as $j\rightarrow\infty$ $\widehat{\mathbb{P}}$-a.s., and $\widehat{u}_{\epsilon_j} \rightarrow \widehat{u}$ strongly in $L^2(0,T;L^2(\mathcal {O}))$ as $j\rightarrow\infty$ $\widehat{\mathbb{P}}$-a.s. The later conclusion is ensured by $c)$ of Lemma \ref{lem5.5}, and it remains to prove the former one.

Recalling the following properties:
\begin{equation}\label{wc}
\begin{split}
 \| \textbf{L}_\epsilon f\|_{L^2}\leq \| f\|_{L^2},~~ \| \textbf{L}_\epsilon f-f\|_{L^2}\rightarrow 0  ~~ \textrm{as} ~~ \epsilon\rightarrow 0, \quad \textrm{for any} ~~f\in L^2_\sigma(\mathcal {O}).
\end{split}
\end{equation}
We have
\begin{equation}\label{5.47}
\begin{split}
 \int_0^T \|\textbf{L}_{\epsilon_j}  \widehat{u}_{\epsilon_j}(t)-\widehat{u} (t)\|_{ L^2 }^2 \mathrm{d} t \leq  \int_0^T\| \widehat{u}_{\epsilon_j}(t)-\widehat{u}(t) \|_{ L^2 } ^2 \mathrm{d} t + \int_0^T\|\textbf{L}_{\epsilon_j}  \widehat{u}(t) -\widehat{u}(t) \|_{ L^2 } ^2\mathrm{d} t.
\end{split}
\end{equation}
The first term on the R.H.S. of \eqref{5.47} converges to zero almost surely according to Lemma \ref{lem5.5}. For the second term, on the one hand we get by \eqref{wc}  that $\| \textbf{L}_\epsilon  \widehat{u}(t) -\widehat{u}(t) \|_{ L^2 } ^2\rightarrow 0$ for $ \widehat{\mathbb{P}}\otimes\mathrm{d} t$-almost all $(\omega,t)\in \widehat{\Omega}\times[0,T]$; on the other hand, we get by \eqref{5.1f}
\begin{equation*}
\begin{split}
  \int_0^T\|\textbf{L}_{\epsilon_j}  \widehat{u}(t) -\widehat{u}(t) \|_{ L^2 } ^2\mathrm{d} t &\leq C \left(\|\textbf{L}_{\epsilon_j}  \widehat{u} \|_{L^2(0,T; L^2) }^{2} +\|\widehat{u}  \|_{ L^2(0,T; L^2)}^{2} \right) \\
 & \leq C \|\widehat{u} \|_{ L^2(0,T; L^2)}^{2}  <\infty, ~~\widehat{\mathbb{P}}\textrm{-a.s.}
\end{split}
\end{equation*}
An application of the Vitali Convergence Theorem leads to $ \int_0^T\|\textbf{L}_{\epsilon_j}  \widehat{u}(t) -\widehat{u}(t) \|_{ L^2 } ^2\mathrm{d} t\rightarrow 0$ $\widehat{\mathbb{P}}$-a.s. Therefore,
$$
 \int_0^T\|\textbf{L}_{\epsilon_j} \widehat{u}_{\epsilon_j} (t) -\widehat{u}(t) \|_{ L^2 } ^2\mathrm{d} t\rightarrow 0, ~~\textrm{as} ~~j\rightarrow\infty,~~\widehat{\mathbb{P}}\textrm{-a.s.}
$$
This proves  \eqref{5.46}.

Moreover, by uniform bound \eqref{5.1f}, we get for $p\geq4$
\begin{equation*}
\begin{split}
\widehat{\mathbb{E}}\left| \int_0^t\left\langle\textbf{L}_{\epsilon_j} \widehat{u}_{\epsilon_j}\otimes \widehat{u}_{\epsilon_j}-\widehat{u} \otimes \widehat{u}, \nabla\psi \right\rangle\mathrm{d}s  \right|^p \leq C \widehat{\mathbb{E}} \left(\|\widehat{u}_{\epsilon_j}\|^p_{L^2(0,T;L^2)}+\|\widehat{u} \|^p_{L^2(0,T;L^2)}\right)\leq C,
\end{split}
\end{equation*}
which infers the uniform integrability of $| \int_0^t\left\langle\textbf{L}_{\epsilon_j} \widehat{u}_{\epsilon_j}\otimes \widehat{u}_{\epsilon_j}-\widehat{u} \otimes \widehat{u}, \nabla\psi \right\rangle\mathrm{d}s|^2$. In view of \eqref{5.46} and the Vitali Convergence Theorem, we get
\begin{equation}\label{5.48}
\begin{split}
\widehat{\mathbb{E}}\left| \int_0^t\left\langle\textbf{L}_{\epsilon_j} \widehat{u}_{\epsilon_j}\otimes \widehat{u}_{\epsilon_j}-\widehat{u} \otimes \widehat{u}, \nabla\psi \right\rangle\mathrm{d}s  \right|^2 \rightarrow 0 \quad \textrm{as} \quad j\rightarrow\infty.
\end{split}
\end{equation}
Making use of the Dominated Convergence Theorem, \eqref{5.48} implies \eqref{5.49}.

Since $(\widehat{u}_{\epsilon_j})_{j\geq 1}$ is uniformly bounded in $L^4\left(\Omega;L^\infty_{loc}(0,\infty;L^2_\sigma(\mathcal {O}) )\right)$,  the sequence $(\|\widehat{u}_{\epsilon_j}\|_{L^2(0,T;L^2)}^2)_{j\geq 1}$ is uniformly integrable, and the Vitali Convergence Theorem infers that
\begin{equation}\label{5.50}
\begin{split}
\widehat{u}_{\epsilon_j}\rightarrow  \widehat{u}\quad \textrm{strongly in} \quad L^2\left(\Omega;L^2_{loc}(0,\infty;(L^2_\sigma(\mathcal {O}) ))^3\right)~~\textrm{as}~~j\rightarrow\infty.
\end{split}
\end{equation}
Based on \eqref{5.50}, in view of the assumptions on $h$ and $g$, we get
\begin{equation*}
\begin{split}
h(t,\widehat{u}_{\epsilon_j})\rightarrow  h(t,\widehat{u} )\quad \textrm{strongly in}\quad L^2(\Omega;L^2_{loc}(0,\infty;L^2_\sigma(\mathcal {O}) ))~~\textrm{as}~~j\rightarrow\infty,
\end{split}
\end{equation*}
which combined with the Dominated Convergence Theorem leads to \eqref{5.51}.

\textsc{Part 2}: Let us deal with the convergence of stochastic integrals. From \eqref{5.50} and the assumtions on $g$, $K$ and $G$, there hold
\begin{align}
g(t,\widehat{u}_{\epsilon_j})\rightarrow  g(t,\widehat{u} )\quad\textrm{strongly in} \quad L^2\left(\Omega;L^2_{loc}(0,\infty;\mathcal {L}_2(U;L^2_\sigma(\mathcal {O})) )\right)~~\textrm{as}~~j\rightarrow\infty,\label{5.52}
\end{align}
and
\begin{align}
&\lim\limits_{j\rightarrow\infty}\widehat{\mathbb{E}} \int_0^T \int_{Z_0} \int_\mathcal {O}|K(\widehat{u}_{\epsilon_j},z)-K(\widehat{u} ,z) |^2\mathrm{d}x\mu( \mathrm{d}z)\mathrm{d} t=0,\\
&\lim\limits_{j\rightarrow\infty}\widehat{\mathbb{E}} \int_0^T \int_{Z\backslash Z_0} \int_\mathcal {O}|G(\widehat{u}_{\epsilon_j},z)-G(\widehat{u} ,z)|^2\mathrm{d}x \mu( \mathrm{d}z)\mathrm{d} t=0.
\end{align}
By using the BDG inequality and the fact of $(\widehat{W}_{\epsilon_j},\widehat{\pi}_{\epsilon_j})=(\widehat{W} ,\widehat{\pi})$, $\widehat{\mathbb{P}}$-a.s., we get from \eqref{5.52} that
\begin{equation*}
\begin{split}
 \lim\limits_{j\rightarrow\infty}\widehat{\mathbb{E}}\left| \int_0^t\left\langle \mathcal {P}g(s,\widehat{u}_{\epsilon_j})-\mathcal {P}g(s,\widehat{u} ) ,\psi\right\rangle \mathrm{d}\widehat{W}_s \right|^2\leq C\lim\limits_{j\rightarrow\infty}\widehat{\mathbb{E}} \int_0^T\| g(t,\widehat{u}_{\epsilon_j})\rightarrow  g(t,\widehat{u} )\|_{\mathcal {L}_2(U;L^2)}^2\mathrm{d} t=0.
\end{split}
\end{equation*}
On the other hand, we get by using \eqref{5.1f} and the BDG inequality 
that
\begin{equation*}
\begin{split}
& \int_0^T\left(\widehat{\mathbb{E}} \left| \int_0^t\left\langle \mathcal {P}g(s,\widehat{u}_{\epsilon_j})-\mathcal {P}g(s,\widehat{u} ) ,\psi\right\rangle \mathrm{d}\widehat{W} _s \right|^2 \right)^2\mathrm{d} t\\
&\quad\leq C  \int_0^T\left[\widehat{\mathbb{E}}  \int_0^t\left(\|\langle \mathcal {P}g(s,\widehat{u}_{\epsilon_j}) ,\psi\rangle\|_{\mathcal {L}_2(U;\mathbb{R})}^2+ |\langle \mathcal {P}g(s,\widehat{u} ) ,\psi\rangle|_{\mathcal {L}_2(U;\mathbb{R})}^2\right)\mathrm{d}s\right]^2\mathrm{d} t\\
&\quad\leq C \widehat{\mathbb{E}}\sup_{t\in [0,T]} \left(\|\widehat{u}_{\epsilon_j}(s)\|^4+ \|\widehat{u}(s)\|_{L^2}^4+1\right)  \leq C,
\end{split}
\end{equation*}
which implies the uniform integrability of $\widehat{\mathbb{E}}| \int_0^t \langle \mathcal {P}g(s,\widehat{u}_{\epsilon_j})-\mathcal {P}g(s,\widehat{u} ) ,\psi \rangle \mathrm{d}\widehat{W}_s|^2$ for $t\in [0,T]$, and it follows from Vitali Convergence Theorem that
\begin{equation}\label{5.55}
\begin{split}
 \lim\limits_{j\rightarrow\infty} \int_0^T\widehat{\mathbb{E}}\left[\left| \int_0^t\left\langle \mathcal {P}g(s,\widehat{u}_{\epsilon_j})-\mathcal {P}g(s,\widehat{u} ) ,\psi\right\rangle \mathrm{d}\widehat{W} _s \right|^2\right]\mathrm{d} t=0.
\end{split}
\end{equation}
For the stochastic integral with respect to $\widetilde{\widehat{\pi}}$,  similar to \eqref{5.55}, 
one can derive that
\begin{equation*}
\begin{split}
 \lim\limits_{j\rightarrow\infty} \int_0^T\widehat{\mathbb{E}}\left[\left| \int_0^t \int_{Z_0} \langle\mathcal {P}K(\widehat{u}_{\epsilon_j}(s-),z)-\mathcal {P}K(\widehat{u} (s-),z) ,\psi\rangle \widetilde{\widehat{\pi}} (\mathrm{d}s,\mathrm{d}z) \right|^2\right]\mathrm{d} t=0.
\end{split}
\end{equation*}
For the last integral with respect to $\widehat{\pi}$, by making use of the facts of $\widetilde{\widehat{\pi}}(\mathrm{d}t,\mathrm{d}z)=\widehat{\pi}(\mathrm{d}t,\mathrm{d}z)-\mu(\mathrm{d}z) \mathrm{d}t$, the It\^{o} isometry and the assumption on $G$, we have
\begin{equation*}
\begin{split}
&\widehat{\mathbb{E}}\left[ \left| \int_0^t \int_{Z\backslash Z_0} \langle\mathcal {P}G(\widehat{u}_{\epsilon_j}(s-),z)-\mathcal {P}G(\widehat{u} (s-),z) ,\psi\rangle\widehat{\pi}(\mathrm{d}s,\mathrm{d}z)\right|^2\right]\\
&\quad\leq\widehat{\mathbb{E}}  \int_0^t \int_{Z\backslash Z_0} |\langle\mathcal {P}G(\widehat{u}_{\epsilon_j}(s-),z)-\mathcal {P}G(\widehat{u} (s-),z) ,\psi\rangle|^2 \mu(\mathrm{d}z)\mathrm{d}s \\
&\quad\leq\widehat{\mathbb{E}}  \int_0^t\|\widehat{u}_{\epsilon_j}-\widehat{u}\|^2 _{L^2} \mathrm{d}s\rightarrow0~~\textrm{as}~~j\rightarrow\infty.
\end{split}
\end{equation*}
By applying \eqref{5.1f} and the assumption on $G$, we also have
\begin{equation*}
\begin{split}
&\widehat{\mathbb{E}} \left[\left| \int_0^t \int_{Z\backslash Z_0} \langle\mathcal {P}G(\widehat{u}_{\epsilon_j}(s-),z)-\mathcal {P}G(\widehat{u} (s-),z) ,\psi\rangle\widehat{\pi}(\mathrm{d}s,\mathrm{d}z)\right|^2\right]\\
&\quad\leq C \widehat{\mathbb{E}}  \int_0^t \int_{Z\backslash Z_0} \left( \| G(\widehat{u}_{\epsilon_j}(s-),z)\|^2_{L^2}+\| G(\widehat{u} (s-),z)\|^2_{L^2} \right) \mu(\mathrm{d}z)\mathrm{d}s\leq C.
\end{split}
\end{equation*}
Thereby it follows from the Dominated Convergence Theorem that
\begin{equation}\label{5.56}
\begin{split}
&\lim\limits_{j\rightarrow\infty} \int_0^T\widehat{\mathbb{E}} \left[\left| \int_0^t \int_{Z\backslash Z_0} \langle\mathcal {P}G(\widehat{u}_{\epsilon_j}(s-),z)-\mathcal {P}G(\widehat{u} (s-),z) ,\psi\rangle\widehat{\pi}(\mathrm{d}s,\mathrm{d}z)\right|^2\right] \mathrm{d} t=0.
\end{split}
\end{equation}
Putting \eqref{5.43}, \eqref{5.45}, \eqref{5.49}, \eqref{5.51}, \eqref{5.55} and \eqref{5.56} together leads to \eqref{5.42b}, which implies that $(\widehat{W},\widehat{\pi},\widehat{n},\widehat{c},\widehat{u})$ satisfies the $u$-equation in \eqref{SCNS}.

\textsc{Step 4  (Weak entropy-energy inequality).}  Thanks to Lemma \ref{lem5.5} and Lemma \ref{lem4.5}, the  new random variables $(\widehat{n}_{\epsilon_j},\widehat{c}_{\epsilon_j},\widehat{u}_{\epsilon_j})$ satisfy the entropy-energy inequality \eqref{(4.6)} on the new probability space (see Lemma \ref{lem4.5} for details), which can also be rewritten in the differential form (by using the identity $\tilde{\pi}(\textrm{d}t,\textrm{d}z)=\pi(\textrm{d}t,\textrm{d}z)-\mu (\textrm{d}z)\textrm{d}t$)
\begin{equation} \label{5.57}
\begin{split}
&\mathrm{d} \mathscr{E} [\widehat{n}_{\epsilon_j},\widehat{c}_{\epsilon_j},\widehat{u}_{\epsilon_j}](t) +  \mathscr{D}[\widehat{n}_{\epsilon_j},\widehat{c}_{\epsilon_j},\widehat{u}_{\epsilon_j}](t)\mathrm{d}t \leq C + \| \widehat{u}_{\epsilon_j}(t)\|_{L^2}^2\mathrm{d}t + \mathrm{d}\mathcal {M}_E^{{\epsilon_j}},
\end{split}
\end{equation}
where $\mathscr{E} [\cdot]$ and $\mathscr{D}[\cdot]$ are defined as before, and $\mathcal {M}_E^{{\epsilon_j}}$ is a real-valued square integrable martingale (thanks to the uniform bounds in Lemma \ref{lem5.5}) given by
{\wuhao\begin{equation*}
\begin{split}
\mathcal {M}_E^{{\epsilon_j}}&=c^\dag \int_0^t 2\langle \widehat{u}_{\epsilon_j}, \mathcal {P}g(s,\widehat{u}_{\epsilon_j}) \mathrm{d}\widehat{W}_{\epsilon_j}(s)  \rangle+ c^\dag \int_0^t \int_{Z_0}  \| K(\widehat{u}_{\epsilon_j}(s-),z)\|_{L^2}^2  \widetilde{\widehat{\pi}}_{\epsilon_j}(\mathrm{d} s, \mathrm{d} z)\\
 & +c^\dag  \int_0^t \int_{Z_0}   2\langle \widehat{u}_{\epsilon_j}, \mathcal {P}K(\widehat{u}_{\epsilon_j}(s-),z)\rangle \widetilde{\widehat{\pi}}_{\epsilon_j}(\mathrm{d} s, \mathrm{d} z) +c^\dag \int_0^t \int_{Z\backslash Z_0} \| G(\widehat{u}_{\epsilon_j}(s-),z)\|_{L^2}^2 \widetilde{\widehat{\pi}}_{\epsilon_j}(\mathrm{d} s, \mathrm{d} z)\\
 & +c^\dag \int_0^t \int_{Z\backslash Z_0}  2\langle \widehat{u}_{\epsilon_j}, \mathcal {P}G(\widehat{u}_{\epsilon_j}(s-),z)\rangle   \widetilde{\widehat{\pi}}_{\epsilon_j} (\mathrm{d} s, \mathrm{d} z).
\end{split}
\end{equation*}}\noindent
Here we have used the assumption $(\textbf{A}_4)$ and the fact of $\widetilde{\widehat{\pi}}_{\epsilon_j} (\mathrm{d} s, \mathrm{d} z)=\widehat{\pi}_{\epsilon_j} (\mathrm{d} s, \mathrm{d} z)-\mu(\textrm{d}z)\textrm{d}t$. We also denote by $\mathcal {M}_E$ the associated limit process by removing all of the subscripts $\epsilon_j$ in the definition of $\mathcal {M}_E^{{\epsilon_j}}$. Note that, by Lemma \ref{lem5.5}, $\mathcal {M}_E$ is a square integrable martingale. Multiplying both sides of \eqref{5.57} by any deterministic smooth test function $\phi (t)\geq 0$ with $\phi(T)=0$, and then integrating over $[0,T]$, we get
\begin{equation}\label{5.58}
\begin{split}
&- \int_0^T \phi' (t)\mathscr{E} \left[\widehat{n}_{\epsilon_j},\widehat{c}_{\epsilon_j},\widehat{u}_{\epsilon_j}\right](t)\mathrm{d}t +  \int_0^T\phi(t) \mathscr{D}\left[\widehat{n}_{\epsilon_j},\widehat{c}_{\epsilon_j},\widehat{u}_{\epsilon_j}\right](t)\mathrm{d}t\\
&\quad\leq \phi(0) \mathscr{E} \left[\widehat{n}_{\epsilon_j},\widehat{c}_{\epsilon_j},\widehat{u}_{\epsilon_j}\right](0) +  \int_0^T\phi(t)\| \widehat{u}_{\epsilon_j}(t)\|_{L^2}^2\mathrm{d}t +C  \int_0^T\phi(t)\mathrm{d} t+  \int_0^T\phi(t)\mathrm{d}\mathcal {M}_E^{{\epsilon_j}}.
\end{split}
\end{equation}
To take the limit in the first integral of \eqref{5.58}, we first get by \eqref{5.25} and the inequality $- \frac{1}{e}\leq x \ln x \leq \frac{3}{2} x^{\frac{5}{3}}$ for all $x>0$ that
$$
- \frac{\textrm{Leb}(\mathcal {O})}{e}\leq \int_\mathcal {O}\widehat{n}_{\epsilon_j} \ln \widehat{n}_{\epsilon_j} \mathrm{d} x \leq \frac{3}{2} \int_\mathcal {O} \widehat{n}_{\epsilon_j}^{\frac{5}{3}}  \mathrm{d} x <\infty,\quad \widehat{\mathbb{P}}\textrm{-a.s.}
$$
where $\textrm{Leb}(\mathcal {O})$ denotes the Lebesgue measure of $\mathcal {O}$. On the other hand, it follows from $c)$ of Lemma \ref{lem5.5} that $\|\widehat{n}_{\epsilon_j}-\widehat{n}\|_{L^{\frac{5}{4}}}\rightarrow 0$ $\widehat{\mathbb{P}}\otimes \mathrm{d}t$-almost surely. Thus by extracting a subsequence denoted by itself, we obtain
\begin{equation}\label{5.59}
\begin{split}
 \int_\mathcal {O}\widehat{n}_{\epsilon_j}(t) \ln \widehat{n}_{\epsilon_j}(t) \mathrm{d} x\rightarrow  \int_\mathcal {O}\widehat{n} (t)\ln \widehat{n}(t)  \mathrm{d} x\quad \textrm{as}~~ j\rightarrow\infty,\quad \widehat{\mathbb{P}}\otimes \mathrm{d} t\textrm{-alomst  all}~ (\omega,t).
\end{split}
\end{equation}
Direct calculation shows that
$$
\Delta \Psi (c_{\epsilon})= \frac{\Delta c_{\epsilon}}{\sqrt{\theta(c_{\epsilon})}} + \frac{\theta'(c_{\epsilon})}{2\theta^{\frac{3}{2}}(c_{\epsilon})} |\nabla c_{\epsilon}|^2.
$$
Based on this identity, we deduce from \eqref{(4.5)}, Remark \ref{remark4.2} and Lemma \ref{lem4.6} that $ \Psi (c_{\epsilon})\in L^p(\Omega;L^2(0,T;W^{2,2}(\mathcal {O})))$, which together with the compact embedding $L^2(0,T;W^{1,2}(\mathcal {O}))\subset L^2(0,T;L^{2}(\mathcal {O}))$ implies that  $\nabla \Psi (\widehat{c}_{\epsilon})\rightarrow \nabla \Psi (\widehat{c} )$ strongly in $L^p(\Omega;L^2(0,T;L^{2}(\mathcal {O})))$, and so one can extract a subsequence denoted by itself  such that
\begin{equation}\label{5.60}
\begin{split}
 \int_\mathcal {O}  \frac{1}{2}|\nabla \Psi(\widehat{c}_{\epsilon_j}(t))|^2  \mathrm{d} x\rightarrow  \int_\mathcal {O}  \frac{1}{2}|\nabla \Psi(\widehat{c}(t) )|^2 \mathrm{d} x\quad \textrm{as}~~ j\rightarrow\infty,\quad \widehat{\mathbb{P}}\otimes \mathrm{d} t\textrm{-almost all}~ (\omega,t).
\end{split}
\end{equation}
Furthermore, by virtue of $c)$ of Lemma \ref{lem5.5}, we have
\begin{equation}\label{5.61}
\begin{split}
 \int_\mathcal {O} | \widehat{u}_{\epsilon_j} (t)|^2  \mathrm{d} x\rightarrow  \int_\mathcal {O} |  \widehat{u} (t) |^2  \mathrm{d} x\quad \textrm{as}~~ j\rightarrow\infty,\quad \widehat{\mathbb{P}}\otimes \mathrm{d} t\textrm{-almost all}~ (\omega,t).
\end{split}
\end{equation}
From \eqref{5.59}-\eqref{5.61}, we deduce that  $\mathscr{E} [\widehat{n}_{\epsilon_j},\widehat{c}_{\epsilon_j},\widehat{u}_{\epsilon_j}]\rightarrow \mathscr{E} [\widehat{n},\widehat{c},\widehat{u}]$  as $j\rightarrow\infty$,  $\widehat{\mathbb{P}}\otimes \mathrm{d} t$-almost all $(\omega,t)$, which combined with the Dominated Convergence Theorem yields that
\begin{equation*}
\begin{split}
- \int_0^T \phi' (t)\mathscr{E} \left[\widehat{n}_{\epsilon_j},\widehat{c}_{\epsilon_j},\widehat{u}_{\epsilon_j}\right](t)
\mathrm{d}t\rightarrow - \int_0^T \phi' (t)\mathscr{E} \left[\widehat{n},\widehat{c},\widehat{u}\right](t)
\mathrm{d}t\quad \textrm{as}~~ j\rightarrow\infty,\quad \widehat{\mathbb{P}}\textrm{-a.s.}
\end{split}
\end{equation*}
Similarly,
$$
\phi(0) \mathscr{E} \left[\widehat{n}_{\epsilon_j},\widehat{c}_{\epsilon_j},\widehat{u}_{\epsilon_j}\right](0)\rightarrow \phi(0) \mathscr{E} \left[\widehat{n},\widehat{c},\widehat{u}\right](0)\quad \textrm{as}~~ j\rightarrow\infty,\quad \widehat{\mathbb{P}}\textrm{-a.s.}
$$
Taking the limit $j\rightarrow\infty$ in \eqref{5.58}, and using the Fatou Lemma, the lower-continuous of norm as well as the  Lemma \ref{lem5.1}, we get
\begin{equation}\label{5.62}
\begin{split}
&- \int_0^T \phi' (t)\mathscr{E} \left[\widehat{n},\widehat{c},\widehat{u}\right](t)\mathrm{d}t +  \int_0^T\phi(t) \mathscr{D}\left[\widehat{n},\widehat{c},\widehat{u}\right](t)\mathrm{d}t\leq \phi(0) \mathscr{E} \left[\widehat{n},\widehat{c},\widehat{u}\right](0)\\
 &\quad+  \int_0^T\phi(t)\| \widehat{u}(t)\|_{L^2}^2\mathrm{d}t +C  \int_0^T\phi(t)\mathrm{d} t+ \lim\limits_{j\rightarrow\infty} \int_0^T\phi(t)\mathrm{d}\mathcal {M}_E^{{\epsilon_j}}.
\end{split}
\end{equation}
Moreover, by applying the BDG inequality, the uniform bounds in Lemma \ref{lem5.1}, the Lemma \ref{lem5.5} and the similar arguments as those in Step 3, one can verify that
$$
\widehat{\mathbb{E}}\left[ \left| \int_0^T\phi(t)\mathrm{d}\mathcal {M}_E^{{\epsilon_j}}- \int_0^T\phi(t)\mathrm{d}\mathcal {M}_E \right|^2\right]\rightarrow 0\quad \textrm{as}~~ j\rightarrow\infty,
$$
which implies that, by extracting  a subsequence (still denoted by itself), there holds
\begin{equation}\label{5.63}
\begin{split}
\lim\limits_{j\rightarrow\infty} \int_0^T\phi(t)\mathrm{d}\mathcal {M}_E^{{\epsilon_j}}= \int_0^T\phi(t)\mathrm{d}\mathcal {M}_E,\quad \widehat{\mathbb{P}}\textrm{-a.s.}
\end{split}
\end{equation}
The desired energy-type inequality is a consequence of \eqref{5.62} and \eqref{5.63}. For the martingale $\mathcal {M}_E $, by applying the BDG inequality, the assumptions on $K$ and $G$ and the uniform bound in Lemma \ref{lem5.1}, we have
\begin{equation*}
\begin{split}
\widehat{\mathbb{E}}\sup_{t\in [0,T]}|\mathcal {M}_E|^p &\leq C\left[\mathbb{E}\left( \int_0^T |\langle \widehat{u}, \mathcal {P}g(s,\widehat{u}) |^2 \mathrm{d}t\right)^{\frac{p}{2}} + \mathbb{E}\left(  \int_0^T(1+\| \widehat{u}(s )\|_{L^2}^4)  \mathrm{d} s\right)^{\frac{p}{2}}\right.\\
 &\left.+  \widehat{\mathbb{E}}\left(  \int_0^T \int_{Z\backslash Z_0} |\langle \widehat{u}, \mathcal {P}G(\widehat{u}(s-),z)\rangle|^2 \mu( \mathrm{d} z)\mathrm{d} s\right)^{\frac{p}{2}} \right.\\
& \left.+  \widehat{\mathbb{E}}\left(  \int_0^T \int_{Z_0} |\langle \widehat{u}, \mathcal {P}K(\widehat{u}(s-),z)\rangle |^2 \mu( \mathrm{d} z)\mathrm{d} s\right)^{\frac{p}{2}}  \right]\\
&\leq C \widehat{\mathbb{E}}\left(\sup_{t\in [0,T]}\| \widehat{u}(t )\|_{L^2}^{2p} +1\right)\\
&\leq C\left(\widehat{\mathbb{E}}\left[ \int_\mathcal {O} \left(\widehat{n}(0)\ln \widehat{n}(0)+ \frac{1}{2}|\nabla \Psi(\widehat{c}(0))|^2+c^\dag |\widehat{u}(0)|^2\right)\mathrm{d} x\right]^p +1 \right),
\end{split}
\end{equation*}
which implies the inequality  \eqref{1.8}. The proof of Theorem \ref{thm} is now completed.
\end{proof}

\section*{Conflict of interest statement}

The authors declared that they have no conflicts of interest to this work.

\section*{Data availability}

No data was used for the research described in the article.

\section*{Acknowledgements}

The authors thank the anonymous referees for their constructive comments and suggestions which improved the quality of this article significantly. This work was partially supported by the National Natural Science Foundation of China (No. 12231008).

\bibliographystyle{plain}%
\bibliography{SCNS}

\begin{bibdiv}
\begin{biblist}

\bib{albritton2022non}{article}{
      author={Albritton, D.},
      author={Bru{\'e}, E.},
      author={Colombo, M.},
       title={Non-uniqueness of leray solutions of the forced navier-stokes
  equations},
        date={2022},
     journal={Annals of Mathematics},
      volume={196},
      number={1},
       pages={415\ndash 455},
}

\bib{aldous1978stopping}{article}{
      author={Aldous, D.},
       title={Stopping times and tightness},
        date={1978},
     journal={Annals of Probability},
       pages={335\ndash 340},
}

\bib{aldous1989stopping}{article}{
      author={Aldous, D.},
       title={Stopping times and tightness. ii},
        date={1989},
     journal={Annals of Probability},
       pages={586\ndash 595},
}

\bib{applebaum2009levy}{book}{
      author={Applebaum, D.},
       title={L{\'e}vy processes and stochastic calculus},
   publisher={Cambridge university press},
        date={2009},
}

\bib{arumugam2021keller}{article}{
      author={Arumugam, G.},
      author={Tyagi, J.},
       title={Keller-segel chemotaxis models: a review},
        date={2021},
     journal={Acta Applicandae Mathematicae},
      volume={171},
      number={1},
       pages={1\ndash 82},
}

\bib{bensoussan1995stochastic}{article}{
      author={Bensoussan, A.},
       title={Stochastic navier-stokes equations},
        date={1995},
     journal={Acta Applicandae Mathematica},
      volume={38},
      number={3},
       pages={267\ndash 304},
}

\bib{billingsley2013convergence}{book}{
      author={Billingsley, P.},
       title={Convergence of probability measures},
   publisher={John Wiley \& Sons},
        date={2013},
}

\bib{brzezniak20132d}{article}{
      author={Brze{\'z}niak, Z.},
      author={Hausenblas, E.},
      author={Zhu, J.},
       title={2d stochastic navier--stokes equations driven by jump noise},
        date={2013},
     journal={Nonlinear Analysis: Theory, Methods \& Applications},
      volume={79},
       pages={122\ndash 139},
}

\bib{brzezniak2019weak}{article}{
      author={Brze{\'z}niak, Z.},
      author={Manna, U.},
       title={Weak solutions of a stochastic landau--lifshitz--gilbert equation
  driven by pure jump noise},
        date={2019},
     journal={Communications in Mathematical Physics},
      volume={371},
      number={3},
       pages={1071\ndash 1129},
}

\bib{Brzezniak2017Maximal}{article}{
      author={Brze\'{z}niak, Z.},
      author={Zhu, J.},
      author={Hausenblas, E.},
       title={Maximal inequalities for stochastic convolutions driven by
  compensated poisson random measures in banach spaces},
        date={2017},
     journal={Annales de L'Institut Henri Poincare. Probabilit\'{e}s et
  statistiques},
      volume={53},
      number={3},
       pages={937\ndash 956},
}

\bib{buckmaster2019nonuniqueness}{article}{
      author={Buckmaster, T.},
      author={Vicol, V.},
       title={Nonuniqueness of weak solutions to the navier-stokes equation},
        date={2019},
     journal={Annals of Mathematics},
      volume={189},
      number={1},
       pages={101\ndash 144},
}

\bib{chae2014global}{article}{
      author={Chae, M.},
      author={Kang, K.},
      author={Lee, J.},
       title={Global existence and temporal decay in keller-segel models
  coupled to fluid equations},
        date={2014},
     journal={Communications in Partial Differential Equations},
      volume={39},
      number={7},
       pages={1205\ndash 1235},
}

\bib{chen2019martingale}{article}{
      author={Chen, R.~M.},
      author={Wang, D.},
      author={Wang, H.},
       title={Martingale solutions for the three-dimensional stochastic
  nonhomogeneous incompressible navier--stokes equations driven by l{\'e}vy
  processes},
        date={2019},
     journal={Journal of Functional Analysis},
      volume={276},
      number={7},
       pages={2007\ndash 2051},
}

\bib{chen2022sharp}{article}{
      author={Chen, W.},
      author={Dong, Z.},
      author={Zhu, X.},
       title={Sharp non-uniqueness of solutions to stochastic navier-stokes
  equations},
        date={2022},
     journal={arXiv preprint arXiv:2208.08321},
}

\bib{cyr2020review}{article}{
      author={Cyr, J.},
      author={Nguyen, P.},
      author={Tang, S.},
      author={Temam, R.},
       title={Review of local and global existence results for stochastic pdes
  with l{\'e}vy noise},
        date={2020},
     journal={Discrete \& Continuous Dynamical Systems},
      volume={40},
      number={10},
       pages={5639},
}

\bib{cyr2018euler}{article}{
      author={Cyr, J.},
      author={Tang, S.},
      author={Temam, R.},
       title={The euler equations of an inviscid incompressible fluid driven by
  a l{\'e}vy noise},
        date={2018},
     journal={Nonlinear Analysis: Real World Applications},
      volume={44},
       pages={173\ndash 222},
}

\bib{de2013dissipative}{article}{
      author={De~Lellis, C.},
      author={Sz{\'e}kelyhidi, L.},
       title={Dissipative continuous euler flows},
        date={2013},
     journal={Inventiones Mathematicae},
      volume={193},
      number={2},
       pages={377\ndash 407},
}

\bib{de2014dissipative}{article}{
      author={De~Lellis, C.},
      author={Sz{\'e}kelyhidi~Jr, L.},
       title={Dissipative euler flows and onsager's conjecture},
        date={2014},
     journal={Journal of the European Mathematical Society},
      volume={16},
      number={7},
       pages={1467\ndash 1505},
}

\bib{dombrowski2004self}{article}{
      author={Dombrowski, C.},
      author={Cisneros, L.},
      author={Chatkaew, S.},
      author={Goldstein, R.},
      author={Kessler, J.},
       title={Self-concentration and large-scale coherence in bacterial
  dynamics},
        date={2004},
     journal={Physical Review Letters},
      volume={93},
      number={9},
       pages={098103},
}

\bib{dong2011ergodicity}{article}{
      author={Dong, Z.},
      author={Xie, Y.},
       title={Ergodicity of stochastic 2d navier--stokes equation with l{\'e}vy
  noise},
        date={2011},
     journal={Journal of Differential Equations},
      volume={251},
      number={1},
       pages={196\ndash 222},
}

\bib{dong2011martingale}{article}{
      author={Dong, Z.},
      author={Zhai, J.},
       title={Martingale solutions and markov selection of stochastic 3d
  navier--stokes equations with jump},
        date={2011},
     journal={Journal of Differential Equations},
      volume={250},
      number={6},
       pages={2737\ndash 2778},
}

\bib{duan2010global}{article}{
      author={Duan, R.},
      author={Lorz, A.},
      author={Markowich, P.},
       title={Global solutions to the coupled chemotaxis-fluid equations},
        date={2010},
     journal={Communications in Partial Differential Equations},
      volume={35},
      number={9},
       pages={1635\ndash 1673},
}

\bib{flandoli1995martingale}{article}{
      author={Flandoli, F.},
      author={Gatarek, D.},
       title={Martingale and stationary solutions for stochastic navier-stokes
  equations},
        date={1995},
     journal={Probability Theory and Related Fields},
      volume={102},
      number={3},
       pages={367\ndash 391},
}

\bib{fujikawa1989fractal}{article}{
      author={Fujikawa, H.},
      author={Matsushita, M.},
       title={Fractal growth of bacillus subtilis on agar plates},
        date={1989},
     journal={Journal of the Physical Society of Japan},
      volume={58},
      number={11},
       pages={3875\ndash 3878},
}

\bib{galdi2011introduction}{book}{
      author={Galdi, G.},
       title={An introduction to the mathematical theory of the navier-stokes
  equations: Steady-state problems},
   publisher={Springer Science \& Business Media},
        date={2011},
}

\bib{giga1986solutions}{article}{
      author={Giga, Y.},
       title={Solutions for semilinear parabolic equations in lp and regularity
  of weak solutions of the navier-stokes system},
        date={1986},
     journal={Journal of Differential Equations},
      volume={62},
      number={2},
       pages={186\ndash 212},
}

\bib{giga1991abstract}{article}{
      author={Giga, Y.},
      author={Sohr, H.},
       title={Abstract $l^p$ estimates for the cauchy problem with applications
  to the navier-stokes equations in exterior domains},
        date={1991},
     journal={Journal of Functional Analysis},
      volume={102},
      number={1},
       pages={72\ndash 94},
}

\bib{gyongy1980stochastic}{article}{
      author={Gy{\"o}ngy, I.},
      author={Krylov, N.V.},
       title={On stochastic equations with respect to semimartingales i.},
        date={1980},
     journal={Stochastics: An International Journal of Probability and
  Stochastic Processes},
      volume={4},
      number={1},
       pages={1\ndash 21},
}

\bib{hausenblas2023existence}{article}{
      author={Hausenblas, E.},
      author={Moghomye, B.J.},
      author={Razafimandimby, P.A.},
       title={On the existence and uniqueness of solution to a stochastic
  chemotaxis-navier-stokes model},
        date={2023},
     journal={arXiv preprint arXiv:2301.00654},
}

\bib{hausenblas2022one}{article}{
      author={Hausenblas, E.},
      author={Mukherjee, D.},
      author={Tran, T.},
       title={The one-dimensional stochastic keller--segel model with
  time-homogeneous spatial wiener processes},
        date={2022},
     journal={Journal of Differential Equations},
      volume={310},
       pages={506\ndash 554},
}

\bib{hausenblas2020existence}{article}{
      author={Hausenblas, E.},
      author={Razafimandimby, P.},
       title={Existence of a density of the 2-dimensional stochastic navier
  stokes equation driven by l{\'e}vy processes or fractional brownian motion},
        date={2020},
     journal={Stochastic Processes and their Applications},
      volume={130},
      number={7},
       pages={4174\ndash 4205},
}

\bib{hillen2009user}{article}{
      author={Hillen, T.},
      author={Painter, K.},
       title={A user's guide to pde models for chemotaxis},
        date={2009},
     journal={Journal of Mathematical Biology},
      volume={58},
      number={1},
       pages={183\ndash 217},
}

\bib{hofmanova2021ill}{article}{
      author={Hofmanov{\'a}, M.},
      author={Zhu, R.},
      author={Zhu, X.},
       title={On ill-and well-posedness of dissipative martingale solutions to
  stochastic 3d euler equations},
        date={2022},
     journal={Communications on Pure and Applied Mathematics},
      volume={75},
      number={11},
       pages={2446\ndash 2510},
}

\bib{hofmanova2021global}{article}{
      author={Hofmanov{\'a}, M.},
      author={Zhu, R.},
      author={Zhu, X.},
       title={Global-in-time probabilistically strong and markov solutions to
  stochastic 3d navier--stokes equations: existence and non-uniqueness},
        date={2023},
     journal={The Annals of Probability},
      volume={51},
      number={2},
       pages={524\ndash 579},
}

\bib{hofmanova2019non}{article}{
      author={Hofmanov{\'a}, M.},
      author={Zhu, R.},
      author={Zhu, X.},
       title={Non-uniqueness in law of stochastic 3d navier--stokes equations},
        date={2023},
     journal={Journal of the European Mathematical Society},
}

\bib{huang2021microscopic}{article}{
      author={Huang, H.},
      author={Qiu, J.},
       title={The microscopic derivation and well-posedness of the stochastic
  keller--segel equation},
        date={2021},
     journal={Journal of Nonlinear Science},
      volume={31},
      number={1},
       pages={1\ndash 31},
}

\bib{ikea1966construction}{article}{
      author={Ikea, N.},
      author={Nagasawa, M.},
      author={Watanabe, S.},
       title={A construction of markov processes by piecing out},
        date={1966},
     journal={Proceedings of the Japan Academy},
      volume={42},
      number={4},
       pages={370\ndash 375},
}

\bib{jakubowski1998almost}{article}{
      author={Jakubowski, A.},
       title={The almost sure skorokhod representation for subsequences in
  nonmetric spaces},
        date={1998},
     journal={Theory of Probability \& Its Applications},
      volume={42},
      number={1},
       pages={167\ndash 174},
}

\bib{jeong2022well}{article}{
      author={Jeong, I.-J.},
      author={Kang, K.},
       title={Well-posedness and singularity formation for inviscid
  keller--segel--fluid system of consumption type},
        date={2022},
     journal={Communications in Mathematical Physics},
      volume={390},
      number={3},
       pages={1175\ndash 1217},
}

\bib{kallenberg1997foundations}{book}{
      author={Kallenberg, O.},
       title={Foundations of modern probability},
   publisher={Springer},
      volume={2},
}

\bib{kang2017existence}{article}{
      author={Kang, K.},
      author={Kim, H.K.},
       title={Existence of weak solutions in wasserstein space for a chemotaxis
  model coupled to fluid equations},
        date={2017},
     journal={SIAM Journal on Mathematical Analysis},
      volume={49},
      number={4},
       pages={2965\ndash 3004},
}

\bib{KR79}{book}{
      author={Krylov, N.~V.},
      author={Rozowski\u{I}, B.~L.},
       title={Stochastic evolution equations,},
   publisher={Current problems in mathematics, Vol. 14 (Russian), Akad. Nauk
  SSSR, Vsesoyuz. Inst. Nauchn. i Tekhn. Informatsii, Moscow,},
}

\bib{liu2011coupled}{article}{
      author={Liu, J.-G.},
      author={Lorz, A.},
       title={A coupled chemotaxis-fluid model: global existence},
        date={2011},
     journal={Annales de l'Institut Henri Poincar{\'e} C, Analyse non
  lin{\'e}aire},
      volume={28},
      number={5},
       pages={643\ndash 652},
}

\bib{lorz2010coupled}{article}{
      author={Lorz, A.},
       title={Coupled chemotaxis fluid model},
        date={2010},
     journal={Mathematical Models and Methods in Applied Sciences},
      volume={20},
      number={06},
       pages={987\ndash 1004},
}

\bib{manna2021well}{article}{
      author={Manna, U.},
      author={Panda, A.A.},
       title={Well-posedness and large deviations for 2d stochastic constrained
  navier-stokes equations driven by l{\'e}vy noise in the marcus canonical
  form},
        date={2021},
     journal={Journal of Differential Equations},
      volume={302},
       pages={64\ndash 138},
}

\bib{martini2022additive}{article}{
      author={Martini, A.},
      author={Mayorcas, A.},
       title={An additive noise approximation to keller-segel-dean-kawasaki
  dynamics part i: Local well-posedness of paracontrolled solutions},
        date={2022},
     journal={arXiv preprint arXiv:2207.10711},
}

\bib{mayorcas2021blow}{article}{
      author={Mayorcas, A.},
      author={Tomasevic, M.},
       title={Blow-up for a stochastic models of chemotaxis driven by
  conservative noise on $\mathbb{R}^2$},
        date={2021},
     journal={arXiv preprint arXiv:2111.02245},
}

\bib{misiats2022global}{article}{
      author={Misiats, O.},
      author={Stanzhytskyi, O.},
      author={Topaloglu, I.},
       title={On global existence and blowup of solutions of stochastic
  keller--segel type equation},
        date={2022},
     journal={Nonlinear Differential Equations and Applications NoDEA},
      volume={29},
      number={1},
       pages={1\ndash 29},
}

\bib{mizoguchi2014nondegeneracy}{article}{
      author={Mizoguchi, N.},
      author={Souplet, P.},
       title={Nondegeneracy of blow-up points for the parabolic keller--segel
  system},
        date={2014},
     journal={Annales de l'Institut Henri Poincar{\'e} C, Analyse non
  lin{\'e}aire},
      volume={31},
      number={4},
       pages={851\ndash 875},
}

\bib{nguyen2021nonlinear}{article}{
      author={Nguyen, P.},
      author={Tawri, K.},
      author={Temam, R.},
       title={Nonlinear stochastic parabolic partial differential equations
  with a monotone operator of the ladyzenskaya-smagorinsky type, driven by a
  l{\'e}vy noise},
        date={2021},
     journal={Journal of Functional Analysis},
      volume={281},
      number={8},
       pages={109157},
}

\bib{nirenberg1959}{article}{
      author={Nirenberg, L.},
       title={On elliptic partial differential equations},
        date={1959},
     journal={Annali della Scuola Normale Superiore di Pisa, Classe di
  Scienze},
      volume={3},
      number={13},
       pages={115\ndash 162},
}

\bib{peszat2007stochastic}{book}{
      author={Peszat, S.},
      author={Zabczyk, J.},
       title={Stochastic partial differential equations with l{\'e}vy noise: An
  evolution equation approach},
   publisher={Cambridge University Press},
        date={2007},
      volume={113},
}

\bib{prevot2007concise}{book}{
      author={Pr{\'e}v{\^o}t, C.},
      author={R{\"o}ckner, M.},
       title={A concise course on stochastic partial differential equations},
   publisher={Springer},
        date={2007},
      volume={1905},
}

\bib{Revuz1999}{book}{
      author={Revuz, D.},
      author={Yor, M.},
       title={Continuous martingales and brownian motion},
   publisher={Springer Science \& Business Media},
        date={2013},
      volume={293},
}

\bib{sakthivel2012martingale}{article}{
      author={Sakthivel, K.},
      author={Sritharan, S.},
       title={Martingale solutions for stochastic navier-stokes equations
  driven by l{\'e}vy noise},
        date={2012},
     journal={Evolution Equations \& Control Theory},
      volume={1},
      number={2},
       pages={355},
}

\bib{shang2019asymptotic}{article}{
      author={Shang, Y.},
      author={Tian, J.P.},
      author={Wang, B.},
       title={Asymptotic behavior of the stochastic keller-segel equations},
        date={2019},
     journal={Discrete and Continuous Dynamical Systems-B},
      volume={24},
      number={3},
       pages={1367},
}

\bib{skorokhod1961existence}{article}{
      author={Skorokhod, A.V.},
       title={On the existence and uniqueness of solutions of stochastic
  differential equations},
        date={1961},
     journal={Sibirskii Matematicheskii Zhurnal},
      volume={2},
      number={1},
       pages={129\ndash 137},
}

\bib{tan2014decay}{article}{
      author={Tan, Z.},
      author={Zhang, X.},
       title={Decay estimates of the coupled chemotaxis--fluid equations in
  $\mathbb{R}^3$},
        date={2014},
     journal={Journal of Mathematical Analysis and Applications},
      volume={410},
      number={1},
       pages={27\ndash 38},
}

\bib{tuval2005}{article}{
      author={Tuval, I.},
      author={Cisneros, L.},
      author={Dombrowski, C.},
      author={Wolgemuth, C.},
      author={Kessler, J.},
      author={Goldstein, R.},
       title={Bacterial swimming and oxygen transport near contact lines},
        date={2005},
     journal={Proceedings of the National Academy of Sciences},
      volume={102},
      number={7},
       pages={2277\ndash 2282},
}

\bib{wang2021local}{article}{
      author={Wang, Y.},
      author={Winkler, M.},
      author={Xiang, Z.},
       title={Local energy estimates and global solvability in a
  three-dimensional chemotaxis-fluid system with prescribed signal on the
  boundary},
        date={2021},
     journal={Communications in Partial Differential Equations},
      volume={46},
      number={6},
       pages={1058\ndash 1091},
}

\bib{winkler2010aggregation}{article}{
      author={Winkler, M.},
       title={Aggregation vs. global diffusive behavior in the
  higher-dimensional keller--segel model},
        date={2010},
     journal={Journal of Differential Equations},
      volume={248},
      number={12},
       pages={2889\ndash 2905},
}

\bib{winkler2012global}{article}{
      author={Winkler, M.},
       title={Global large-data solutions in a chemotaxis-(navier--) stokes
  system modeling cellular swimming in fluid drops},
        date={2012},
     journal={Communications in Partial Differential Equations},
      volume={37},
      number={2},
       pages={319\ndash 351},
}

\bib{winkler2014stabilization}{article}{
      author={Winkler, M.},
       title={Stabilization in a two-dimensional chemotaxis-navier--stokes
  system},
        date={2014},
     journal={Archive for Rational Mechanics and Analysis},
      volume={211},
      number={2},
       pages={455\ndash 487},
}

\bib{winkler2015boundedness}{article}{
      author={Winkler, M.},
       title={Boundedness and large time behavior in a three-dimensional
  chemotaxis-stokes system with nonlinear diffusion and general sensitivity},
        date={2015},
     journal={Calculus of Variations and Partial Differential Equations},
      volume={54},
      number={4},
       pages={3789\ndash 3828},
}

\bib{winkler2016global}{article}{
      author={Winkler, M.},
       title={Global weak solutions in a three-dimensional
  chemotaxis--navier--stokes system},
organization={Elsevier},
        date={2016},
     journal={Annales de l'Institut Henri Poincar{\'e} C, Analyse non
  lin{\'e}aire},
      volume={33},
      number={5},
       pages={1329\ndash 1352},
}

\bib{winkler2017far}{article}{
      author={Winkler, M.},
       title={How far do chemotaxis-driven forces influence regularity in the
  navier-stokes system?},
        date={2017},
     journal={Transactions of the American Mathematical Society},
      volume={369},
      number={5},
       pages={3067\ndash 3125},
}

\bib{winkler2022does}{article}{
      author={Winkler, M.},
       title={Does leray's structure theorem withstand buoyancy-driven
  chemotaxis-fluid interaction?},
        date={2022},
     journal={Journal of the European Mathematical Society},
}

\bib{winkler2022reaction}{article}{
      author={Winkler, M.},
       title={Reaction-driven relaxation in three-dimensional
  keller--segel--navier--stokes interaction},
        date={2022},
     journal={Communications in Mathematical Physics},
      volume={389},
      number={1},
       pages={439\ndash 489},
}

\bib{zhai2015large}{article}{
      author={Zhai, J.},
      author={Zhang, T.},
       title={Large deviations for 2-d stochastic navier--stokes equations
  driven by multiplicative l{\'e}vy noises},
        date={2015},
     journal={Bernoulli},
      volume={21},
      number={4},
       pages={2351\ndash 2392},
}

\bib{zhai20202d}{article}{
      author={Zhai, J.},
      author={Zhang, T.},
       title={2d stochastic chemotaxis-navier-stokes system},
        date={2020},
     journal={Journal de Math{\'e}matiques Pures et Appliqu{\'e}es},
      volume={138},
       pages={307\ndash 355},
}

\bib{2010Stochastic}{article}{
      author={Zhang, X.},
       title={Stochastic volterra equations in banach spaces and stochastic
  partial differential equation},
        date={2010},
     journal={Journal of Functional Analysis},
      volume={258},
      number={4},
}

\bib{Brzezniak2019Maximal}{article}{
      author={Zhu, J.},
      author={Brze\'{z}niak, Z.},
       title={Maximal inequalities and exponential estimates for stochastic
  convolutions driven by l\'{e}vy-type processes in banach spaces with
  application to stochastic quasi-geostrophic equation},
        date={2019},
     journal={SIAM Journal on Mathematical Analysis},
      volume={51},
      number={3},
       pages={2121\ndash 2167},
}

\end{biblist}
\end{bibdiv}

\end{document}